\documentclass[11pt,reqno]{amsart}

\usepackage{etex}
\usepackage{amsmath,amssymb,amsthm,amsfonts,mathrsfs,dsfont,mathtools}
\usepackage[frame,cmtip,arrow,matrix,line,graph,curve]{xy}
\usepackage{graphpap,color,paralist,pstricks}
\usepackage[mathscr]{eucal}
\usepackage{mathabx}
\usepackage[pdftex,colorlinks,backref=page,citecolor=blue,linkcolor=blue]{hyperref}
\usepackage[latin1]{inputenc}
\usepackage{tikz}
\usepackage{tikz-cd}
\usetikzlibrary{fit,matrix}
\usepackage{epic,eepic}
\usepackage{yfonts}
\usepackage{enumerate}
\usepackage{caption}
\usepackage[alphabetic]{amsrefs}
\usepackage{extpfeil}
\usepackage[normalem]{ulem}

\usepackage{todonotes}

\theoremstyle{definition}
\newtheorem*{theorem*}{Theorem}
\newtheorem*{conditionA}{Condition A}
\newtheorem*{conditionB}{Condition B}
\newtheorem*{theoremA}{Theorem A}
\newtheorem{theorem}{Theorem}[section]
\newtheorem{definition}[theorem]{Definition}

\newtheorem{lemma}[theorem]{Lemma}
\newtheorem{proposition}[theorem]{Proposition}
\newtheorem{corollary}[theorem]{Corollary}
\theoremstyle{remark}
\newtheorem{remark}[theorem]{Remark}
\newtheorem{example}[theorem]{Example}
\newtheorem{convention}[theorem]{Convention}

\let\Horig\H

\renewcommand*{\o}{\circ}
\newcommand*{\A}{i}
\newcommand*{\ua}{\underline{\alpha}}
\newcommand*{\ub}{\underline{\smash{\beta}}}
\newcommand{\excise}[1]{}
\newcommand*{\udeg}{\underline{\smash{\deg}}}
\newcommand*{\uf}{\underline{f}}

\newcommand{\id}{\operatorname{id}}

\newcommand{\C}{{\mathrm{C}}}
\newcommand{\Z}{{\mathbb{Z}}}
\newcommand{\Q}{{\mathbb{Q}}}
\renewcommand{\H}{\operatorname{H}}
\newcommand{\HH}{\mathcal{H}}
\newcommand{\IH}{\operatorname{IH}}

\newcommand{\N}{\mathrm{N}}

\renewcommand{\a}{\alpha}
\renewcommand{\b}{\ub}

\newcommand{\crk}{\operatorname{crk}}

\newcommand{\mmod}{\text{-}\mathrm{mod}}

\newcommand{\OS}{\operatorname{OS}}

\newcommand{\VRep}{\operatorname{VRep}}
\newcommand{\Ind}{\operatorname{Ind}}

\newcommand{\M}{\mathrm{M}}
\newcommand{\Mo}{{}}              

\newcommand{\rk}{\operatorname{rk}}
\newcommand{\IC}{\operatorname{IC}}

\newcommand{\cL}{\mathscr{L}}
\newcommand{\cE}{\mathcal{E}}
\newcommand{\Mi}{\M\setminus i}
\newcommand{\Mip}{(\M\setminus i)}
\newcommand{\CHi}{\CH_{(i)}}
\newcommand{\tIH}{\widetilde\IH}
\newcommand{\tuJ}{\widetilde\uJ}
\renewcommand{\and}{\;\;\text{and}\;\;}
\newcommand{\sS}{\mathscr{S}}

\newcommand{\Hom}{\operatorname{Hom}}
\newcommand{\Ext}{\operatorname{Ext}}

\newcommand{\Ann}{\operatorname{Ann}}

\newcommand{\mm}{\mathfrak{m}}

\newcommand{\uJ}{\underline{\operatorname{J}}}

\newcommand{\uH}{\underline{\operatorname{H}}}

\newcommand{\CD}{\textsf{CD}}
\newcommand{\uCD}{\underline{\CD}}
\newcommand{\PD}{\textsf{PD}}
\newcommand{\uPD}{\underline{\PD}}
\newcommand{\HL}{\textsf{HL}}
\newcommand{\uHL}{\underline{\HL}}
\newcommand{\HR}{\textsf{HR}}
\newcommand{\uHR}{\underline{\HR}}
\newcommand{\NS}{\textsf{NS}}
\newcommand{\uNS}{\underline{\NS}}
\newcommand{\CH}{\operatorname{CH}}
\newcommand{\uCH}{\underline{\CH}}
\newcommand{\uIH}{\underline{\IH}}

\newcommand{\upsi}{\underline{\smash{\psi}}}
\newcommand{\uvarphi}{\underline{\smash{\varphi}}}
\newcommand{\uL}{\underline{\operatorname{L}}}
\newcommand{\K}[2]{\operatorname{K}_{#1}(#2)}
\newcommand{\uK}[2]{{\underline{\operatorname{K}}_{#1}(#2)}}
\newcommand{\oL}{\operatorname{L}}

\newcommand{\sig}{\operatorname{sig}}

\DeclareRobustCommand{\uIHi}{\underline{\IH}_i(\M)}

\begin{document}

\title{Singular Hodge theory for combinatorial geometries}

\author{Tom Braden}
\address{Department of Mathematics and Statistics, University of Massachusetts, Amherst, MA.}
\email{braden@umass.edu}

\author{June Huh}
\address{Department of Mathematics, Princeton University, Princeton, NJ, and Korea Institute for Advanced Study, Seoul, Korea.}
\email{huh@princeton.edu}

\author{Jacob P. Matherne}
\address{Department of Mathematics, North Carolina State University, Raleigh, NC.}
\email{jpmather@ncsu.edu}

\author{Nicholas Proudfoot}
\address{Department of Mathematics, University of Oregon, Eugene, OR.}
\email{njp@uoregon.edu}

\author{Botong Wang}
\address{Department of Mathematics, University of Wisconsin-Madison, Madison, WI.}
\email{wang@math.wisc.edu}

\thanks{June Huh received support from NSF Grant DMS-1638352 and the Ellentuck Fund.  Jacob Matherne received support from NSF Grants DMS-1638352 and DMS-2452179, the Association of Members of the Institute for Advanced Study, the Max Planck Institute for Mathematics, the Deutsche Forschungsgemeinschaft (DFG) under Germany's Excellence Strategy - GZ 2047/1, Projekt-ID 390685813, and Simons Foundation Travel Support for Mathematicians Award MPS-TSM00007970.  Nicholas Proudfoot received support from NSF Grants DMS-1565036, DMS-1954050, and DMS-2039316.  Botong Wang received support from NSF Grant DMS-1701305 and the Alfred P. Sloan Foundation.}

\begin{abstract}
We introduce the intersection cohomology module of a matroid and prove that it satisfies Poincar\'e duality, the hard Lefschetz theorem, and the Hodge--Riemann relations.
As applications, we obtain proofs of Dowling and Wilson's Top-Heavy conjecture
and  the nonnegativity of the coefficients of  Kazhdan--Lusztig polynomials for all matroids.
\end{abstract}

\maketitle
\setcounter{tocdepth}{1}
\tableofcontents

\section{Introduction}\label{Introduction}

\subsection{Results}\label{sec:results}
A {\bf matroid} $\mathrm{M}$ on a finite set $E$ is a nonempty collection of subsets of $E$, called {\bf flats} of $\mathrm{M}$,
that satisfies the following properties:
\begin{enumerate}[$\bullet$]\itemsep 5pt
\item If $F_1$ and $F_2$ are flats, then their intersection $F_1 \cap F_2$ is a flat.
\item If $F$ is a flat, then any element in $E \setminus F$ is in exactly one flat  that is minimal among the flats strictly containing $F$.
\end{enumerate}
For notational convenience, we assume throughout that $\mathrm{M}$ is {\bf loopless}:
\begin{enumerate}[$\bullet$]\itemsep 5pt
\item The empty subset of $E$ is a flat.
\end{enumerate}
We write $\mathscr{L}(\mathrm{M})$ for the poset of flats of $\mathrm{M}$, which is a geometric lattice \cite[Theorem 1.7.5]{Oxley}, and we write $\leq$ for containment of flats.
For every flat $F$, any maximal flag of flats strictly contained in $F$ has the same cardinality, which is called the {\bf rank} of $F$
and denoted $\rk F$. The {\bf rank} of $\M$ is defined to be $\rk E$. 
For any nonnegative integer $k$, we write $\cL^k(\M)$ to denote the set of rank $k$ flats of $\M$.
A matroid can be equivalently defined in terms of its {\bf independent sets}, {\bf circuits}, or the {\bf rank function}.
For background in matroid theory, we refer to \cite{Oxley} and \cite{Welsh}.  

Let $\Gamma$ be a finite group acting on $\M$.  By definition, 
 $\Gamma$ permutes the elements of $E$ in such a way that it sends flats to flats.  

\begin{theorem}\label{thm:top-heavy}
The following holds for any $k\leq j \leq \rk \M-k$.
\begin{enumerate}[(1)]\itemsep 5pt
\item The cardinality of  $\cL^k(\M)$ is at most the cardinality of $\cL^{j}(\M)$. 
\item There is an injective map $\iota\colon \cL^k(\M)\to\cL^j(\M)$ satisfying $F\leq \iota(F)$ for all $F\in\cL^k(\M)$.
\item There is an injective map 
$\Q\cL^k(\M)\to\Q\cL^j(\M)$
of permutation representations of $\Gamma$.\footnote{
One might hope to combine the last two parts of Theorem \ref{thm:top-heavy} by asking the map $\iota$ to be $\Gamma$-equivariant,
but this is not possible, even if we drop the condition that $F\leq \iota(F)$.  
For example, when $\M$ is the uniform matroid of rank $3$ on $4$ elements,
 there is no $S_4$-equivariant map from $\cL^1(\M)$ to $\cL^2(\M)$.  
}
\end{enumerate}
\end{theorem}

The first two parts of Theorem \ref{thm:top-heavy} were conjectured by Dowling and Wilson \cite{DW1,DW2}, and have come to be known as
the {\bf Top-Heavy conjecture}.  
Its best known instance is the de Bruijn--Erd\Horig{o}s theorem on point-line incidences in projective planes \cite{dBE}: 
\begin{quote}
\emph{Let $E$ be a finite set of points in a projective plane that is not contained in any line. Then there are at least $|E|$ distinct lines that intersect $E$ in at least two points.}
\end{quote}
The second statement of Theorem~\ref{thm:top-heavy} implies that the lattice $\cL=\cL(\M)$ admits order-matchings
\[
\cL^0 \rightarrow \cL^1 \rightarrow \cdots \rightarrow \cL^{\lfloor\frac{\rk \M}{2} \rfloor}   \leftrightarrow \cL^{\lceil\frac{\rk \M}{2}\rceil}  \leftarrow \cdots \leftarrow \cL^{\rk\M-1} \leftarrow \cL^{\rk \M},
\]
and hence has the Sperner property \cite[Chapter 4]{StanleyAC}.  

When $\cL$ is a Boolean lattice or a projective geometry, Theorem~\ref{thm:top-heavy} is a classical result; see for example \cite[Corollary 4.8 and Exercise 4.4]{StanleyAC}.  In these cases, the Sperner property has the following interpretation:
\begin{quote}
\emph{The maximal number of pairwise incomparable subsets of $[n]$ is the maximum among the  binomial coefficients ${n \choose k}$. 
Similarly, the maximal number of pairwise incomparable subspaces of $\mathbb{F}_q^n$ is the maximum among the  $q$-binomial coefficients ${n \choose k }_q$.}
\end{quote}
Other earlier versions of Theorem \ref{thm:top-heavy}, for specific classes of matroids or small values of $k$,
can be found in \cite{Motzkin, BK, Greene, Mason, Heron,  KungRadonI, KungRadon, KungRadonII, KungLinesPlanes}.
In \cite{HW}, Theorem \ref{thm:top-heavy} was proved
for matroids realizable over some field.
See Section \ref{sec:realizable} for an overview of that proof.
Although realizable matroids provide the primary
motivation for the definition of a matroid, almost all matroids are not realizable over any field.  More precisely, the portion of matroids on the ground
set $[n]$ that are realizable over some field goes to zero as $n$ goes to infinity \cite{Nelson}.

Our proof of Theorem \ref{thm:top-heavy} is closely related to Kazhdan--Lusztig theory of matroids, as developed in \cite{EPW}.
For any flat $F$ of $\M$, we define the {\bf localization} of $\M$ at $F$ to be the matroid $\M^F$
on the ground set $F$ whose flats are the flats of $\M$ contained in $F$. 
Similarly,
 we define the {\bf contraction} of $\M$ at $F$
to be the matroid $\M_F$ on the ground set $E\setminus F$ whose flats are $G \setminus F$ for flats $G$ of $\M$ containing $F$.\footnote{In 
\cite{EPW}, as well as several other references on Kazhdan--Lusztig polynomials of matroids, the localization is denoted $\M_F$
and the contraction is denoted $\M^F$.  Our notational choice here is consistent with \cite{AHK} and \cite{BHMPW}.}
We also consider the {\bf characteristic polynomial}
\[
\chi_{\M}(t) \coloneq \sum_{I\subseteq E} (-1)^{|I|} t^{\crk I},
\]
where $\crk I$ is the corank of $I$ in $\M$.
According to \cite[Theorem 2.2]{EPW},
there is a unique way to assign a polynomial $P_{\M}(t)$ to each matroid $\M$, called the {\bf Kazhdan--Lusztig polynomial} of $\M$,  subject to the following  three conditions:
\begin{enumerate}[(a)]\itemsep 5pt
\item If the ground set is empty, then $P_{\M}(t)$ is the constant polynomial $1$.
\item For every matroid $\M$ on a nonempty ground set, the degree of  $P_\M(t)$ is strictly less than  $\rk \M / 2$.
\item For every matroid $\M$, we have $t^{\rk \M} P_\M(t^{-1}) = \displaystyle\sum_{F\in \cL(\M)} \chi_{\M^F}(t) \cdot P_{\M_F}(t)$.
\end{enumerate}
Alternatively \cite[Theorem 2.2]{BV}, one may define Kazhdan--Lusztig polynomials of matroids by replacing the third condition above with the following condition not involving $\chi_\M(t)$: 
\begin{enumerate}[(a)]\itemsep 5pt
\item[(c)'] For every matroid $\M$, the polynomial $Z_\M(t)\coloneq \displaystyle\sum_{F \in \cL(\M)} t^{\rk F} P_{\M_F}(t)$ satisfies the identity
\[
t^{\rk \M} Z_{\M}(t^{-1})=Z_{\M}(t).
\]
\end{enumerate}
The polynomial $Z_{\M}(t)$, called the {\bf \boldmath{$Z$}-polynomial} of $\M$, was introduced in \cite{PXY} using the first definition of $P_\M(t)$,
where it was shown to satisfy the displayed identity.
The degree of the $Z$-polynomial of $\M$ is exactly the rank of $\M$, and its leading coefficient is $1$.

\begin{theorem}\label{thm:KL}
The following holds for any matroid $\M$.
\begin{enumerate}[(1)]\itemsep 5pt
\item The polynomial $P_{\M}(t)$ 
has nonnegative coefficients.
\item The polynomial $Z_{\M}(t)$ is unimodal:  
The  coefficient of $t^k$  in  $Z_{\M}(t)$ is less than or equal to the coefficient of $t^{k+1}$ in  $Z_{\M}(t)$  for  all $k< \rk \M / 2$.
\end{enumerate}
\end{theorem}

The first part of Theorem \ref{thm:KL} was  conjectured in \cite[Conjecture 2.3]{EPW},
where it was proved for matroids realizable over some field
using $l$-adic \'etale intersection cohomology theory.
See Section \ref{sec:realizable} for an overview of that proof.
For sparse paving matroids, a combinatorial proof of the nonnegativity  was given in  \cite{SparsePaving}, and a proof for all matroids, independent of the results of this paper, recently appeared in \cite{coron2024operadic}.  


Kazhdan--Lusztig polynomials of matroids are special cases of Kazhdan--Lusztig--Stanley polynomials \cite{subdivisions,KLS}.
Several important
families of Kazhdan--Lusztig--Stanley polynomials turn out to have nonnegative coefficients, including classical Kazhdan--Lusztig polynomials associated with Bruhat intervals \cite{EW}
and $g$-polynomials of convex polytopes \cites{Karu,BrLu05}.  For more on this analogy, see Section \ref{sec:antecedents}.
 
For a finite group $\Gamma$ acting on $\M$, one can define the {\bf equivariant Kazhdan--Lusztig polynomial} $P_{\M}^{\Gamma}(t)$ and the {\bf equivariant \boldmath{$Z$}-polynomial} $Z_{\M}^{\Gamma}(t)$; see Appendix \ref{Appendix} 
for  formal definitions.   
These are polynomials with coefficients in the ring of virtual representations of $\Gamma$, with the property that taking dimensions
recovers the ordinary polynomials  \cite{GPY,PXY}.  
Our proof shows the following strengthening of Theorem \ref{thm:KL}. 

\begin{theorem}\label{thm:equivariantKL}
The following holds for any matroid $\M$ and any finite group $\Gamma$ acting on $\M$.
\begin{enumerate}[(1)]\itemsep 5pt
\item The polynomial $P_{\M}^{\Gamma}(t)$ has nonnegative coefficients: 
The coefficients of $P_{\M}^{\Gamma}(t)$ are isomorphism classes of honest, rather than virtual,  
representations of $\Gamma$.
\item The polynomial $Z^\Gamma_{\M}(t)$ is unimodal: The  coefficient of $t^k$ in  $Z^\Gamma_{\M}(t)$  is isomorphic to a subrepresentation of the coefficient of $t^{k+1}$  in  $Z^\Gamma_{\M}(t)$ for all $k<\rk \M / 2$.
\end{enumerate}
\end{theorem}

 Theorem \ref{thm:equivariantKL} specializes to Theorem \ref{thm:KL} when we take $\Gamma$ to be the trivial group.
The first part of Theorem \ref{thm:equivariantKL}  was conjectured in  \cite[Conjecture 2.13]{GPY},  
where it was proved for  matroids that are $\Gamma$-equivariantly realizable over some field.\footnote{It is much easier to construct matroids that are not $\Gamma$-equivariantly realizable than it is to construct  matroids that are not realizable. For example, the uniform matroid of rank $2$ on $4$ elements is realizable over any field with at least three elements, but it is not $S_4$-equivariantly realizable over any field.} 
For uniform matroids, a combinatorial proof of the equivariant nonnegativity  was given in \cite[Section 3]{GPY}.

We also prove the following monotonicity result for equivariant Kazhdan--Lusztig polynomials of matroids.
The non-equivariant case (when $\Gamma$ is the trivial group) is analogous to a monotonicity result for classical Kazhdan--Lusztig polynomials in Weyl groups \cite{Irv}, \cite[Corollary 3.7]{BrMacP}.
\begin{theorem}\label{thm:monotonicity}
	Let $\M$ be a matroid acted on by a finite group $\Gamma$ which fixes a nonempty flat $F \in \cL(\M)$.  Then the polynomial
	\[P^\Gamma_\M(t) - P^\Gamma_{\M_F}(t)\]
	has coefficients which are honest, rather than virtual, representations of the stabilizer group $\Gamma_F$.  In particular when $\Gamma$ is the trivial group we have that the polynomial $P_\M(t) - P_{\M_F}(t)$ has nonnegative coefficients.
\end{theorem}

By \cite[Theorem 1.2]{GX}, there is 
a unique way to assign a polynomial $Q_{\M}(t)$ to each matroid $\M$, called the {\bf inverse Kazhdan--Lusztig polynomial} of $\M$,  subject to the following  three conditions:
\begin{enumerate}[(a)]\itemsep 5pt
\item If the ground set of $\M$ is empty, then $Q_{\M}(t)$ is the constant polynomial $1$.
\item For every matroid $\M$ on a nonempty ground set, the degree of  $Q_\M(t)$ is strictly less than  $\rk \M / 2$.
\item For every matroid $\M$, we have $(-t)^{\rk \M} Q_\M(t^{-1}) = \displaystyle\sum_{F\in \cL(\M)} (-1)^{\rk \M^F}Q_{\M^F}(t) \cdot t^{\rk \M_F}\chi_{\M_F}(t^{-1})$.
\end{enumerate}
Just as the last of the three conditions characterizing $P_{\M}(t)$ can be replaced by
a condition saying that the $Z$-polynomial is palindromic, the last condition characterizing $Q_{\M}(t)$ 
can be replaced by the following analogous statement, which can be proved in the same way:
\begin{enumerate}[(a)]\itemsep 5pt
\item[(c)'] For every $\M$, the polynomial $Y_\M(t)\coloneq \displaystyle\sum_{F \in \cL(\M)} (-t)^{\crk F}\mu(F,E) \,Q_{\M^F}(t)$ satisfies
\[
t^{\rk \M} Y_{\M}(t^{-1})=Y_{\M}(t),
\]
where $\mu$ is the M\"obius function on $\cL(\M)$.
\end{enumerate}
We also prove the following result, which was conjectured in \cite[Conjecture 4.1]{GX}.

\begin{theorem}\label{thm:inverse}
The polynomial $Q_{\M}(t)$ has nonnegative coefficients.  
\end{theorem}

In fact, our proof shows that 
the coefficients of the {\bf equivariant inverse Kazhdan--Lusztig polynomial} $Q_{\M}^{\Gamma}(t)$ defined in Appendix \ref{Appendix} are isomorphism classes of honest, rather than virtual, representations of $\Gamma$.  


\subsection{Proof strategy}\label{sec:strategy}
We now provide an outline of the proofs of Theorems \ref{thm:top-heavy}, \ref{thm:KL}, and \ref{thm:equivariantKL}. 
The algebro-geometric
motivations for these arguments will appear in Section \ref{sec:realizable}.

For any matroid $\M$ of rank $d$, consider the {\bf graded M\"obius algebra} 
\[
\H(\M) \coloneq \bigoplus_{F\in \cL(\M)} \Q y_F.
\]
The grading is defined by declaring the degree of the element $y_F$ to be  $\rk F$, the rank of $F$ in $\M$.
The multiplication is defined by the formula
\[
y_Fy_G \coloneq \begin{cases} y_{F\vee G} & \text{if $\rk F + \rk G = \rk (F\vee G)$,}\\
0 & \text{if $\rk F + \rk G > \rk (F\vee G)$,}
\end{cases}
\]
where $\vee$ stands for the join of flats in the lattice $\cL(\M)$.
Let $\CH(\M)$ be the {\bf augmented Chow ring} of $\M$, introduced in \cite{BHMPW}.
We will review the definition of $\CH(\M)$ in Section \ref{sec:Chow}, but for now it will suffice to know the following three things:
\begin{enumerate}[$\bullet$]\itemsep 5pt
\item $\CH(\M)$ contains $\H(\M)$ as a graded subalgebra  \cite[Proposition 2.18]{BHMPW}.
\item $\CH(\M)$ is equipped with a {\bf degree isomorphism} $\deg_\M \colon \CH^d(\M)
\rightarrow \Q$  \cite[Definition 2.15]{BHMPW}.
\item By the Krull--Schmidt theorem, $\CH(\M)$ can be written as a direct sum of indecomposable graded $\H(\M)$-submodules, and their isomorphism classes and multiplicities are unique.\footnote{For the Krull--Schmidt theorem, see, for example, \cite[Theorem 1]{Atiyah}. By \cite[Corollary 2]{CF} or \cite[Theorem 3.2]{GG}, the indecomposability in the category of graded $\H(\M)$-modules implies the indecomposability  in the category of  $\H(\M)$-modules.} Since $\CH^0(\M)=\H^0(\M)=\Q$, there is up to isomorphism a unique indecomposable summand containing $\H(\M)$. 

\end{enumerate}
In this introduction, we temporarily define the {\bf intersection cohomology} of $\M$ to be the indecomposable graded $\H(\M)$-module $\IH(\M)$ described by the last point above.
This defines the intersection cohomology of $\M$  up to isomorphism of graded $\H(\M)$-modules. 
In Section \ref{sec:IC of a matroid}, we will construct a canonical submodule $\IH(\M) \subseteq\CH(\M)$  that contains $\H(\M)$ and is preserved by all symmetries of $\M$.  Proving that it is an indecomposable direct summand, and thus that it agrees with our temporary definition, requires most of the results of the rest of the paper.
The construction of $\IH(\M)$ as an explicit submodule of $\CH(\M)$, or more generally the construction of the {\bf canonical decomposition} of $\CH(\M)$ as a graded $\H(\M)$-module,
will be essential in our proofs of the main results but not in their statements.\footnote{By canonical, we mean 
only that the  symmetries of the matroid preserve the decomposition.}
To illustrate the main definitions and statements, we give several
 explicit descriptions of the canonical decomposition in simple cases; see Examples \ref{ex:rank2_1b}, \ref{ex:boolean1}, \ref{ex:uniform1}, \ref{ex:rank2_2},  \ref{ex:boolean2}, \ref{ex:uniform2}, \ref{ex:deletion1}, and \ref{ex:deletion2}.

We fix any decomposition of the graded $\H(\M)$-module $\CH(\M)$ as above, and consider any positive linear combination 
 \begin{equation}\label{eq_star}
 \ell = \sum_{F\in\cL^1(\M)} c_Fy_F\in \H^1(\M), \ \  \text{$c_F$ is positive for every rank $1$ flat $F$ of $\M$.}
 \end{equation}
Our central result is that  $\IH(\M)$  satisfies the {\bf K\"ahler package} with respect to $\ell$.

\begin{theorem}\label{prop:kahler}
The following holds for any matroid $\M$ of rank $d$.
\begin{enumerate}[(1)]\itemsep 5pt
\item (Poincar\'e duality theorem) For every nonnegative  $k \le d/2$, the bilinear pairing
\[
\IH^k(\M) \times \IH^{d-k}(\M) \longrightarrow \Q, \qquad (\eta_1,\eta_2) \longmapsto \deg_\M (\eta_1\eta_2)
\]
is non-degenerate.
\item (Hard Lefschetz theorem) For every nonnegative  $k \le d/2$, the multiplication map
\[
 \IH^k(\M)\longrightarrow \IH^{d-k}(\M), \qquad \eta \longmapsto \ell^{d-2k} \eta
\]
 is an isomorphism.
\item (Hodge--Riemann relations)  For every nonnegative  $k \le d/2$, the bilinear form
\[
\IH^k(\M) \times \IH^{k}(\M) \longrightarrow \Q, \qquad (\eta_1,\eta_2) \longmapsto (-1)^k \deg_\M ( \ell^{d-2k}\eta_1\eta_2)
\]
is positive definite on the kernel of multiplication by $\ell^{d-2k+1}$.
\end{enumerate}
\end{theorem}

We now show how Theorem \ref{prop:kahler} implies Theorem \ref{thm:top-heavy}.

\begin{proof}[Proof of Theorem \ref{thm:top-heavy}, assuming Theorem \ref{prop:kahler}]
It follows from the hard Lefschetz theorem that the multiplication map 
$\ell^{j-k}\colon\IH^k(\M)\to\IH^{j}(\M)$ is injective for $j \le d-k$.  
After restricting the multiplication map to the $\H(\M)$-submodule $\H(\M)\subseteq\IH(\M)$, we obtain
an  injection  
\[
\ell^{j-k}\colon\H^k(\M)\longrightarrow \H^{j}(\M).
\]
Define $\ell$ by putting $c_F=1$ for all $F$ in Equation \eqref{eq_star}. Then $\ell$ is invariant under the action of $\Gamma$, and $\ell^{j-k}$ is
the injective map desired by part (3). 
If we write this injection as a matrix in terms of the natural bases, 
the matrix is supported on the pairs satisfying $F\leq G$.
An injective matrix has a nonvanishing maximal minor, so there is at least one way
to choose a row for each column such that the entries in the matrix are nonzero. Any such
choice gives a map $\iota$ as desired by part (2).
Clearly, part (1) follows from either part (2) or part (3).
\end{proof}

The following propositions will be key ingredients in the proof of Theorem \ref{thm:KL}.
We write $\mm$ for the graded maximal ideal of $\H(\M)$,
write $\mathbb{Q}$ for the one-dimensional graded $\H(\M)$-module  in degree zero, 
and write $\IH(\M)_{\varnothing}$ for the graded vector space 
\[
\IH(\M) \otimes_{\H(\M)} \mathbb{Q} \cong \IH(\M) / \mm \IH(\M).
\]

\begin{proposition}\label{prop:no socle}
For every nonempty matroid $\M$,  $\IH(\M)_{\varnothing}$ vanishes in degrees  $\ge \rk \M/2$.
\end{proposition}


\begin{proposition}\label{prop:gr}
 For all nonnegative $p$, there is a canonical graded vector space isomorphism
\begin{equation}\label{eq:gr iso}
\mm^p \IH(\M) / \mm^{p+1} \IH(\M) \cong \bigoplus_{F\in\cL^p(\M)} \IH(\M_F)_{\varnothing}[-p].
\end{equation}
\end{proposition}

For a geometric description in the realizable case, see Section \ref{sec:realizable}.
When a finite group $\Gamma$ acts on $\M$, it preserves the direct sum of $\IH(\M_F)_{\varnothing}[-p]$ over $F$ in each orbit of $\Gamma$ on $\cL^p(\M)$. This sum over an orbit is $\Gamma$-equivariantly isomorphic to $\Ind_{\Gamma_F}^\Gamma  \IH(\M_F)_{\varnothing}[-p]$, where $\Gamma_F \subseteq \Gamma$ is the subgroup of elements fixing  $F$. This isomorphism upgrades \eqref{eq:gr iso} to an isomorphism of $\Gamma$-representations
\[
\mm^p \IH(\M) / \mm^{p+1} \IH(\M) \cong  \bigoplus_{[F]\in\cL^p(\M)/\Gamma} 
\Ind_{\Gamma_F}^\Gamma  \IH(\M_F)_{\varnothing}[-p],
\]
where we are taking one flat $F$ in each orbit of the action of $\Gamma$ on $\cL^p(\M)$. 

\begin{proof}[Proofs of Theorems \ref{thm:KL} and \ref{thm:equivariantKL}, assuming Theorem \ref{prop:kahler} and Propositions \ref{prop:no socle} and \ref{prop:gr}]
We define polynomials
\[
\tilde P_\M(t) \coloneq \sum_{k\geq 0} \dim \left( \IH^k(\M)_{\varnothing}\right) t^k \ \  \text{and} \ \ 
\tilde Z_\M(t) \coloneq \sum_{k\geq 0} \dim \left( \IH^k(\M) \right)  t^k.
\]
We argue $\tilde P_\M(t) = P_\M(t)$ and $\tilde Z_\M(t) = Z_\M(t)$ by induction on the rank of $\M$.  
The statement is clear when the rank is zero, so assume that $\M$ has positive rank and  that the statement holds for matroids of strictly smaller rank.
Taking Poincar\'e polynomials of the graded vector spaces in 
Proposition \ref{prop:gr} and summing over all $k$, we get
\[
\tilde Z_{\M}(t) = \sum _{F\in \cL(\M)} t^{\rk F} \tilde P_{\M_F}(t).
\]
When  combined with our inductive hypothesis, the above gives
\[
\tilde Z_{\M}(t) = \tilde P_{\M}(t) + \sum_{F\neq \varnothing} t^{\rk F} P_{\M_F}(t).
\]
On the other hand, by Theorem \ref{prop:kahler} and 
 Proposition \ref{prop:no socle}, we have 
\[
\tilde Z_\M(t) = t^{\rk \M}\tilde Z_\M(t^{-1}) \ \ \text{and} \ \  \deg \tilde P_{\M}(t)<\rk \M/2.
\]
The desired identities 
now  follow from the second definition of Kazhdan--Lusztig polynomials of matroids given above  \cite[Theorem 2.2]{BV}. 

The nonnegativity of the coefficients of $P_\M(t)$ is immediate from the fact that it is the Poincar\'e polynomial of a graded vector space. 
The unimodality of $ Z_\M(t)$ follows from 
the hard Lefschetz theorem for $\IH(\M)$.  All of the
steps of this argument still hold when interpreted equivariantly with respect to any group of symmetries of $\M$ 
by Lemma \ref{induced reps}, Definition \ref{eqZ}, and Corollary \ref{Zcor}.
\end{proof}

We record the numerical identities for $P_{\M}(t)$ and $Z_{\M}(t)$ obtained in the above proof.

\begin{theorem}
For any matroid $\M$, we have 
\[
P_{\M}(t)= \sum_{k\geq 0}  \dim \left( \IH^k(\M)_{\varnothing} \right)  \, t^k \ \ \text{and} \ \ 
Z_{\M}(t)= \sum_{k\geq 0}  \dim \left( \IH^k(\M) \right) \, t^k.
\]
When a finite group $\Gamma$ acts on $\M$, the analogous identities hold for $P^\Gamma_{\M}(t)$ and $Z^\Gamma_{\M}(t)$.
\end{theorem}

\begin{remark}\label{indecomposability is subtle}
The explicit construction of $\IH(\M)$ as a submodule of $\CH(\M)$ appears in Section \ref{sec:IC of a matroid}, but the fact that it is an indecomposable direct summand of $\CH(\M)$ is not established
until much later.  It follows from Proposition \ref{prop:indecomposable}, which can only be applied after we have proved 
Theorem \ref{prop:kahler}.  See Remark \ref{ignore the first} for why this is the case.
\end{remark} 

\begin{remark}\label{why HR?}
The astute reader will note that the only part of Theorem \ref{prop:kahler}
that appears in the applications is the hard Lefschetz theorem.  However, we know of no way to prove the hard Lefschetz theorem
by itself. 
 Instead, we roll all three statements in Theorem \ref{prop:kahler} up into a grand induction.  See Remark \ref{unification} for more on this philosophy.
%
\end{remark}

\begin{remark}
	In a forthcoming work \cite{BHMPW-IH}, we will give a more direct and intrinsic definition of $\IH(\M)$ as an $\H(\M)$-module, without reference to an embedding in the larger space $\CH(\M)$.  See Remark \ref{rmk:IH of a matroid}. 
\end{remark}

\begin{remark}
We have not yet commented on our strategy for proving Theorem \ref{thm:inverse}.  This proof will also rely on Theorem \ref{prop:kahler},
and will proceed by interpreting $Q_{\M}(t)$ as the graded multiplicity of the trivial graded $\H(\M)$-module in a complex of $\H(\M)$-modules called the
{\bf Rouquier complex}.  See Sections \ref{sec:introduction-Rouquier} and \ref{sec:multiplicities and inverse KL} for more details.
\end{remark}

\subsection{The realizable case}\label{sec:realizable}
We now give the geometric motivation for the statements in Sections \ref{sec:results} and \ref{sec:strategy}, and in particular review the proofs
of Theorems \ref{thm:top-heavy} and \ref{thm:KL} for realizable matroids.

Let $V$ be a vector space of dimension $d$ over a field $\mathbb{F}$,
let $E$ be a finite set, and let $\sigma\colon E\to V^\vee$ be a map
whose image spans the dual vector space $V^\vee$.  
The collection of subsets $S\subseteq E$ for which $\sigma$ is injective on $S$ and 
$\sigma(S)$ is a linearly independent set in $V^\vee$ forms the 
independent sets of a matroid $\M$ of rank $d$.  Any matroid which arises in this way is called {\bf realizable} over $\mathbb{F}$, and $\sigma$ is called a {\bf realization} of $\M$ over $\mathbb{F}$.\footnote{When a finite group $\Gamma$ acts on $\M$, we say that $\M$ is {\bf \boldmath{$\Gamma$}-equivariantly realizable} over $\mathbb{F}$ if there is a $\Gamma$-equivariant map $\sigma\colon E \to V^\vee$ for some representation $V$ of $\Gamma$ over $\mathbb{F}$.}
We continue to assume that $\M$ is loopless, which means that the image of $\sigma$ does not contain the zero vector.

For any flat $F$ of $\M$, let $V_F\subseteq V$ 
be the subspace perpendicular to $\{\sigma(e)\}_{e \in F}$, and let $V^F$ be the quotient space $V/V_F$.  Then we have canonical maps
\[
\sigma^F \colon F\to (V^F)^\vee \ \ \text{and} \ \  \sigma_F \colon E\setminus F\to (V_F)^\vee
\]
realizing the localization $\M^F$ and the contraction $\M_F$, respectively.

Consider the linear map $V\to \mathbb{F}^E$ whose $e$-th 
coordinate is given by $\sigma(e)$.  
The assumption that the image of $\sigma$ spans $V^\vee$ implies that this map is injective.  The decomposition $\mathbb{P}^1_{\mathbb{F}} = \mathbb{F} \sqcup \{\infty\}$ gives an embedding of $\mathbb{F}^E$ into $(\mathbb{P}^1_\mathbb{F})^E$, and we let $Y \subseteq (\mathbb{P}^1_\mathbb{F})^E$ 
denote the closure of the image of $V$.  This projective variety is called the {\bf arrangement Schubert variety} of $\sigma$.
The terminology is chosen to suggest that $Y$ has many similarities to classical Schubert varieties.  
It has a stratification by affine spaces, whose strata are the orbits of  the additive group $V$ on $Y$, indexed by flats of $\M$.
For any flat $F$ of  $\M$, let 
\[
U^F \coloneq \{p\in Y\, | \, \text{$p_e = \infty$ if and only if $e\notin F$}\}.
\]
For example, $U^E$ is the vector space $V$ and $U^\varnothing$ is the point $\infty^E$.
More generally, $U^F$ is isomorphic to $V^F$, and these subvarieties form 
a stratification of $Y$ with $U^F$ contained in the closure of $U^G$ if and only if $F$ is contained in  $G$ \cite[Lemmas 7.5 and 7.6]{PXY}.

The arrangement Schubert variety $Y$ is singular, and it admits a canonical resolution $X$
called the {\bf augmented wonderful variety}, obtained by first blowing up the point $U^\varnothing$, then blowing up the proper transforms
of the closures of $U^F$ for all rank 1 flats $F$, and so on.  The preimage of $U^\varnothing$ in $X$ is the {\bf wonderful variety} $\underline{X}$ of de Concini--Procesi \cite{DeCP}.  A different description of $X$ as an iterated blow-up of a projective space
appears in \cite[Section 2.4]{BHMPW}.  

For the remainder of this section, we will assume for simplicity that $\mathbb{F} = \mathbb{C}$; see Remark \ref{characteristic}
for a discussion of what happens over other fields.
The rings and modules introduced in Section~\ref{sec:strategy} have the following interpretations in terms of the varieties $X$ and $Y$.
The graded M\"obius algebra $\H(\M)$ is isomorphic to the rational cohomology ring $\H^{\bullet}(Y)$ \cite[Theorem 14]{HW},
and the augmented Chow ring $\CH(\M)$ is isomorphic to the rational Chow ring of $X$, or equivalently to the
rational cohomology ring $\H^{\bullet}(X)$.  
The {\bf Chow ring} $\uCH(\M)$, which does not feature prominently in this introduction but will play a crucial role in the body of the paper, is isomorphic to the rational Chow ring of $\underline{X}$,
or equivalently to the rational cohomology ring $\H^{\bullet}(\underline{X})$.  

By applying the decomposition theorem \cite{BBD} to the proper map from $X$ to $Y$, 
we find that the intersection cohomology $\IH^{\bullet}(Y)$ is isomorphic as a graded $\H^{\bullet}(Y)$-module 
to a direct summand of $\H^{\bullet}(X)$.\footnote{All of these
cohomology rings and intersection cohomology groups of varieties vanish in odd degree, and our isomorphisms double degree.  
So $\H^1(\M) \cong \H^2(Y)$, $\CH^1(\M)\cong \H^2(X)$, $\IH^1(\M)\cong \IH^2(Y)$, and so on.}  A slight extension of an argument of Ginzburg \cite{Ginsburg-perverse} shows that $\IH^\bullet(Y)$ is indecomposable as an $\H^\bullet(Y)$-module, which implies that it coincides with our module $\IH(\M)$.\footnote{To be precise, two hypotheses of \cite{Ginsburg-perverse} are not satisfied by $Y$: it is not the closure of a Bia\l ynicki-Birula cell for a torus action on a smooth projective variety, and the natural torus which acts is one-dimensional, so it is not possible to find an attracting cocharacter at each fixed point.  However, each fixed point has an affine neighborhood with an attracting action of 
the multiplicative group, and this is enough.}  Theorem \ref{prop:kahler} is a standard
result in Hodge theory for singular projective varieties; see, for example, \cite[Theorem 2.2.3]{dCM05}.

For each flat $F$ of $\M$, let $\H^\bullet(\IC_{Y,F})$ denote the cohomology of the stalk of the intersection cohomology complex 
$\IC(Y)$ at a point in $U^F$.
The restriction map on global sections from $\IH^\bullet(Y)$ to $\H^\bullet(\IC_{Y,\varnothing})$ descends to $\IH^\bullet(Y)_\varnothing$,
and another application of the result of \cite{Ginsburg-perverse} implies that the induced map from $\IH^\bullet(Y)_\varnothing$
to $\H^\bullet(\IC_{Y,\varnothing})$ is an isomorphism.  
A fundamental property of the intersection cohomology sheaf $\IC(Y)$ is that, if the dimension $d$ of $Y$ is positive, then the stalk cohomology group 
$\H^{2k}(\IC_{Y,F})$ vanishes for $k\geq d$.
This proves Proposition \ref{prop:no socle} in the realizable case.

Let $Y_F$ be the arrangement Schubert variety associated with the realization $\sigma_F$ of $\M_F$.  
We have a canonical inclusion $Y_F\to Y$, whose image is equal to the subvariety
\[
\{p\in Y\, | \, \text{$p_e = 0$ if and only if $e\in F$}\}.
\]
The inclusion realizes $Y_F$ as a normally nonsingular slice to the stratum $U^F$. 
Thus it induces an isomorphism from $\H^\bullet(\IC_{Y,F})$ to $\H^\bullet(\IC_{Y_F,\varnothing})$, 
see \cite[Proposition 4.11]{KLS}.
Let $j_F \colon U^F\to Y$ denote the inclusion of the stratum $U^F$.
Our stratification of $Y$ induces a spectral sequence 
with 
\[
E_1^{p,q} = \displaystyle \bigoplus_{F\in \cL^{p}(\M)} \H_c^{p+q}\left(j_F^*\IC(Y)\right)
\]
that converges to $\IH^\bullet(Y)$.
The summands  of $E_1^{p,q}$ satisfy
\[
\H_c^{p+q}\left(j_F^*\IC(Y)\right) \cong \Big(\H^{\bullet}(\IC_{Y,F})\otimes \H_c^\bullet(U^F)\Big)^{p+q}
\cong \Big(\H^{\bullet}(\IC_{Y,F})[-2p]\Big)^{p+q} \cong  \H^{q-p}(\IC_{Y_F,\varnothing}).
\]
Since  $\H^{\bullet}(\IC_{Y_F,\varnothing})$ vanishes in odd degree,  our
spectral sequence degenerates at the $E_1$ page \cite[Section 7]{PXY}.  This means that 
$\IH^\bullet(Y)$ vanishes in odd degree, 
and that the degree $2k$ part of the graded vector space
$\mm^p\IH^\bullet(Y)/\mm^{p+1}\IH^{\bullet}(Y)$ is isomorphic to
\begin{equation*}\label{spectral-gr}
E_\infty^{p,2k-p} = E_1^{p,2k-p} 
\cong \bigoplus_{F\in \cL^p(\M)} \H^{2(k-p)}(\IC_{Y_F,\varnothing})
\cong \bigoplus_{F\in \cL^p(\M)} \IH^{2(k-p)}(Y_F)_\varnothing.\end{equation*}
This proves Proposition \ref{prop:gr} in the realizable case.

\begin{remark}\label{characteristic}
If the field $\mathbb{F}$ is not equal to the complex numbers, then we can mimic all of the geometric arguments
in this section using $l$-adic \'etale cohomology for some prime $l$ not equal to the characteristic of $\mathbb{F}$.  
In this setting there is no geometric analogue of the Hodge--Riemann relations, so Hodge theory does not give us the full 
K\"ahler package of Theorem \ref{prop:kahler}.  
It is interesting to note that Theorem \ref{prop:kahler} gives us a truly new result
for matroids that are realizable only in positive characteristic.  Namely, it says that there is a rational form for the $l$-adic
 \'etale intersection cohomology of the arrangement Schubert variety for which the Hodge--Riemann relations hold.
 We suspect that $\IH(\M)$ is a Chow analogue of the intersection cohomology of $Y$.
\end{remark}

\begin{remark}
If one wants to write down a maximally streamlined proof of Theorem \ref{thm:top-heavy}
for realizable matroids, it is not necessary to know that $\H^\bullet(Y)$ is isomorphic to the graded M\"obius algebra of $\M$,
and it is not necessary to consider the augmented wonderful variety $X$ or the augmented
Chow ring of $\M$.  One considers $\IH^\bullet(Y)$ as a module over $\H^\bullet(Y)$
and applies the same argument outlined in Section \ref{sec:strategy}.
The statements that $\IH^\bullet(Y)$ contains $\H^\bullet(Y)$ as a submodule, that $\H^\bullet(Y)$
has a basis indexed by flats, and that the matrix for the multiplication by a power of an ample class in this basis 
is supported on pairs  $F\leq G$ follow from \cite[Theorem 2.1, Theorem 3.1, and Lemma 5.1]{BE}.
For the proof of Theorem \ref{thm:KL}, we need to know that the cohomology groups $\H^{\bullet}(\IC_{Y,F})$
vanish in odd degree in order to conclude that the spectral sequence degenerates.  To see this, we can either embed $\IH^\bullet(Y)$
in $\H^\bullet(X)$ as in the text above, or we can rely on an inductive argument as in \cite[Theorem 3.6]{KLS}.
\end{remark}

\subsection{Kazhdan--Lusztig--Stanley polynomials}\label{sec:antecedents}
In this section, we will discuss two antecedents to our work in the context of Kazhdan--Lusztig--Stanley theory.
Let $P$ be a locally finite ranked poset.  For all $x\leq y\in P$, let $r_{xy} \coloneq \rk y - \rk x$.
A {\bf \boldmath{$P$}-kernel} is a collection of polynomials 
$\kappa_{xy}(t) \in \Z[t]$ for each $x\leq y \in P$
satisfying the following conditions:
\begin{enumerate}[$\bullet$]\itemsep 5pt
\item For all $x\in P$, $\kappa_{xx}(t) = 1$.
\item For all $x\leq y \in P$, $\deg\kappa_{xy}(t) \leq r_{xy}$.
\item For all $x < z\in P$, $\displaystyle\sum_{x\leq y \leq z} t^{r_{xy}}\kappa_{xy}(t^{-1})\kappa_{yz}(t) = 0$.
\end{enumerate}
Given such a collection of polynomials, Stanley \cite{subdivisions} showed that there exists a unique collection of polynomials
$f_{xy}(t) \in \Z[t]$ for each $x\leq y \in P$
satisfying the following conditions:
\begin{enumerate}[$\bullet$]\itemsep 5pt
\item For all $x\in P$, $f_{xx}(t) = 1$.
\item For all $x< y \in P$, $\deg f_{xy}(t) < r_{xy}/2$.
\item For all $x \le  z\in P$, $t^{r_{xz}}f_{xz}(t^{-1}) = \displaystyle\sum_{x\leq y \leq z} \kappa_{xy}(t)f_{yz}(t)$.
\end{enumerate}
The polynomials $f_{xy}(t)$ are called {\bf Kazhdan--Lusztig--Stanley polynomials}.

The first motivation for this construction comes from classical Kazhdan--Lusztig polynomials.  If we take the poset to be a Coxeter
group $W$ equipped with the Bruhat order and the $W$-kernel to be the $R$-polynomials $R_{xy}(t)$, then the polynomials
$f_{xy}(t)$ are called {\bf Kazhdan--Lusztig polynomials}.  These polynomials were introduced by Kazhdan and Lusztig
in \cite{KL79}, where they were conjectured to have nonnegative coefficients.  This was proved for Weyl groups in \cite{KL80} 
by interpreting $f_{xy}(t)$ as the Poincar\'e polynomial
for a stalk of the intersection cohomology sheaf of a classical Schubert variety.  For arbitrary Coxeter groups, the conjecture remained
open for $34$ years before it was proved by Elias and Williamson \cite{EW}, who used Soergel bimodules as a combinatorial
replacement for intersection cohomology groups of classical Schubert varieties.

The second motivation for this definition comes from convex polytopes.  Let $\Delta$ be a convex polytope, and let
$P$ be the poset of faces of $\Delta$, ordered by reverse inclusion and ranked by codimension, with the convention that the codimension
of the empty face is $\dim\Delta+1$.
This poset is Eulerian, which means that the polynomials $(t-1)^{r_{xy}}$ form a $P$-kernel.  The polynomial $g_{\Delta}(t) \coloneq f_{\Delta\varnothing}(t)$
is called the {\bf \boldmath{$g$}-polynomial} of $\Delta$.  When $\Delta$ is rational, this polynomial can be shown to have nonnegative
coefficients by interpreting it as the Poincar\'e polynomial for a stalk of the intersection cohomology sheaf of a toric variety \cite{DL,Fieseler}.
For arbitrary convex polytopes, nonnegativity of the coefficients of the $g$-polynomial was proved $13$ years later by Karu \cite{Karu}, who used the theory
of combinatorial intersection cohomology of fans \cites{BBFK,BrLu,Braden-CICF} as a replacement for intersection cohomology
groups of toric varieties.

In our setting, we consider the ranked poset $\cL(\M)$ along with the $\cL(\M)$-kernel consisting of the characteristic
polynomials $\chi_{FG}(t) \coloneq \chi_{\M_F^G}(t)$,
and we find that $f_{\varnothing E}(t)$ is equal to the Kazhdan--Lusztig polynomial $P_\M(t)$.
When $\M$ is realizable, this polynomial can be shown to have nonnegative coefficients 
by interpreting it as the Poincar\'e polynomial for a stalk of the intersection cohomology sheaf of the arrangement Schubert variety
$Y$, as explained in Section \ref{sec:realizable}.  Theorem \ref{thm:KL} is obtained by using $\IH(\M)$ as a replacement for the intersection
cohomology group of $Y$.

\begin{remark}\label{unification}
It is reasonable to ask to what extent these three nonnegativity results can be unified.  In the geometric setting (Weyl groups,
rational polytopes, realizable matroids), it is possible to write down a general theorem that has each of these results as a special
case \cite[Theorem 3.6]{KLS}.  However, the problem of finding algebraic or combinatorial replacements for the intersection
cohomology groups of stratified algebraic varieties is not one for which we have a general solution.  Each of the three
theories described above involves numerous details that are unique to that specific case.  One insight that we can
take away is that, while the hard Lefschetz theorem is typically the main statement needed for applications, it is always
necessary to prove Poincar\'e duality, the hard Lefschetz theorem, and the Hodge--Riemann relations together as a single package.
This approach can be traced back to the work of McMullen on the polytope algebra \cite{McMullen} and to that of De Cataldo and Migliorini
on the intersection cohomology of complex algebraic varieties \cite{dCM}.
\end{remark}

\begin{remark}
We note that Kazhdan--Lusztig--Stanley polynomials do not {\em always} have nonnegative coefficients. 
Indeed, we can take any collection of polynomials $\{f_{xy}(t)\mid x\leq y\in P\}$ satisfying the conditions
$f_{xx}(t)=1$ and $\deg f_{xy}(t) < r_{xy}/2$ for all $x<y$, and these will be the Kazhdan--Lusztig--Stanley polynomials
for some (unique) $P$-kernel.  It is only in very special settings that nonnegativity will hold.
\end{remark}

\begin{remark}
The analogue of Theorem \ref{thm:top-heavy} for Weyl groups
appears in \cite{BE}, and for general Coxeter groups (using Soergel bimodules) in \cite{Melvin-Slofstra}.
There is no analogous result for convex polytopes because toric varieties associated with non-simple polytopes do not in general admit stratifications by affine spaces.
\end{remark}

\begin{remark}
For a locally finite poset $P$, consider the incidence algebra
\[
I(P) \coloneq \prod_{x \le y \in P} \mathbb{Z}[t], \ \ \text{where $\displaystyle (uv)_{xz}(t) \coloneq  \sum_{x \le y \le z} u_{xy}(t) v_{yz}(t)$ for $u,v \in I(P)$}.
\]
An element $h \in I(P)$ has an inverse, left or right, if and only if $h_{xx}(t)=\pm 1$ for all $x \in P$. In this case, the left and right inverses are unique and they coincide \cite[Lemma 2.1]{KLS}. 
In terms of the incidence algebra, the inverse Kazhdan--Lusztig polynomial of $\M$ can be interpreted as 
\[
Q_{\M}(t) = (-1)^{\rk \M} (f^{-1})_{\varnothing E}(t),
\]
where $f$ is the Kazhdan--Lusztig polynomial viewed as an element of $I(\cL(\M))$.
We note that the analogous constructions for finite Coxeter groups and convex polytopes do not produce
any new families of polynomials.  Specifically, for a  finite Coxeter group, we have
\[
(-1)^{r_{xy}}(f^{-1})_{xy}(t) = f_{(w_0y)(w_0x)}(t),
\]
where $w_0\in W$ is the longest word \cite[Example 2.12]{KLS}. 
For a convex polytope, we have
\[
(-1)^{\dim\Delta+1}(f^{-1})_{\Delta\varnothing}(t) = g_{\Delta^{\!*}}(t),
\]
where $\Delta^{\!*}$ is the dual polytope of $\Delta$ \cite[Example 2.14]{KLS}.
The explanation for these statements is that the corresponding $P$-kernels are {\bf alternating} \cite[Proposition 2.11]{KLS}, 
which means that 
$
(-t)^{r_{xy}}\kappa_{xy}(t^{-1}) = \kappa_{xy}(t)$. 
The same is not true for characteristic polynomials, which is why inverse Kazhdan--Lusztig polynomials of matroids are fundamentally
different from ordinary Kazhdan--Lusztig polynomials of matroids.
\end{remark}

\subsection{Outline}
In Section~\ref{sec:Chow}, we recall the definitions of the Chow ring and the augmented Chow ring of a matroid, then we review properties established in \cite{BHMPW} of various pushforward and pullback maps between these rings.  In Section~\ref{sec:IC of a matroid}, we define the intersection cohomology modules of matroids, explain how these modules behave under the pullback and pushforward maps, and define the host of statements that make up our main inductive proof.  

With all the key players defined, we provide Section~\ref{sec:guide} as a guide to the inductive proof of the main theorem of the paper, Theorem~\ref{theorem_all}. 
No definitions or proofs are given here, and the section is meant only to provide intuition for the structure of the proof.  This section may be skipped, but we hope that the reader benefits from flipping back to this section to ``see what the authors were thinking" as they read the rest of the paper.

The proof of the main theorem begins in Section~\ref{sec:modules} and continues for the remainder of the paper.  We use Sections~\ref{sec:modules} and~\ref{sec:ihmodulesmobius} to establish some general results about modules over the graded M\"obius algebra, and in particular about the intersection cohomology modules.  
The results in Section~\ref{sec:modules} are not inductive in nature and are established outside of the inductive loop.  
Section~\ref{sec:starting induction} studies the Poincar\'e pairings on various $\H(\M)$-submodules of $\CH(\M)$ and how they behave under linear-algebraic operations such as tensor products.
Section~\ref{sec:Rouquier complexes} is dedicated to introducing and studying the so-called Rouquier complexes; as in \cite{EW}, we use these to prove a version of weak Lefschetz, which for us is a certain vanishing condition for the socles of our intersection cohomology modules.  In Section~\ref{sec:multiplicities and inverse KL}, we explain how Theorem~\ref{theorem_all} can be used to deduce Theorem \ref{thm:inverse}.
Sections~\ref{sec:deletion induction} and~\ref{sec:underlined deletion induction} use the semi-small decomposition developed in \cite{BHMPW} to perform an induction involving the deletion $\M \setminus i$ of a single element $i$ from $\M$.  Section~\ref{sec:deform} explores how the hard Lefschetz theorem and Hodge--Riemann relations behave when deforming Lefschetz operators.  Section~\ref{SectionProof} puts all of the results from the previous sections together to finish the inductive proof of Theorem~\ref{theorem_all}, from which Theorems~\ref{thm:top-heavy}, ~\ref{thm:KL}, and ~\ref{prop:kahler} follow.  We also show in Section~\ref{subsec:monotonicity} how Lemma~\ref{mult by y} can be used to deduce Theorem~\ref{thm:monotonicity}.  Finally, the appendix establishes the framework needed to deduce Theorem~\ref{thm:equivariantKL} as well as the equivariant part of Theorem~\ref{thm:top-heavy}.

\subsection*{Acknowledgements}
The authors would like to thank both the Institute for Advanced Study and the Korea Institute for Advanced Study for their
hospitality during the preparation of this paper.
We thank Matt Baker and Richard Stanley for helpful comments, and we thank the anonymous referees for giving the paper a very careful reading and making suggestions that have greatly improved the paper.

\section{The Chow ring and the augmented Chow ring of a matroid}\label{sec:Chow}

For the remainder of this paper, we write $d$ for the rank of $\M$ and $n$ for the cardinality of $E$.
We continue to assume that $\M$ is a loopless matroid on $E$.
Under this assumption, $n$ is positive if and only if $d$ is positive.

\subsection{Definitions of the rings} \label{subsec:therings}

We  recall the definitions of the Chow ring of a matroid introduced  in \cite{FY} and the augmented Chow ring of a matroid introduced in \cite{BHMPW}.
To each matroid $\M$ on $E$, we assign two polynomial rings with rational coefficients
\begin{align*}
\underline{S}_\M &\coloneq \mathbb{Q}[x_F\,  |\,  \text{$F$ is a nonempty proper flat of $\mathrm{M}$}] \ \ \text{and} \ \ \\
S_\M & \coloneq \mathbb{Q}[x_F \, | \, \text{$F$ is a proper flat of $\mathrm{M}$}] \otimes \mathbb{Q}[y_i \, | \,  \text{$i$ is an element of $E$}].
\end{align*}
Note that $S_\M$ contains a variable $x_\varnothing$, while $\underline{S}_\M$ does not.

\begin{definition}
The {\bf Chow ring} of $\mathrm{M}$ is the quotient algebra
\[
\underline{\mathrm{CH}}(\mathrm{M}) \coloneq \underline{S}_\M/(\underline{A}_\M+\underline{B}_\M),
\]
where $\underline{A}_\M$ is the ideal generated by the linear forms
\[
\sum_{i_1 \in F} x_F -\sum_{i_2 \in F} x_F, \ \ \text{for every pair of distinct elements $i_1$ and $i_2$ of $E$},
\]
and $\underline{B}_\M$  is the ideal generated by the quadratic monomials
\[
x_{F_1}x_{F_2}, \ \ \text{for every pair of incomparable nonempty proper flats $F_1$ and $F_2$ of $\mathrm{M}$.}
\]
\end{definition}
This ring has two distinguished classes
\[
\underline{\alpha}=
\underline{\alpha}_{\mathrm{M}} \coloneq  \sum_{i \in G} x_G \in \underline{\mathrm{CH}}^1(\mathrm{M}),
\]
where the sum is over all nonempty proper flats $G$ of $\mathrm{M}$ containing a given element $i$ in $E$, and
\[
\ub=\ub_{\mathrm{M}} \coloneq \sum_{i \notin G} x_G  \in \underline{\mathrm{CH}}^1(\mathrm{M}),
\]
where the sum is over all nonempty proper flats $G$ of $\mathrm{M}$ not containing a given element $i$ in $E$.
The fact that $\underline{\alpha}$ and $\underline{\beta}$ do not depend on $i$ follows from the relations imposed by the generators of $\underline{A}_\M$.
When $d$ is positive, 
the Chow ring of $\M$ is the  Chow ring of an $(n-1)$-dimensional smooth toric variety defined by a $(d-1)$-dimensional fan $\underline{\Pi}_\M$,  called the {\bf Bergman fan} of $\M$  \cite[Theorem 3]{FY}.  In particular, this implies that it is spanned as a $\Q$-vector space by the classes of square-free monomials.

\begin{definition}
The {\bf augmented Chow ring} of $\mathrm{M}$ is the quotient algebra
\[
\mathrm{CH}(\mathrm{M}) \coloneq S_\M/ (A_\M+B_\M),
\]
where $A_\M$ is the ideal generated by the linear forms
 \[
y_i - \sum_{i \notin F} x_F,  \ \  \text{for every element  $i$ of $E$},
\]
and $B_\M$ is the ideal generated by  the quadratic monomials
\begin{align*}
x_{F_1}x_{F_2}, \ \  &\text{for every pair of incomparable proper flats $F_1$ and $F_2$ of $\mathrm{M}$, and}\\
y_i \hspace{0.5mm} x_F,  \ \  &\text{for every element $i$ of $E$ and every proper flat  $F$ of $\mathrm{M}$ not containing $i$.}
\end{align*}
\end{definition}
This ring has a distinguished class
\[
\alpha=\alpha_{\mathrm{M}} \coloneq \sum_G x_G \in \mathrm{CH}^1(\mathrm{M}),
\]
where the sum is over all proper flats $G$ of $\mathrm{M}$.
The augmented Chow ring of $\M$ is the Chow ring of an $n$-dimensional smooth toric variety defined by a $d$-dimensional fan $\Pi_\M$,  called the {\bf augmented Bergman fan} of $\M$  \cite[Proposition 2.12]{BHMPW}.
As in the case of the Chow ring, this implies that it is spanned as a $\Q$-vector space by the classes of square-free monomials.

\begin{lemma}\label{CH-to-uCH}
There is an isomorphism $\CH(\M)/\langle y_i\mid i\in E\rangle\cong\uCH(\M)$
that sends $\alpha$ to $\underline{\alpha}$, $x_\varnothing$ to $-\underline{\beta}$, and 
$x_F$ to $x_F$ for any nonempty proper flat $F$.
\end{lemma}

\begin{proof}
For any $i\in E$, we have
$$x_\varnothing = -\sum_{i\notin F\neq\varnothing} x_F\in \CH(\M)/\langle y_i\mid i\in E\rangle.$$
One of these relations can be used to rewrite $x_\varnothing$ in terms of the generators $x_F$ for $F$ proper
and nonempty, and the differences between these relations for various $i\in E$
are precisely those relations in $\uCH(\M)$ coming from the ideal $\underline{A}_\M$.
\end{proof}

Recall that we defined the graded M\"obius algebra $\H(\M)$ in Section \ref{sec:strategy}.
For any $i\in E$, let $\overline{i}$ denote the unique rank $1$ flat of $\M$ containing $i$.
By \cite[Lemma 2.11(2)]{BHMPW}, we have $y_i = y_j \in\CH(\M)$ if and only if $\overline{i} = \overline{j}$.
By \cite[Proposition 2.18]{BHMPW}, there is an injection $\H(\M)\to\CH(\M)$ of graded algebras
taking $y_{\overline{i}}$ to $y_i$ for all $i\in E$.  (In particular, we have $y_i^2 = 0$ for all $i\in E$.)
Thus, we may identify the graded M\"obius algebra with the subalgebra of the augmented Chow ring  generated by the classes $\{y_i\mid i\in E\}$ and define $y_F \coloneq \prod_{i \in I}y_i \in \CH(\M)$ for any flat $F$ of $\M$ and any basis $I$ of $F$; see \cite[Lemma 2.11]{BHMPW} and the discussion thereafter for well-definedness of these $y_F$.
One of the principal goals of this paper is to understand  the $\H(\M)$-module structure of  $\CH(\M)$.
The Chow ring $\uCH(\M)$ will play an important supporting role.

\begin{example}\label{ex:rank2_1}
Let $\M$ be the rank $2$ uniform matroid on  $E = \{1,2,3\}$. The Chow ring of $\M$ is the quotient 
\[
\uCH(\M) = \mathbb{Q}[x_1,x_2,x_3]/\langle x_1x_2,x_1x_3,x_2x_3,x_1-x_2,x_1-x_3,x_2-x_3 \rangle \cong \mathbb{Q}[x]/\langle x^2 \rangle,
\]
where $x$ is the class of any one of the $x_i$'s.
The augmented Chow ring of $\M$ is the quotient 
\[
\CH(\M)
=
\frac{\mathbb{Q}[x_\varnothing, x_{1}, x_{2}, x_{3}, y_{1}, y_{2}, y_{3}]}
{
\langle\,
x_\varnothing y_{1},\, x_\varnothing y_{2},\, x_\varnothing y_{3},\,
x_{1}y_{2},\, x_{1}y_{3},\, x_{2}y_{1},\, x_{2}y_{3},\, x_{3}y_{1},\, x_{3}y_{2},x_{1}x_{2},\, x_{1}x_{3},\, x_{2}x_{3},\,
\ell_{1},\, \ell_{2},\, \ell_{3}
\,\rangle},
\]
where $\ell_1,\ell_2,\ell_3$ are the linear forms
$y_1-x_\varnothing-x_2-x_3$,  $y_2-x_\varnothing-x_1-x_3$,  $y_3-x_\varnothing-x_1-x_2$, respectively. 
\end{example}

In general, the description of $\uCH(\M)$ in terms of  $\underline{\Pi}_\M$ reveals that $\uCH(\M)$ vanishes  in degrees $\ge d$.
Similarly, the description of $\CH(\M)$ in terms of $\Pi_\M$ reveals that $\CH(\M)$ vanishes in degrees $>d$.
Furthermore, one can construct distinguished isomorphisms from  the graded pieces $\uCH^{d-1}(\M)$ and $\CH^d(\M)$  to $\mathbb{Q}$.

\begin{definition}\label{DefinitionDegreemap}
Let $\M$ be a matroid of rank $d$.
\begin{enumerate}[(1)]\itemsep 5pt
\item When $d$ is positive, we define the {\bf degree map} for $\uCH(\M)$ to be the unique linear map
\[
\udeg_\M \colon \uCH^{d-1}(\M)\longrightarrow \mathbb{Q}, \qquad \prod_{F\in \underline{\mathscr{F}}}x_F\longmapsto 1,
\]
where $\underline{\mathscr{F}}$ is any complete flag of nonempty proper flats of $\M$. 
We define the {\bf Poincar\'e pairing} on $\uCH(\M)$
by the formula $(\eta_1,\eta_2)\mapsto \udeg_{\M}(\eta_1\eta_2)$.
\item We define the {\bf degree map} for $\CH(\M)$ to be the unique linear map
\[
{\deg}_\M \colon \CH^{d}(\M)\longrightarrow \mathbb{Q}, \qquad \prod_{F\in \mathscr{F}}x_F\longmapsto 1,
\]
where $\mathscr{F}$ is any complete flag of proper flats of $\M$.  We define the {\bf Poincar\'e pairing} on $\CH(\M)$
by the formula $(\eta_1,\eta_2)\mapsto \deg_{\M}(\eta_1\eta_2)$.
\end{enumerate}
By \cite[Proposition 2.8]{BHMPW}, the degree maps are unique, well-defined, and bijective.
The degree map on the augmented Chow ring has the additional property that $\deg_{\M}(y_E) = 1$ \cite[Section 2.4]{BHMPW}.
\end{definition}

\begin{example}\label{ex:rank2_1b}
We continue analyzing the case where $\M$ is uniform of rank 2 on $E = \{1,2,3\}$, as in Example \ref{ex:rank2_1}.
Using the relations $\ell_i$, one can express each $y_i$ as a linear combination of classes of the form
$x_F$, which demonstrates that $\{x_F\mid \text{$F$ a proper flat}\}$ is a basis for $\CH^1(\M)$.
However, to emphasize the $\H(\M)$-module structure, we prefer to take a different approach.
Indeed, we can solve for $x_i$ in terms of $y_i$ because $2$ is invertible in $\mathbb{Q}$,
and we obtain the vector space decompositions
\[
\CH^0(\M)=\mathbb{Q}1, \ \  \CH^1(\M)= \mathbb{Q}  y_1  \oplus \mathbb{Q}  y_2 \oplus \mathbb{Q}  y_3 \oplus \mathbb{Q}  x_\varnothing, \ \ \CH^2(\M)= \mathbb{Q}  y_1y_2.
\]
Note that $y_1y_2 = y_1y_3=y_2y_3 = y_E$.  We also have $x_\varnothing x_1 = y_1x_1 = y_1y_2$,
using the relations $\ell_1$ and $\ell_2$.  More generally, for any matroid $\M$ and any complete flag
$\varnothing = F_0\subsetneq F_1\subsetneq\cdots\subsetneq F_{d-1}$ of proper flats, 
we know that $\deg_{\M}(x_{F_0}x_{F_1}\cdots x_{F_{d-1}}) = 1 = \deg_{\M}(y_E)$,
and therefore $x_{F_0}x_{F_1}\cdots x_{F_{d-1}} = y_E$.

Since $x_\varnothing y_i=0$ in $\CH(\M)$, we have a decomposition of  $\CH(\M)$ into indecomposable  $\H(\M)$-modules 
\[
\CH(\M)=\H(\M) \oplus \mathbb{Q}x_\varnothing.
\]
Since the degree of $y_iy_j$ is equal to 1 for all $i\neq j$, the Gram matrix for the Poincar\'e pairing on
$\H^1(\M)\subseteq\CH^1(\M)$ has determinant 2, and the restriction of the pairing to $\H^1(\M)$
is therefore non-degenerate.  Since $x_\varnothing^2 = -2y_E$, the restriction of the pairing to $\Q x_\varnothing$
is also non-degenerate.  (See Example \ref{ex:uniform2} for a discussion of the degree of the top power
of $x_\varnothing$ for a general matroid.)
One can check that the same decomposition holds for any loopless rank $2$ matroid $\M$ provided that $n \neq 1$ in $\mathbb{Q}$, where $n$ is the number of rank $1$ flats of $\M$.\footnote{The $\H(\M)$-module structure of $\CH(\M)$ for a matroid $\M$ over coefficient fields of nonzero characteristic will be investigated in the forthcoming paper \cite{BHMPW-IH}.}
\end{example}

\subsection{The pullback and pushforward maps}\label{sec:pbpf}
In this subsection, we assume that $E$ is nonempty.
%
Let $F$ be a proper flat of $\M$.
The following definition is motivated by the geometry of augmented Bergman fans \cite[Propositions 2.20 and 2.21]{BHMPW}.

\begin{definition}\label{DefinitionXPullback}
The {\bf pullback} $\varphi^F = \varphi^F_\M$ is the unique graded algebra homomorphism
\[
\mathrm{CH}(\mathrm{M}) \longrightarrow  \underline{\mathrm{CH}}(\mathrm{M}_F) \otimes \mathrm{CH}(\mathrm{M}^F) 
\]
that satisfies the following properties:
\begin{enumerate}[$\bullet$]\itemsep 5pt
\item If $G$ is a flat properly contained in $F$, then $\varphi^F_\Mo(x_G)=1 \otimes x_G$.
\item If $G$ is a flat properly containing $F$, then $\varphi^F_\Mo(x_G)=x_{G \setminus F} \otimes 1$.
\item If $G$ is a flat incomparable to $F$, then $\varphi^F_\Mo(x_G)=0$.
\item If $G$ is the flat $F$, then $\varphi^F_\Mo(x_F)=-1\otimes\alpha_{\mathrm{M}^F}-\ub_{\mathrm{M}_F} \otimes 1$.
\end{enumerate}
The {\bf pushforward} $\psi^F_\Mo$ is the unique linear map 
\[
 \underline{\mathrm{CH}}(\mathrm{M}_F)  \otimes \mathrm{CH}(\mathrm{M}^F) \longrightarrow \mathrm{CH}(\mathrm{M})
 \]
that maps the monomial $\prod_{F'} x_{F' \setminus F}  \otimes \prod_{F''} x_{F''}$ to the monomial  $x_F \prod_{F'} x_{F'} \prod_{F''} x_{F''}$.
That is, it concatenates flags below $F$ with flags above $F$.
Note that this map increases degree by one.
\end{definition}

Of particular importance will be the pullback $\varphi^\varnothing_\Mo:\CH(\M)\to \uCH(\M)$,
which is the surjective graded algebra homomorphism that induces the isomorphism of Lemma \ref{CH-to-uCH}.
The following results can be found in \cite[Section 2]{BHMPW}. 

\begin{proposition}\label{lemma_xdegree}
The pullback $\varphi^F_\Mo$ and the pushforward $\psi^F_\Mo$ have the following properties:
\begin{enumerate}[(1)]\itemsep 5pt
\item If $i$ is an element of $F$, then $\varphi^F_\Mo(y_i)=1 \otimes y_i$. 
\item If $i$ is not an element of $F$, then $\varphi^F_\Mo(y_i)=0$. 
\item The equality $\varphi^F_\Mo(\alpha_\Mo)=\underline{\alpha}_{\mathrm{M}_F}\otimes 1$ holds. 
\item The pullback $\varphi_\Mo^F$ is surjective.
\item  The pushforward $\psi_\Mo^F$ is injective.
\item The pushforward $\psi^F_\Mo$ commutes with the degree maps: 
$
\udeg_{\mathrm{M}_F} \otimes \deg_{\mathrm{M}^F} = \deg_\M \circ\; \psi_\Mo^F.
$
\item The pushforward $\psi_\Mo^F$ is a homomorphism of  $\mathrm{CH}(\mathrm{M})$-modules:\footnote{For realizable matroids, this is an instance of the projection formula; see Section \ref{CD geometry}.}
\[
\eta\psi_\Mo^F(\xi)=\psi_\Mo^F\big(\varphi_\Mo^F(\eta)\xi\big) \ \ \text{for any $\eta\in \mathrm{CH}(\mathrm{M})$ and $\xi \in  \underline{\mathrm{CH}}(\mathrm{M}_F)  \otimes \mathrm{CH}(\mathrm{M}^F)$.}
\]
\end{enumerate}
\end{proposition}

We use the pullback map to make $\underline{\mathrm{CH}}(\mathrm{M}_F) \otimes \mathrm{CH}(\mathrm{M}^F)$ into a module over $\CH(\M)$ and $\H(\M)$.  
By part (1) of the above proposition,  $\H(\M)$
acts only on the second tensor factor.  

For later use, we record here the following immediate consequence of  Proposition \ref{lemma_xdegree}. 

\begin{lemma}\label{adjoint}
For any $\eta\in \mathrm{CH}(\mathrm{M})$ and $\xi \in  \underline{\mathrm{CH}}(\mathrm{M}_F)  \otimes \mathrm{CH}(\mathrm{M}^F)$, we have
\[
\deg_\M \left(\eta\psi_\Mo^F(\xi)\right)= \udeg_{\mathrm{M}_F} \otimes \deg_{\mathrm{M}^F}\left(\varphi_\Mo^F(\eta)\xi\right).
\]
\end{lemma}

Since the pushforward $\psi^F_\Mo$ is injective, the  statement below shows that the graded $\mathrm{CH}(\mathrm{M})$-module $ \underline{\mathrm{CH}}(\mathrm{M}_F) \otimes \mathrm{CH}(\mathrm{M}^F)[-1]$ 
is isomorphic to  the principal ideal of $x_F$ in $\mathrm{CH}(\mathrm{M})$.

\begin{proposition}[{\cite[Proposition 2.21]{BHMPW}}]\label{Prop_xmult}\hfill
\begin{enumerate}[(1)]\itemsep 5pt
\item The composition $\psi^F_\Mo\circ \varphi^F_\Mo \colon \CH(\M)\to \CH(\M)$ is multiplication by $x_F$.
\item The composition $\varphi^F_\Mo \circ \psi^F_\Mo  \colon \underline{\mathrm{CH}}(\mathrm{M}_F)  \otimes \mathrm{CH}(\mathrm{M}^F)\to \underline{\mathrm{CH}}(\mathrm{M}_F)  \otimes \mathrm{CH}(\mathrm{M}^F)$ is multiplication by $\varphi^F_\Mo(x_F)$.
\end{enumerate}
\end{proposition}


We next introduce the analogous maps for Chow rings (rather than augmented Chow rings).
Let $F$ be a nonempty proper flat of $\M$.
The following definition is motivated by the geometry of Bergman fans \cite[Propositions 2.24 and 2.25]{BHMPW}.

\begin{definition}\label{DefinitionUnderlinedPullPush}
The {\bf pullback} $\uvarphi^F = \uvarphi^F_\M$ is the unique graded algebra homomorphism
\[
\underline{\mathrm{CH}}(\mathrm{M}) \longrightarrow  \underline{\mathrm{CH}}(\mathrm{M}_F) \otimes \underline{\mathrm{CH}}(\mathrm{M}^F) 
\]
that satisfies the following properties:
\begin{enumerate}[$\bullet$]\itemsep 5pt
\item If $G$ is a flat properly contained in $F$, then $\uvarphi^F_\Mo(x_G)=1 \otimes x_G$.
\item If $G$ is a flat properly containing $F$, then $\uvarphi^F_\Mo(x_G)=x_{G \setminus F} \otimes 1$.
\item If $G$ is a flat incomparable to $F$, then $\uvarphi^F_\Mo(x_G)=0$.
\item If $G$ is the flat $F$, then $\uvarphi^F_\Mo(x_F)=-1\otimes\underline{\alpha}_{\mathrm{M}^F}-\ub_{\mathrm{M}_F} \otimes 1$.
\end{enumerate}
The {\bf pushforward} $\upsi^F_\Mo$ is the unique linear map
\[
 \underline{\mathrm{CH}}(\mathrm{M}_F)  \otimes \underline{\mathrm{CH}}(\mathrm{M}^F) \longrightarrow \underline{\mathrm{CH}}(\M)
 \]
that maps the monomial $\prod_{F'} x_{F' \setminus F}  \otimes \prod_{F''} x_{F''}$ to the monomial  $x_F \prod_{F'} x_{F'} \prod_{F''} x_{F''}$.
Like $\psi^F$, this map increases degree by one.
\end{definition}
The following analogue of Proposition \ref{lemma_xdegree} can be found in  \cite[Section 2]{BHMPW}.

\begin{proposition}\label{lemma_uab}
The pullback $\uvarphi^F_\Mo$ and the pushforward $\upsi^F_\Mo$ have the following properties:
\begin{enumerate}[(1)]\itemsep 5pt
\item We have
$
\uvarphi^F_\Mo(\underline{\alpha}_\Mo)=\underline{\alpha}_{\mathrm{M}_F}\otimes 1$
and  $\uvarphi^F_\Mo(\ub_\Mo)=1 \otimes \ub_{\mathrm{M}^F}$.
\item The pullback $\uvarphi^F_\Mo$ is surjective.
\item The pushforward $\upsi^F_\Mo$ is  injective.
\item The pushforward $\upsi^F_\Mo$ commutes with the degree maps: 
$
\udeg_{\mathrm{M}_F} \otimes \udeg_{\mathrm{M}^F} = \udeg_\M \circ \upsi_\Mo^F$.
\item The pushforward $\upsi_\Mo^F$ is a homomorphism of  $\uCH(\mathrm{M})$-modules:
\[
\eta\upsi_\Mo^F(\xi)=\upsi_\Mo^F\big(\uvarphi_\Mo^F(\eta)\xi\big) \ \ \text{for any $\eta\in \underline{\mathrm{CH}}(\mathrm{M})$ and $\xi \in  \underline{\mathrm{CH}}(\mathrm{M}_F)  \otimes \underline{\mathrm{CH}}(\mathrm{M}^F)$.}
\]
\end{enumerate}
\end{proposition}



The following analogue of Lemma \ref{adjoint} immediately follows from Proposition \ref{lemma_uab}. 

\begin{lemma}\label{underline adjoint}
For any $\eta\in \underline{\mathrm{CH}}(\mathrm{M})$ and $\xi \in  \underline{\mathrm{CH}}(\mathrm{M}_F)  \otimes \underline{\mathrm{CH}}(\mathrm{M}^F)$, we have
\[
\udeg_{\M}\left(\eta\upsi_\Mo^F(\xi)\right)= \udeg_{\mathrm{M}_F} \otimes \udeg_{\mathrm{M}^F}\left(\uvarphi_\Mo^F(\eta)\xi\right).
\]
\end{lemma}

Since the pushforward $\upsi^F_\Mo$ is injective, the  statement below shows that  the graded $\underline{\mathrm{CH}}(\mathrm{M})$-module $ \underline{\mathrm{CH}}(\mathrm{M}_F) \otimes \underline{\mathrm{CH}}(\mathrm{M}^F)[-1]$ 
is isomorphic to  the principal ideal of $x_F$ in $\underline{\mathrm{CH}}(\mathrm{M})$.

\begin{proposition}[{\cite[Proposition 2.25]{BHMPW}}]\label{Prop_uxmult}
\hfill
\begin{enumerate}[(1)]\itemsep 5pt
\item The composition $\upsi^F_\Mo\circ \uvarphi^F_\Mo \colon \underline{\CH}(\M)\to \underline{\CH}(\M)$ is multiplication by $x_F$.
\item The composition 
$\uvarphi^F_\Mo \circ \upsi^F_\Mo\colon \underline{\mathrm{CH}}(\mathrm{M}_F)  \otimes \underline{\mathrm{CH}}(\mathrm{M}^F)\to \underline{\mathrm{CH}}(\mathrm{M}_F)  \otimes \underline{\mathrm{CH}}(\mathrm{M}^F)$ is multiplication by $\uvarphi^F_\Mo(x_F)$. 
\end{enumerate}
\end{proposition}

Finally, we introduce a third flavor of pullback and pushforward maps, this time relating the augmented
Chow ring of $\mathrm{M}$ to the augmented Chow ring of $\mathrm{M}_F$ for any flat $F$ of $\M$, with no tensor products.  Whereas the previous pushforward and pullback maps gave a factorization of multiplication by a generator $x_F$, the maps we now describe give a factorization of multiplication by $y_F$.
The notational difference is that $F$ is now in the subscript rather than the superscript. The following definition can be found in \cite[Propositions 2.28 and 2.29]{BHMPW}.

\begin{definition}\label{DefinitionYPullPush}
The {\bf pullback} $\varphi^\Mo_F = \varphi_F^\M$ is the unique graded algebra homomorphism
\[
\mathrm{CH}(\mathrm{M}) \longrightarrow  \mathrm{CH}(\mathrm{M}_F) 
\]
that satisfies the following properties:
\begin{enumerate}[$\bullet$]\itemsep 5pt
\item If $G$ is a proper flat containing $F$, then $\varphi_F^\Mo(x_G)=x_{G\setminus F}$.
\item If $G$ is a proper flat not containing $F$, then $\varphi_F^\Mo(x_G)=0$.
\end{enumerate}
The {\bf pushforward}  $\psi_F^\Mo = \psi_F^\M$ is the unique degree $\rk F$ linear map 
\[
\mathrm{CH}(\mathrm{M}_F) \longrightarrow \mathrm{CH}(\mathrm{M})
\]
that maps the monomial $ \prod_{F'} x_{F' \setminus F}$  to the monomial $y_F  \prod_{F'} x_{F'}$.
\end{definition}

The next results can be found  in \cite[Section 2]{BHMPW}. 
\begin{proposition}\label{lemma_degy}
The pullback $\varphi_F^\Mo$ and the pushforward $\psi_F^\Mo$ have the following properties:
\begin{enumerate}[(1)]\itemsep 5pt
\item If $i$ is an element of $F$, then $\varphi_F^\Mo(y_i)=0$. 
\item If $i$ is not an element of $F$, then $\varphi_F^\Mo(y_i)=y_i$.
\item The equality $\varphi_F^\Mo(\alpha_\Mo)=\alpha_{\mathrm{M}_F}$ holds.
\item The pullback $\varphi_F^\Mo$ is surjective.
\item The pushforward $\psi_F^\Mo$ is injective.
\item The pushforward $\psi_F^\Mo$ commutes with the degree maps: 
 $
 \deg_{\mathrm{M}_F}=\deg_\M \circ\  \psi_F^\Mo$.
\item The pushforward $\psi_F^\Mo$ is a  homomorphism of $\mathrm{CH}(\mathrm{M})$-modules: 
\[
\eta\psi_F^\Mo(\xi)=\psi_F^\Mo\big(\varphi_F^\Mo(\eta)\xi\big) \ \ \text{for any $\eta\in \mathrm{CH}(\mathrm{M})$ and $\xi \in \mathrm{CH}(\mathrm{M}_F)$.}
\]
\end{enumerate}
\end{proposition}





The following analogue of Lemmas \ref{adjoint} and \ref{underline adjoint} follows from Proposition \ref{lemma_degy}. 

\begin{lemma}\label{subscript adjoint}
For any $\eta\in \mathrm{CH}(\mathrm{M})$ and $\xi \in \mathrm{CH}(\mathrm{M}_F)$, we have
\[
\deg_\M\left(\eta\,\psi_F^\Mo(\xi)\right)=\deg_{\M_F}\left(\varphi_F^\Mo(\eta)\xi\right).
\]
\end{lemma}

Since the pushforward $\psi_F^\Mo$ is injective, the statement below shows that the graded $\mathrm{CH}(\mathrm{M})$-module $\mathrm{CH}(\mathrm{M}_F)[-\rk F]$ is isomorphic to the principal ideal of $y_F$ in $\mathrm{CH}(\mathrm{M})$.

\begin{proposition}\label{Prop_ymult}
The composition $\psi_F^\Mo\circ \varphi_F^\Mo \colon \CH(\M)\to \CH(\M)$ is multiplication by $y_F$. 
\end{proposition}


\begin{corollary}\label{cor:y_FN is an H(M_F) module}
	The homomorphism $\varphi_F$ restricts to a surjection $\H(\M) \to \H(\M_F)$ whose kernel is the annihilator of $y_F$.
	Thus for any $\H(\M)$-module $\N$, the submodule $y_F\N$ can naturally be regarded as an $\H(\M_F)$-module.
\end{corollary}

\subsection{New lemmas}
Until now, everything that has appeared in Section \ref{sec:Chow} was proved in \cite{BHMPW}.
In this section, we state a few additional lemmas about the pushforward and pullback maps that will be needed in this paper.

The following lemma will be needed for the proof of Proposition \ref{summands}.

\begin{lemma}\label{lemma_incomparable}
Suppose that $F$ and $G$ are incomparable proper flats of $\M$. Then
\[
\varphi^G_\Mo\psi^F_\Mo=0 \ \ \text{and} \ \ \uvarphi_\Mo^G\upsi^F_\Mo=0. 
\]
\end{lemma}
\begin{proof}
We only prove the first equality. The second one follows from the same arguments. By Definition \ref{DefinitionXPullback} and Proposition \ref{lemma_xdegree}, 
the pushforward $\psi^G_\Mo$ is injective and the pullback $\varphi_\Mo^F$ is surjective. Thus, it is sufficient to show $\psi^G_\Mo\varphi_\Mo^G\psi^F_\Mo\varphi_\Mo^F=0$. Since the compositions $\psi^G_\Mo\varphi_\Mo^G$ and $\psi^F_\Mo\varphi_\Mo^F$ are equal to the multiplications by $x_G$ and $x_F$ respectively (Proposition \ref{Prop_xmult}), the assertion follows because $x_G x_F=0$ in $\CH(\M)$.
\end{proof}

The next lemma will be used in the proofs of Propositions \ref{reduction}, \ref{uHL1CD}, \ref{alt}, and \ref{block1}.

\begin{lemma}\label{we need this}
Let $F$ be a proper flat of $\M$.
\begin{enumerate}[(1)]\itemsep 5pt
\item For any $\mu,\nu\in\uCH(\M_F)\otimes \CH(\M^F)$, we have
$$\deg_\M\left(\psi^F_\Mo(\mu)\cdot\psi^F_\Mo(\nu)\right) =-\udeg_{\M_F}\otimes \deg_{\M^F}\bigl((\ub_{\M_F}\otimes 1+1\otimes \alpha_{\M^F})\mu\nu\bigr).$$
\item When $F$ is nonempty, for any  $\mu,\nu\in\uCH(\M_F)\otimes \uCH(\M^F)$, we have
$$\udeg_\M\left(\upsi^F_\Mo(\mu)\cdot\upsi^F_\Mo(\nu)\right) =-\udeg_{\M_F}\otimes \udeg_{\M^F}\bigl((\ub_{\M_F}\otimes 1+1\otimes \ua_{\M^F})\mu\nu\bigr).$$
\end{enumerate}
\end{lemma}

\begin{proof}
We prove only part (1); the proof of part (2) is identical.
By Proposition \ref{lemma_xdegree} (6) and (7), 
we have
\[
\deg_{\M}\left(\psi^F_\Mo(\mu)\cdot\psi^F_\Mo(\nu)\right)
= \udeg_{\M_F}\otimes \deg_{\M^F}\left(\varphi_\Mo^F\psi^F_\Mo(\mu)\cdot\nu\right).
\]
By Definition \ref{DefinitionXPullback}  and Proposition \ref{Prop_xmult}, 
\[
\varphi_\Mo^F\psi^F_\Mo(\mu)\cdot\nu=\varphi_\Mo^F(x_F)\cdot \mu\nu=-(\ub_{\M_F}\otimes 1+1\otimes \alpha_{\M^F})\mu\nu.
\]
Thus, 
\[
\deg_{\M}\left(\psi^F_\Mo(\mu)\cdot\psi^F_\Mo(\nu)\right)=-\udeg_{\M_F}\otimes \deg_{\M^F}\left( (\ub_{\M_F}\otimes 1+1\otimes \alpha_{\M^F})\mu\nu \right).\qedhere
\]
\end{proof}

For later use, we collect here useful commutative diagrams involving the pullback and the pushforward maps.

\begin{lemma}\label{lemma_diagrams}
Let $F$ be a proper flat of $\M$.
\begin{enumerate}[(1)] \itemsep 5pt
\item The following diagram commutes:
\[
\begin{tikzcd}
\uCH(\M_{F})\otimes \CH(\M^{F}) \arrow{r}{\id \otimes \varphi^\varnothing_{\M^F}} \arrow{d}{\psi^{F}_\M} &\uCH(\M_{F})\otimes \uCH(\M^{F}) \arrow{d}{\upsi^{F}_\M}\arrow{r}{\id \otimes \psi^\varnothing_{\M^F}} &\uCH(\M_{F})\otimes \CH(\M^{F}) \arrow{d}{\psi^{F}_\M} \\
 \CH(\M)\arrow{r}{\varphi_\M^\varnothing} & \uCH(\M)\ar{r}{\psi^\varnothing_\M} & \CH(\M).
\end{tikzcd}
\]
\item More generally, for any flat $G < F$, the following diagram commutes:
\[
\begin{tikzcd}
\uCH(\M_{F})\otimes \CH(\M^{F}) \arrow{r}{\id\otimes \varphi_{\M^F}^{G}} \arrow{d}{\psi^{F}_\M} &\uCH(\M_{F})\otimes \uCH(\M^{F}_{G})\otimes \CH(\M^{G})\arrow{d}{\upsi^{F\setminus G}_{\M_{G}}\otimes \id}\arrow{r}{\id\otimes \psi_{\M^F}^{G}}  & \uCH(\M_{F})\otimes \CH(\M^{F})\ar{d}{\psi^{F}_\M}\\
 \CH(\M)\arrow{r}{\varphi_\M^{G}} & \uCH(\M_{G})\otimes \CH(\M^{G})\ar{r}{\psi^{G}_\M} & \CH(\M).
\end{tikzcd}
\]
\item For any nonempty flat $G < F$, the following diagram commutes:
\[
\begin{tikzcd}
\uCH(\M_{F})\otimes \uCH(\M^{F}) \arrow{r}{\id\otimes \uvarphi_{\M^F}^{G}} \arrow{d}{\upsi^{F}_\M} &\uCH(\M_{F})\otimes \uCH(\M^{F}_{G})\otimes \uCH(\M^{G})\arrow{d}{\upsi^{F\setminus G}_{\M_{G}}\otimes \id}\arrow{r}{\id\otimes \upsi_{\M^F}^{G}}  & \uCH(\M_{F})\otimes \uCH(\M^{F})\ar{d}{\upsi^{F}_\M}\\
 \uCH(\M)\arrow{r}{\uvarphi_\M^{G}} & \uCH(\M_{G})\otimes \uCH(\M^{G})\ar{r}{\upsi^{G}_\M} & \uCH(\M).
\end{tikzcd}
\]
\item For any flat $F \le G$, the following diagram commutes:
\[
\begin{tikzcd}
\uCH(\M_G)\otimes \CH(\M^G)  \ar[d,"\psi^G_{\M}"] \ar[r,"\id \otimes \varphi_F^{\M^G}"] & \uCH(\M_G) \otimes \CH(\M^G_F) \ar[d,"\psi^{G\setminus F}_{\M_F}"] \ar[r,"\id \otimes \psi_F^{\M^G}"] & \uCH(\M_G)\otimes \CH(\M^G) \ar[d,"\psi^G_{\M}"]\\
\CH(\M) \ar[r,"\varphi^\M_F"] & \CH(\M_F) \ar[r,"\psi_F^\M"] & \CH(\M).
\end{tikzcd}
\]
\end{enumerate}
\end{lemma}

We omit the proof, which is a straightforward computation.

\subsection{Hodge theory of the Chow ring and the augmented Chow ring}

The following results appear in \cite[Theorem 1.6]{BHMPW}.  They will be used in conjunction with Proposition \ref{why hancock?} to deduce that $\uCH(\M)$ and $\CH(\M)$ satisfy the Hancock condition of Section \ref{sec:hancock}.

\begin{theorem}\label{TheoremChowKahlerPackage}
Let $\mathrm{M}$ be a matroid on $E$.   There is a nonempty subset $\mathscr{K}(\mathrm{M})\subseteq\CH^1(\M)$, defined by a finite collection of linear inequalities, with the property that, for any $\ell\in\mathscr{K}(\mathrm{M})$, the following statements hold.
\begin{enumerate}[(1)]\itemsep 5pt
\item (Poincar\'e duality theorem) For every nonnegative integer $k \le d/2$, the bilinear pairing
\[
\mathrm{CH}^k(\mathrm{M})  \times \mathrm{CH}^{d-k}(\mathrm{M})  \longrightarrow \mathbb{Q}, \quad (\eta_1,\eta_2) \longmapsto \deg_\M(\eta_1 \eta_2)
\]
is non-degenerate.
\item  (Hard Lefschetz theorem)  For every nonnegative integer $k \le d/2$, the multiplication map
\[
\mathrm{CH}^k(\mathrm{M})  \longrightarrow  \mathrm{CH}^{d-k}(\mathrm{M}), \quad \eta \longmapsto \ell^{d-2k} \eta
\]
is an isomorphism.
\item (Hodge--Riemann relations) For every nonnegative integer $k \le d/2$, the bilinear form
\[
\mathrm{CH}^k(\mathrm{M})  \times \mathrm{CH}^{k}(\mathrm{M})  \longrightarrow \mathbb{Q}, \quad (\eta_1,\eta_2) \longmapsto (-1)^k \deg_\M( \ell^{d-2k} \eta_1 \eta_2)
\]
is positive definite on the kernel of the multiplication by $\ell^{d-2k+1}$.
\end{enumerate}
\end{theorem}

\begin{theorem}\label{theorem_underlineKahler}
Let $\mathrm{M}$ be a matroid on $E$.   There is a nonempty subset $\underline{\mathscr{K}}(\mathrm{M})\subseteq\uCH^1(\M)$, defined by a finite collection of linear inequalities, with the property that, for any $\underline{\ell}\in\underline{\mathscr{K}}(\mathrm{M})$, the following statements hold.
\begin{enumerate}[(1)]\itemsep 5pt
\item (Poincar\'e duality theorem) For every nonnegative integer $k < d/2$, the bilinear pairing
\[
\underline{\mathrm{CH}}^k(\mathrm{M})  \times \underline{\mathrm{CH}}^{d-k-1}(\mathrm{M})  \longrightarrow \mathbb{Q}, \quad (\eta_1,\eta_2) \longmapsto \udeg_\M(\eta_1 \eta_2)
\]
is non-degenerate.
\item (Hard Lefschetz theorem)  For every nonnegative integer $k < d/2$, the multiplication map
\[
\underline{\mathrm{CH}}^k(\mathrm{M})  \longrightarrow  \underline{\mathrm{CH}}^{d-k-1}(\mathrm{M}), \quad \eta \longmapsto \underline{\ell}^{d-2k-1}  \eta
\]
is an isomorphism.
\item (Hodge--Riemann relations) For every nonnegative integer $k <d/2$, the bilinear form
\[
\underline{\mathrm{CH}}^k(\mathrm{M})  \times \underline{\mathrm{CH}}^{k}(\mathrm{M})  \longrightarrow \mathbb{Q}, \quad (\eta_1,\eta_2) \longmapsto (-1)^k \udeg_\M(\underline{\ell}^{d-2k-1} \eta_1 \eta_2)
\]
is positive definite on the kernel of the multiplication by $\underline{\ell}^{d-2k}$.
\end{enumerate}
\end{theorem}

Theorem \ref{theorem_underlineKahler} was first proved as the main result of \cite{AHK}.

\section{The intersection cohomology of a matroid}\label{sec:IC of a matroid}

The purpose of this section is to define the $\H(\M)$-module $\IH(\M)$ along with various related
objects, and to state the litany of results that will be proved in our inductive argument.

\subsection{Definition of $\uIH(\M)$}\label{sec:ihdefs}


Before defining $\IH(\M)$, we first define an ``underlined version" inside $\uCH(\M)$.  Let $\uH(\M)$ be the unital subalgebra of $\uCH(\M)$ generated by the element $\ub$, introduced in Section \ref{subsec:therings}. 

For any subspace $V$ of $\uCH(\M)$,
we set 
\[
V^\perp \coloneq \Big\{ \eta \in \uCH(\M) \ | \ \udeg_\M(v\eta)=0 \ \ \text{for all $v \in V$}\Big\}.
\]
Note that
$V$ is an $\uH(\M)$-submodule  if and only if
$V^\perp$ is an $\uH(\M)$-submodule.

We recursively construct subspaces
 $\uK{F}{\M}$, $\uIH(\M)$, and $\uJ(\M)$ of $\uCH(\M)$ as follows. 
 See the end of Section \ref{sec:canonical decomposition} for a discussion of the motivation behind the definition of $\uJ(\M)$.

\begin{definition}\label{def:thesubspaces}
Let $\M$ be a matroid of positive rank $d$. 
\begin{enumerate}[(1)]\itemsep 5pt
\item For a nonempty proper flat $F$ of $\M$, we define
\[\uK{F}{\M} \coloneq \upsi^F_\Mo\! \left( \uJ(\M_F)\otimes \uCH(\M^F)\right).\]
Proposition \ref{lemma_uab} shows that this is an $\uH(\M)$-submodule of $\uCH(\M)$. 
\item We define the $\uH(\M)$-submodule $\uIH(\M)$ of $\uCH(\M)$ by
\[
\uIH(\M)\coloneq\left(\sum_{\varnothing < F < E} \ \uK{F}{\M} \right)^\perp,
\]
where the sum is over all nonempty proper flats $F$ of $\M$.
\item We define the graded subspace $\uJ(\M)$ of $\uCH(\M)$ by setting
\[
\uJ^k(\M) \coloneq 
\begin{cases}
\uIH^k(\M) & \text{if $k\leq (d-2)/2$,}\\
\ub^{2k-d+2} \ \uIH^{d - k - 2}(\M) & \text{if $k\geq (d-2)/2$.}
\end{cases}
\]
\end{enumerate}
Note that $\uJ^k(\M) = 0$ if $k > d-2$, since $d - k - 2$ is negative.
\end{definition}

For example, when $\M$ is a rank $1$ matroid, we have
\[
\uIH(\M)=\uCH(\M)=\Q 1 \ \ \text{and} \ \ \uJ(\M)=0,
\]
and when $\M$ is a rank $2$ matroid, we have
\[
\uIH(\M)=\uCH(\M)=\Q 1  \oplus \Q\ub \ \ \text{and} \ \ \uJ(\M)=\Q 1.
\]

\begin{example}\label{ex:boolean1}
Let $\M$ be the rank $3$ Boolean matroid on $\{1,2,3\}$.  By definition, $\uCH^0(\M)=\Q 1$, $\uCH^1(\M)$ is generated by $x_1, x_2, x_3, x_{12}, x_{13}, x_{23}$ subject to the relations
\[
x_1+x_{12}+x_{13}=x_2+x_{12}+x_{23}=x_3+x_{13}+x_{23},
\]
and $\uCH^2(\M)$ is generated by $x_1x_{12}=\beta^2$. 
Using the above calculations for the rank 1 matroid $\M^1$ and the rank 2 matroid $\M_1$, we compute
\[
\uK{1}{\M}=\underline{\psi}^F(\Q 1\otimes \Q 1)=\Q x_1.
\]
Meanwhile, since $\uJ(\M_{12})=0$, it follows that $\uK{12}{\M}=0$. Therefore, we have
\[
\uK{1}{\M}=\mathbb{Q}x_1,\ \ \uK{2}{\M}=\mathbb{Q}x_2,\ \ \uK{3}{\M}=\mathbb{Q}x_3, \ \ \text{and} \ \ \uK{12}{\M}=\uK{13}{\M}=\uK{23}{\M}=0.
\]
Note that $\uCH^1(\M)$ is a $4$-dimensional vector space spanned by $\ub, x_1,x_2,x_3$. 
Since $\ub x_1=\ub x_2=\ub x_3=0$ in the Chow ring of $\M$, we have a decomposition of modules over $\uH(\M) = \Q[\ub]$ of the following form:
\[
\uCH(\M)=\uH(\M) \oplus \uK{1}{\M} \oplus \uK{2}{\M} \oplus \uK{3}{\M}.
\]
Thus, in this case, $\uIH(\M)=\uH(\M)=\mathbb{Q}1 \oplus \mathbb{Q} \ub\oplus \mathbb{Q} \ub^2$ and $\uJ(\M)=\mathbb{Q}1 \oplus \mathbb{Q} \ub$.
\end{example}

\begin{example}\label{ex:uniform1}
Let $\M$ be the rank $3$ uniform matroid on $\{1,2,3,4\}$. As above,
\[
\uK{i}{\M}=\mathbb{Q}x_i \ \ \text{for all $i$} \ \ \text{and} \ \ \uK{ij}{\M}=0 \ \ \text{for all distinct $i$ and $j$}.
\]
Note that $\uCH^1(\M)$ is a $7$-dimensional vector space spanned by $\ua,x_{12},x_{13},x_{14},x_{23},x_{24},x_{34}$.
By definition, the orthogonal complement of the span of $x_1,x_2,x_3,x_4$ in $\uCH^1(\M)$ is the $3$-dimensional subspace $\uIH^1(\M)$. 
In fact, we have a decomposition of $\uCH(\M)$ into $\uH(\M)$-modules
\[
\uCH(\M)=\uIH(\M)\oplus  \uK{1}{\M} \oplus \uK{2}{\M} \oplus \uK{3}{\M} \oplus \uK{4}{\M}.
\]
The validity of the displayed decomposition is equivalent to the statement that the restriction of the Poincar\'e pairing of $\uCH(\M)$ to the span of $x_1,x_2,x_3,x_4$ is non-degenerate. This follows from the relations $x_ix_j=0$ for $i \neq j$, together with the computation $\udeg_\M(x_i^2)=2$.
In this case, $\uIH^0(\M)$, $\uIH^1(\M)$, and $\uIH^2(\M)$ have dimensions 1, 3, and 1, respectively. Meanwhile, the subspace $\uH(\M)=\Q1\oplus \Q\ub\oplus \Q\ub^2$ has codimension 2 in $\uIH(\M)$. In this example, though not in general, $\uJ(\M)$ is contained in $\uH(\M)$, and in fact $\uJ(\M)=\Q1\oplus \Q\ub$.
\end{example}

In Section \ref{SectionProof}, we will prove that $\uIH(\M)$ satisfies the hard Lefschetz theorem with respect to $\ub$:
For every nonnegative integer $k < d/2$, the multiplication map
\[
\uIH^k(\M)  \longrightarrow  \uIH^{d-k-1}(\M), \quad \eta \longmapsto \ub^{d-2k-1}  \eta
\]
is an isomorphism. 
Equivalently,
 $\uIH(\M)$ is the unique representation of the Lie algebra
 \[
 \mathfrak{sl}_2=\operatorname{Span}_\Q\Bigg\{\left(\begin{array}{cc} 0&1\\0&0\end{array}\right),\left(\begin{array}{cr} 1&0\\0&-1\end{array}\right),\left(\begin{array}{cc} 0&0\\1&0\end{array}\right)\Bigg\}
 \]
 such that the first matrix acts via multiplication by $\ub$ 
 and the second matrix acts on $\uIH^k(\M)$ via multiplication by $2k-d+1$.
In terms of the $\mathfrak{sl}_2$-action, we have
\[
\uJ(\M)=\left(\begin{array}{cc} 0&0\\1&0\end{array}\right) \cdot \uIH(\M).
\] 
Some further intuitive justification for this definition can be found in Section \ref{sec:canonical decomposition}.



\subsection{Definition of $\IH(\M)$}
We now consider the graded algebras
\begin{align*}
\H(\M) & \coloneq  \text{the unital subalgebra of $\CH(\M)$ generated by $y_i$ for $i \in E$, and }\\
\H_\o(\M)& \coloneq \text{the unital subalgebra of $\CH(\M)$ generated by $y_i$ for $i\in E$ and $x_\varnothing$.}
\end{align*}
As mentioned before, the subalgebra $\H(\M)$ can be identified with the graded M\"obius algebra of $\M$ defined in the introduction \cite[Proposition 2.18]{BHMPW}.
If $E$ is the empty set,
then $x_\varnothing$ does not exist, and we do not define $\H_\o(\M)$.
Note that since the homomorphism  $\varphi^\varnothing\colon \CH(\M) \to \uCH(\M)$ sends $-x_\varnothing$ to $\ub$, it sends $\H_\o(\M)$ to $\uH(\M)$.
For a subspace $V$ of $\CH(\M)$,
we set 
\[
V^\perp \coloneq \Big\{ \eta \in \CH(\M) \ | \ \deg_\M(v\eta)=0 \ \ \text{for all $v \in V$}\Big\}.
\]
If $V$ is an $\H(\M)$-submodule or an $\H_\o(\M)$-submodule, then so is $V^\perp$.


\begin{definition}\label{def:ihdef}
Let $\M$ be a matroid.
\begin{enumerate}[(1)]\itemsep 5pt
\item For any proper flat $F$ of $\M$, we let
\[\K{F}{\M} \coloneq \psi^F_\Mo\!\left(\uJ(\M_F)\otimes \CH(\M^F)\right).\]	
Proposition \ref{lemma_xdegree} shows that this is an $\H(\M)$-submodule of $\CH(\M)$. Moreover, when $F$ is nonempty, Definition \ref{DefinitionXPullback} implies that $\K{F}{\M}$ is in fact an $\H_\circ(\M)$-submodule of $\CH(\M)$.

\item We define the $\H(\M)$-submodule $\IH(\M)$ of $\CH(\M)$ by
\[
\IH(\M)\coloneq \left(\sum_{F<E}\ \K{F}{\M}\right)^\perp,
\]
where the sum is over all proper flats $F$ of $\M$. Note that $1\in \IH(\M)$ for degree reasons, and thus $\H(\M)\subseteq \IH(\M)$. 
\item If $E$ is nonempty, we define the $\H_\circ(\M)$-submodule $\IH_\circ(\M)$ of $\CH(\M)$ by
\[
\IH_\circ(\M) \coloneq \left(\sum_{\varnothing<F<E}\ \K{F}{\M}\right)^\perp,
\]
where the sum is over all nonempty proper flats $F$ of $\M$.
\end{enumerate}
\end{definition}

In Proposition~\ref{summands} below, we will show that the subspaces $\K{F}{\M}$ are mutually orthogonal in $\CH(\M)$, so that $\CH(\M) = \IH(\M) + \sum_{F<E} \K{F}{\M}$ is an orthogonal decomposition of $\CH(\M)$ into $\H(\M)$-submodules.

\begin{example}\label{ex:rank2_2}
Let $\M$ be the rank $2$ uniform matroid on $\{1,2,3\}$.
Since  $\uJ(\N)=0$ for any  rank $1$ matroid $\N$, we have
$\K{1}{\M}=\K{2}{\M}=\K{3}{\M}=0$. 
Similarly, since $\uJ(\N)=\mathbb{Q}1$ for any  rank $2$ matroid $\N$ and $\CH(\N)=\mathbb{Q}1$ for the rank $0$ matroid $\N$, we have
$\K{\varnothing}{\M}=\mathbb{Q}x_\varnothing$.
 By Example \ref{ex:rank2_1b}, we have an orthogonal decomposition 
into indecomposable $\H(\M)$-modules
$\CH(\M)=\H(\M) \oplus \mathbb{Q}x_\varnothing$.
Thus, in this case, $\IH(\M)=\H(\M)$ and $\IH_\circ(\M)=\CH(\M)$.
\end{example}

\begin{example}\label{ex:boolean2}
Let $\M$ be the rank $3$ Boolean matroid on $\{1,2,3\}$. In this case, we have
\[
\K{12}{\M}=\K{13}{\M}=\K{23}{\M}=0,
\]
because $\uJ(\N)=0$ for any rank $1$ matroid $\N$. 
Similarly, since $\uJ(\N)=\mathbb{Q}1$ for any  rank $2$ matroid $\N$ and $\CH(\N)=\mathbb{Q}1\oplus \mathbb{Q}y$ for any rank $1$ matroid $\N$,  we have
\[
\K{1}{\M}=\mathbb{Q}x_1 \oplus \mathbb{Q}x_1 y_1, \ \ 
\K{2}{\M}=\mathbb{Q}x_2 \oplus \mathbb{Q}x_2 y_2,  \ \ 
\K{3}{\M}=\mathbb{Q}x_3 \oplus \mathbb{Q}x_3 y_3.
\]
Lastly, by Example \ref{ex:boolean1}, we have $\uJ(\M)=\mathbb{Q}1 \oplus \mathbb{Q}\ub$, and hence
$\K{\varnothing}{\M}=\mathbb{Q} x_\varnothing \oplus \mathbb{Q} x_\varnothing^2$.
One can check by direct computation that the augmented Chow ring of $\M$ admits the vector space decomposition
\begin{align*}
\CH^3(\M)&=\mathbb{Q}y_1y_2y_3,\\
\CH^2(\M)&=\mathbb{Q}y_1y_2 \oplus \mathbb{Q}y_1y_3 \oplus \mathbb{Q}y_2y_3 \oplus \mathbb{Q}x_\varnothing^2 \oplus \mathbb{Q}x_1y_1 \oplus \mathbb{Q}x_2y_2 \oplus \mathbb{Q}x_3y_3,\\
\CH^1(\M)&=\mathbb{Q}y_1 \oplus \mathbb{Q}y_2 \oplus \mathbb{Q}y_3 \oplus \mathbb{Q}x_\varnothing \oplus \mathbb{Q}x_1 \oplus \mathbb{Q}x_2 \oplus \mathbb{Q}x_3,\\
\CH^0(\M)&=\mathbb{Q}1.
\end{align*}
It follows that $\CH(\M)$ admits an orthogonal decomposition into $\H(\M)$-submodules
\begin{align*}
\CH(\M)&=\H(\M) \oplus (\mathbb{Q}x_\varnothing \oplus \mathbb{Q}x_\varnothing^2) \oplus (\mathbb{Q}x_1 \oplus \mathbb{Q}x_1y_1) \oplus (\mathbb{Q}x_2 \oplus \mathbb{Q}x_2y_2) \oplus (\mathbb{Q}x_3 \oplus \mathbb{Q}x_3y_3)\\
&=\H(\M) \oplus \K{\varnothing}{\M} \oplus \K{1}{\M} \oplus \K{2}{\M} \oplus \K{3}{\M}. 
\end{align*}
Thus, in this case, $\IH(\M)=\H(\M)$ and $\IH_\circ(\M)=\H(\M) \oplus \K{\varnothing}{\M} $. 
\end{example}

\begin{example}\label{ex:uniform2}
Let $\M$ be the rank $3$ uniform matroid on $\{1,2,3,4\}$. Both $\CH^1(\M)$ and $\CH^2(\M)$ are 11-dimensional, and a direct computation shows that the elements $y_iy_j$, $x_iy_i$, and $x_\varnothing^2$ form a basis of $\CH^2(\M)$. 
As in the earlier examples, we have
\[
\K{ij}{\M}=0 \ \ \text{for any distinct $i$, $j$}, \ \ \K{i}{\M}=\mathbb{Q}x_i \oplus \mathbb{Q}x_iy_i \ \ \text{for any $i$}, \ \ \text{and} \ \ \K{\varnothing}{\M}= \mathbb{Q}x_\varnothing \oplus \mathbb{Q}x_\varnothing^2.
\]
Hence $\dim \IH^2(\M)=6$. Note that $\dim \H^2(\M)=6$ and $\H^2(\M)\subseteq \IH^2(\M)$, and therefore $\H^2(\M)= \IH^2(\M)$. In contrast, $\IH^1(\M)$ contains $\H^1(\M)$ as a codimension $2$ subspace. In fact, $\CH(\M)$ admits decompositions into $\H(\M)$-submodules
\[
\CH(\M)=\IH_\circ(\M) \oplus \bigoplus_{i=1}^4 \K{i}{\M}=\IH(\M) \oplus \K{\varnothing}{\M}\oplus  \bigoplus_{i=1}^4  \K{i}{\M}.
\]
Since the summands are mutually orthogonal, the validity of the first decomposition is equivalent to the statement that the restriction of the Poincar\'e pairing of $\CH(\M)$ to each  $\mathbb{Q}x_i \oplus \mathbb{Q}x_iy_i$ is non-degenerate, which can be verified directly.  
On the other hand, the validity of the decomposition 
\[
\IH_\circ(\M)=\IH(\M) \oplus \K{\varnothing}{\M}
\]
 is equivalent to the fact that $x_\varnothing^3 \neq 0$ in $\CH^3(\M)$, and hence to the fact that 
$\ub^2 \neq 0$ in $\uCH^2(\M)$.
In general, by \cite[Proposition 9.5]{AHK}, the degree of the top self-intersection of $\ub$ in the Chow ring of $\M$ is the absolute value of the M\"obius number of the lattice of flats of $\M$, which is $3$ in this example, and is nonzero in general \cite[Theorem 7.1.8]{Zaslavsky}.
\end{example}

\subsection{Pulling and pushing the intersection cohomology modules}
We now state some basic properties of the pullbacks and pushforwards for  the subspaces
we have defined.

\begin{lemma}\label{push uIH}
For any proper flats $F < G$ of $\M$, we have
\begin{enumerate}[(1)]\itemsep 5pt
	\item $\psi^F\!\big( \uK{G\setminus F}{\M_F} \otimes \CH(\M^F)\big) \subseteq \K{G}{\M}$ and
	$\varphi^F(\K{G}{\M}) = \uK{G\setminus F}{\M_F} \otimes \CH(\M^F)$,
	\item $\upsi^F\!\big( \uK{G\setminus F}{\M_F} \otimes \uCH(\M^F)\big) \subseteq \uK{G}{\M}$ and $\uvarphi^F(\uK{G}{\M}) = \uK{G\setminus F}{\M_F} \otimes \uCH(\M^F)$,
	\item $\varphi_\Mo^F \IH_\o(\M) \subseteq \uIH(\M_F) \otimes \CH(\M^F)$ and $\uvarphi_\Mo^F \uIH(\M) \subseteq \uIH(\M_F) \otimes \uCH(\M^F).$
\end{enumerate}
\end{lemma}
\begin{proof}
For the first part of statement (1), we use the right square of Lemma \ref{lemma_diagrams} (2):
\begin{align*}
\psi^F\big(\uK{G\setminus F}{\M_F} \otimes \CH(\M^F)\big) 
& = \psi^F_\Mo \!\big(\upsi^{G\setminus F}_{\M_F}\! \left(\uJ(\M_G) \otimes \uCH(\M^G_F)\right) \otimes \CH(\M^F)\big)\\
&= \psi^G\!\left(\uJ(\M_G) \otimes \psi^F_{\M^G}\!\left(\uCH(\M^G_F)\otimes \CH(\M^F)\right)\right)\\
&\subseteq \psi^G\!\left(\uJ(\M_G) \otimes \CH(\M^G)\right) = \K{G}{\M}.
\end{align*}	
The second part follows similarly, using the left square of Lemma \ref{lemma_diagrams} (2) and the surjectivity of $\varphi^F_{\M^G}$.	 Statement (2) follows by the same arguments, using Lemma \ref{lemma_diagrams} (3).
	
For the first part of statement (3), we need to show that, for any proper flat $G$ of $\M$ properly containing $F$,
 $\varphi^F_\Mo\IH_\o(\M)$ is orthogonal to $\uK{G\setminus F}{\M_F} \otimes \CH(\M^F)$ in  $\uCH(\M_F) \otimes \CH(\M^F)$.
By Lemma \ref{underline adjoint}, this is equivalent to the statement that $\IH_\o(\M)$ is orthogonal to
$\psi^F(\uK{G\setminus F}{\M_F} \otimes \CH(\M^F))$ in $\CH(\M)$.  But this follows from the first part of statement (1).  The second part of (3) follows similarly, using the first part of statement (2). \qedhere


\end{proof}

\begin{lemma}\label{lem_imageofIH}
The following holds for any matroid $\M$. 
\begin{enumerate}[(1)]\itemsep 5pt
\item For any nonempty proper flat $F$ of $\M$, we have 
$\varphi^\Mo_F \IH_\o(\M)\subseteq \IH(\M_F)$.
\item For any proper flats $F \le G$ of $\M$, we have 
$\varphi^\Mo_F \K{G}{\M}= \K{G\setminus F}{\M_F}$.
\end{enumerate}
\end{lemma}
\begin{proof}
To prove (1), it suffices to show that for any flat $G$ containing $F$,
\[
\varphi^\Mo_F \IH_\o(\M) \ \text{and}\ \K{G\setminus F}{\M_F} \ \text{are orthogonal in} \ \CH(\M_F).
\]
Note that we have 
\begin{align*}
\psi_F\!\left(\K{G\setminus F}{\M_F}\right) & = \psi_F^\Mo \psi^{G\setminus F}_{\M_F}\big(\uJ(\M_G)\otimes \CH(\M_F^G)\big) \\
& = \psi^G_\Mo\big(\uJ(\M_G)\otimes \psi^{\M^G}_F\CH(\M_F^G)\big)\\
& \subseteq \K{G}{\M}.
\end{align*} 
By Lemma \ref{subscript adjoint} and the right commutative square of Lemma \ref{lemma_diagrams} (4), the orthogonality statement that we need is equivalent to the statement that $\IH_\o(\M)$ and
$\psi_F\!\left(\K{G\setminus F}{\M_F}\right)$ are orthogonal in $\CH(\M)$.  But $\IH_\o(\M)$ is by definition orthogonal to all of $\K{G}{\M}$.
The second statement follows similarly using the left square of Lemma \ref{lemma_diagrams} (4) and the surjectivity of $\varphi_F^{\M^G}$. 
\end{proof}

\begin{proposition}\label{summands}
The graded linear subspaces 
\[
\uK{F}{\M}\subseteq \uCH(\M),
\]
where $F$ varies through all nonempty  proper flats of $\M$, are mutually orthogonal in $\uCH(\M)$.  
Similarly, the graded linear subspaces
\[
\K{F}{\M}\subseteq \CH(\M),
\]
where $F$ varies through all proper flats of $\M$, are mutually orthogonal in $\CH(\M)$.
\end{proposition}

\begin{proof}
We only prove the second statement. The first statement follows from the same arguments.

Let $F$ and $G$ be distinct nonempty proper flats.
We want to show that $\K{F}{\M}$ 
is orthogonal to
$\K{G}{\M}$ in $\CH(\M)$.  By Lemma \ref{adjoint} and the fact that $\K{F}{\M}  = \psi^F_\Mo\!\left( \uJ(\M_F)\otimes \CH(\M^F)\right) $, this is equivalent to showing that
\[
\varphi_\Mo^F\K{G}{\M}\ \ \text{is orthogonal to}\ \ \uJ(\M_F)\otimes \CH(\M^F) \ \ \text{in} \ \ \uCH(\M_F)\otimes \CH(\M^F) .
\]
If $F$ and $G$ are incomparable, this follows from Lemma \ref{lemma_incomparable}, so we may assume
without loss of generality that $F<G$.  But then by Lemma \ref{push uIH} (1), we have $\varphi^F(\K{G}{\M}) = \uK{G\setminus F}{\M_F} \otimes \CH(\M^F)$.
Since $\uJ(\M_F)$ is contained in $\uIH(\M_F)$, which is orthogonal to $\uK{G\setminus F}{\M_F}$, the result follows.
\end{proof}

Finally, we need one more variant of the module $\uIH(\M)$, which treats one element $i\in E$ differently than the others.
Let 
$\uH_i(\M)$ be the unital subalgebra of $\uCH(\M)$ generated by $\ub$ and $x_{\{i\}}$, with the convention that $x_{\{i\}}=0$ when $\{i\}$ is not a flat.


As before, $V$ is an $\uH_i(\M)$-submodule 
if and only if
$V^\perp$ is an $\uH_i(\M)$-submodule. 
Proposition \ref{lemma_uab} shows that 
$\uK{F}{\M}$ is an $\uH_i(\M)$-submodule of $\uCH(\M)$
for every nonempty proper flat $F$ different from $\{i\}$.

\begin{definition}\label{def:uIHi}
	We define the $\uH_i(\M)$-submodule $\uIH_i(\M)$ of $\uCH(\M)$ by
	\[
	\uIH_i(\M) \coloneq \left(\sum_{F\neq \{i\}} \ \uK{F}{\M}\right)^\perp,
	\]
	where the sum is over all nonempty proper flats $F$ of $\M$ different from $\{i\}$.\footnote{Note that  $\uIH_i(\M)=\uIH(\M)$ when $\{i\}$ is not a flat.}
\end{definition}

In general, $\uIH_i(\M)$ is a graded $\uH(\M \setminus i)$-module, which will be shown to decompose as
\[
\uIH_i(\M)=\uIH(\M) \oplus  \uK{i}{\M}. 
\]
For example, in Example~\ref{ex:uniform1}, we have $\uIH_4(\M)=\uIH(\M) \oplus  \uK{4}{\M}$. 
The submodule $\uIH_i(\M)$ will appear in a crucial step of our inductive argument in Section \ref{sec:underlined deletion induction}. 
See also Section \ref{sec:udel} in our guide to the proof.\footnote{In contrast, $\uIH(\M)$ does not have a natural structure of a graded $\uH(\M \setminus i)$-module. See Example~\ref{ex:deletion1}.}

\subsection{The statements}\label{sec:statements}

Let $\N=\bigoplus_{k\geq 0}\N^k$ be a finite-dimensional graded $\mathbb{Q}$-vector space endowed with a bilinear form 
\[
\langle -,- \rangle \colon \N\times \N\to \Q
\]
and  a linear operator $\oL \colon \N\to \N$  of degree $1$ that satisfies 
$\langle \oL(\eta), \xi \rangle=\langle \eta, \oL(\xi) \rangle$ for all $\eta, \xi\in \N$. 

\begin{definition}
Using the notation above, we define three properties for $\N$.
\begin{enumerate}[(1)]\itemsep 5pt
\item We say that $\N$ satisfies {\bf Poincar\'e duality of degree \boldmath{$d$}}  if the bilinear form $\langle -,- \rangle$ is non-degenerate, and for  $\eta\in \N^j$ and $\xi\in \N^k$, the pairing $\langle \eta, \xi \rangle$ is nonzero only when $j+k=d$. 
\item We say that $\N$ satisfies the {\bf hard Lefschetz theorem of degree \boldmath{$d$}}  if the linear map 
\[
\oL^{d-2k} \colon \N^k\to \N^{d-k}
\]
is an isomorphism for all $k\leq d/2$. 
\item We say that $\N$ satisfies the {\bf Hodge--Riemann relations of degree \boldmath{$d$}}  if the restriction of 
\[
\N^k\times \N^k\longrightarrow \Q, \qquad (\eta, \xi)\longmapsto (-1)^k\langle \oL^{d-2k} (\eta), \xi \rangle
\]
to the kernel of 
$
\oL^{d-2k+1} \colon \N^k\to \N^{d-k+1}
$
is positive definite for all $k\leq d/2$. Elements of $\ker \oL^{d-2k+1}$ are called {\bf primitive classes}.
\end{enumerate}
\end{definition}


We now define  the central statements that appear in the induction. 

Our first group of statements says that the augmented Chow ring admits canonical decompositions into $\H(\M)$-modules, and the Chow ring admits canonical decompositions into $\uH(\M)$-modules. 

\begin{definition}[Canonical decompositions]\label{def:canonical decomps}
\
\begin{enumerate}[]\itemsep 5pt
\item \hypertarget{cd}{$\CD(\M)$}: We have the direct sum decomposition
\[
\CH(\M) = \IH(\M) \oplus \bigoplus_{F<E}\ \K{F}{\M},
\]
where the sum is over all proper flats $F$ of $\M$.
\item \hypertarget{cdo}{$\CD_\circ(\M)$}: We have the direct sum decomposition
\[
\CH(\M) = \IH_\circ(\M) \oplus \bigoplus_{\varnothing<F<E}\ \K{F}{\M},
\]
where the sum is over all nonempty proper flats $F$ of $\M$.
\item \hypertarget{ucd}{$\uCD(\M)$}: We have the direct sum decomposition
\[
\uCH(\M) = \uIH(\M) \oplus \bigoplus_{\varnothing<F<E}\ \uK{F}{\M},
\]
where the sum is over all nonempty proper flats $F$ of $\M$.
\end{enumerate}
\end{definition}

\begin{convention}
We will use a superscript to denote that the decompositions hold in certain degrees. For example, \hyperlink{cd}{$\CD^{\leq k}(\M)$} means that the direct sum decomposition holds in degrees less than or equal to $k$. 
\end{convention}

\begin{remark}\label{remark_cd}
Let $V$ and $W$ be finite-dimensional $\Q$-vector spaces with subspaces $V_1\subseteq V$ and $W_1\subseteq W$. Given a non-degenerate pairing 
$V\times W\to \mathbb{Q}$,
we can define the orthogonal subspaces $W_1^\perp \subseteq V$ and $V_1^\perp\subseteq W$. It is straightforward to check that
$W=W_1\oplus V_1^\perp$
if and only if 
$V=V_1\oplus W_1^\perp$.
Applying this fact repeatedly, we have
\[
\hyperlink{cd}{\CD^{k}(\M)}  \Longleftrightarrow  \hyperlink{cd}{\CD^{d-k}(\M)}, \;\;\hyperlink{cd}{\CD(\M)}  \Longleftrightarrow  \hyperlink{cd}{\CD^{\leq \frac{d}{2}}(\M)}, \;\;\text{and}\;\; \hyperlink{cdo}{\CD_\o(\M)}  \Longleftrightarrow  \hyperlink{cdo}{\CD_\o^{\leq \frac{d}{2}}(\M)}.
 \]
Similarly, we have
$\hyperlink{ucd}{\uCD(\M)} \Longleftrightarrow  \hyperlink{ucd}{\uCD^{\leq \frac{d-1}{2}}(\M)}$. 
\end{remark}

\begin{definition}[Poincar\'e dualities]\label{def:pds}
	\
	\begin{enumerate}[]\itemsep 5pt
		\item \hypertarget{pd}{$\PD(\M)$}: The graded vector space $\IH(\M)$ satisfies Poincar\'e duality of degree $d$ with respect to  the Poincar\'e pairing on $\CH(\M)$.
		\item \hypertarget{pdo}{$\PD_\o(\M)$}: The graded vector space $\IH_\o(\M)$ satisfies Poincar\'e duality of degree $d$ with respect to the Poincar\'e pairing on $\CH(\M)$.
		\item \hypertarget{upd}{$\uPD(\M)$}: The graded vector space $\uIH(\M)$ satisfies Poincar\'e duality of degree $d-1$ with respect to  the Poincar\'e pairing on $\uCH(\M)$.
	\end{enumerate}
\end{definition}

\begin{remark}\label{CDPD}
Let $V$ be a finite-dimensional $\Q$-vector space equipped with a non-degenerate symmetric bilinear form,
and let $W\subseteq V$ be a subspace.  Then the restriction of the form to $W$ is non-degenerate if and only if $V = W\oplus W^\perp$.
In light of Remark \ref{remark_cd}, this implies that $$\hyperlink{cd}{\CD^k(\M)}\Longleftrightarrow\hyperlink{pd}{\PD^k(\M)},\;\;\hyperlink{cdo}{\CD_\circ^k(\M)}\Longleftrightarrow\hyperlink{pdo}{\PD_\circ^k(\M)},
\;\;\text{and}\;\;\hyperlink{ucd}{\uCD^k(\M)}\Longleftrightarrow\hyperlink{upd}{\uPD^k(\M)}.$$
\end{remark}

Let $R$ be a graded $\Q$-algebra with degree zero part equal to $\Q$, and let $\mm\subseteq R$
denote the unique graded maximal ideal.  For any graded $R$-module $\N$, the {\bf socle} of $\N$ is defined to be the graded submodule 
\[
\operatorname{soc}(\N) \coloneq \{n\in \N\ | \ \mm\cdot n = 0\}.
\]
Equivalently, it is the direct sum of all simple submodules of $\N$.
The next conditions assert that the socles of the intersection cohomology modules 
defined in Section~\ref{sec:ihdefs} 
vanish in low degrees.  As before, the symbol $d$ stands for the rank of the matroid $\M$.

\begin{definition}[No socle conditions]\label{def:ns}
\
\begin{enumerate}[]\itemsep 5pt
\item \hypertarget{ns}{$\NS(\M)}$: The socle of the $\H(\M)$-module $\IH(\M)$ vanishes in degrees less than or equal to $d/2$.
\item \hypertarget{nso}{$\NS_\o(\M)$}: The socle of the $\H_\circ(\M)$-module $\IH_\circ(\M)$ vanishes in degrees less than or equal to $d/2$.
\item \hypertarget{uns}{$\uNS(\M)$}: The socle of the $\uH(\M)$-module $\uIH(\M)$ vanishes in degrees less than or equal to $(d-2)/2$.
\end{enumerate}
\end{definition}

In particular, for  even $d$, the no socle condition for $\IH(\M)$ says that the socle of the $\H(\M)$-module $\IH(\M)$ 
is concentrated in degrees strictly larger than the middle degree $d/2$.
On the other hand, for an odd number $d$, the socle of the $\uH(\M)$-module $\uIH(\M)$ 
may be nonzero in the middle degree $(d-1)/2$.

Recall that we have Poincar\'e pairings on $\CH(\M)$ and $\uCH(\M)$ defined by
\[
\langle \eta, \,\xi \rangle_{\CH(\M)} \coloneq \deg_{\M}(\eta\,\xi) 
\;\;\text{and}\;\;
\langle \underline\eta, \, \underline\xi \rangle_{\uCH(\M)}\coloneq \udeg_{\M}(\underline\eta\,\underline\xi). 
\]
Moreover, with respect to the above bilinear forms, $\CH(\M)$ satisfies Poincar\'e duality of degree $d$ and $\uCH(\M)$ satisfies Poincar\'e duality of degree $d-1$,
by Theorems \ref{TheoremChowKahlerPackage} and \ref{theorem_underlineKahler}.

\begin{definition}[Hard Lefschetz theorems]\label{def:hls}
\
\begin{enumerate}[]\itemsep 5pt
\item \hypertarget{hl}{$\HL(\M)$}: For any positive linear combination $y=\sum_{j \in E} c_jy_j$, the graded vector space $\IH(\M)$ satisfies the hard Lefschetz theorem of degree $d$ with respect to multiplication by $y$. 
\item \hypertarget{hlo}{$\HL_\circ(\M)$}: For any positive linear combination $y=\sum_{j \in E} c_jy_j$, there is a positive $\epsilon$ such that the graded vector space $\IH_\o(\M)$ satisfies the hard Lefschetz theorem of degree $d$ with respect to multiplication by $y-\epsilon x_\varnothing$. 
\item \hypertarget{hli}{$\HL_i(\M)$}: For any positive linear combination $y'=\sum_{j \in E\setminus i} c_jy_j$, the graded vector space $\IH(\M)$ satisfies the hard Lefschetz theorem of degree $d$ with respect to multiplication by $y'$. 
\item  \hypertarget{uhl}{$\uHL(\M)$}: The graded vector space $\uIH(\M)$ satisfies the hard Lefschetz theorem of degree $d-1$ with respect to multiplication by $\b$.
\item \hypertarget{uhli}{$\uHL_i(\M)$}: The graded vector space $\uIH_i(\M)$ satisfies the hard Lefschetz theorem of degree $d-1$ with respect to multiplication by $\b-x_{\{i\}}$. Here we recall our convention that $x_{\{i\}}=0$ if $\{i\}$ is not a flat. 
\end{enumerate}
\end{definition}

\begin{definition}[Hodge--Riemann relations]\label{def:hrs}
\
\begin{enumerate}[]\itemsep 5pt
\item \hypertarget{hr}{$\HR(\M)$}: For any positive linear combination $y=\sum_{j \in E} c_jy_j$, the graded vector space $\IH(\M)$ satisfies the Hodge--Riemann relations of degree $d$ with respect to the Poincar\'e pairing on $\CH(\M)$ and the multiplication by $y$. 
\item \hypertarget{hro}{$\HR_\o(\M)$}: For any positive linear combination $y=\sum_{j \in E} c_jy_j$, there is a positive $\epsilon$ such that the graded vector space $\IH_\o(\M)$ satisfies the Hodge--Riemann relations of degree $d$ with respect to the Poincar\'e pairing on $\CH(\M)$ and the multiplication by $y-\epsilon x_\varnothing$. 
\item \hypertarget{hri}{$\HR_i(\M)$}: For any positive linear combination $y'=\sum_{j \in E\setminus i} c_jy_j$, the graded vector space $\IH(\M)$ satisfies the Hodge--Riemann relations of degree $d$ with respect to the Poincar\'e pairing on $\CH(\M)$ and the multiplication by $y'$. 
\item \hypertarget{uhr}{$\uHR(\M)$}: The graded vector space $\uIH(\M)$ satisfies the Hodge--Riemann relations of degree $d-1$ with respect to the Poincar\'e pairing on $\uCH(\M)$ and the multiplication by $\b$. 
\item \hypertarget{uhri}{$\uHR_i(\M)$}: The graded vector space $\uIH_i(\M)$ satisfies the Hodge--Riemann relations of degree $d-1$ with respect to the Poincar\'e pairing on $\uCH(\M)$ and the multiplication by $\b-x_{\{i\}}$. 
\end{enumerate}
\end{definition}

As before, we will use a superscript to denote that the conditions hold in certain degrees. For example, $\hyperlink{pd}{\PD^k(\M)}$ means the Poincar\'e pairing on $\CH(\M)$ induces a non-degenerate pairing between $\IH^k(\M)$ and $\IH^{d-k}(\M)$, and $\hyperlink{hl}{\HL^k(\M)}$ means the hard Lefschetz map 
from $\IH^k(\M)$ to $\IH^{d-k}(\M)$ is an isomorphism. 

Now we state the main result of this paper, which will be proved using induction on the cardinality of the ground set $E$. 

\begin{theorem}\label{theorem_all}
Let $\M$ be a matroid on $E$.  If $E$ is nonempty, the following statements hold:
\begin{center}
{\renewcommand{\arraystretch}{1.3}%
\begin{tabular}{lllll}
$\hyperlink{cd}{\CD(\M)}$, \; \; \; & $\hyperlink{ns}{\NS(\M)}$, \; \; \; & $\hyperlink{pd}{\PD(\M)}$, \; \; \; & $\hyperlink{hl}{\HL(\M)}$, \; \; \; & $\hyperlink{hr}{\HR(\M)}$, \; \; \; \\
$\hyperlink{cdo}{\CD_\o(\M)}$, \; \; \;& $\hyperlink{nso}{\NS_\o(\M)}$, \; \; \; & $\hyperlink{pdo}{\PD_\o(\M)}$, \; \; \; & $\hyperlink{hlo}{\HL_\o(\M)}$, \; \; \; & $\hyperlink{hro}{\HR_\o(\M)}$, \; \; \; \\
$\hyperlink{ucd}{\uCD(\M)}$, \; \; \; & $\hyperlink{uns}{\uNS(\M)}$, \; \; \; & $\hyperlink{upd}{\uPD(\M)}$, \; \; \; & $\hyperlink{uhl}{\uHL(\M)}$, \; \; \; & $\hyperlink{uhr}{\uHR(\M)}$.
\end{tabular}}
\end{center}
\end{theorem}

As intermediate steps in the induction, we will also prove the statements $\hyperlink{hli}{\HL_i(\M)}$, $\hyperlink{hri}{\HR_i(\M)}$, $\hyperlink{uhli}{\uHL_i(\M)}$, and $\hyperlink{uhri}{\uHR_i(\M)}$. 
However, we will not use these statements in our applications, and we do not need them in the main inductive hypothesis.

\begin{remark}
If $E$ is the empty set, the statements $\hyperlink{cd}{\CD(\M)}$, $\hyperlink{pd}{\PD(\M)}$, $\hyperlink{hl}{\HL(\M)}$, and $\hyperlink{hr}{\HR(\M)}$ hold tautologically.
The statement $\hyperlink{ns}{\NS(\M)}$ fails, as we have $\H(\M) = \CH(\M) = \IH(\M) = \Q$, so the socle is nonvanishing in degree $0$.
This is directly related to the fact that the Kazhdan--Lusztig polynomial of the rank zero matroid has larger than expected degree.
The remaining statements do not make sense because $\IH_\o(\M)$ and $\uIH(\M)$ are not defined when $E$ is empty. 
\end{remark}







\section{Guide to the proof}\label{sec:guide}

\tikzset{
block/.style={
    rectangle,
    draw, thick,
    text width=10em,
    text centered,
    rounded corners
},
blocka/.style={
    rectangle,
    draw=none,
    text width=10em,
    minimum height=1.2em
},
line/.style={>=latex,->,thick},
liney/.style={>=latex,-,thick},
}

\begin{figure}[!htpb]
\begin{tikzpicture}
\matrix (m)[matrix of nodes, column  sep=1.75cm, row  sep=1.6cm, align=center, nodes={rectangle,draw, anchor=center} ]
{
\node[blocka](top1) {}; & \node[blocka](top2) {}; & \node[blocka](top3) {};\\
|[block]| {$\underline{\mathsf{NS}}^{<\frac{d-2}{2}}(\mathrm{M})$}  & |[block]| {$\PD_\o(\M)$, $\mathsf{CD}_\o(\mathrm{M})$ \\ $\uPD(\M)$, $\uCD(\M)$}       & |[block]| {$\uHL_\A(\M), \underline{\mathsf{HR}}_\A(\mathrm{M})$} \\
 |[block]| {$\underline{\mathsf{HL}}^{<\frac{d-2}{2}}(\mathrm{M})$}  & |[block]| {$\mathsf{CD}^{<\frac{d}{2}}(\mathrm{M})$, $\mathsf{PD}^{<\frac{d}{2}}(\mathrm{M})$}  &  |[block]| {$\mathsf{NS}^{<\frac{d}{2}}(\mathrm{M})$}                                \\
 |[block]| {$\underline{\mathsf{HR}}^{<\frac{d-2}{2}}(\mathrm{M})$}          &  |[block]| {$\HL_i(\M), \HR_i^{<\frac d 2}(\M)$}          &  |[block]| {$\mathsf{HL}(\mathrm{M})$}                               \\
 & |[block]| {$\mathsf{HR}^{<\frac{d}{2}}(\mathrm{M})$} & 
 |[block]| {$\mathsf{HL}_\o(\mathrm{M})$} \\
 |[block]| {$\mathsf{HR}_\circ^{<\frac{d}{2}}(\mathrm{M})$}  &       &  |[block]| {$\mathsf{HR}_\circ 
 (\mathrm{M})$}                         \\
 |[block]| {$\uNS(\M)$} &  & |[block]| {$\mathsf{NS}_\circ 
  (\mathrm{M})$}             \\
 |[block]| {$\underline{\mathsf{HL}} 
 (\mathrm{M})$}    &   |[block]| {$\mathsf{CD} 
 (\mathrm{M})$, $\mathsf{PD}(\mathrm{M})$}         &  |[block]| {$\mathsf{NS} 
 (\mathrm{M})$}                     \\
 |[block]| {$\underline{\mathsf{HR}} 
 (\mathrm{M})$}         &  & |[block]| {$\mathsf{HR} 
 (\mathrm{M})$}    &      \\
};

\node[block] (outer) [fit=(top1)(top2)(top3)] {\raisebox{-.6em}{All statements for matroids on fewer elements}};

\draw [line] ([yshift=-.14cm]top2.south) -- node[midway,left]{\ref{CD for free}}(m-2-2);

\draw [line] (m-2-2) -- node[midway,above]{\ref{uHL_i}}node[midway,below]{\ref{uHR_i}}(m-2-3);

\draw [line] (m-2-2) -- node[midway,above]{\ref{uNS1}}(m-2-1);

\draw [liney] (m-2-1) -- ([yshift=-1cm]m-2-1.south);
\draw [liney] (m-2-3) |- ([yshift=-0.8cm]m-2-1.south);
\draw [line] ([yshift=-0.78593cm]m-2-1.south) -- node[pos=0,left]{\ref{NS under half}} (m-3-1);

\draw [line] (m-3-1) -- node[midway,above]{\ref{uHL1CD}} (m-3-2);
\draw [line] (m-3-2) -- node[midway,above]{\ref{CDNS}} (m-3-3);

\draw [line] (m-3-1) -- node[midway,left]{\ref{alt}} (m-4-1);
\draw [line] (m-3-2) -- node[midway,left]{\shortstack{\ref{HLi}\\ \\ \ref{HRi}}} 
(m-4-2);
\draw [liney] (m-3-3) |- ([yshift=1cm]m-4-2.north);

\draw [line] ([xshift=.5cm]m-3-3.south) -- node[midway,right]{\ref{dCM}} ([xshift=.5cm]m-4-3.north);


\draw [line] (m-4-1) -- node[pos=0.79,left]{\ref{block1}} (m-6-1);
\draw [line] (m-4-2) -- node[midway,left]{\ref{delnodel}} (m-5-2);
\draw [liney] (m-4-3) |- ([yshift=1cm]m-5-2.north);

\draw [line] ([xshift=.5cm]m-4-3.south) -- node[midway,right]{\ref{block_HL}} ([xshift=.5cm]m-5-3.north);

\draw [liney] (m-5-2) -- ([xshift=-3.81cm]m-5-2.west);

\draw [liney] (m-5-3) |- ([yshift=.78593cm]m-6-3.north);

\draw [liney] ([yshift=.8cm]m-6-3.north) -| ([xshift=-3.75cm]m-6-3.west);

\draw [line] (m-6-1) -- node[midway,below]{\ref{sig}} (m-6-3);
\draw [line] (m-6-3) -- node[midway,left]{\ref{HR1NS1}} (m-7-3);

\draw [line] (m-7-3) --  node[midway,above]{\ref{NS1uNS1}} (m-7-1);

\draw [line] (m-7-1) -- node[midway,left]{\ref{NS1uHL}} (m-8-1);

\draw [line] (m-8-1) -- node[midway,above]{\ref{uHL1CD}} (m-8-2);
\draw [line] (m-8-2) -- node[midway,above]{\ref{CDNS}} (m-8-3);

\draw [line] (m-8-1) -- node[midway,left]{\ref{alt}} (m-9-1);

\draw [liney] (m-8-2) -- ([xshift=-3.81cm]m-9-3.west);

\draw [line] (m-9-1) -- node[midway,below]{\ref{other sig}} (m-9-3);
\end{tikzpicture}
\caption{Diagram of the proof}\label{Proof diagram}
\end{figure}
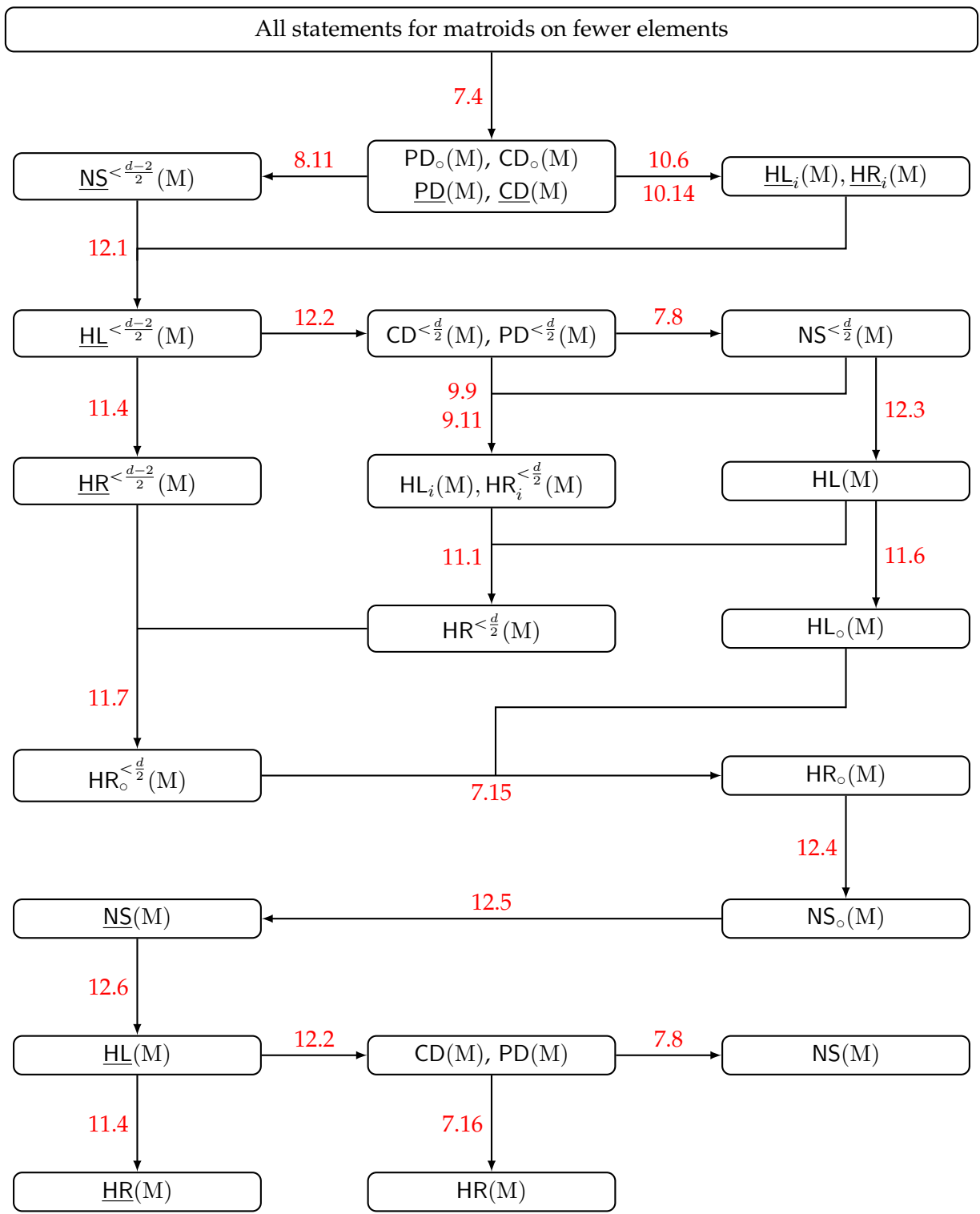

The proof of our main result, Theorem \ref{theorem_all}, is a complex induction involving all of the statements introduced in the previous section.  A more or less complete diagram of the steps of the induction appears in Figure~\ref{Proof diagram}.  The purpose of this section is to highlight the main steps in the proof, to explain what these steps mean in the geometric setting when $\M$ is realizable, and to make some comparisons with the structure of the proofs of Karu \cite{Karu} and Elias--Williamson \cite{EW}.  

We hope that readers will benefit from flipping back to this section frequently as they read the rest of the paper.  However, this section is not needed for establishing the results in this paper; it is included only to communicate the overall structure and geometric insight behind the main ingredients of the proof.  It may be skipped in full by readers who would like to stick to a purely formal treatment.

\subsection{Canonical decomposition}\label{sec:canonical decomposition}
As discussed in Section~\ref{sec:realizable}, when the matroid $\M$ is realizable, $\CH(\M)$ is the cohomology ring of a resolution $X$ of the arrangement Schubert variety $Y$.  Applying the Beilinson--Bernstein--Deligne--Gabber decomposition theorem to the proper map $\pi\colon X \to Y$ gives a decomposition of $\pi_*\underline{\Q}_X$ as a direct sum of shifted intersection complexes on $Y$.  Since $\pi$
is constructible for the stratification of $Y$ by affine spaces $U^F$, these intersection complexes are all of the form $\IC^\bullet(\overline{U^F}; \Q)$ extended by zero to $Y$.  Furthermore, the closure $\overline{U^F}$ is isomorphic to the arrangement Schubert variety $Y^F$ associated with the realization $\sigma^F$ of the matroid $\M^F$.
Taking cohomology, $\CH(\M) \cong \H^\bullet(\pi_*\underline{\Q}_X)$ is a direct sum of graded $\H(\M)$-submodules, each isomorphic to a shift of $\IH(\M^F)$ for some flat $F$.\footnote{The surjection $\H(\M) \to \H(\M^F)$ defined by setting $y_G = 0$ unless $G \le F$ makes $\IH(\M^F)$ an $\H(\M)$-module.}  As was noted in Section \ref{sec:realizable}, an argument of Ginzburg \cite{Ginsburg-perverse} implies that these modules are indecomposable, so the summands and their multiplicities are well-defined by Krull--Schmidt.

In our proof, we obtain such a decomposition as a consequence of the coarser decomposition $\hyperlink{cd}{\CD(\M)}$ (Definition~\ref{def:canonical decomps}).  The summand in $\hyperlink{cd}{\CD(\M)}$ indexed by the proper flat $F$ is isomorphic as an $\H(\M)$-module to a direct sum of shifts of copies of $\CH(\M^F)$, so it can be further decomposed using the same formula.  Iterating this, one can obtain a decomposition of $\CH(\M)$ into shifted copies of $\IH(\M^F)$ for various flats $F$. We prove that these modules $\IH(\M^F)$ are indecomposable in Proposition \ref{prop:indecomposable}.

The decomposition $\hyperlink{cd}{\CD(\M)}$ has several properties which make proving it easier than proving the full decomposition into indecomposable modules directly.  First, the summands $\K{F}{\M}$ in $\hyperlink{cd}{\CD(\M)}$ are canonical, since the definition
of $\uJ(\M)$ does not involve any choices (Definition~\ref{def:thesubspaces}).
Second, these summands are orthogonal to each other with respect to the Poincar\'e pairing on $\CH(\M)$ (Proposition~\ref{summands}), and we define $\IH(\M)$ to be the perpendicular space to them (Definition~\ref{def:ihdef}).

The problem then is to show that the terms actually do form a direct sum.  Note that, because the Poincar\'e pairing on $\CH(\M)$ is non-degenerate, to prove $\hyperlink{cd}{\CD(\M)}$  it is enough to show that the restriction of the pairing to each summand $\K{F}{\M}$ is non-degenerate.  In Corollary~\ref{CD for free}, we use the formal properties of our push/pull operators to show that this holds when $F$ is a \emph{nonempty} proper flat.  This is possible because the matroid $\M_F$ has a smaller ground set than $\M$, and so our inductive assumption says that all of our results hold for $\M_F$.  
Similarly, the pairing on $\uCH(\M)$ restricts to a non-degenerate pairing on the summands of the decomposition $\hyperlink{ucd}{\uCD(\M)}$.  Thus at the very beginning of our induction we are able to deduce $\hyperlink{cdo}{\CD_\o(\M)}$ and $\hyperlink{ucd}{\uCD(\M)}$.  The Poincar\'e duality statements $\hyperlink{pdo}{\PD_\o(\M)}$ and $\hyperlink{upd}{\uPD(\M)}$ then follow, as noted in Remark \ref{CDPD}.

In order to prove $\hyperlink{cd}{\CD(\M)}$, we need to show that the pairing on $\IH(\M)$ restricts to a non-degenerate pairing on the subspace $\K\varnothing\M = \psi^\varnothing(\uJ(\M))$.  In fact, $\uJ(\M)$ is a maximal subspace for which this will be true.  We show in Proposition \ref{uHL1CD} that this is a consequence of the hard Lefschetz property $\hyperlink{uhl}{\uHL(\M)}$ for $\uIH(\M)$.  The idea is the following.
Let $k \le d/2$, and take elements $\nu$ in $\uJ^{k-1}(\M) = \uIH^{k-1}(\M)$ and $\mu$ in $\uJ^{d-k-1}(\M) = \ub^{d-2k}\uIH^{k-1}(\M)$. (The degrees are chosen so that $\psi^\varnothing(\mu)$ and $\psi^\varnothing(\nu)$ are in complementary degrees.)  
Then the adjointness of the operators $\psi^\varnothing$ and $\varphi^\varnothing$ implies that the pairing of 
$\psi^\varnothing(\mu)$ and $\psi^\varnothing(\nu)$ is 
\begin{equation}\label{eqn:fiber pairing}
\deg_\M(\psi^\varnothing(\mu)\cdot \psi^\varnothing(\nu)) = -\udeg_{\M}(\ub\mu \cdot \nu).
\end{equation}
The non-degeneracy of this pairing on $\K\varnothing\M$ then follows because $\hyperlink{uhl}{\uHL(\M)}$ and $\hyperlink{upd}{\uPD(\M)}$ imply that the pairing of $\uIH^{k-1}(\M)$ and $\ub^{d-2k+1}\uIH^{k-1}(\M) = \uIH^{d-k}(\M)$ is non-degenerate.

\subsection{Geometric interpretation}\label{CD geometry}
Let us explain the geometry behind these definitions and statements when $\M$ is realizable over the field $\mathbb{C}$, as in Section \ref{sec:realizable}.  Recall that the augmented wonderful variety $X$ is obtained from the arrangement Schubert variety $Y$ by blowing up the proper transforms of the closures $\overline{U^F}$ 
of strata $U^F$ in order of increasing dimension, and in particular the exceptional divisor has a component $D_F$ for any proper flat $F$.    The divisor $D_\varnothing$ is the fiber of the resolution $\pi\colon X \to Y$ over the point stratum $U^\varnothing$; it is equal to the wonderful variety from Section \ref{sec:realizable}.
Its cohomology ring is identified with $\uCH(\M)$, and the restriction $\H^\bullet(X) \to \H^\bullet(D_\varnothing)$ is identified with the pullback $\varphi_\Mo^\varnothing\colon\CH(\M) \to \uCH(\M)$ of Definition~\ref{DefinitionXPullback}, while the Gysin pushforward $\H^\bullet(\underline{X}) \to \H^\bullet(X)$ is identified with $\varphi^\varnothing$. 
(In this section, all cohomology and intersection cohomology groups are taken with $\Q$ coefficients.)

More generally, for an arbitrary proper flat $F$, the map $\psi^F_\Mo$ of Definition~\ref{DefinitionXPullback} is the Gysin pushforward for the divisor $D_F$, and the map $\varphi^F$ is restriction to $D_F$.  The divisor $D_F$ is isomorphic to the product $\underline{X}_F \times X^F$, where $\underline{X}_F$ is the fiber of the resolution $X_F$ of the arrangement Schubert variety $Y_F$ over the point stratum, and
$X^F$ is the resolution of $\overline{U^F}$.  This gives the tensor product decomposition
\[\H^\bullet(D_F) = \H^\bullet(\underline{X}_F) \otimes_\Q \H^\bullet(X^F) = \uCH(\M_F) \otimes_\Q \CH(\M^F)\]
on the domain of $\psi^F$.

The ``lower" push/pull maps $\varphi_F$ and $\psi_F$, which factor multiplication by $y_F$, also have a simple geometric description. 
Recall from Section \ref{sec:realizable} that 
the arrangement Schubert variety $Y_F$ defined by the subspace $V_F \subseteq V$
embeds into $Y$ as a normally nonsingular slice to the stratum $U^F$. 
The closed embedding $Y_F\to Y$ lifts to an embedding of the wonderful models $X_F\to X$
and the maps $\varphi_F$ and $\psi_F$ are the restriction and Gysin pushforward with respect to this embedding.  This explains why these maps play an essential role in our inductive proof: knowing that our results hold for all matroids on smaller ground sets means that we understand the slices $Y_F$ and their resolutions $X_F$ for all $F$ of positive rank. 

To close the inductive loop, we must understand what happens at the point stratum $U^\varnothing$, which is not in any of the smaller slices.  The resolution $\pi\colon X \to Y$ factors as $X \stackrel{p}{\longrightarrow} Y_\o \stackrel{q}{\longrightarrow} Y$, where $q$ is the blow-up of $Y$ at   $U^\varnothing$.  The cohomology class of the exceptional divisor pulls back to the element $x_\varnothing$ in $\CH(\M)$, and the cohomology ring of $Y_\o$ is the ring $\H_\o(\M)$ obtained by adjoining $x_\varnothing$ to $\H(\M)$.  There is a natural stratification $Y_\o = \coprod_{\varnothing < F} U^F_\o$, and the stratum closure $\overline{U^F_\o}$ is isomorphic to the blow-up of $\overline{U^F}$ at the point stratum.
Applying the BBDG decomposition theorem to $p\colon X \to Y_\o$ gives an isomorphism between $p_*\underline{\mathbb{Q}}_X$ and a direct sum of shifts of intersection complexes $\IC^\bullet(\overline{U^F_\o})$, and taking cohomology gives a (non-canonical) expression of $\CH(\M)$ as a direct sum of shifts of modules $\IH_\o(\M^F)$.  Our decomposition $\hyperlink{cdo}{\CD_\o(\M)}$ is then a (canonical) coarsening of this direct sum decomposition.  In particular, $\IH_\o(\M)$ is isomorphic to the intersection cohomology of $Y_\o$ as
a module over $\H_\o(\M) = \H^\bullet(Y_\o)$. 

Applying the decomposition theorem to $q\colon Y_\o \to Y$ tells us that there exists an isomorphism
\begin{equation}\label{eqn:DT for blow-up}
q_*\IC(Y_\o) \cong \IC(Y) \oplus \bigoplus_{j\in \Z} i_*\underline{\Q}_{U^\varnothing}[j]^{\oplus m_j},
\end{equation}
where $i\colon U^\varnothing\to Y$ denotes the inclusion and $m_j \in \Z_{\ge 0}$.  We can understand 
the multiplicities $m_j$ of these skyscraper summands inside $q_*\IC(Y_\o)$ in the following way.  We have a 
bilinear pairing
\begin{equation}\label{eqn:sheaf pairing}
\Hom(i_*\underline{\Q}_{U^\varnothing}[j], q_*\IC(Y_\o)) \times \Hom(q_*\IC(Y_\o), i_*\underline{\Q}_{U^\varnothing}[j]) \to \Hom(i_*\underline{\Q}_{U^\varnothing}[j],i_*\underline{\Q}_{U^\varnothing}[j]) \cong \Q
\end{equation} 
given by composition, and the multiplicity $m_j$ is the rank of this pairing.  In other words, the skyscraper summands in the decomposition \eqref{eqn:DT for blow-up} can be identified with a maximal subspace on which this pairing is non-degenerate.  
A discussion of this fact can be found for instance in \cite[Section 3]{JMW}, see particularly Proposition 3.2. (Note that there it is assumed that the blow-up results in a smooth variety, but the argument still applies to the pushforward of the IC complex when the blow-up is singular.)

This pairing can be identified with the pairing \eqref{eqn:fiber pairing}, which explains why we are able to show that the subspace $\K\varnothing\M = \psi^\varnothing(\uJ(\M))$ is a maximal direct sum of trivial $\H(\M)$-modules appearing inside $\IH_\o(\M)$ as a direct summand.  Indeed, the fact that the generators $y_i$ of $\H(\M)$ must act trivially implies that the subspace is in the image of $\psi^\varnothing$, and the non-degeneracy of the pairing on this subspace implies that it is a direct summand.

To identify the pairings, we use adjunction and base change to get
\[\Hom(i_*\underline{\Q}_{U^\varnothing}[j], q_*\IC^\bullet(Y_\o)) \cong \H^{-j}(i^!q_*\IC^\bullet(Y_\o)) = \IH^{-j}(Y_\o, Y_\o \setminus \underline{Y}) \cong \IH^{-j-2}(\underline{Y}),\]
where $\underline{Y} = \pi_\o^{-1}(U^\varnothing)$ is the exceptional fiber of the blow-up $q$.  The last isomorphism holds by Poincar\'e-Verdier duality, because a neighborhood of $\underline{Y}$ in $Y_\o$ is isomorphic to a line bundle $L$ over $\underline{Y}$.  Applying adjunction and base change to the second term gives
\[\Hom( q_*\IC^\bullet(Y_\o),i_*\underline{\Q}_{U^\varnothing}[j]) \cong \H^{-j}(i^*q_*\IC(Y_\o))^* \cong \IH^{-j}(\underline{Y})^*,\]
and with these identifications the pairing \eqref{eqn:sheaf pairing} is given by multiplication 
\[\IH^{-j-2}(\underline{Y}) \to \IH^{-j}(\underline{Y})\]
 by the first Chern class $c_1(L)$, followed by the evaluation map.
But $\IH^\bullet(\underline{Y})$ is isomorphic to the module $\uIH(\M)$, and $c_1(L)$ acts as multiplication by $\varphi^\varnothing(x_\varnothing) = -\ub$, so the pairing \eqref{eqn:sheaf pairing} is identified with the pairing \eqref{eqn:fiber pairing}.

The variety $\underline{Y}$ can be viewed as a ``local" counterpart to $Y$, since the singularity of $Y$ at the point stratum is the affine cone over the projective variety $\underline{Y}$.  One of the reasons for the complexity of our inductive argument is the need to prove statements in both the ``local" and ``global" setting: we prove a canonical decomposition $\hyperlink{ucd}{\uCD(\M)}$ of $\uCH(\M)$ analogous to $\hyperlink{cd}{\CD(\M)}$, we prove the Hodge--Riemann relations $\hyperlink{uhr}{\uHR(\M)}$ for $\uIH(\M)$, and so on.  This is in contrast to Karu's proof for the combinatorial intersection cohomology of fans \cite{Karu}, where an important role is played by the fact that any affine toric variety is a (weighted) cone over a projective toric variety of dimension one less.

\subsection{Rouquier complexes}\label{sec:introduction-Rouquier}

As an intermediate step to proving $\hyperlink{uhl}{\uHL(\M)}$, 
we prove the weaker statement $\hyperlink{uns}{\uNS(\M)}$ (Definition~\ref{def:ns}).
When $d$ is even, the statement that there is no socle in degree exactly $(d-2)/2$ is equivalent to hard Lefschetz in that degree, since $\uIH^{\frac{d-2}{2}}(\M)$ and $\uIH^{\frac{d}{2}}(\M)$ have the same dimension by Poincar\'e duality.  
The no socle condition in this middle degree requires a more elaborate argument (discussed in Section~\ref{sec:middeg}), and our first step is to prove that $\uIH(\M)$ has no socle in degrees strictly less than $(d-2)/2$ (Proposition~\ref{uNS1}).

We do this by constructing a map of graded $\uH(\M)$-modules of the form
\begin{equation}\label{eq:map for uNS}
\uIH(\M) \to \bigoplus_{F} \uIH(\M^F)^{\oplus {m_F}}\left[\frac{1-\crk F}{2}\right],
\end{equation}
where $F$ runs over nonempty flats of odd corank, and $\{m_F\}$ are some nonnegative integers.  We show that this map is injective except in the top degree $d-1$, so except in that degree the socle of $\uIH(\M)$ is contained in the socle of the right hand side.  Because the maximal flat $E$ has even corank, all of the matroids $\M^F$ on the right side have smaller ground sets, so we can assume $\hyperlink{uns}{\uNS(\M^F)}$ holds by induction. This means that the socle of the summand indexed by $F$ vanishes in degrees less than or equal to
\[\frac{\rk F - 2}{2} + \frac{\crk F -1}{2} = \frac{d-3}{2},\]
and so we can conclude that $\hyperlink{uns}{\uNS^{<\frac{d-2}{2}}(\M)}$ holds.

The map \eqref{eq:map for uNS} arises by taking the stalk at the flat $\varnothing$ of the first differential of a complex $\bar{\C}^\bullet_\o(\M)$ of
graded $\H_\o(\M)$-modules, which we call the reduced {\bf Rouquier complex}.  It has the form
\begin{equation}\label{eqn:minimal Rouquier}
\IH_\o(\M) \to \bigoplus_F \IH_\o(\M^F)^{\oplus m_F}\left[\frac{1-\crk F}{2}\right] \to \bigoplus_G \IH_\o(\M^G)^{\oplus n_G}\left[\frac{2-\crk G}{2}\right] \to \cdots
\end{equation}
where the sums are over nonempty flats $F$, $G$, etc.\ for which the indicated shifts are nonpositive integers, and the first term $\IH_\o(\M)$ is placed in cohomological degree $0$.  

We find the complex $\bar{\C}^\bullet_\o(\M)$ as a minimal subcomplex of a larger but combinatorially simpler complex $\C^\bullet_\o(\M)$ defined as follows.  We put $\C^0_\o(\M) \coloneq \CH(\M)$, and for positive $k$, we put 
\begin{equation}\label{eqn:big Rouquier}
\C_{\o}^k(\M) \coloneq \bigoplus_{\varnothing < F_1 < \cdots < F_k<E} x_{F_1}\cdots x_{F_k}\CH(\M)[k].
\end{equation}
The entries of the differential are multiplication by monomials $x_F$, up to sign.  This complex will contain a number of acyclic two-step complexes $\dots \to 0 \to \N \stackrel{\sim}\to \N \to 0 \to \cdots$ as direct summands, and taking a complementary summand to all of them gives the complex $\bar{\C}^\bullet_\o(\M)$.  It is well-defined up to isomorphism of complexes of graded $\H_\o(\M)$-modules.

The modules $\C^k_\o(\M)$ are isomorphic to direct sums of graded $\H_\o(\M)$-modules of the form $\CH(\M^F)[\ell]$ (Lemma \ref{terms of Rouquier}). 
 We call $\H_\o(\M)$-modules of this form, and more generally direct summands of such modules,   
  {\bf pure} $\H_\o(\M)$-modules, in analogy with pure mixed Hodge modules and pure $l$-adic complexes in algebraic geometry.
Using the canonical decompositions $\hyperlink{cdo}{\CD_\o(\M^G)}$ for all nonempty flats $G$, we show that an $\H_\o(\M)$-module is pure if and only if it is a direct sum of modules of the form $\IH_\o(\M^G)[\ell]$ (Corollary \ref{cor:characterizing purity}).  An important step to proving this is showing that $\IH_\o(\M^F)$ is indecomposable as an $\H_\o(\M)$-module (Proposition \ref{prop:indecomposable}).

The fact that the summands in the minimal complex $\bar\C^\bullet_\o(\M)$ appear with shifts as in \eqref{eqn:minimal Rouquier} follows from the fact that the complex $\C^\bullet_\o(\M)$ is {\bf $\circ$-perverse} (Definition \ref{def:perverse complex}).  This condition is an algebraic analogue of perversity for constructible complexes on $Y_\o$, and it is defined in terms of stalk and costalk functors 
\[(\;)_{F}, (\;)_{[F]}\colon \H(\M)\mmod \to \Q\mmod\]
which are defined for any $F \in \cL(\M)$ (Definition \ref{def:stalk/costalk}).   A pure module $\N$ has a filtration whose subquotients give all costalks $\N_{[F]}$ and another filtration whose subquotients are the stalks $\N_F$, up to a shift (Proposition \ref{pure}).

 Applying these functors to a complex $\C^\bullet$ of pure graded $\H_\o(\M)$-modules gives complexes of graded vector spaces $\C^\bullet_F$, $\C^\bullet_{[F]}$.  The complex $\C^\bullet$ is said to be $\o$-perverse if, for every \emph{nonempty} flat $F$, the cohomology $\HH^i(\C^\bullet_F)$ vanishes in all grading degrees $j$ for which $i + 2j > \crk F$ and $\HH^i(\C^\bullet_{[F]})$ vanishes in all degrees $j$ with $i + 2j < \crk F$.   

Our main result about perverse complexes is Theorem \ref{thm:perverse complexes}, which says that if $\C^\bullet$ is a complex of pure $\H_\o(\M)$-modules which is $\o$-perverse and minimal, meaning that it does not contain any acyclic direct summands, then for a direct summand $\IH_\o(\M^F)[k]$ of $\C^i$, the shift must be $k = (i-\crk F)/2$.  This result is a version of the ``diagonal miracle" for complexes of Soergel bimodules appearing in the work of Elias and Williamson \cite[Section 6.5]{EW} \cite[Theorem 19.47]{EMTW}.  Proving Theorem \ref{thm:perverse complexes} requires estimates on the vanishing of stalks and costalks of $\IH_\o(\M^F)$ at nonempty flats $G < F$ (Proposition \ref{zero map}).  These estimates in particular imply that any complex of the form \eqref{eqn:minimal Rouquier} is $\o$-perverse, even if all differentials are zero.  

We show by directly computing the stalks and costalks that $\C^\bullet_\o(\M)$ is $\o$-perverse (Proposition \ref{prop:Rouquier perversity}).  Since the complex $\bar{\C}_\o^\bullet(\M)$ is obtained by splitting off acyclic direct summands of $\C^\bullet_\o(\M)$, it has the same stalk and costalk cohomology, and so is also $\o$-perverse.  Theorem \ref{thm:perverse complexes} then shows that $\bar\C^\bullet_\o(\M)$ has the form \eqref{eqn:minimal Rouquier}, except for showing that the first term is isomorphic to $\IH_\o(\M)$, which requires a small additional argument.

\begin{remark}
	We also construct a ``non-reduced" Rouquier complex $\bar\C^\bullet(\M)$, which is a complex of graded $\H(\M)$-modules which are pure, meaning that they are isomorphic to direct sums of direct summands of $\H(\M)$-modules $\CH(\M^F)[k]$.   This complex has a form just like \eqref{eqn:minimal Rouquier}, but with summands $\IH(\M^F)[k]$ in place of $\IH_\o(\M^F)[k]$, and including summands for the flat $F = \varnothing$.  We define perversity for complexes of graded $\H(\M)$-modules analogously to $\o$-perversity, but now the restrictions on stalks and costalks are imposed at every flat, including the empty flat.  Just as $\bar{\C}_\o^\bullet(\M)$ is $\o$-perverse, the non-reduced complex $\bar\C^\bullet(\M)$ is perverse.
	
	The argument to construct $\bar\C^\bullet(\M)$ is essentially the same as for $\bar\C_\o^\bullet(\M)$, except that the indecomposability of $\IH(\M)$ and the stalk and costalk estimates at the empty flat $\varnothing$ require the statements $\hyperlink{cd}{\CD(\M)}$ and $\hyperlink{ns}{\NS(\M)}$, which are not established until the end of our induction loop.  As a result, this complex does not play a role in our main induction.  We include it because it is more natural than $\IH_\o(\M)$, and because it can be used to prove that the inverse Kazhdan--Lusztig polynomial of $\M$ has nonnegative coefficients (Theorem \ref{thm:inverse} and Proposition \ref{prop:Rouquier multiplicities}).
\end{remark}

\begin{remark}\label{rmk:mixed motivation}
	The natural setting for studying these complexes would be  
	$K^b(\operatorname{Pure}(\H_\o(\M)))$ and $K^b(\operatorname{Pure}(\H(\M)))$, the homotopy categories of bounded complexes of pure $\H_\o(\M)$-modules or $\H(\M)$-modules.  These will be triangulated categories equipped with $t$-structures whose hearts are the categories of perverse and $\o$-perverse complexes, and in the realizable case they should be mixed versions (in the sense of \cite[Section 4]{BGS}) of the derived categories of sheaves on $Y$ (respectively on $Y_\o$), constructible with respect to the stratification by $U^F$ (respectively $U^F_\o$).  This is analogous to the use of the homotopy categories of Soergel bimodules or parity sheaves on flag varieties to model mixed sheaves with modular coefficients in the works of Achar--Riche and Makisumi \cite{AR16,Mak17}.  
	
	However, developing this formalism would add an additional layer of machinery from homological algebra to this paper, and since the key properties of the $t$-structure rely on results (Propositions \ref{zero map} and \ref{prop:indecomposable}) which are only known to hold as a result of the main induction, doing so would not offer any significant simplifications.  
	So we have elected not to pursue this approach here.	
\end{remark}

\begin{remark}	
  When $\M$ is realizable, the complexes $\bar\C^\bullet(\M)$ and $\bar\C^\bullet_\o(\M)$ can be viewed as representing certain ``Verma-type" perverse sheaves on the varieties $Y$ and $Y_\o$, respectively.  We discuss the case of $\bar\C^\bullet(\M)$; the complex $\bar\C^\bullet_\o(\M)$ can be understood similarly.  
  
  Consider the proper pushforward $j_!\underline{\Q}_{U^E}$ of the constant sheaf along the inclusion 
  $j\colon U^E \to Y$ of the open stratum into $Y$.  Since $U^E$ is affine, this is a perverse sheaf, up to a shift in degree. 	It is naturally a mixed sheaf, using either Saito's mixed Hodge modules or mixed $l$-adic sheaves, so it carries a weight filtration whose graded pieces are semisimple perverse sheaves.  The modules $\bar{\C}^j(\M)$ are the cohomologies of these graded pieces, and the differentials are induced by the $\Ext^1$ classes between successive pieces.  
	
  The quasi-isomorphic complex $\C^\bullet(\M)$ has a similar description in terms of the resolution $p \colon X \to Y$.  The map $p$ restricts to an isomorphism from $U \coloneq p^{-1}(U^E_\o)$ to $U^E$, so we have
  	$j_!\underline{\Q}_{U^E} = p_!(j_U)_!\underline{\Q}_U$, where $j\colon U \to X$ is the inclusion.
  The complement $X \setminus U$ is a divisor with normal crossings, with one component for each proper flat, and the nonempty intersections of these divisors are indexed by chains of flats.  The $i$-th graded piece of the weight filtration of the perverse sheaf $(j_U)_!\underline{\Q}_U$  
  is (up to a shift) the direct sum of constant sheaves on all $i$-fold intersections of divisors.  Then $\C^i(\M)$ is the cohomology of this graded piece as a module over $\H^\bullet(Y;\Q) = \H(\M)$.
\end{remark}

\subsection{Hard Lefschetz for $\IH(\M)$}

The proof of the statement $\hyperlink{hl}{\HL(\M)}$ (Definition \ref{def:hls}) follows a standard argument similar to one which appears in \cite{Karu} and \cite{EW}, using restriction to divisors to deduce the hard Lefschetz theorem from the Hodge--Riemann relations for smaller matroids (Proposition~\ref{dCM}).  The key fact is that multiplication by $y_F$ factors
as the composition of the maps $\varphi_F^{\Mo}$ and $\psi_F^{\Mo}$ (Proposition \ref{Prop_ymult}).
We take a class $\ell = \sum_{F\in \cL^1(\M)} c_F y_F$ in $\H^1(\M)$ with all $c_F$ positive, as in the statement of Theorem \ref{prop:kahler}.
If we have a class $\eta \in \IH^k(\M)$ for $k < d/2$ for which $\ell^{d-2k}\eta = 0$, applying $\varphi_F^{\Mo}$ for any $F \in \cL^1(\M)$ gives
\[\varphi_F^{\Mo}(\ell)^{d-2k}\cdot \varphi_F^{\Mo}(\eta) = 0.\]
Since $\rk \M_F = d -1$, this says that $\varphi_F^{\Mo}(\eta)$ is a primitive class in $\IH^k(\M_F)$ with respect to the class $\ell' \coloneq \varphi_F^{\Mo}(\ell)$.
This class satisfies the hypotheses of Theorem \ref{prop:kahler} for the matroid $\M_F$, so we can assume inductively that the Hodge--Riemann relations hold for $\ell'$.
By Proposition \ref{lemma_degy} and Lemma \ref{lem_imageofIH}~(1), we have
\[
0  = \deg_\M(\ell^{d-2k}\eta^2) 
 = \sum_F c_F \deg_{\M_F}((\ell')^{d-2k-1}\varphi_F^{\Mo}(\eta)^2).
\]
The Hodge--Riemann relations for $\M_F$ imply that all of the degrees in the sum have the same sign.
Since the $c_F$ are all positive, these degrees must all vanish.  Since the Hodge--Riemann forms are non-degenerate, we must have $\varphi_F^{\Mo}(\eta) = 0$ for every $F$, and so $\eta$ is annihilated by every $y_F$.  In other words, $\eta$ is in the socle of the $\H(\M)$-module $\IH(\M)$.  However, we show in Proposition~\ref{CDNS} that the socle of $\IH(\M)$ vanishes in any degree less than or equal to $ d/2$ for which the canonical decomposition $\hyperlink{cd}{\CD(\M)}$ holds.  At this point in the induction, we only know that this decomposition holds outside of the middle degree $d/2$, but this is enough to conclude $\hyperlink{hl}{\HL(\M)}$.

\subsection{Deletion induction for $\IH(\M)$}

An important step of our argument is deducing the Hodge--Riemann relations $\hyperlink{hr}{\HR(\M)}$ and $\hyperlink{uhr}{\uHR(\M)}$ (Definition~\ref{def:hrs}), except possibly in the middle degree (postponed until Section~\ref{sec:middeg}), by inductively using the Hodge--Riemann relations for matroids on smaller sets.
The arguments for $\IH(\M)$ and $\uIH(\M)$ are somewhat parallel, but the case of $\IH(\M)$ is simpler, so we begin with it even though it appears later in the structure of the whole proof.

This step uses the relation between $\M$ and the {\bf deletion} $\M \setminus i$. 
This is a matroid on the set $E \setminus i$ whose independent sets are the independent sets of $\M$ which do not contain $i$.  The flats of $\Mi$ are precisely the sets of the form $F\setminus \{i\}$ where $F$ is a flat of $\M$ (see \cite[Proposition 3.3.1]{Oxley}, or Lemma~\ref{lemma:semi-small flats}).
We assume that $i$ is not a coloop of $\M$, which means that there is at least one basis which does not contain $i$, and so $\M$ and $\M \setminus i$ have the same rank.  If all elements of $E$ are coloops, then $\M$ is a Boolean matroid.  This is the base case of our induction; we prove Theorem~\ref{theorem_all} in this case by a direct calculation in Section~\ref{sec:boo}.
For simplicity, we assume in this section and in Section \ref{sec:udel} that all of the rank one flats are singletons, and in particular that $\{i\}$ is a flat.

There is a homomorphism $\theta_i^\Mo \colon \CH(\M \setminus i) \to \CH(\M)$ which takes $y_j$ to $y_j$ for each $j\ne i$, and so it sends $\H(\M\setminus i)$ injectively to $\H(\M)$ (Section~\ref{sec_deletion}).  This map plays a major role in the semi-small decomposition of $\CH(\M)$ obtained in \cite{BHMPW}.  In Section \ref{sec:deletion induction}, we prove the following result about the pullback $\theta^*_i\IH(\M)$ of our intersection cohomology module by this homomorphism (modulo a technical issue described in Remark \ref{rmk:IH tilde} below).  


\begin{theoremA} When considered as a complex of pure graded $\H(\Mi)$-modules placed in degree $0$, the module $\theta_i^*\IH(\M)$ is perverse.  As a consequence, $\theta_i^*\IH(\M)$ is isomorphic to a direct sum of graded  $\H(\M \setminus i)$-modules of the form 
\begin{equation}\label{eqn:shift}
\IH((\M\setminus i)^G)[-(\crk G)/2], \tag{$*$}
\end{equation}\renewcommand*{\theHequation}{notag.\theequation}where $G$ is a flat of $\M\setminus i$ of even corank.
\end{theoremA}

\begin{remark}\label{rmk:IH tilde}
At the stage of the induction at which this argument appears, we can only assume that the canonical decomposition $\hyperlink{cd}{\CD(\M)}$ holds in degrees outside of the middle degree when $d$ is even.  This means that we do not even know that $\IH(\M)$ is pure.  So we actually prove (Corollary \ref{centering}) that Theorem A holds for the pullback of a modified module $\tIH(\M)$, defined in Section~\ref{sec:delhl}, which we can prove is a direct summand of $\CH(\M)$ (Lemma~\ref{tilde PD}).  It equals $\IH(\M)$ except in the middle degree $d/2$, where it equals $\IH^{\frac{d}{2}}_\o(\M)$. Because of this, the argument below only gives the Hodge--Riemann relations for $\IH(\M)$ in degrees strictly less than $d/2$.  We need a separate argument later to handle the middle degree, which we highlight in Section~\ref{sec:middeg}.  Theorem A as stated above is true, but it can only be proved after the entire induction is finished.
\end{remark}

To prove that $\theta^*_i\tIH(\M)$ is a pure $\H(\Mi)$-module, we use the fact, proved in \cite{BHMPW},  that $\theta^*_i\CH(\M)$ is a direct sum of $\CH(\M\setminus i)$-modules of the form $\CH((\M\setminus i)^F)[k]$ for various flats $F \in \cL(\M\setminus i)$ and $k \in \Z$.  The proof that it is perverse relies on Proposition \ref{prop:semi-small stalk}, which says that the stalk $(\theta_i^*\N)_F$ of the pullback of a pure $\H(\M)$-module $\N$  at a flat $F \in \cL(\Mi)$ is isomorphic to the direct sum of the stalks of $\N$ at the flats $F$, $F\cup i \in \cL(\M)$ with certain shifts (if either $F$ or $F\cup i$ are not flats of $\M$, their contribution is zero).  Using the degree restrictions on the stalks of $\IH(\M)$ given by Proposition \ref{zero map}, the stalk conditions for perversity of $\theta_i^*\tIH(\M)$ follow.  Since stalks and costalks are interchanged by duality (Lemma~\ref{stalk-costalk}) and $\hyperlink{pd}{\PD(\M)}$ implies that $\tIH(\M)$ is self-dual, we also get the costalk conditions.

Because $\M\setminus i$ has a smaller ground set than $\M$, we can inductively assume that all of our statements hold for all of the matroids $(\M\setminus i)^G$ appearing in Theorem A.  In particular, $\IH((\M\setminus i)^G)$ satisfies hard Lefschetz and the Hodge--Riemann relations for any positive linear combination $\ell' = \sum_{j \ne i} c_j y_j \in \H(\M \setminus i)$.
The shift by $-(\crk G)/2$ in the summand \eqref{eqn:shift} ensures that each summand is centered at the same middle degree as $\IH(\M)$, so Theorem A implies that $\tIH(\M)$ satisfies hard Lefschetz for the class $\ell'$.  Since hard Lefschetz is vacuous in the middle degree, the same holds for $\IH(\M)$, proving the statement $\hyperlink{hli}{\HL_i(\M)}$ (Proposition~\ref{HLi}).
By keeping careful track of how the Poincar\'e pairing restricts to  the summand \eqref{eqn:shift} (Lemma~\ref{pair and compare}), we also deduce that the Hodge--Riemann inequalities hold for $\ell'$ acting on $\tIH(\M)$. Since $\IH(\M)$ and $\tIH(\M)$ are equal except for the middle degree, this shows that the statement $\hyperlink{hri}{\HR_i^{< \frac{d}{2}}(\M)}$ also holds (Corollary~\ref{HRi}).

Next we use a standard deformation argument to pass from the special class $\ell'$ to a class $\ell = \ell' + c_iy_i$ with positive $c_i$. We have already shown $\hyperlink{hl}{\HL(\M)}$, $\hyperlink{hli}{\HL_i(\M)}$, and $\hyperlink{hri}{\HR_i^{< \frac{d}{2}}(\M)}$; that is, $\IH(\M)$ satisfies hard Lefschetz for both $\ell$ and $\ell'$, and the Hodge--Riemann relations hold for $\ell'$.  But for a continuous family of classes all of which satisfy hard Lefschetz, the signature of the associated pairings cannot change, so the Hodge--Riemann relations for $\ell'$ imply them for $\ell$.  Hence, we have deduced the statement $\hyperlink{hr}{\HR^{< \frac{d}{2}}(\M)}$ (Proposition~\ref{delnodel}).

\begin{remark}\label{rmk:deletion geometry}
When $\M$ is realizable, the theorem above follows from a study of the properties of a map $q\colon Y \to Y'$, obtained as the restriction of the projection $(\mathbb{P}^1)^E \to (\mathbb{P}^1)^{E\setminus i}$ to $Y$.  The image $Y' = q(Y)$ is again an arrangement Schubert variety as considered in Section \ref{sec:realizable}, given by restricting the map $\sigma\colon E \to V^\vee$ to $E\setminus i$. 
The map is compatible with the stratifications: we have $q(U^F) = U^{F\setminus i}$ for each $F \in \cL(\M)$.  It is also clear that the fibers of $q$ are either points or rational curves $\mathbb{P}^1$.  An easy computation with stalks shows that $q_*\IC(Y)$ is perverse, and by the decomposition theorem, it is semisimple. These two properties together give the theorem.  We point to \cite[Section 1.1]{BV} for more geometric insight in this direction.

The map $q\colon Y \to Y'$ resembles a map
which naturally appears in the inductive computation of intersection cohomology of classical Schubert varieties.
Let $G \supset B \supset T$ be a reductive algebraic group along with a choice of Borel subgroup and maximal torus, and let $W$ be the associated Weyl group.  For any $y \in W$, the intersection cohomology complex of the classical Schubert variety $X_y \coloneq ByB/B \subseteq G/B$ corresponds to the Kazhdan--Lusztig basis element $C_y$ in the Hecke algebra of $W$.  If $s$ is a simple reflection and $ys>y$, then the map
\[X_y \ast X_s \coloneq ByB \times_B BsB/B \to BysB/B = X_{ys}\]
induced by multiplication has fibers that are either points or rational curves, and the source is 
a $\mathbb{P}^1$-bundle over $X_y$.  The pushforward of $\IC(X_y\ast X_s)$ along this map is perverse, and it is isomorphic to a direct sum of $\IC(X_{ys})$ and the IC complexes of smaller classical Schubert varieties, all with the appropriate perverse shifts.  This is reflected in the fact that in the formula
\[C_yC_s = C_{ys} + \sum_{\substack{x<y\\xs < x}} \mu(x,y)C_x\]
(see, for example, \cite[Section 1.5, Formula (2)]{Sp82}) the coefficients $\mu(x,y)$ are integers, not more general Laurent polynomials.

Despite these similarities, the roles of the source and target in the two situations are different.  In our case, the base $Y'$ is a simpler variety which we can assume inductively that we already understand.  In contrast, the classical Schubert variety map uses inductive knowledge about $X_y$ to deduce results about the base $X_{ys}$.
\end{remark}

\subsection{Deletion induction for $\uIH(\M)$}\label{sec:udel}
In Section \ref{sec:underlined deletion induction}, we use a similar argument to deduce hard Lefschetz and the Hodge--Riemann relations for $\uIH(\M)$ from the same statements for matroids on smaller ground sets.  There is one significant difficulty, however.  We would like to decompose $\uIH(\M)$ as a direct sum of terms of the form 
\begin{equation}\label{eqn:underlined deletion summand}
\uIH((\M \setminus i)^G)[-(\crk G)/2], \tag{$**$}
\end{equation}\renewcommand*{\theHequation}{notag2.\theequation}but these are not modules over the same ring.  The operators $\ub_\M$ and $\ub_{\M\setminus i}$ which act on these spaces are the images of $-x_\varnothing$ in $\CH(\M)$ and $\CH(\M\setminus i)$, respectively.  However, the natural map $\CH(\M\setminus i) \to \CH(\M)$ sends $x_\varnothing$ to $x_\varnothing + x_{\{i\}}$, so $\ub_{\M\setminus i}$ is sent to $\ub_{\M} - x_{\{i\}}$.  But $x_{\{i\}}$ does not act on $\uIH(\M)$, so we must consider the larger space $\uIH_i(\M)$ (Definition \ref{def:uIHi}).
It is this space that we decompose into a sum of terms of the form \eqref{eqn:underlined deletion summand} (Corollary~\ref{uIH_i in terms of uIH_1}). 

As a result, we can use the inductive assumptions for matroids $(\M\setminus i)^G$ to show that 
hard Lefschetz and Hodge--Riemann hold for the action of
$\ub_\M - x_{\{i\}}$ on $\uIH_i(\M)$ (Corollaries \ref{uHL_i} and \ref{uHR_i}). This statement, combined with $\hyperlink{uns}{\uNS(\M)}$, implies hard Lefschetz for $\ub_\M$ on $\uIH(\M)$ (Proposition~\ref{NS under half}).  By deforming $\ub_{\M} - x_{\{i\}}$ to $\ub_{\M}$, we get the Hodge--Riemann relations as well (Proposition~\ref{alt}). However, as noted in Section \ref{sec:introduction-Rouquier}, in our first pass we only prove $\hyperlink{uns}{\uNS(\M)}$ strictly below the critical degree $(d-2)/2$, so we only get hard Lefschetz and Hodge--Riemann in that range as well.  When $d$ is even, we need an additional chain of arguments to finish the proof in this degree.

\begin{remark}\label{rmk:underlined semismall geometry}
	Let us explain the geometric meaning of $\uIH_i(\M)$ when $\M$ is realizable.  The map $Y \to Y'$ from Remark \ref{rmk:deletion geometry} induces a rational map $\underline{Y} \dashrightarrow \underline{Y}'$ between the corresponding ``local" varieties.  Blowing up at the point $p$ where this map is not defined gives a regular map $\underline{q}\colon \operatorname{Bl}_p \underline{Y} \to \underline{Y}'$.  The module $\uIH_i(\M)$ is then the intersection cohomology of $\operatorname{Bl}_p \underline{Y}$.	
	Just as in Remark \ref{rmk:deletion geometry}, the fibers of $\underline{q}$ are again either points or rational curves, so a similar argument shows that $\underline{q}_*\IC(\operatorname{Bl}_p \underline{Y})$ is perverse. Together with the decomposition theorem this shows that $\uIH_i(\M)$ is a direct sum of terms of the form \eqref{eqn:underlined deletion summand}.
\end{remark}

\subsection{The middle degree}\label{sec:middeg}
Finally, we are left with the problem of proving the Hodge--Riemann relations in the middle degree $\IH^{\frac{d}{2}}(\M)$. 
Although the space of primitive classes depends on the choice of an ample class $\ell$, if we already know the Hodge--Riemann relations in degrees below $d/2$, then showing them in middle degree is equivalent to showing that the signature of the Poincar\'e pairing on the whole space $\IH^{\frac{d}{2}}(\M)$ is $\sum_{k\ge 0} (-1)^k \dim \IH^k(\M)$ (Proposition \ref{why hancock?}).

We say that a graded vector space with non-degenerate pairing that satisfies this condition on the pairing in middle degree is {\bf Hancock} (that is, ``has a nice signature").  This condition is preserved by taking tensor products and orthogonal direct sums (Lemma~\ref{Lemma_Hproduct}).  In \cite{BHMPW}, we showed that $\CH(\M)$ satisfies Hodge--Riemann, so in particular it is Hancock.  The fact that $\uIH(\M)$ satisfies hard Lefschetz and Hodge--Riemann implies that $\uJ(\M)$ does too, so we can deduce that each summand $\K{F}{\M}$ in the decomposition $\hyperlink{cd}{\CD(\M)}$ is Hancock (Corollary~\ref{Hancock piece}).
If every term but one in an orthogonal direct sum decomposition is Hancock, and the whole space is as well, then the remaining summand is Hancock (Lemma~\ref{lemma_Hadditive}). Thus, once we have the canonical decomposition $\hyperlink{cd}{\CD(\M)}$, we can deduce that $\IH(\M)$ is Hancock and thus satisfies Hodge--Riemann in middle degree (Proposition~\ref{other sig}).

At this point, our induction still has a gap because we have not proved the decomposition $\hyperlink{cd}{\CD(\M)}$ in the middle degree $d/2$.  To fix this, we first work with $\IH_\o(\M)$, which we do know is a direct summand of $\CH(\M)$.  Following the argument of the previous paragraph shows that $\IH_\o(\M)$ satisfies the Hodge--Riemann relations in all degrees (Propositions~\ref{sig} and~\ref{block1}), and this implies that $\IH_\o(\M)$ has no socle in degrees less than or equal to $d/2$ as an $\H_\o(\M)$-module (Proposition~\ref{HR1NS1}).  Because $\uIH(\M)$ is the quotient of $\IH_\o(\M)$ by the action of the generators of $\H(\M)$, this implies the full condition $\hyperlink{uns}{\uNS(\M)}$, including in the missing degree $(d-2)/2$ (Proposition~\ref{NS1uNS1}).  But the lack of socle in $\uIH^{\frac{d-2}{2}}(\M)$ is equivalent to hard Lefschetz in that degree (Proposition~\ref{NS1uHL}), which gives the final ingredient needed to close the induction loop and prove the full canonical decomposition $\hyperlink{cd}{\CD(\M)}$ (Proposition~\ref{uHL1CD}).

\section{Modules over the graded M\"obius algebra}\label{sec:modules}
Let $\M$ be a matroid on a nonempty ground set $E$.
In this section we study some basic constructions 
involving graded modules over the graded M\"obius algebra $\H(\M)$.
This section is entirely independent of Section \ref{sec:IC of a matroid}.

\subsection{Annihilators}\label{sec:annihilators}
We begin with a general lemma about annihilators of ideals in Poincar\'e duality algebras.
\begin{lemma}
Let $R$ be a finite-dimensional commutative algebra 
equipped with a degree map with respect to which $R$ satisfies 
Poincar\'e duality as in Theorems \ref{TheoremChowKahlerPackage} (1) and \ref{theorem_underlineKahler} (1). For any ideal $I$,
\[
\Ann(I)=I^\perp,
\]
where the perp is taken with respect to the Poincar\'e duality pairing of $R$.
\end{lemma}
\begin{proof}
The inclusion $\Ann(I)\subseteq I^\perp$ is obvious. Conversely, if $f\in R$ and $f\notin \Ann(I)$, then there exists $g\in I$ such that $fg\neq 0$. By Poincar\'e duality of $R$, there exists $h\in R$ such that $\deg(fgh)\neq 0$. Since $gh\in I$, it follows that $f\notin I^\perp$. Therefore, $I^\perp\subseteq \Ann(I)$.
\end{proof}

Using the preceding lemma, the next result follows directly from elementary linear algebra.
\begin{lemma}\label{PD annihilators}
Let $R$ be as in the preceding lemma. 
Let $I,J\subseteq R$ be ideals.
Let $\Ann(I)$ denote the annihilator of $I$ in $R$.  The following identities hold:
\begin{enumerate}[(1)]\itemsep 5pt
\item If $J = \Ann(I)$, then $I = \Ann(J)$;
\item $\Ann(I + J) = \Ann(I) \cap \Ann(J)$;
\item $\Ann(I \cap J) = \Ann(I) + \Ann(J)$.
\end{enumerate}
\end{lemma}


\begin{lemma}\label{mutual annihilators}
The ideals $\langle x_\varnothing\rangle$ and $\langle y_i\ | \ i\in E\rangle$
are mutual annihilators inside of $\CH(\M)$.
\end{lemma}

\begin{proof}
Since multiplication by $x_\varnothing$ factors as $\psi^\varnothing\varphi^\varnothing$ (Proposition \ref{Prop_xmult}) and 
$\psi^\varnothing$ is injective (Proposition \ref{lemma_xdegree}), 
the annihilator of $x_\varnothing$ is equal to the kernel of $\varphi_\Mo^\varnothing$.  By Lemma \ref{CH-to-uCH}, this kernel is exactly the ideal $\langle y_i\ | \ i\in E\rangle$.
The opposite statement follows from 
Theorem \ref{TheoremChowKahlerPackage} (1) and Lemma \ref{PD annihilators} (1).
\end{proof}

An upwardly closed subset $\Sigma \subseteq \cL(\M)$ is called an {\bf order ideal}. For any flat $F$ of $\M$, we will denote the order ideals $\{G\ | \ G\geq F\}$ and $\{G\ | \ G> F\}$ by $\Sigma_{\geq F}$ and $\Sigma_{>F}$, respectively. 

\begin{definition}\label{ideal order filter}
For any order ideal $\Sigma$, we define an ideal of the graded M\"obius algebra 
$$\Upsilon_\Sigma \coloneq \operatorname{Span}_\Q \{ y_G \ | \ G\in \Sigma\} \subseteq\H(\M).$$
Note that $y_\varnothing = 1$ and $\Upsilon_{\cL(\M)} = \H(\M)$.
We will write 
\[
\Upsilon_{\geq F} \coloneq \Upsilon_{\Sigma_{\ge F}} \and \Upsilon_{> F} \coloneq \Upsilon_{\Sigma_{> F}}.
\]
Note that $\Upsilon_{\ge F}$ is the ideal $\langle y_F\rangle$, $\Upsilon_{>\varnothing}$ is the graded maximal ideal $\langle y_i\ | \ i\in E\rangle$ of $\H(\M)$, and 
\[\Upsilon_{>F} = \Upsilon_{\ge F}\Upsilon_{>\varnothing} = y_F\Upsilon_{>\varnothing}.\]
\end{definition}

The following lemma generalizes Lemma \ref{mutual annihilators}.

\begin{lemma}\label{annihilators}
For any order ideal $\Sigma$, the ideals $\CH(\M)\cdot \Upsilon_\Sigma$ and $\CH(\M)\cdot \{ x_F \ | \  F\notin \Sigma\}$ 
are mutual annihilators in  $\CH(\M)$.
\end{lemma}

\begin{proof}
By Lemma~\ref{PD annihilators}~(1), it is sufficient to prove that  $\CH(\M)\cdot \Upsilon_\Sigma$
is the annihilator of the set $\{ x_F \ | \  F\notin \Sigma\}$.
If $F\not\in\Sigma$ and $G\in\Sigma$, then $G\not\leq F$, and hence \[
y_{G}x_F = 0.
\]
This proves that $\CH(\M)\cdot \Upsilon_\Sigma$ annihilates $\{ x_F \ | \  F\notin \Sigma\}$.
For the opposite inclusion, we use downward induction on the cardinality of $\Sigma$.  

Suppose that $\Sigma$ is an order ideal and that the statement is true for all order ideals strictly containing $\Sigma$.
Let $\eta$ be an element  of $\CH(\M)$ satisfying $\eta x_F=0$ for all $F\notin \Sigma$.  We need to show that
$\eta$ is in the ideal $\Upsilon_\Sigma\cdot \CH(\M)$.

The base case $\Sigma=\cL(\M)$ is trivial.  
Now assume that $\Sigma\neq\cL(\M)$, and
let $H$ be a maximal flat not in $\Sigma$. Then $\eta x_H=0$,
and applying our inductive hypothesis to the order ideal $\Sigma\cup\{H\}$, we find that
\[
\eta \in \Upsilon_{\Sigma\cup\{H\}}\cdot \CH(\M) = y_H\CH(\M) + \Upsilon_\Sigma\cdot \CH(\M).
\]
This means that there exist elements $\xi$ and $\{\xi_F\mid F\in\Sigma\}$ in $\CH(\M)$ such that
\[
\eta = y_H \xi + \sum_{F\in \Sigma} y_F \xi_F.
\]
Since $H\notin\Sigma$, we have  $x_H y_F = 0$ for all $F\in\Sigma$, and hence
\[
0 = x_H \eta = x_Hy_H\xi + \sum_{F\in \Sigma} x_H y_F \xi_F = x_Hy_H\xi
= x_H\psi^\Mo_H\varphi^\Mo_H(\xi) = \psi^\Mo_H\bigl(x_\varnothing\varphi^\Mo_H(\xi)\bigr).
\]
Since $\psi^\Mo_H$ is injective, we have $x_\varnothing\varphi^\Mo_H(\xi) = 0 \in \CH(\M_H)$.
By Lemma \ref{mutual annihilators}, it follows that $\varphi^\Mo_H(\xi)$ is in the ideal $\langle y_{K\setminus H} \ | \ K > H\rangle
\subseteq \CH(\M_H)$.
Applying $\psi^\Mo_H$, we see that $y_H\xi = \psi^\Mo_H\varphi^\Mo_H(\xi)$ is in the ideal
$\langle y_{K} \ | \ K > H\rangle\subseteq\CH(\M)$.  By the maximality of $H$, any flat $K$ strictly containing $H$ is in $\Sigma$. Thus, $y_H\xi$ is in $ \Upsilon_\Sigma\cdot \CH(\M)$, and we conclude that
$\eta$ is in  $\Upsilon_\Sigma\cdot \CH(\M)$.
\end{proof}

\subsection{Stalks and costalks}\label{sec:stalks}

For an order ideal $\Sigma$ and a graded $\H(\M)$-module $\N$, we define
\[
\N_\Sigma \coloneq \Upsilon_{\Sigma} \cdot\N \and \N^{\Sigma} \coloneq \{n\in \N\ | \ \Upsilon_\Sigma\cdot n=0\}.
\]
Note that, if $\Sigma'\subseteq \Sigma$, then
$\N_{\Sigma'}\subseteq \N_{\Sigma}$ and $\N^{\Sigma}\subseteq \N^{\Sigma'}$.
We will write 
\[
\N_{\geq F} \coloneq \N_{\Sigma_{\ge F}} \and \N^{\geq F} \coloneq \N^{\Sigma_{\ge F}},
\]
and similarly for the order ideal $\Sigma_{> F}$.  


\begin{definition}\label{def:stalk/costalk}
We define the {\bf stalk} of $\N$ at $F$ to be the quotient
$$\N_F \coloneq \frac{\N_{\geq F}[\rk F]}{\N_{> F}[\rk F]}.$$
Dually, we define the {\bf costalk} of $\N$ at $F$ to be the quotient
$$\N_{[F]} \coloneq \frac{\N^{>F}}{\N^{\geq F}}.$$
\end{definition}

The stalk and costalk of $\N$ are again $\H(\M)$-modules, but because $\Upsilon_{>F} = \Upsilon_{\ge F}\Upsilon_{>\varnothing}$, all of the generators $y_i$ act by zero.  So we will consider them as functors from graded $\H(\M)$-modules to graded $\Q$-modules.

The stalk and costalk at the smallest flat $\varnothing$ are particularly simple: $\N_{\ge F} = \N$ and $\N^{\ge F} = 0$, so $\N_\varnothing$ is the quotient $\N/\Upsilon_{>\varnothing}\N$ and 
$\N_{[\varnothing]}$ is the socle of $\N$.  Since the kernel of $\varphi^\varnothing$ is $\Upsilon_{>\varnothing}\CH(\M)$, we have the following result.

\begin{lemma}\label{empty (co)stalk}
If $\N$ is a direct summand of $\CH(\M)$ as an $\H(\M)$-module, then we 
have a natural isomorphism
$\N_\varnothing \cong \varphi_\Mo^\varnothing(\N)\subseteq\uCH(\M)$.
\end{lemma}

The stalk or costalk functors at a flat $F$ can be described in terms of the stalk or costalk functors at the empty flat for the contraction matroid $\M_F$, by the following result. 
Recall that, by Corollary \ref{cor:y_FN is an H(M_F) module}, the submodule $y_F\N$ can naturally be regarded as an $\H(\M_F)$-module.

\begin{lemma}\label{everything is empty}
For any graded $\H(\M)$-module $\N$, there are natural isomorphisms
$$\N_F \cong (y_F\N[\rk F])_{\varnothing} \and \N_{[F]} \cong (y_F\N[\rk F])_{[\varnothing]},$$
where the stalk and costalk are taken for the flat $\varnothing \in \cL(\M_F)$.
\end{lemma}

\begin{proof}
We have $\N_{\ge F} = y_F\N = (y_F\N)_{\ge \varnothing}$ and
$\N_{>F} = \Upsilon_{>F}\N = \Upsilon_{>\varnothing}(y_F\N)$, so we have
$$\N_F = \frac{\N_{\geq F}[\rk F]}{\N_{> F}[\rk F]} = \frac{(y_F\N)_{\geq\varnothing}[\rk F]}{(y_F\N)_{>\varnothing}[\rk F]} \cong (y_F\N)_\varnothing[\rk F]
\cong (y_F\N[\rk F])_\varnothing.$$ 
Note in the second equality that the image of the maximal ideal $\Upsilon_{>\varnothing} \subseteq \H(\M)$ under the homomorphism $\varphi_F\colon\H(\M) \to \H(\M_F)$ is $\Upsilon_{>\varnothing} \subseteq \H(\M_F)$.

For the second statement, note that $\N^{\ge F}$ is the kernel of multiplication by $y_F$, so we have
$$\N_{[F]} =  \frac{\N^{>F}}{\N^{\geq F}} \cong y_F (\N^{>F})[\rk F] = (y_F\N)^{>\varnothing}[\rk F] 
= (y_F\N)_{[\varnothing]}[\rk F]\cong (y_F\N[\rk F])_{[\varnothing]}. \qedhere$$  
\end{proof}

With these results, we can describe all of the stalks and costalks of 
$\CH(\M)$.

\begin{proposition}
	We have
	\[\CH(\M)_E \cong \CH(\M)_{[E]} \cong \Q.\]
	For a proper flat $F \in \cL(\M)$, we have isomorphisms
	\[\CH(\M)_F \cong \uCH(\M_F)\;\;\text{and}\;\;\CH(\M)_{[F]}\cong \uCH(\M_F)[-1].\]
\end{proposition}
\begin{proof}
	The first part follows immediately from the fact that $y_E\CH(\M) \cong \CH(\M_E)[-\rk E] \cong \Q[-\rk E]$, since $\M_E$ is the empty matroid.
	
	For the second part, we first consider the case $F = \varnothing$.  The isomorphism $\CH(\M)_\varnothing \cong \uCH(\M)$
	is a special case of Lemma \ref{empty (co)stalk}, while we have
	\[\CH(\M)_{[\varnothing]} = x_\varnothing\CH(\M) = \psi^\varnothing(\CH(\M)) \cong \uCH(\M)[-1],\]
	by Lemma \ref{mutual annihilators}.  The general result now follows using Lemma \ref{everything is empty} and the isomorphism $y_F\CH(\M) = \psi_F(\CH(\M)) \cong \CH(\M_F)[-\rk F]$, which holds by Proposition \ref{lemma_degy} and Proposition \ref{Prop_ymult}.
\end{proof}

Next we describe how the stalk and costalk functors interact with duality. For a graded $\H(\M)$-module $\N$, we write $\N^*$ for $\Hom_\Q(\N,\Q)$.
Note that $\N^*$  has a natural graded $\H(\M)$-module structure.

\begin{lemma}\label{stalk-costalk}
For any graded $\H(\M)$-module $\N$ and any flat $F$, there is a natural isomorphism of graded $\Q$-modules
\[
(\N_F)^* \cong (\N^*)_{[F]}.
\]
\end{lemma}

\begin{proof}
We first prove the lemma when $F = \varnothing$.  
The module $(\N_\varnothing)^*$ is equal to the submodule of $\N^*$ consisting of functions that vanish on $\N_{>\varnothing}$, 
which is the same as $(\N^*)_{[\varnothing]}$.  

Now consider an arbitrary flat $F$. By Lemma \ref{everything is empty} and the case
that we just proved,
we have 
\[(\N_F)^* \cong \bigl((y_F\N[\rk F])_\varnothing\bigr)^* \cong \bigl((y_F\N[\rk F])^*\bigr)_{[\varnothing]} \cong (y_F\N)^*[-\rk F]_{[\varnothing]}.
\]
Since multiplication by $y_F$ is an $\H(\M)$-module homomorphism of degree $\rk F$, we have 
\[
(y_F\N)^*[- \rk F]\cong y_F(\N^*)[\rk F].
\]
Therefore,  we have 
\[
(\N_F)^* \cong \big(y_F(\N^*)[\rk F]\big)_{[\varnothing]}\cong (\N^*)_{[F]},
\]
where the second isomorphism follows from Lemma \ref{everything is empty}.
\end{proof}

\subsection{Pure modules}\label{sec:pure}

In this section we introduce a special class of graded $\H(\M)$-modules called pure modules, which in a sense are the main objects of study in this paper.  In particular, once the big induction is complete, our main result Theorem \ref{theorem_all} implies that $\IH(\M)$ is pure.

\begin{definition}\label{def:pure}
We say that a graded $\H(\M)$-module (respectively a graded $\H_\o(\M)$-module) is {\bf pure} if it is a direct sum of modules which are isomorphic to direct summands of graded $\H(\M)$-modules (respectively of graded $\H_\o(\M)$-modules) of the form $\CH(\M^F)[k]$, where $F\in \cL(\M)$ and $k\in \Z$.	
\end{definition}

\begin{remark}
	It is clear that a pure $\H_\o(\M)$-module is also pure when considered as an $\H(\M)$-module, but a pure $\H(\M)$-module need not even admit an $\H_\o(\M)$-module structure.
\end{remark}

\begin{remark}
	The notion of pure modules is motivated by the notion of pure mixed Hodge modules or pure mixed $\ell$-adic sheaves.   More precisely, following the discussion of Remark \ref{rmk:mixed motivation}, there will be a mixed structure on the category of complexes of pure $\H(\M)$-modules (or $\H_\o(\M)$-modules), such that pure modules placed in cohomological degree zero would be pure of weight zero.
\end{remark}

\begin{proposition}\label{prop:purity characterization}
	Let $\N$ be a pure graded $\H(\M)$-module.  Then for any order ideals $\Sigma_1$, $\Sigma_2 \subseteq \cL(\M)$, we have
	\begin{enumerate}
		\item $\N_{\Sigma_1} \cap \N_{\Sigma_2} = \N_{\Sigma_1\cap\Sigma_2}$ and
		\item $\N^{\Sigma_1} + \N^{\Sigma_2} = \N^{\Sigma_1\cap \Sigma_2}$.
	\end{enumerate}
\end{proposition}

\begin{proof}
	The properties (1) and (2) are preserved under taking direct sums, passing to direct summands, and shifting  degree, so we may assume that $\N = \CH(\M^F)$ for some flat $F$.  Furthermore, the $\H(\M)$-module structure on $\CH(\M^F)$ factors through the quotient $\H(\M)\to \H(\M^F)$ induced by 
	the pullback map $\varphi_F$ (see Proposition \ref{lemma_degy}), and so we can further assume that $F = E$, or in other words that $\M^F = \M$.
	
	To see that (1) holds, we compute: 
	\begin{eqnarray*}
		\CH(\M)_{\Sigma_1}\cap \CH(\M)_{\Sigma_{2}} &=& \CH(\M) \cdot \Upsilon_{\Sigma_1} \cap \CH(\M) \cdot \Upsilon_{\Sigma_{2}}\\
		&=& \Ann\{x_G\ | \  G \notin \Sigma_1\} \cap  \Ann\{x_G\ | \  G \notin\Sigma_{2}\}\\
		&=& \Ann\{x_G\ | \  G \notin \Sigma_1\cap \Sigma_2\}\\
		&=& \CH(\M)\cdot \Upsilon_{\Sigma_1 \cap \Sigma_2}\\
		&=& \CH(\M)_{\Sigma_1\cap \Sigma_2}
	\end{eqnarray*}
    where the second and fourth equalities follow from Lemma \ref{annihilators}.  Similarly, we get (2) from
	\begin{eqnarray*}
		\CH(\M)^{\Sigma_{1}} + \CH(\M)^{\Sigma_2}
		&=& \Ann\Upsilon_{\Sigma_{1}} +  \Ann\Upsilon_{\Sigma_2}\\
		&=& \CH(\M)\cdot \{ x_G \ | \  G\notin \Sigma_{1}\} + \CH(\M)\cdot \{ x_G \ | \  G\notin\Sigma_2\}\\
		&=& \CH(\M)\cdot \{ x_G \ | \  G\notin \Sigma_1\cap\Sigma_2\}\\
		&=& \Ann \Upsilon_{\Sigma_1\cap\Sigma_2}\\
		&=& \CH(\M)^{\Sigma_1\cap \Sigma_2},
	\end{eqnarray*}
	again by Lemma \ref{annihilators}.
\end{proof}

\begin{remark}\label{rmk:IH of a matroid}
	In a forthcoming work \cite{BHMPW-IH}, we show that the converse of Proposition \ref{prop:purity characterization} also holds, so pure modules are characterized by the properties (1) and (2).  For a module $\N$ satisfying the property (2), the assignment $\Sigma \mapsto \N/\N^\Sigma$ defines a sheaf on the finite topological space $\cL(\M)$, with order ideals as open sets.  This sheaf is automatically flabby, so we refer to modules satisfying (2) as {\bf flabby}.  Dually, we call a module satisfying the property (1) {\bf co-flabby}, because it implies that $\Sigma \mapsto \N_\Sigma$ is a co-sheaf with the property that the sections on any open set inject into global sections.  It is easy to see that $\N$ is flabby if and only if $\N^*$ is co-flabby, and vice-versa.
The converse to Proposition \ref{prop:purity characterization} allows us to prove that $\IH(\M)$ is uniquely characterized by the following conditions: 
	\begin{itemize}
		\item $\IH(\M)$ is indecomposable and $y_E\IH(\M) \ne 0$,
		\item there is an isomorphism $\IH(\M)^* \cong \IH(\M)[\rk \M]$ of graded $\H(\M)$-modules, and
		\item $\IH(\M)$ is flabby (or co-flabby).
	\end{itemize} 
Furthermore, we show in Proposition \ref{prop:indecomposable} that the only endomorphisms of $\IH(\M)$ are multiplication by scalars, so the isomorphism is unique once it is fixed in degree zero. This also leads to an inductive construction of $\IH(\M)$ directly as an $\H(\M)$-module, without reference to an embedding into $\CH(\M)$.
\end{remark}

	
We defined the stalks and costalks of an $\H(\M)$-module $\N$ as subquotients associated to the ideals $\Upsilon_{>F} \subseteq \Upsilon_{\ge F}$.  The next proposition says that, when $\N$ is pure, any pair of monomial ideals differing by $y_F$ will give isomorphic subquotients.  This implies that $\N$ has a filtration whose successive quotients give its stalks at all of the flats, and another filtration that gives the costalks.  This is the main ingredient in the proof of Proposition \ref{prop:gr} from the introduction.

\begin{proposition}\label{pure}
Let $\N$ be a pure graded $\H(\M)$-module, and let $\Sigma,\Sigma' \subseteq \cL(\M)$ be order ideals with $\Sigma' = \Sigma \cup \{F\}$, so 
$F$ is a minimal element of $\Sigma'$.
\begin{enumerate}[(1)]\itemsep 5pt
\item the inclusions $\Upsilon_{> F}\subseteq \Upsilon_{\Sigma}$ and $\Upsilon_{\geq F}\subseteq \Upsilon_{\Sigma'}$ induce an isomorphism 
$$\N_{F} = \frac{\N_{\geq F}[\rk F]}{\N_{> F}[\rk F]}
\overset{\cong}{\longrightarrow}
\frac{\N_{\Sigma'}[\rk F]}{\N_{\Sigma}[\rk F]}.$$
\item For all $k$, the above inclusions induce an isomorphism 
$$
\frac{\N^{\Sigma}}{\N^{\Sigma'}}
\overset{\cong}{\longrightarrow}
\frac{\N^{>F}}{\N^{\geq F}}=\N_{[F]}.$$
\end{enumerate}
\end{proposition}

\begin{proof}

Since  $\Upsilon_{\Sigma'} = \Upsilon_{\Sigma} + \Upsilon_{\ge F}$, we have $\N_{\Sigma'}=\N_{\Sigma} + \N_{\geq F}$, and so the
map in (1) is surjective.  To see that it is injective, apply Proposition \ref{prop:purity characterization} to get \[\N_{\ge F} \cap \N_{\Sigma} = \N_{\Sigma_{\ge F} \cap \Sigma} = \N_{>F}.\]

Turning to the statement (2), note that $\N^{\Sigma} \cap \N^{\ge F} = 
\N^{\Sigma \cup \Sigma_{\ge F}} = \N^{\Sigma'}$, which shows injectivity.  Surjectivity follows from the second part of Proposition \ref{prop:purity characterization}.
\end{proof}

\begin{remark}	
	To see examples of modules for which the conclusions of Propositions \ref{prop:purity characterization} and \ref{pure} fail, let $\M$ be the Boolean matroid of rank two on the set $E = \{i,j\}$.  We denote the flats of $\cL(\M)$ by $\varnothing, i, j, ij$,
	omitting the braces for simplicity.  Then the quotient $\H(\M)/\langle y_i+y_j\rangle$ doesn't satisfy either part of Proposition \ref{prop:purity characterization}. Next, let $\N = \Upsilon_{>\varnothing}$ be the ideal in $\H(\M)$ generated by $y_i$ and $y_j$, or in other words the subspace spanned by $y_i$, $y_j$, and $y_{ij}$.  If $\Sigma = \{j,ij\}$ and $\Sigma' = \Sigma \cup \{i\}$, we have $\N_{\Sigma} = \N_{\Sigma'} = \Q y_{ij}$, so $\N_{\Sigma'}/\N_{\Sigma}=0$, but the stalk $\N_i$ is one-dimensional.  Thus part (1) of Proposition \ref{pure} fails for $\N$.  Part (2) will fail for the dual module $\N^\ast \cong \H(\M)/y_{ij}\H(\M)$.  
	 
	For another example, let $\M$ be the uniform matroid on $\{1,2,3,4\}$.  Then $\N = \H(\M)$, considered as a module over itself, satisfies statement (1) of Proposition \ref{pure} but not statement (2).  To see this, let $\Sigma'$ contain all flats of rank $2$ or $3$, and let $\Sigma = \Sigma'\setminus \{F\}$, where $F$ is any rank $2$ flat.  Then 
	$\N^\Sigma = \N^{\Sigma'} = \Upsilon_{\Sigma'}$, but $\N_{[F]}$ is one-dimensional.  The dual module $\H(\M)^*$ satisfies (2) but not (1).
\end{remark}

\subsection{Orlik--Solomon algebra}
Our second result about pure modules is Proposition \ref{costalk complex}, which gives a chain complex to compute the costalk $\N_{[\varnothing]}$ (and then, using Lemma \ref{everything is empty}, it can be used for any costalk).  It is only used once, in Section \ref{sec:Rouquier perversity}.  By definition, $\N_{[\varnothing]}$ is the kernel of the homomorphism
\[\N \to \bigoplus_{F \in \cL^1(\M)} y_F\N\]
given by multiplication by each generator $y_F \in \H^1(\M)$. 
Proposition \ref{costalk complex} says that if $\N$ is pure, this can be extended to the right to give a complex which is exact in all but the first place, whose $k$-th step is a direct sum of submodules $y_F\N$ for $\rk(F) = k$.
This complex can be viewed as performing a sort of inclusion-exclusion computation, but because the lattice $\cL(\M)$ is not Eulerian, the module $y_F\N$ may have to appear with multiplicity greater than one.  
More precisely, the multiplicity is $|\mu(\varnothing,F)|$, where $\mu$ is the M\"obius function of $\cL(\M)$.  The appropriate vector space we need with this dimension is the dual of a piece of the Orlik--Solomon algebra of $\M$, so we first recall a few facts about this algebra.  
We refer to \cite[Section 3.1]{Orlik-Terao} for more details. 

Let $\mathcal{E}^1$ be the vector space over $\Q$ with basis $\{e_i\mid i\in E\}$, and let $\mathcal{E}$ be the exterior algebra generated by $\mathcal{E}^1$. Define a degree $-1$ linear map $\partial_\cE \colon \cE\to \cE$ by setting $\partial_\cE 1=0$, $\partial_\cE e_i=1$, and
\[
\partial_\cE(e_{i_1}\cdots e_{i_l})=\sum_{k=1}^{l}(-1)^k e_{i_1}\cdots \widehat{e_{i_k}}\cdots e_{i_l} \ \ \text{for any $i_1, \ldots, i_l\in E$. }
\]
For any ordered subset $S=\{i_1, \ldots, i_l\}\subseteq E$, we denote $e_{i_1}\cdots e_{i_l}$ by $e_S$. 
The {\bf Orlik--Solomon} algebra of $\M$, denoted by $\OS(\M)$, is the quotient of $\cE$ by the ideal generated by $\partial_\cE e_S$ for all dependent sets $S$ of $\M$. The differential $\partial_\cE$ descends to a differential $\partial$ on $\OS(\M)$, and the complex $(\OS(\M), \partial)$ is acyclic whenever the rank of $ \M$ is positive.

For any flat $F$ of $\M$, we define a graded subspace $\cE_F$ of $\cE$ generated by those monomials $e_S$ 
for all subsets $S\subseteq E$ with closure $F$.
Then we have a direct sum decomposition 
\[
\cE=\bigoplus_{F\in \cL(\M)}\cE_F,
\]
which descends to a direct sum decomposition \cite[Proposition 2.10]{Orlik-Solomon}
\begin{equation}\label{eqn:Brieskorn}
\OS(\M)=\bigoplus_{F\in \cL(\M)}\OS_F(\M).
\end{equation}
Moreover, $\OS_F(\M)$ is concentrated in degree $\rk F$, and the natural ring map $\OS(\M^F)\to \OS(\M)$ induces an isomorphism of vector spaces
\[
\OS^{\rk F}(\M^F)\cong \OS_F(\M).
\]

\subsection{A complex to compute costalks}\label{sec:costalk}
In this section we present a chain complex which computes the costalk of a pure module at the empty flat.  This result is only used in Section \ref{sec:Rouquier complexes}.

Let $\N$ be a graded $\H(\M)$-module.  For all $0\leq k \leq d = \rk \M$, let
$$\N^k_! \coloneq \bigoplus_{F \in \cL^k(\M)} \OS_F(\M)^* \otimes y_F \N.$$
Note that since $\OS_F(\M)^*$ is concentrated in degree $-\rk F$, tensoring with $\OS_F(\M)^*$ and multiplying by $y_F$ has no net effect on degrees.

We define a differential $\delta^k \colon \N^k_!\to \N^{k+1}_!$ as follows.
If $F\in \cL^k(\M)$ and $G\in \cL^{k+1}(\M)$, then the $(F,G)$-component of $\delta^k$ is zero unless $F<G$.
If $F<G$, choose $i\in G\setminus F$ so that $y_G = y_i y_F$.  Then the $(F,G)$-component of $\delta^k$ is given on the first tensor factor
by the $(F,G)$-component of $\partial^* \colon \OS_F(\M)^*\to \OS_G(\M)^*$ and on the second tensor factor by multiplication by $y_i$.  The fact that $\delta\circ \delta = 0$ follows easily from the fact that $\partial^*$ is a differential.

This defines a functor $\N \mapsto \N^\bullet_!$ from graded $\H(\M)$-modules to complexes of graded $\Q$-vector spaces.   We have $\OS_\varnothing(\M) \cong \OS_F(\M) \cong \Q$ for any rank $1$ flat $F$,
so the degree $0$ and $1$ parts of this complex are isomorphic to
\[\N \to \bigoplus_{\rk F = 1} y_F\N,\]
where the differential is the sum of multiplication by $y_F$ for all rank $1$ flats $F$.  In particular the inclusion $\N_{[\varnothing]} \to \N$ gives a natural map of complexes
\[\N_{[\varnothing]} \to \N^\bullet_!,\]
where $\N_{[\varnothing]}$ is considered as a complex concentrated in degree zero. 

This is the first place where we need a notation for cohomology of a complex (the next will be in Section \ref{sec:Rouquier complexes}).  To avoid a clash with $\H(\M)$, we will denote cohomology by $\HH$.

\begin{proposition}\label{costalk complex}
If $\N$ is pure, then the above map of complexes induces an isomorphism on cohomology, so we have natural isomorphisms
$\HH^0(\N^\bullet_!) \cong \N_{[\varnothing]}$
and $\HH^m(\N^\bullet_!) = 0$ for all $m>0$.
\end{proposition}

\begin{proof}
Fix a total order $\varnothing = F_0,\ldots,F_r$ of $\cL(\M)$ refining the natural partial order, so that for any $k$, the set 
	$$\Sigma_k \coloneq \{F_k,\ldots,F_r\}$$
	is an order ideal.  
Consider the filtration
\[0 = (\N^{\Sigma_0})^\bullet_! \subseteq \dots \subseteq  (\N^{\Sigma_r})^\bullet_! \subseteq (\N^{\Sigma_{r+1}})^\bullet_! = \N^\bullet_!\]
obtained by applying the functor $(\cdot)^\bullet_!$ to the filtration
$0 = \N^{\Sigma_0} \subseteq \dots \subseteq \N^{\Sigma_r}\subseteq \N^{\Sigma_{r+1}} = \N$.

We claim that the quotient complex 
$$\label{eq:quotient complex}
\frac{(\N^{\Sigma_{k+1}})^\bullet_!}{(\N^{\Sigma_k})^\bullet_!}
$$
is acyclic when $k\neq 0$, and when $k=0$ it is quasi-isomorphic to the module $\N_{[\varnothing]}$ concentrated in cohomological degree zero.  Given this claim, the desired result then follows from the spectral sequence relating the cohomology of a filtered complex to the cohomology
of its associated graded complexes.

To show the above claim, we consider the short exact sequence
\[
0\to \N^{\Sigma_k\cup \Sigma_{\geq F}} \to \N^{\Sigma_k}\overset{\cdot y_F}{\longrightarrow} y_F\N^{\Sigma_k}\to 0,
\]
for any $k$ and any flat $F$.
This sequence maps injectively into the same sequence with $k$ replaced by $k+1$, 
so the snake lemma gives a short exact sequence
$$0\to \frac{\N^{\Sigma_{k+1}\cup \Sigma_{\geq F}}}{\N^{\Sigma_{k}\cup \Sigma_{\geq F}}} \to 
\frac{\N^{\Sigma_{k+1}}}{\N^{\Sigma_{k}}} \to 
\frac{y_F\N^{\Sigma_{k+1}}}{y_F\N^{\Sigma_{k}}}\to 0.$$
By Proposition \ref{pure} (2), the middle term of this sequence is isomorphic to $\N_{[F_k]}$.
If $F\leq F_{k}$, then $\Sigma_{k+1}\cup \Sigma_{\geq F} = \Sigma_{k}\cup \Sigma_{\geq F}$, 
and the first term in our sequence is therefore zero.
On the other hand, if $F\not\leq F_{k}$, then Proposition \ref{pure} (2) implies that the first
term of our sequence is $\N_{[F_k]}$, and that the first map in our sequence is an isomorphism.  Putting
these two observations together, we get an isomorphism
$$\frac{y_F\N^{\Sigma_{k+1}}}{y_F\N^{\Sigma_{k}}} \cong
\begin{cases}
\N_{[F_k]} &\text{if $F\leq F_{k}$,}\\
0 &\text{otherwise.}
\end{cases}$$
Furthermore, if $F < G\le F_k$ and $y_G = y_Fy_i$, the natural map 
\[\frac{y_F\N^{\Sigma_{k+1}}}{y_F\N^{\Sigma_{k}}} \to \frac{y_G\N^{\Sigma_{k+1}}}{y_G\N^{\Sigma_{k}}}\] induced by multiplication by $y_i$ becomes the identity on $\N_{[F_k]}$ under this isomorphism.

It follows that there is an isomorphism of complexes
$$\frac{(\N^{\Sigma_{k+1}})^\bullet_!}{(\N^{\Sigma_k})^\bullet_!} \cong \OS(\M^{F_k})^*\otimes \N_{[F_k]},$$
where the right-hand side has the differential $\partial^*\otimes\id_{\N_{[F_k]}}$.
This complex is acyclic unless $\rk \M^{F_k} = 0$, which happens only when $k=0$.  We have $\N^{\Sigma_0}=0$ and $\N^{\Sigma_1} = \N_{[F_0]} = \N_{[\varnothing]}$, with all $y_F$ with $\rk F >0$ acting by zero.
So the map of complexes $\N_{[\varnothing]} \to \N^\bullet_!$ is injective with image $(\N^{\Sigma_1})^\bullet_!$.  The result follows.
\end{proof}

\section{Intersection cohomology as a module over the graded M\"obius algebra}\label{sec:ihmodulesmobius}
In this section, we apply some of the constructions from Section \ref{sec:modules} to the intersection cohomology module $\IH(\M)\subseteq\CH(\M)$.  In particular, we study its stalks and costalks, and prove that under suitable hypotheses it is indecomposable as an $\H(\M)$-module.  We also study the $\H_\o(\M)$-module $\IH_\o(\M)$ similarly.

For most of the remainder of the paper, we will prove very few absolute statements.  Most of what we prove
will be of the form ``If X holds, then so does Y.''  At the end, we will use all of these results in a modular way
to complete our inductive proof of Theorem \ref{theorem_all}.

\begin{remark}\label{ignore the first}
The main results of this section are Propositions \ref{zero map} and \ref{prop:indecomposable}, and Corollary~\ref{cor:characterizing purity}.
Each of these results has two parts, the first pertaining to the module $\IH(\M)$ and the second pertaining to the module $\IH_\o(\M)$.
We note that only the second parts of these three results will be used in our large induction.  The first parts require that we know $\hyperlink{cd}{\CD(\M)}$ (Definition \ref{def:canonical decomps}),
and will only be applied after the induction is complete.  This was alluded to earlier in Remark \ref{indecomposability is subtle}.
\end{remark}

Recall that by Corollary \ref{cor:y_FN is an H(M_F) module}, $y_F\N$ is an $\H(\M_F)$-module for any 
$\H(\M)$-module $\N$ and any flat $F$.

\subsection{Stalks and costalks of the intersection cohomology modules}
\begin{lemma}\label{mult by y}
Let $F$ be a nonempty flat such that $\hyperlink{cd}{\CD(\M_F)}$ holds.
\begin{enumerate}[(1)]\itemsep 5pt
\item If $\hyperlink{cd}{\CD^k(\M)}$ holds, then $\varphi^\Mo_F(\IH^k(\M)) = \IH^k(\M_F)$.  If this holds for all $k$, then $\psi_F$
restricts to a graded $\H(\M_F)$-module isomorphism 
$\IH(\M_F)[-\rk F]\cong y_F\IH(\M)$. 
\item If $\hyperlink{cdo}{\CD^k_\o(\M)}$ holds, then $\varphi^\Mo_F(\IH^k_\o(\M)) = \IH^k(\M_F)$.  If this holds for all $k$, then 
$\psi_F$ restricts to a graded $\H(\M_F)$-module isomorphism 
$\IH(\M_F)[-\rk F]\cong y_F\IH_\o(\M)$. 
\end{enumerate}
\end{lemma}

\begin{proof}
We prove statement (1); the proof of (2) is identical.  For notational convenience,
we will assume that $\hyperlink{cd}{\CD^k(\M)}$ holds for all $k$, but in fact the argument makes sense one degree at a time.

For any nonempty proper flat $G$ of $\M$, we apply $\varphi^\Mo_F$ to the direct summand $\K{G}{\M}$. By \cite[Proposition 2.28]{BHMPW}, if $G\not\geq F$, then 
$\varphi^\Mo_F\K{G}{\M}=0$.  Thus applying Lemma \ref{lem_imageofIH} (2) gives 
\[
\varphi^\Mo_F\Big( \bigoplus_{G < E} \K{G}{\M} \Big)= \bigoplus_{F\leq G < E} \K{G\setminus F}{\M_F}.
\]
By Lemma \ref{lem_imageofIH} (1), we also have $\varphi^\Mo_F(\IH(\M)) \subseteq\varphi^\Mo_F(\IH_\o(\M)) \subseteq \IH(\M_F)$. Therefore, the map $\varphi^\Mo_F$ is compatible with the canonical decompositions in the sense that it maps $\IH(\M)$ to $\IH(\M_F)$ and it maps the sum of the smaller summands to the sum of the smaller summands. 
Since $\varphi^\Mo_F$ is surjective, it must restrict to a surjective map from $\IH(\M)$ to $\IH(\M_F)$, so
$\varphi^\Mo_F(\IH(\M)) = \IH(\M_F)$. Applying the injective map $\psi_F^\Mo$ to this equality, 
we obtain the second part of statement (1) from Proposition \ref{Prop_ymult}.
\end{proof}

\begin{proposition}\label{zero map}
Suppose that $F$ is a proper flat for which $\hyperlink{cd}{\CD(\M_F)}$, $\hyperlink{pd}{\PD(\M_F)}$, and $\hyperlink{ns}{\NS(\M_F)}$ hold.
\begin{enumerate}[(1)] \itemsep 5pt
\item If $\hyperlink{cd}{\CD(\M)}$ holds, then the costalk $\IH(\M)_{[F]}$ vanishes in degrees  less than or equal to $(\crk F)/2$ and 
the stalk $\IH(\M)_{F}$ vanishes in degrees greater than or equal to $(\crk F)/2$.  
\item If $F\neq\varnothing$ and $\hyperlink{cdo}{\CD_\o(\M)}$ holds, then 
the costalk $\IH_\o(\M)_{[F]}$ vanishes in degrees less than or equal to $(\crk F)/2$ and 
the stalk $\IH_\o(\M)_{F}$ vanishes in degrees greater than or equal to $(\crk F)/2$.  
\end{enumerate}
\end{proposition}

\begin{proof}
For any nonempty proper flat $F$, it follows from Lemmas \ref{everything is empty} and  
\ref{mult by y} (2) that
\[
\IH_\o(\M)_{[F]} \cong \bigl(y_F\IH_\o(\M)[\rk F]\bigr)_{[\varnothing]} \cong \IH(\M_F)_{[\varnothing]}.
\]
Since the costalk at the empty flat is equal to the socle, 
$\hyperlink{ns}{\NS(\M_F)}$ says that this graded vector space vanishes in degrees less than or equal to $(\rk \M_F)/2 = (\crk F)/2$.
Similarly, we have
\[
\IH_\o(\M)_F \cong \bigl(y_F\IH_\o(\M)[\rk F]\bigr)_{\varnothing} \cong \IH(\M_F)_{\varnothing}.
\]
By $\hyperlink{pd}{\PD(\M_F)}$, there is a natural isomorphism $\IH(\M_F)^* \cong \IH(\M_F)[\crk F]$ of $\H(\M)$-modules. 
Then by Lemma \ref{stalk-costalk}, we have
\[
\IH(\M_F)_\varnothing \cong \bigl((\IH(\M_F)^*)_{[\varnothing]}\bigr)^* \cong \bigl(\IH(\M_F)_{[\varnothing]}[\crk F]\bigr)^*.
\]
By $\hyperlink{ns}{\NS(\M_F)}$, it follows that $\IH(\M_F)_{[\varnothing]}$ vanishes in degrees less than or equal to $(\crk F)/2$, 
and hence $\IH(\M_F)_{[\varnothing]}[\crk F]$ vanishes in degrees less than or equal to $-(\crk F)/2$. 
Thus, $\IH_\o(\M)_F \cong\IH(\M_F)_{[\varnothing]}[\crk F]^*$ vanishes in degrees greater than or equal to $(\crk F)/2$.

This concludes the proof of statement (2).
When $F$ is a nonempty flat, the proof of (1) is identical. When $F=\varnothing$, $\hyperlink{ns}{\NS(\M)}$ implies that $\IH(\M)_{[\varnothing]}$ vanishes in degrees 
less than or equal to $d/2$. By $\hyperlink{pd}{\PD(\M)}$, $\IH(\M)_{\varnothing}$ vanishes in degrees greater than or equal to $d/2$.
\end{proof}

\subsection{Indecomposability of $\IH(\M)$ and $\IH_\o(\M)$ and pure modules}
The next result concerns the endomorphisms and indecomposability of the graded $\H(\M)$-module $\IH(\M)$ and the graded $\H_\o(\M)$-module $\IH_\o(\M)$.  
\begin{proposition}\label{prop:indecomposable}
Let $\M$ be a matroid with ground set $E$.
\begin{enumerate}[(1)]\itemsep 5pt
\item  Suppose that $\hyperlink{cd}{\CD(\M_F)}$, $\hyperlink{pd}{\PD(\M_F)}$, and $\hyperlink{ns}{\NS(\M_F)}$ hold for all proper flats $F$.  
Any endomorphism of the graded $\H(\M)$-module $\IH(\M)$ that induces the zero map on the stalk $\IH(\M)_E$
is in fact the zero endomorphism of $\IH(\M)$.
In particular, $\IH(\M)$ has only scalar endomorphisms, and is therefore indecomposable as an $\H(\M)$-module.
\item  Suppose that $E$ is nonempty, $\hyperlink{cdo}{\CD_\o(\M)}$ holds, and $\hyperlink{cd}{\CD(\M_F)}$, $\hyperlink{pd}{\PD(\M_F)}$, and $\hyperlink{ns}{\NS(\M_F)}$ hold for all nonempty proper flats $F$.  
Any endomorphism of the graded $\H_\o(\M)$-module $\IH_\o(\M)$ that induces an automorphism of the stalk $\IH_\o(\M)_E$
is in fact an automorphism of $\IH_\o(\M)$.  In particular, $\IH_\o(\M)$ is indecomposable as an $\H_\o(\M)$-module.
\end{enumerate}
\end{proposition}

\begin{proof}
For statement (1), we proceed by induction on the cardinality of the ground set $E$.  When $E$ is empty or consists of a singleton, the proposition is trivial.
Let $f$ be an endomorphism of $\IH(\M)$ that induces the zero map on $\IH(\M)_E$.  For each rank one flat $G$, Lemma \ref{mult by y}~(1)
implies that $y_G\IH(\M)\cong\IH(\M_G)[-1]$.  
Since $f$ restricts to an endomorphism of the graded $\H(\M_G)$-module $\IH(\M_G)$
that induces the zero map on the stalk $\IH(\M_G)_{E\setminus G} \cong \IH(\M)_E$, the inductive hypothesis implies that $f$
restricts to zero on each submodule $y_G\IH(\M)$.  Thus, the map $f \colon \IH(\M)\to\IH(\M)$ factors through the stalk
$\IH(\M)_\varnothing$ of $\IH(\M)$
and has image contained in the costalk $\IH(\M)_{[\varnothing]}$ of $\IH(\M)$.  Proposition~\ref{zero map}
tells us that there is no degree in which the stalk and the costalk are both nonzero, hence $f$ must be the zero map.
By Definition \ref{def:ihdef}, for every proper flat $F$, every element in $\K{F}{\M}$ has positive degree. As the orthogonal complement of these submodules, $\IH(\M)$ contains the top-degree subspace of $\CH(\M)$
which is spanned by $y_E$. Therefore, the stalk 
$$\IH(\M)_E = y_E\IH(\M)[\rk\M] = \Q y_E[\rk\M]$$
is one-dimensional, and $\IH(\M)$ has only scalar endomorphisms. 


Next, we prove statement (2).  Suppose that $f$ is an endomorphism, but not an automorphism, of $\IH_\o(\M)$ that induces an automorphism of the stalk $\IH_\o(\M)_E$. Since $\IH_\o(\M)_E=\Q y_E[\rk\M]$ is one-dimensional, the induced automorphism of $f$ on the stalk $\IH_\o(\M)_E$ must be multiplication by a nonzero scalar, which we denote by $c$. 

By Lemma \ref{mult by y} (2), we have $y_F\IH_\o(\M)\cong \IH(\M_F)[-\rk F]$ for any nonempty flat $F$.  By statement~(1),
the restriction of $f$ to $\IH_\o(\M)_{>\varnothing} = \sum_{F\neq \varnothing} y_F\IH_\o(\M)$ is equal to multiplication by $c$.
Since $f$ is not an automorphism and $\IH_\o(\M)$ is finite-dimensional, $f$ must not be injective.
Choose a nonzero homogeneous element $\eta$ of minimal degree in the kernel of $f$.
For any nonempty flat $F$, we have 
\[
cy^{}_F\eta = f(y^{}_F\eta) = y^{}_Ff(\eta) = y^{}_F\cdot 0 = 0.
\] 
Thus, $y^{}_F\eta=0$ for any nonempty flat $F$.
By Lemma \ref{mutual annihilators}, this implies that $\eta$ is a multiple of $x_\varnothing$ in $\CH(\M)$.  
By $\hyperlink{cdo}{\CD_\o(\M)}$, $\IH_\o(\M)$ is a direct summand of $\CH(\M)$ as an $\H_\o(\M)$-module. Hence, $\eta = x_\varnothing\xi$ for some $\xi\in\IH_\o(\M)$.  We have
$$0 = f(\eta) = f(x_\varnothing\xi) = x_\varnothing f(\xi).$$
Thus $f(\xi)$ is in the intersection of the annihilator of $x_\varnothing$ and 
$\IH_\o(\M)$, which is equal to $\IH_\o(\M)_{>\varnothing}$.

Let $\xi' = f(\xi)/c$.  Since $\xi'\in \IH_\o(\M)_{>\varnothing}$, we have $f(\xi') = c\xi' = f(\xi)$, and hence $f(\xi-\xi')=0$.  Since 
\[x_\varnothing(\xi-\xi') = x_\varnothing\xi - x_\varnothing f(\xi)/c =  x_\varnothing\xi = \eta \neq 0,\]
we have $\xi-\xi' \neq 0$.  This contradicts the minimality of the degree of $\eta$.
\end{proof}

\begin{remark}
	It is also true that the only endomorphisms of $\IH_\o(\M)$ as a graded $\H_\o(\M)$-module are multiplication by scalars.  We will prove this later, as Lemma \ref{lem:IHo endomorphisms}.  Although the statement would fit as part of Proposition \ref{prop:indecomposable}, the proof needs some results from the next section, so we postpone it until the section where it is used.
\end{remark}

Using Proposition \ref{prop:indecomposable}, we get the following basic characterization of pure $\H(\M)$-modules and pure $\H_\o(\M)$-modules.	
	
\begin{corollary}\label{cor:characterizing purity}
	Let $\M$ be a matroid with ground set $E$.
	\begin{enumerate}[(1)]\itemsep 5pt
		\item  Suppose that $\hyperlink{cd}{\CD(\M^G_F)}$, $\hyperlink{pd}{\PD(\M^G_F)}$, and $\hyperlink{ns}{\NS(\M^G_F)}$ hold for all flats $F < G$.  
		Then a graded $\H(\M)$-module is pure if and only if it is isomorphic to a direct sum of modules of the form
		$\IH(\M^G)[k]$ for $G \in \cL(\M)$ and $k \in \Z$.
		\item  Suppose that $E$ is nonempty, $\hyperlink{cdo}{\CD_\o(\M^G)}$ holds for all nonempty flats $G$, and $\hyperlink{cd}{\CD(\M^G_F)}$, $\hyperlink{pd}{\PD(\M^G_F)}$, and $\hyperlink{ns}{\NS(\M^G_F)}$ hold for all flats $\varnothing < F < G$.  Then a graded $\H_\o(\M)$-module is pure if and only if it is isomorphic to a direct sum of modules of the form
		$\IH_\o(\M^F)[k]$ for $F \in \cL(\M)\setminus \{\varnothing\}$ and $k \in \Z$.
	\end{enumerate}
\end{corollary}
\begin{proof}
	To prove (1), note that the decomposition $\hyperlink{cd}{\CD(\M)}$ expresses $\CH(\M)$ as a direct sum of $\IH(\M)$ and $\H(\M)$-submodules isomorphic to $\CH(\M^G)[k]$.  So using the decompositions $\hyperlink{cd}{\CD(\M^G)}$ inductively, we can write $\CH(\M)$ as a direct sum of modules of the form $\IH(\M^G)[k]$.  Proposition \ref{prop:indecomposable} then shows that these summands are all indecomposable as $\H(\M)$-modules.  The result follows.  Statement (2) follows similarly.
\end{proof}

\section{The submodules indexed by flats}\label{sec:starting induction}
In order to define the modules $\IH(\M)\subseteq\IH_\o(\M)\subseteq \CH(\M)$ and $\uIH(\M)\subseteq \uCH(\M)$,
we made use of the submodules $$\K{F}{\M} = \psi^F_\Mo\!\left(\uJ(\M_F)\otimes\CH(\M^F)\right)\subseteq \CH(\M)$$
for all proper flats $F$, and the submodules 
$$\uK{F}{\M} = \upsi^F_\Mo\!\left(\uJ(\M_F)\otimes\uCH(\M^F)\right)\subseteq \uCH(\M)$$
for all nonempty proper flats $F$.  The purpose of this section is to understand the 
relationship between the intrinsic Poincar\'e pairings on these pieces and the pairings induced by the inclusions
into the Chow ring and augmented Chow ring of $\M$.

\subsection{The Poincar\'e pairing on $\K{F}{\M}$}\label{sec:bigraded}
Suppose that
\[
\N = \bigoplus_{0\leq i,j\leq d} \N^{i,j}
\] is a finite-dimensional bigraded $\Q$-vector space.
Suppose further that $\N$ is equipped with a bilinear pairing $\langle-,-\rangle$
such that, if $\mu\in \N^{i,j}$ and $\nu\in \N^{k,l}$, then $\langle \mu, \nu\rangle \neq 0$ only when $i+j+k+l = d$.
Let $r\in \mathbb{N}$ be given.  We say that the pairing is {\bf adapted to \boldmath{$r$}} 
if it satisfies the following properties:
\begin{enumerate}[(1)]\itemsep 5pt
\item $\dim \N^{i, j}=\dim \N^{r-i, d-r-j}$ for any $0\leq i\leq r$ and $0\leq j\leq d-r$;
\item if $\mu\in \N^{i,j}$, $\nu\in \N^{k,l}$, and $i+k< r$, then $\langle\mu,\nu\rangle = 0$.
\end{enumerate}
Assuming that the original pairing is adapted to $r$, we define its {\bf \boldmath{$r$}-reduction} by
$$\langle\mu,\nu\rangle_r \coloneq \sum_{\substack{i,j,k,l\\ i+k=r}}\langle\mu_{ij},\nu_{kl}\rangle,$$
where $\mu_{ij}$ is the projection of $\mu$ to $\N^{i,j}$, and similarly for $\nu_{kl}$.

\begin{lemma}\label{triangular}
Suppose that the bilinear form $\langle-,-\rangle$ is adapted to $r$. Then $\langle-,-\rangle_r$ is non-degenerate if and only if $\langle-,-\rangle$ is non-degenerate. 
\end{lemma}

\begin{proof}
This translates to the statement that if a matrix is block upper triangular and its block diagonal part is nonsingular, then the original matrix is nonsingular. 
\end{proof}

We define a bilinear pairing on the space $\uJ(\M)[-1]$ by
\[
\langle \eta, \xi \rangle =-\udeg_{\M}(\ub\,\eta\,\xi).
\]
Notice that elements of degree $k$ can only pair non-trivially with elements of degree $d-k$.

\begin{lemma}\label{lemma_J}
Suppose that $\hyperlink{upd}{\uPD(\M)}$ and $\hyperlink{uhl}{\uHL(\M)}$ hold.
Then $\uJ(\M)[-1]$ satisfies Poincar\'e duality of degree $d$ with respect to this pairing.
\end{lemma}
\begin{proof}
	Take $k \le d/2$.  Then $k-1 \le (d-2)/2$, so
	\[\uJ(\M)[-1]^k = \uJ^{k-1}(\M) = \uIH^{k-1}(\M).\]
By $\hyperlink{upd}{\uPD(\M)}$ and $\hyperlink{uhl}{\uHL(\M)}$, this space is dually paired (under the pairing 
$(\eta,\xi) \mapsto \udeg_\M(\eta\xi)$) with
\[\uIH^{d-k}(\M) = \ub^{d-2k+1}\uIH^{k-1}(\M) = \ub \big(\uJ^{d-k-1}(\M)\big) = \ub \big( \uJ(\M)[-1]^{d-k}\big). \]	
It follows immediately that the pairing $\langle -,- \rangle$ is non-degenerate on $\uJ(\M)[-1]$.
\end{proof}

Let $F$ be a proper flat.  To understand the pairing on $\K{F}{\M}$, we will apply Lemma \ref{triangular} to the bigraded vector space $\uJ(\M_F)[-1]\otimes \CH(\M^F)$.
This vector space has two natural bilinear pairings.  The first, which we denote 
$\langle\cdot\,, \cdot\rangle_{F}$, is the tensor product of the Poincar\'e pairings on $\uJ(\M_F)[-1]$ and $\CH(\M^F)$.  The second, which we denote
$\langle\cdot\,, \cdot\rangle_{\CH(\M)}$,  is the restriction of the Poincar\'e pairing on $\CH(\M)$ via the inclusion
$$\uJ(\M_F)[-1]\otimes \CH(\M^F)\to \CH(\M)$$
induced by $\psi^F_\Mo$, which matches the total grading on the source with the grading on the target.
Similarly, the bigraded vector space $\uJ(\M_F)[-1]\otimes \uCH(\M^F)$
has two natural bilinear pairings.  The first, which we denote 
$\langle\cdot\,, \cdot\rangle_{{\underline{F}}}$, is the tensor product of the Poincar\'e pairings on $\uJ(\M_F)[-1]$ and $\uCH(\M^F)$.  The second, which we denote
$\langle\cdot\,, \cdot\rangle_{\uCH(\M)}$,  is the restriction of the Poincar\'e pairing on $\uCH(\M)$ via the inclusion 
$$ \uJ(\M_F)[-1]\otimes \uCH(\M^F)\to \uCH(\M)$$
induced by $\upsi^F_\Mo$.

\begin{proposition}\label{reduction}
Suppose that $\hyperlink{upd}{\uPD(\M)}$ and $\hyperlink{uhl}{\uHL(\M)}$ hold. Let $r = \crk F$.
\begin{enumerate}[(1)]\itemsep 5pt
\item
The pairing $\langle\cdot\,, \cdot\rangle_{F}$ 
on $\uJ(\M_F)[-1]\otimes \CH(\M^F)$ is adapted to $r$, and its $r$-reduction is
equal to the pairing $\langle\cdot\,, \cdot\rangle_{\CH(\M)}$.
\item
The pairing $\langle\cdot\,, \cdot\rangle_{{\underline{F}}}$ 
on $\uJ(\M_F)[-1]\otimes \uCH(\M^F)$ is adapted to $r$, and its $r$-reduction is
equal to the pairing $\langle\cdot\,, \cdot\rangle_{\uCH(\M)}$.
\end{enumerate}
\end{proposition}

\begin{proof}
We prove only part (1); the proof of part (2) is identical.
The first condition for adaptedness follows from the Poincar\'e duality statements of Lemma \ref{lemma_J} and Theorem \ref{TheoremChowKahlerPackage}.
For the second condition, let 
\[
\mu\in \uJ(\M_F)[-1]^{i}\otimes \CH^j(\M^F) = \uJ^{i-1}(\M_F)\otimes \CH^j(\M^F)
\] and 
\[\nu\in \uJ(\M_F)[-1]^{k}\otimes \CH^l(\M^F)= \uJ^{k-1}(\M_F)\otimes \CH^l(\M^F).
\]
By Lemma \ref{we need this} (1), 
we have
$$\langle\mu,\nu\rangle_{\CH(\M)}=
\deg_\M\left(\psi^F_\Mo(\mu)\cdot\psi^F_\Mo(\nu)\right)=
-\udeg_{\M_F}\otimes \deg_{\M^F}\bigl((\ub_{\M_F}\otimes 1+1\otimes \alpha_{\M^F})\mu\nu\bigr).$$
If $i+k<r$, then 
\[
(\ub_{\M_F}\otimes 1+1\otimes \alpha_{\M^F})\mu\nu\in \uCH^{<\crk(F)-1}(\M_F)\otimes \CH(\M^F)
\]
and hence $\langle\mu,\nu\rangle_{\CH(\M)}=0$.  This proves that the first pairing is adapted to $r$.
If $i+k=r$, then 
\[
(1\otimes \alpha_{\M^F})\mu\nu\in \uCH^{r-2}(\M_F)\otimes \CH(\M^F),
\]
hence we have
\[
\langle\mu,\nu\rangle_{\CH(\M)}=-\udeg_{\M_F}\otimes \deg_{\M^F}\bigl((\ub_{\M_F}\otimes 1)\mu\nu\bigr)
= \langle\mu, \nu\rangle_{F}.
\]
This completes the proof.
\end{proof}

\subsection{Beginning the induction: the coarse canonical decomposition}
In this section, we use Proposition \ref{reduction} and the assumption
that Theorem \ref{theorem_all} holds for smaller matroids to show that the decomposition $\hyperlink{cdo}{\CD_\o(\M)}$ 
holds.  We then show that $\hyperlink{cdo}{\CD_\o(\M)}$ implies some useful relations between our modules $\IH_\o(\M)$ and $\uIH(\M)$ under the push and pull operators $\psi^\varnothing$, $\varphi^\varnothing$.

Assume throughout the section that $E$ is nonempty.

\begin{corollary}\label{CD for free}
Assume that all of the statements of Theorem \ref{theorem_all} hold for $\M_F$ for every nonempty proper flat $F$.
Then the statements $\hyperlink{pdo}{\PD_\o(\M)}$, $\hyperlink{upd}{\uPD(\M)}$, $\hyperlink{cdo}{\CD_\o(\M)}$, and $\hyperlink{ucd}{\uCD(\M)}$ hold.
\end{corollary}

\begin{proof}
By Proposition \ref{summands}, the subspaces $\K{F}{\M}$ are mutually orthogonal as $F$ varies through all nonempty proper flats of $\M$. 
By Lemmas \ref{triangular} and \ref{lemma_J}, Proposition \ref{reduction}, and Theorem \ref{TheoremChowKahlerPackage} (1),
the restriction of the Poincar\'e pairing on $\K{F}{\M} = \psi^F_\Mo\!\left(\uJ(\M_F)\otimes \CH(\M^F)\right)\subseteq \CH(\M)$ is non-degenerate. 
These statements imply that the sum of these subspaces of $\CH(\M)$ is a direct sum and the restriction of the Poincar\'e pairing to this direct sum is non-degenerate. Since $\IH_\o(\M)$ is defined to be the orthogonal complement of the above direct sum, 
we have an orthogonal decomposition
\[
\CH(\M)=\IH_\o(\M)\oplus \bigoplus_{\varnothing < F < E}\K{F}{\M}
\]
and the restriction of the Poincar\'e pairing to $\IH_\o(\M)$ is also non-degenerate. Thus, $\hyperlink{pdo}{\PD_\o(\M)}$ and $\hyperlink{cdo}{\CD_\o(\M)}$ hold. The statements $\hyperlink{upd}{\uPD(\M)}$ and $\hyperlink{ucd}{\uCD(\M)}$ follow from the same arguments.
\end{proof}

\begin{proposition}\label{summand annihilator}
If $\hyperlink{cdo}{\CD_\o(\M)}$ holds, then $\langle x_\varnothing\rangle \cap \IH_\o(\M) = x_\varnothing
 \IH_\o(\M)$.
\end{proposition}

\begin{proof}
By $\hyperlink{cdo}{\CD_\o(\M)}$, we have
\begin{multline*}
\langle x_\varnothing\rangle \cap \IH_\o(\M) = x_\varnothing \CH(\M)\cap \IH_\o(\M) \\
= \Big(x_\varnothing\IH_\o(\M) \oplus \bigoplus_{\varnothing<F<E}\ x_\varnothing \K{F}{\M} \Big)\cap\IH_\o(\M) = x_\varnothing
\IH_\o(\M). \qedhere
\end{multline*}
\end{proof}

\begin{corollary}\label{underline it}
If $\hyperlink{cdo}{\CD_\o(\M)}$ holds, then
$\varphi_\Mo^\varnothing\!\left( \IH_\o(\M)\right) =  \uIH(\M)$.
\end{corollary}

\begin{proof}
Let $G$ be a nonempty proper flat of $\M$.  By Lemma \ref{push uIH} (1), we have  
\[\psi^\varnothing \uK{G}{\M} \subseteq \K{G}{\M}.\]
Therefore, $\IH_\o(\M)$ is orthogonal to $\psi^\varnothing_\Mo\uK{G}{\M}$ with respect to the Poincar\'e pairing on $\CH(\M)$. Then Proposition \ref{lemma_xdegree} (6) and (7) imply that $\varphi_\Mo^\varnothing\!\left(\IH_\o(\M)\right)$ is orthogonal to $\uK{G}{\M}$ with respect to the Poincar\'e pairing on $\uCH(\M)$.  Thus $\varphi_\Mo^\varnothing\!\left(\IH_\o(\M)\right)\subseteq\uIH(\M)$.

On the other hand, Lemma \ref{push uIH} (1) also gives 
\[\varphi_\Mo^\varnothing\K{G}{\M} = \uK{G}{\M}\]
for any nonempty proper flat $G$.
Hence $\uIH(\M)$ is orthogonal to $\varphi_\Mo^\varnothing\K{G}{\M}$ with respect to the Poincar\'e pairing on $\uCH(\M)$,
or equivalently (by Lemma~\ref{adjoint}) $\psi^\varnothing_\Mo\!\left(\uIH(\M)\right)$ is orthogonal to $\K{G}{\M}$ with respect to the Poincar\'e pairing on $\CH(\M)$.  
Thus $\psi^\varnothing_\Mo\!\left(\uIH(\M)\right)\subseteq \IH_\o(\M)$.

By the definition of $\psi_\Mo^\varnothing$, we have 
$\psi^\varnothing_\Mo\!\left(\uIH(\M)\right)\subseteq\langle x_\varnothing\rangle$.  Then by Proposition \ref{summand annihilator}, we have
$$\psi^\varnothing_\Mo\!\left(\uIH(\M)\right)\subseteq  \langle x_\varnothing\rangle\cap \IH_\o(\M)
= x_\varnothing\cdot \IH_\o(\M) = \psi^\varnothing_\Mo\varphi_\Mo^\varnothing\!\left(\IH_\o(\M)\right).$$
By the injectivity of $\psi^\varnothing_\Mo$, it follows that $\uIH(\M)\subseteq\varphi_\Mo^\varnothing\left(\IH_\o(\M)\right)$.
\end{proof}

\begin{corollary}\label{very convenient}
If $\hyperlink{cdo}{\CD_\o(\M)}$ holds, then $\langle x_\varnothing\rangle\cap \IH_\o(\M) = \psi^\varnothing_\Mo\!\left(\uIH(\M)\right)$.
\end{corollary}

\begin{proof}
By Corollary \ref{underline it} and Proposition \ref{Prop_xmult}, we have 
\[ \psi^\varnothing_\Mo\!\left(\uIH(\M)\right) = \psi^\varnothing_\Mo\varphi_\Mo^\varnothing \!\left(\IH_\o(\M)\right) = x_\varnothing\cdot \IH_\o(\M).
\]
The statement then follows from Proposition \ref{summand annihilator}.
\end{proof}

\begin{proposition}\label{uHL1CD}
	If $\hyperlink{cdo}{\CD_\o(\M)}$ holds, then for any $k\leq d/2$, we have
	\[
	\hyperlink{upd}{\uPD^{\leq k-1}(\M)} \;\;\text{and}\;\;\hyperlink{uhl}{\uHL^{\leq k-1}(\M)}\;\; \Longrightarrow\;\;\hyperlink{cd}{\CD^{\leq k}(\M)}.
	\]
\end{proposition}

\begin{proof}
	By $\hyperlink{cdo}{\CD_\o(\M)}$, the statement $\hyperlink{cd}{\CD^{\leq k}(\M)}$ is equivalent to the direct sum decomposition
	\[
	\IH_\o^{\leq k}(\M) = \IH^{\leq k}(\M) \oplus \psi^\varnothing_\Mo\big(\uJ^{\leq k-1}(\M)\big).
	\]
	By definition, $\IH(\M)$ is the orthogonal complement of $\psi^\varnothing_\Mo\big(\uJ(\M)\big)$ in $\IH_\o(\M)$. Thus, the above direct sum decomposition is equivalent to the statement that the Poincar\'e pairing of $\CH(\M)$ restricts to a non-degenerate pairing between $\psi^\varnothing_\Mo\big(\uJ^{\leq k-1}(\M)\big)$ and $\psi^\varnothing_\Mo\big(\uJ^{\geq d-k-1}(\M)\big)$.  Note that $\psi^\varnothing_\Mo\big(\uJ^{k-1}(\M)\big)$ and $\psi^\varnothing_\Mo\big(\uJ^{\leq d-k-1}(\M)\big)$ are in degrees $k$ and $d-k$ respectively, so they are in complementary degrees for the Poincar\'e pairing of $\IH_\o(\M)$. 
	
	By Lemma \ref{we need this} (1) with $F=\varnothing$ and the fact that $\a_{\M^\varnothing}=0$ for degree reasons, we have
	\[
	\deg_{\M}\big(\psi^\varnothing_\Mo(\mu)\cdot \psi^\varnothing_\Mo(\nu)\big)
	= -\udeg_{\M} \big(\ub\cdot \mu\nu\big)
	\]
	for $\mu, \nu \in \uCH(\M)$.
	Thus, by $\hyperlink{upd}{\uPD^{\leq k-1}(\M)}$ and $\hyperlink{uhl}{\uHL^{\leq k-1}(\M)}$, the Poincar\'e pairing of $\CH(\M)$ restricts to a non-degenerate pairing between $\psi^\varnothing_\Mo\big(\uJ^{\leq k-1}(\M)\big)$ and $\psi^\varnothing_\Mo\big(\uJ^{\geq d-k-1}(\M)\big)$.  
\end{proof}

\begin{proposition}\label{CDNS}
If $\hyperlink{cdo}{\CD_\o(\M)}$ holds, then for any $k\leq d/2$ we have
\[
\hyperlink{cd}{\CD^{k}(\M)}\;\; \Longrightarrow \;\;\hyperlink{ns}{\NS^{k}(\M)}.
\] 
\end{proposition}

\begin{proof}
Suppose that $\eta\in \IH^k(\M)$ and $y_i\eta = 0$ for all $i\in E$.
By Lemma \ref{mutual annihilators}, $\eta$ is a multiple of $x_\varnothing$.
Thus, Corollary \ref{very convenient} implies that 
\[
\eta\in\psi^\varnothing_\Mo\big(\uIH^{k-1}(\M)\big) = \psi^\varnothing_\Mo\big(\uJ^{k-1}(\M)\big).
\]
However, $\hyperlink{cd}{\CD^{k}(\M)}$ implies that $\IH^k(\M) \cap \psi^\varnothing_\Mo\!\left(\uJ^{k-1}(\M)\right) = 0$. Therefore, we have $\eta=0$.
\end{proof}

\subsection{The Hancock condition}\label{sec:hancock}
Let 
$\N = \bigoplus_{k\geq 0}\N^k$ be a finite-dimensional graded $\Q$-vector space equipped with a 
symmetric bilinear form.  Let 
$$P_{\N}(t) \coloneq \sum_{k\geq 0} t^k \dim \N^k$$ be the Poincar\'e polynomial of $\N$.  
The {\bf signature} of $\N$ is defined to be the difference between the dimension of the largest space on which our pairing is positive definite and the dimension of the largest space on which it is negative definite.  Equivalently, it is the number of positive eigenvalues minus the number of negative eigenvalues.
We say that $\N$ is {\bf Hancock} if its signature is equal to $P_\N(-1)$.  

\begin{remark}\label{that's odd}
If the symmetric bilinear form on $\N$ satisfies Poincar\'e duality of degree $d$, then
its signature is equal to the signature of its restriction to its piece in degree $d/2$.
In particular, if $d$ is odd, then the signature is necessarily zero, as is $P_\N(-1)$.  
Thus when $d$ is odd, the Hancock condition follows automatically from Poincar\'e duality.
\end{remark}

The motivation for the Hancock condition is the following proposition.

\begin{proposition}\label{why hancock?}
Suppose that $\oL \colon \N\to \N$ is a linear operator of degree $1$ with respect to which
$\N$ satisfies 
the hard Lefschetz theorem of degree $d$.
Suppose that $d$ is even and that $\N$ satisfies the Hodge--Riemann
relations of degree $d$ in all but the middle degree.  Then $\N$ satisfies the Hodge--Riemann relations
in middle degree if and only if $\N$ is Hancock.
\end{proposition}

\begin{proof}
The hard Lefschetz theorem induces a Lefschetz decomposition (see e.g., \cite{AHK}, Definitions 7.3, 7.4 and following)
$$\N^{d/2} = \bigoplus_{k=0}^{d/2} \oL^{(d/2) - k} \ker(\oL^{d-2k+1}).$$
For all $k\leq d/2$, the Hodge--Riemann relations in degree $k$ are equivalent to the statement that
the signature of the restriction of the bilinear form to $\oL^{(d/2) - k} \ker(\oL^{d-2k+1})$
is equal to $(-1)^k(\dim \N^k - \dim \N^{k-1})$.  If we assume the Hodge--Riemann relations in all but one degree,
this means that the Hodge--Riemann relations in the missing degree are equivalent to the statement
that the signature of the bilinear form is equal to
$$\sum_{k=0}^{d/2} (-1)^k(\dim \N^k - \dim \N^{k-1}).$$
By hard Lefschetz and the fact that $d$ is even, $$-(-1)^k\dim \N^{k-1} = (-1)^{d-k+1}\dim \N^{d-k+1},$$
thus the expected signature is $$\sum_{k=0}^{d/2} \left( (-1)^k\dim \N^k + (-1)^{d-k+1}\dim \N^{d-k+1}\right) = P_\N(-1).$$
This completes the proof.
\end{proof}

\begin{lemma}\label{Lemma_Hproduct}
If $\N$ and $\N'$ are both Hancock, then so are $\N\oplus\N'$ and $\N\otimes\N'$.
\end{lemma}

\begin{proof}
This follows from the fact that signature and Poincar\'e polynomial are both additive with respect to direct sum and multiplicative with respect to tensor product.
\end{proof}

\begin{lemma}\label{lemma_Hadditive}
Suppose that $\N$ is Hancock and $\N=\N_0\oplus \N_1\oplus\cdots\oplus \N_l$ is an orthogonal decomposition. 
If $\N_1, \ldots, \N_l$ are all Hancock, then so is $\N_0$. 
\end{lemma}

\begin{proof}
This follows from the fact that the signature and the Poincar\'e polynomial are both additive with respect to the orthogonal decomposition. 
\end{proof}

Now suppose that $\N$ has a bigrading $\N = \oplus_{i,j\ge 0} \N^{i,j}$ refining the given single grading, in the sense that $\N^k = \oplus_{i+j=k} \N^{i,j}$.
\begin{lemma}\label{signature triangular}
A graded bilinear form that is adapted to $r$ is Hancock if and only if its $r$-reduction is Hancock.
\end{lemma}

\begin{proof}
This follows from the fact that the matrix of the bilinear form and its block diagonal part with respect to the decomposition induced by the bigrading have the same multiset of eigenvalues.
\end{proof}

\begin{lemma}\label{lem:uJ is Hancock}
	Suppose that $\hyperlink{upd}{\uPD(\M)}$, $\hyperlink{uhl}{\uHL(\M)}$, and $\hyperlink{uhr}{\uHR(\M)}$ all hold.  Then $\uJ(\M)[-1]$ is Hancock with respect to the pairing $\langle -,-\rangle$.
\end{lemma}
\begin{proof}
	If $d$ is odd, the Hancock condition holds automatically by Lemma \ref{lemma_J}.  So suppose that $d = 2k$ is even, and let $\N = \uJ(\M)[-1]$.  Then the pairing in middle degree on 
	\[\N^k = \uJ^{k-1}(\M) = \uIH^{k-1}(\M)\]
	is precisely the Hodge--Riemann form $(\eta,\xi) \mapsto \udeg_\M(\ub \eta\xi)$ on $\uIH^{k-1}(\M)$ with respect to $\ub$.  So following the proof of Proposition \ref{why hancock?}, we have
	\[\uIH^{k-1}(\M) = \bigoplus_{j=0}^{k-1} \ub^{k-1-j}\, \ker\left((\beta^{d-2j}\cdot)\colon \uIH^j(\M)\to \uIH^{d-j}(\M)\right).\]
	By $\hyperlink{uhr}{\uHR(\M)}$, the signature of the pairing restricted to the summand indexed by $j$ is 
	\begin{align*}
	(-1)^{j+1}(\dim \uIH^j(\M) - \dim \uIH^{j-1}(\M)) &= (-1)^{j+1}\dim \N^{j+1} + (-1)^{j}\dim \N^j\\
	& = (-1)^{j+1} \dim \N^{j+1} + (-1)^{d-j}\dim \N^{d-j},
	\end{align*}
	using duality of the pairing on $\N$ and the fact that $d$ is even.  Adding up over all $0 \le j \le k-1$ gives $P_\N(-1)$, completing the proof that $\N$ is Hancock.
\end{proof}

\begin{remark}
	It is also possible to deduce Lemma \ref{lem:uJ is Hancock} from Proposition \ref{why hancock?}.  Although $\uJ(\M)$ is not closed under multiplication by $\ub$, multiplication by $\ub$ gives an isomorphism $\uJ(\M)[-1] \cong \ub\,\uIH(\M)$, and so it suffices to show $\ub\,\uIH(\M)$ is Hancock.  
By $\hyperlink{uhl}{\uHL(\M)}$, for $k<(d-1)/2$, the isomorphism $\cdot \ub^{d-1-2k}\colon \uIH^{k}(\M)\to \uIH^{d-1-k}(\M)$ induces an isomorphism 
\[
\cdot \ub^{d-2-2k}\colon \ub\,\uIH^{k}(\M)\to \uIH^{d-1-k}(\M)=\ub\,\uIH^{d-2-k}(\M).
\]
Therefore, the hard Lefschetz theorem of $\ub\,\uIH(\M)$ follows from $\hyperlink{uhl}{\uHL(\M)}$. Likewise, the Hodge--Riemann relations of $\ub\,\uIH(\M)$ follow from $\hyperlink{uhr}{\uHR(\M)}$, and hence $\ub\,\uIH(\M)$ is Hancock. 
\end{remark}

\begin{corollary}\label{Hancock piece}
Let $F$ be a nonempty proper flat of $\M$ such that $\hyperlink{upd}{\uPD(\M_F)}$,
$\hyperlink{uhl}{\uHL(\M_F)}$, and $\hyperlink{uhr}{\uHR(\M_F)}$ hold.  
The graded subspace $\K{F}{\M}$ is Hancock with respect to the Poincar\'e pairing on $\CH(\M)$,
and the graded subspace $\uK{F}{\M}$ is Hancock with respect to the Poincar\'e pairing on $\uCH(\M)$.
\end{corollary}

\begin{proof}
We prove the first statement; the proof of the second is the same.  Let $r = \crk F$.
By Proposition \ref{reduction} and Lemma \ref{signature triangular}, this is equivalent to the statement that
the graded vector space $\uJ(\M_F)[-1]\otimes \CH(\M^F)$ is Hancock with respect to the pairing 
$\langle\cdot\,, \cdot\rangle_{F}$.  
By Lemma \ref{Lemma_Hproduct}, it is sufficient to prove that 
 $\CH(\M^F)$ and $\uJ(\M_F)[-1]$ are both Hancock.
The first assertion follows from Theorem \ref{TheoremChowKahlerPackage} and Proposition \ref{why hancock?}.
The second assertion follows from Lemma \ref{lem:uJ is Hancock}.
\end{proof}

\begin{proposition}\label{sig}
Assume that $\hyperlink{upd}{\uPD(\M_F)}$, $\hyperlink{uhl}{\uHL(\M_F)}$, and $\hyperlink{uhr}{\uHR(\M_F)}$ hold for all nonempty proper flats $F$ of $\M$. Then 
\[
\hyperlink{cdo}{\CD_\o(\M)},\;\;\hyperlink{hlo}{\HL_\o(\M)},\;\;\text{and}\;\;\hyperlink{hro}{\HR_\o^{<\frac{d}{2}}(\M)}\;\; \Longrightarrow\;\;\hyperlink{hro}{\HR_\o(\M)}.
\]
\end{proposition}

\begin{proof}
We may assume that $d$ is even, as otherwise the conditions $\hyperlink{hro}{\HR_\o^{<\frac{d}{2}}(\M)}$ and $\hyperlink{hro}{\HR_\o(\M)}$
are the same.
Proposition \ref{why hancock?} tells us that we need to show that $\IH_\o(\M)$ is Hancock.  

By Corollary \ref{Hancock piece}, $\K{F}{\M}$ is Hancock for all nonempty proper flats $F$ of $\M$. 
Theorem \ref{TheoremChowKahlerPackage} tells us that there exists some $\ell\in\CH^1(\M)$
with respect to which $\CH(\M)$ satisfies the hard Lefschetz theorem and the Hodge--Riemann relations.
By Proposition \ref{why hancock?}, this implies that $\CH(\M)$ is Hancock.
Finally, $\hyperlink{cdo}{\CD_\o(\M)}$ combines with Proposition~\ref{summands} and Lemma \ref{lemma_Hadditive}
to imply that the direct summand  $\IH_\o(\M)\subseteq \CH(\M)$ is also Hancock.
\end{proof}

\begin{proposition}\label{other sig}
Suppose that $E$ is nonempty and the following statements hold:
\[
\hyperlink{cd}{\CD(\M)}, \; \hyperlink{hl}{\HL(\M)}, \;\hyperlink{hr}{\HR^{<\frac{d}{2}}(\M)},\; \hyperlink{hlo}{\HL_\o(\M)}, \; \hyperlink{hro}{\HR_\o(\M)}, \;
\hyperlink{upd}{\uPD(\M)}, \; \hyperlink{uhl}{\uHL(\M)}, \;\text{and}\; \hyperlink{uhr}{\uHR(\M)}.
\]
Then $\hyperlink{hr}{\HR(\M)}$ also holds.
\end{proposition}

\begin{proof}
By Proposition \ref{why hancock?}, it suffices to show that $\IH(\M)$ is Hancock.  
By $\hyperlink{cd}{\CD(\M)}$, we have
\[
\IH_\o(\M) = \IH(\M) \oplus \K{\varnothing}{\M}.
\]
Since $\hyperlink{pdo}{\PD_\o(\M)}$, $\hyperlink{hlo}{\HL_\o(\M)}$, and $\hyperlink{hro}{\HR_\o(\M)}$ hold, Proposition \ref{why hancock?} implies that $\IH_\o(\M)$ is Hancock. 
Since $\K{\varnothing}{\M}=\psi^\varnothing_\Mo \uJ(\M)\cong \uJ(\M)[-1]$, Lemma \ref{lem:uJ is Hancock} implies that $\K{\varnothing}{\M}$ is Hancock. Then Lemma \ref{lemma_Hadditive} shows that $\IH(\M)$ is Hancock.
\end{proof}

\section{Rouquier complexes}\label{sec:Rouquier complexes}
The main result of this section is Proposition \ref{uNS1}, which deduces the no socle condition $\hyperlink{uns}{\uNS^{<\frac{d-2}{2}}(\M)}$ in all except the degree closest to the middle, assuming 
$\hyperlink{cdo}{\CD_\o(\M)}$ and all our statements for matroids on fewer elements.  The tool we use to prove this is the \textbf{Rouquier complex}, a complex of pure graded $\H_\o(\M)$-modules which plays a role analogous to the Rouquier complex of Soergel bimodules in \cite{EW}. 

Let $\C^\bullet$ be a complex of pure graded $\H(\M)$-modules.  Applying the stalk and costalk functors to each step of this complex gives complexes $\C^\bullet_F$, $\C^\bullet_{[F]}$ of graded $\Q$-modules.  We will refer to these as the stalk complex and costalk complex of $\C^\bullet$ at $F$, respectively.

\begin{definition}\label{def:perverse complex}
	We say that a complex $\C^\bullet$ of pure graded $\H(\M)$-modules is \textbf{perverse} if, for all flats $F \in \cL(\M)$, we have
	\begin{enumerate}[(a)]\itemsep 5pt
		\item For all $i$, the $i$th cohomology $\HH^i(\C^\bullet_F)$ of the stalk complex vanishes in degrees $j$ for which $i + 2j > \crk F$, and
		\item for all $i$, the $i$th cohomology $\HH^i(\C^\bullet_{[F]})$ of the costalk complex  vanishes in degrees $j$ for which $i + 2j < \crk F$.
	\end{enumerate}
	If $\C^\bullet$ is a complex of pure graded $\H_\o(\M)$-modules, we say that it is \textbf{$\o$-perverse} if the above conditions hold for all nonempty flats $F$.  Note that direct sums and direct summands of perverse
	(respectively $\o$-perverse) complexes are again perverse (respectively $\o$-perverse).
\end{definition}

\begin{remark}
	In the realizable case, the homotopy category $K^b(\operatorname{Pure}(\H(\M)))$ of complexes of pure $\H(\M)$-modules is a ``mixed" analogue of the derived category of sheaves on $Y$ constructible with respect to the stratification by cells $U^F$.  From that point of view, the perverse complexes form the heart of a $t$-structure on $K^b(\operatorname{Pure}(\H(\M)))$, and many of the structures and results in geometric representation theory that hold for mixed perverse sheaves on flag varieties will have analogues in this setting.  But for the purposes of this paper we only need to construct one particular complex, so we do not pursue this formalism here.
\end{remark}

\begin{remark}
	We are using a somewhat nonstandard convention on shifts and grading.  To match with the standard definitions of perverse sheaves in topology, it would be more natural to put the generators $x_F$, $y_i$ in degree two, so that $\H(\M)$ and $\CH(\M)$ have even gradings, and adjust the shifts in the definition of perversity so that $\IH(\M^F)[\rk F]$ placed in cohomological degree $0$ would be perverse.  Then taking the graded dual of a complex would preserve perversity, as would the ``twist" functor which shifts a complex up by $1$ in homological degree and down by $1$ in grading degree.  But the convention we use matches the polynomials $Z_\M(t)$ and $P_\M(t)$ more closely, and it makes the notation somewhat simpler in other parts of the paper.
\end{remark}

\subsection{Minimal subcomplexes and perversity}
We begin by stating the following standard lemma, sometimes described as ``Gaussian elimination'' for complexes, 
whose proof can be found in \cite[Lemma 4.2]{BN07}.

\begin{lemma}\label{basic fact}
	Suppose that $(\C^\bullet,\partial)$ is a complex in some additive category and we have direct sum decompositions of two consecutive objects
	$$\C^k = \mathrm{P}^k \oplus \mathrm{Q}^k \and \C^{k+1} = \mathrm{P}^{k+1} \oplus \mathrm{Q}^{k+1}$$
	for some $k$ with the property that the composition 
	$$\mathrm{P}^k\hookrightarrow\C^k\overset{\partial^k}{\longrightarrow}\C^{k+1}\twoheadrightarrow\mathrm{P}^{k+1}$$ is an isomorphism.
	Then $(\C^\bullet,\partial)$ has as a direct summand a two-step acyclic complex whose $k$-th and $(k+1)$-st graded pieces
	are isomorphic to $\mathrm{P}^k$.
\end{lemma}

If $\C^\bullet$ is a complex of finitely generated graded $\H(\M)$-modules (or $\H_\o(\M)$-modules), we can split off as many two-term acyclic complexes as possible
until there do not exist $k$, $\mathrm{P}^k\neq 0$, $\mathrm{P}^{k+1}$, $\mathrm{Q}^{k}$, and $\mathrm{Q}^{k+1}$ such that the hypotheses of Lemma \ref{basic fact} hold.  We call the resulting subcomplex 
$\bar{\C}^\bullet\subseteq \C^\bullet$ a \textbf{minimal subcomplex} of $\C^\bullet$.   Since $\C^\bullet$ is the direct sum of $\bar{\C}^\bullet$ and an acyclic complex, $\C^\bullet$ and $\bar{\C}^\bullet$ have the same stalk and costalk cohomology.  
In particular, if $\C^\bullet$ is perverse or $\o$-perverse, so is $\bar{\C}^\bullet$.

\begin{remark}
	Even though the subcomplex $\bar{\C}^\bullet$ of $\C^\bullet$ depends on the choices of splitting, its isomorphism class as a complex of $\H(\M)$-modules (or $\H_\o(\M)$-modules) is uniquely determined. In fact, the category of bounded complexes of finitely generated $\H(\M)$-modules is an abelian category in which every element has finite length. By the Krull--Schmidt theorem, the complex $\C^\bullet$ admits a decomposition into a direct sum of indecomposable complexes of $\H(\M)$-modules, and the summands are uniquely determined up to isomorphisms. Removing all two-term acyclic summands, we obtain $\bar{\C}^\bullet$.  
\end{remark}



For the next result and several additional results in this section, we will use the following conditions as hypotheses.  The first condition implies that the conclusions of Proposition \ref{zero map}, Proposition \ref{prop:indecomposable}, and Corollary \ref{cor:characterizing purity} hold for any module $\IH(\M^G)$, and the second condition does the same for modules $\IH_\o(\M^G)$. 

\begin{conditionA}
	$\hyperlink{cd}{\CD(\M^G_F)}$, $\hyperlink{pd}{\PD(\M^G_F)}$, and $\hyperlink{ns}{\NS(\M^G_F)}$ hold for all flats $F < G$.
\end{conditionA}

\begin{conditionB}
	$E$ is nonempty, $\hyperlink{cdo}{\CD_\o(\M^G)}$ holds for all nonempty flats $G$, and $\hyperlink{cd}{\CD(\M^G_F)}$, $\hyperlink{pd}{\PD(\M^G_F)}$, and $\hyperlink{ns}{\NS(\M^G_F)}$ hold for all flats $\varnothing < F < G$.
\end{conditionB}

Note that condition A, and any results which rely on it, will only be known once the main induction loop is finished, while condition B holds at the beginning of our induction by Corollary \ref{CD for free}.
Under these hypotheses we have the following characterization of minimal perverse complexes of $\H(\M)$-modules and minimal $\o$-perverse complexes of $\H_\o(\M)$-modules.

\begin{theorem}\label{thm:perverse complexes}
	Let $\C^\bullet$ be a minimal complex of pure $\H(\M)$-modules (resp. a minimal complex of pure $\H_\o(\M)$-modules) and assume that condition A (resp. condition B) holds.  Then the following are equivalent:
	
	\begin{enumerate}[(a)]\itemsep 5pt
		\item Each $\C^i$ is isomorphic to a direct sum of modules of the form 
		\[\IH(\M^F)[k] \;\;\; \mbox{(resp.\  $\IH_\o(\M^F)[k]$)},\]
		where $F \in \cL(\M)$ (resp. $F \in \cL(\M) \setminus \varnothing$) and $k = (i-\crk F)/2$.
		\item Each module $\C^i$ is perverse (resp.\ $\o$-perverse) when considered as a complex placed in degree $i$.
		\item The complex $\C^\bullet$ is perverse (resp.\ $\o$-perverse).
	\end{enumerate}
%
%
%
\end{theorem}

\begin{proof}
	Suppose that $\C^\bullet$ is a complex of pure $\H(\M)$-modules and condition A holds; the other case is proved in the same way.
	
	Suppose that (a) holds; we will show that (b) holds as well.  Fix an integer $i$, and consider $\C^i$ as a complex placed in degree $i$.  Since conditions (a) and (b) are preserved under taking direct sums and direct summands, we can assume that $\C^i = \IH(\M^F)[k]$, where $k=(i-\crk F)/2$.  Take a flat $G\in \cL(\M)$.  If $G \not \le F$, the stalk and costalk of $\C^i$ at $G$ vanish, so the conditions of Definition \ref{def:perverse complex} hold for $G$.  If $G = F$, then  
	$\C^i_F = \C^i_{[F]} \cong \Q[k]$, which is only nonzero in degree $j = -k$.  Since $i+2j = \crk F$, the conditions of Definition \ref{def:perverse complex} again hold.   If $G < F$, then by Proposition \ref{zero map}, the costalk $\C^i_{[G]} = (\IH(\M^F)[k])_{[G]}$ vanishes in degrees less than or equal to 
	\[(\rk F - \rk G)/2 - k = (\rk F -\rk G + \crk F - i)/2 = (\crk G - i)/2,\]
	and the stalk at $G$ vanishes in degrees greater than or equal to $(\crk G - i)/2$.  Combining these three cases, we see that statement (b) holds.  Note that we have shown a stronger statement for flats $G \ne F$ when $\C^i=\IH(\M^F)[k]$: the strict inequalities in Definition \ref{def:perverse complex} can be replaced by non-strict inequalities.  
	
	If statement (b) holds, it means that the complex $(\C^\bullet, \partial=0)$ with zero differential is perverse.  Since setting the differentials to zero can only make the cohomology larger, this immediately implies statement (c).
	
	Finally, let us suppose (c) holds, so that $\C^\bullet$ is minimal and perverse. By Corollary \ref{cor:characterizing purity}, each $\C^i$ is isomorphic to a direct sum of modules of the form $\IH(\M^F)[k]$, so we need to show that this module can only appear in $\C^i$ with shift $k = (i-\crk F)/2$.  We prove this by induction on $\crk F$.  As the base case we take $\crk F = -1$; there are no such flats, so the statement is trivial.  Now suppose that $\crk F \ge 0$ and the statement holds for all flats of smaller corank.
	
	Let us suppose that $k > (i-\crk F)/2$, and so $k \ge (i +1 - \crk F)/2$.  Then the fact that $\IH(\M^F)_F \cong \Q$ implies that $\bar{\C}^i_F$ is nonzero in degree $j=-k$.  Since $i+2j > \crk F$, the assumption that $\C^\bullet$ is perverse implies that the cohomology $\HH^i(\C^\bullet_F)$ is zero in degree $j$, so either ${\C}^{i-1}_F$ or 
	${\C}^{i+1}_F$ must be nonzero in degree $j$.  Suppose $\IH(\M^G)[\ell]$ is a direct summand of ${\C}^{i+a}$, where $a = \pm 1$, whose stalk at $F$ is nonzero in degree $j$.  The nonvanishing of this stalk implies that $G \ge F$.  If $G > F$, then by our inductive hypothesis $\ell = (i+a - \crk G)/2$, and then Proposition \ref{zero map} implies that the stalk at $F$ of this summand vanishes in degrees greater than or equal to
	\[(\rk G - \rk F)/2 - \ell = (\crk F - (i+a))/2 \ge j.\] 
	In particular the stalk would vanish in degree $j$, contrary to assumption.  So we must have $G = F$.  But then in order for the map between the summands $\IH(\M^G)[\ell]_F$ and $\IH(\M^F)[k]_F$ to be nonzero we must have $\ell = k$, and by Proposition \ref{prop:indecomposable} the map must be an isomorphism, contradicting the minimality of our complex ${\C}^\bullet$.  So $k \le (i - \crk F)/2$.
	
	On the other hand, suppose that $k < (i-\crk F)/2$, so $k \le (i-1-\crk F)/2$.  Now the costalk 
	${\C}^i_{[F]}$ is nonzero in degree $j = -k$, and since $i + 2j < \crk F$, either
	${\C}^{i-1}_{[F]}$ or ${\C}^{i+1}_{[F]}$ must be nonzero in degree $j$.  Take a summand 
	$\IH(\M^G)[\ell]$ of ${\C}^{i+a}$, where $a = \pm 1$, and assume that $G > F$.  Then as before we have $\ell = (i+a - \crk G)/2$, so by Proposition \ref{zero map} the costalk of this summand vanishes in degrees less than or equal to 
	\[(\rk G - \rk F)/2 - \ell = (\crk F - (i+a))/2 \le j.\]
	This is impossible, so we must have $G = F$, which gives the same contradiction as before.  Thus we have $k = (i-\crk F)/2$, as desired.  
\end{proof}

\subsection{Perversity and chain homotopies}

Proposition \ref{prop:no chain homotopies} below will only be needed in Section \ref{sec:multiplicities and inverse KL}, in the proof of the nonnegativity of equivariant inverse Kazhdan--Lusztig polynomials.

\begin{lemma}
	If condition A holds, $\mathrm{P}^\bullet$ is a minimal perverse complex of pure $\H(\M)$-modules, and $F$ is any flat, then 
	$y_F\mathrm{P}^\bullet[\rk F]$ is a minimal perverse complex of pure $\H(\M_F)$-modules.
\end{lemma}
\begin{proof}
	Since $\mathrm{P}^\bullet$ satisfies the criterion (c) of Theorem \ref{thm:perverse complexes} we know that for every $i$, $\mathrm{P}^i$ is a direct sum of modules of the form $\IH(\M^G)[(i-\crk G)/2]$.  By condition A, we know that $\CD(\M)$ and $\CD(\M_F)$ hold, and so Lemma \ref{mult by y} part (1) implies that 
	$y_F\mathrm{P}^i[\rk F]$ is a direct sum of $\H(\M_F)$-modules of the form $\IH(\M^G_F)[(i-\crk G)/2]$ for flats $G \ge F$.  But $\M^G_F = (\M_F)^{G\setminus F}$, and the corank of $G \setminus F$ in $\M_F$ is equal to the corank of $G$ in $\M$, so another application of Theorem \ref{thm:perverse complexes} shows that $y_F\mathrm{P}^\bullet[\rk F]$  is a minimal perverse complex.
\end{proof}

\begin{proposition}\label{prop:no chain homotopies} Suppose that condition A holds. 
	Let $\mathrm{P}^\bullet$, $\mathrm{Q}^\bullet$ be minimal perverse complexes of pure $\H(\M)$-modules.  Then $\Hom(\mathrm{P}^i,\mathrm{Q}^{i-1}) = 0$ for all $i$, so all chain homotopies from $\mathrm{P}^\bullet$ to $\mathrm{Q}^\bullet$ vanish.
\end{proposition}
\begin{proof}
	The proof is similar to the proof of Proposition \ref{prop:indecomposable} (1).  We use induction on the rank of $\M$.  If $\rk \M = 0$, there is only one flat, of rank $0$, and $\H(\M) \cong \Q$.  So $\mathrm{P}^i$ and $\mathrm{Q}^{i-1}$ are just graded $\Q$-vector spaces, and by condition (b) of Theorem \ref{thm:perverse complexes}, $\mathrm{P}^i$ vanishes in all degrees except $-i/2$, and $\mathrm{Q}^{i+1}$ vanishes in all degrees except $-(i+1)/2$.  Thus $\Hom(\mathrm{P}^i,\mathrm{Q}^{i-1}) = 0$.
	
	Now suppose that $\rk \M>0$, and the statement holds for all matroids of smaller rank. Take a map 
	$f\colon \mathrm{P}^i \to \mathrm{Q}^{i-1}$ of graded $\H(\M)$-modules. 
	Take a nonempty flat $F$.  By the previous lemma, $y_F\mathrm{P}^\bullet[\rk F]$ and $y_F\mathrm{Q}^\bullet[\rk F]$ are minimal perverse complexes of pure $\H(\M_F)$-modules.  
	So by the inductive hypothesis $f$ induces the zero map
	$y_F\mathrm{P}^i \to y_F\mathrm{Q}^{i-1}$.  It follows that $f$ factors as
	\[\mathrm{P}^i \to \mathrm{P}^i_\varnothing \stackrel{\bar{f}}\longrightarrow \mathrm{Q}^{i-1}_{[\varnothing]} \to \mathrm{Q}^{i-1}.\]  But by Theorem~\ref{thm:perverse complexes} condition (b), $\mathrm{P}^i_\varnothing$ vanishes in degrees 
	$j$ for which $i + 2j > \rk \M$, and $\mathrm{Q}^{i-1}_{[\varnothing]}$ vanishes in degrees $j$ with 
	$i-1+2j < \rk \M$.  It follows that $\bar{f} = 0$, and so $f = 0$, as required.	
\end{proof}

\subsection{The big complexes}\label{sec:bigcomplexes}
Our Rouquier complexes will be defined as minimal subcomplexes of larger complexes $\C^\bullet(\M)$, $\C^\bullet_\o(\M)$ which we define in this section.
Consider the graded $\CH(\M)$-module 
\[
\C^i(\M) \coloneq \bigoplus_{\varnothing \leq F_1 < \cdots < F_i<E} x_{F_1}\cdots x_{F_i}\CH(\M)[i]
\]
for $i> 0$ and $\C^0(\M)\coloneq \CH(\M)$, along with the module homomorphism 
\[\partial^i \colon \C^i(\M)\to\C^{i+1}(\M)\]
defined component-wise by multiplication by a variable:
\[
x_{F_1} \cdots \widehat{x_{F_j}} \cdots x_{F_{i+1}}\CH(\M)[i] \xrightarrow{(-1)^j x_{F_j}} x_{F_1} \cdots x_{F_{i+1}}\CH(\M)[i+1].
\]
It is straightforward to check that $\partial^{i+1}\circ\partial^i = 0$, and hence $(\C^\bullet(\M),\partial)$
is a complex of graded $\CH(\M)$-modules.  

If $E$ is nonempty, we define 
$\C^\bullet_\o(\M)$ to be the quotient of $\C^\bullet(\M)$ by the subcomplex
consisting of terms with $F_1=\varnothing$.  In other words, it is defined by
\[
\C_{\o}^i(\M) \coloneq \bigoplus_{\varnothing < F_1 < \cdots < F_i<E} x_{F_1}\cdots x_{F_i}\CH(\M)[i],
\]
for $i>0$ and $\C_{\o}^0(\M)\coloneq \CH(\M)$. The differential of $\C^\bullet_\o(\M)$ is given by the same formula as in $\C^\bullet(\M)$.  

Both $\C^\bullet(\M)$ and $\C^\bullet_\o(\M)$ are complexes of $\CH(\M)$-modules, but we will consider $\C^\bullet(\M)$ as a complex of $\H(\M)$-modules and $\C^\bullet_\o(\M)$ as a complex of $\H_\o(\M)$-modules by restriction.

\begin{lemma}\label{terms of Rouquier}
	For all $i>0$ and proper flats $F_1<\cdots <F_i$, 
	$x_{F_1}\cdots x_{F_i}\CH(\M)[i]$ is isomorphic as an $\H(\M)$-module to a direct sum of shifted copies of $\CH(\M^{F_1})$, and if $F_1 \ne \varnothing$, this isomorphism can be taken to be an isomorphism of $\H_\o(\M)$-modules.
	In particular, for all $i$, $\C^i(\M)$ is a pure $\H(\M)$-module and $\C_\o^i(\M)$ is a pure $\H_\o(\M)$-module.
\end{lemma}

\begin{proof}
	By \cite[Proposition 2.23]{BHMPW}, for any proper flat $F$ the map $\psi^F$ gives an isomorphism
	\begin{equation}\label{eqn:psiiso}
	\uCH(\M_F) \otimes \CH(\M^F) \cong x_F \CH(\M)[1].
	\end{equation}
	This is an isomorphism of $\H(\M)$-modules, where the module structure on the left side is given by letting the generators $y_i$ act on $\uCH(\M_F)$ trivially and on $\CH(\M^F)$ by multiplication by $y_i$ if $i \in F$ and by zero if $i\notin F$.  In other words, the action on $\CH(\M^F)$ is via the homomorphism $\H(\M) \to \H(\M^F)$ obtained by restricting the pullback $\varphi^F$.  Furthermore, for any flat $G < F$, multiplication by $x_G$ on the right side of \eqref{eqn:psiiso} is given by multiplication by $1 \otimes x_G$ on the left side.
	
	Applying the isomorphism \eqref{eqn:psiiso} repeatedly, we have an isomorphism of $\H(\M)$-modules
	\[
	\underline{\CH}(\M_{F_i}) \otimes \underline{\CH}(\M^{F_i}_{F_{i-1}}) \otimes \cdots\otimes \underline{\CH}(\M^{F_2}_{F_1})  \otimes \CH(\M^{F_1}) \cong x_{F_1}\cdots x_{F_i}\CH(\M)[i],
	\]
where the action of $\H(\M)$ on the left-hand side is only on the last tensor factor.
If $F_1 \ne \varnothing$, it is even an isomorphism of $\H_\o(\M)$-modules.   
\end{proof}

Let $\bar{\C}^\bullet(\M)$ and $\bar{\C}^\bullet_\o(\M)$ be minimal subcomplexes of $\C^\bullet(\M)$ and $\C^\bullet_\o(\M)$, respectively.  These complexes are well-defined up to isomorphism; we call them the \textbf{Rouquier complex} and \textbf{reduced Rouquier complex} of $\M$, respectively. 

\subsection{Proving ${\uNS^{<\frac{d-2}{2}}(\M)}$}

The main result of this section is Proposition \ref{uNS1} below, which provides one of the first key steps of our main induction loop.  We deduce it from the following three important properties of the reduced Rouquier complex $\bar{\C}^\bullet_\o(\M)$; each of Sections~\ref{sec:Rouquier perversity},~\ref{sec:empty stalk}, and~\ref{sec:Leading term of Rouquier} is dedicated to proving one of these properties.  Two of these propositions also include corresponding statements about $\C^\bullet(\M)$ and  
$\bar{\C}^\bullet(\M)$; we do not need them for our main induction, but we use them later in Section \ref{sec:multiplicities and inverse KL} to prove the nonnegativity of inverse Kazhdan--Lusztig polynomials of matroids.

\begin{proposition}\label{prop:Rouquier perversity}
	The complexes $\C^\bullet(\M)$ and $\C^\bullet_\o(\M)$ are perverse, hence so are $\bar{\C}^\bullet(\M)$ and $\bar{\C}^\bullet_\o(\M)$. 
\end{proposition}

\begin{proposition}\label{empty stalk}
	Suppose that $E \ne \varnothing$, so $\rk \M > 0$. Then for every $i$, the graded $\H_\o(\M)$-module $\HH^i(\bar{\C}_\o^\bullet(\M)_\varnothing) \cong \HH^i({\C}_\o^\bullet(\M)_\varnothing)$ is concentrated in degree $d - 1 - i$.
\end{proposition}

\begin{proposition}\label{prop:Leading term of Rouquier}
	\
	\begin{enumerate}[(1)] \itemsep 5pt
		\item Suppose that condition A holds for $\M$.
		Then $\bar{\C}^0(\M) \cong \IH(\M)$.
		\item Suppose that condition B holds for $\M$.
		Then $\bar{\C}_\o^0(\M) \cong \IH_\o(\M)$.
	\end{enumerate}
\end{proposition}

Assuming these for the moment, we can now prove $\hyperlink{uns}{\uNS^{<\frac{d-2}{2}}(\M)}$.

\begin{proposition}\label{uNS1}
	Suppose that condition B holds for $\M$, and $\hyperlink{uns}{\uNS(\M^F)}$ holds for all proper nonempty flats $F$.
	Then $\hyperlink{uns}{\underline{\NS}^{<\frac{d-2}{2}}(\M)}$ holds.
\end{proposition}

\begin{proof}
Consider the differential   
	\[
\partial^0_\varnothing \colon \bar\C_\o^0(\M)_\varnothing\to \bar\C_\o^1(\M)_\varnothing\]
of the complex $\bar\C_\o^\bullet(\M)_\varnothing$.  
By Proposition \ref{prop:Leading term of Rouquier} (2), the stalk $\bar\C_\o^0(\M)_\varnothing$ is isomorphic to $\IH_\circ(\M)_{\varnothing}$,
which in turn is isomorphic to $\underline{\IH}(\M)$ by Lemma \ref{empty (co)stalk} and Corollary \ref{underline it}.
Note that the differential is an $\uH(\M)$-module map, so it sends the socle of $\bar\C_\o^0(\M)_\varnothing$ to the socle of $\bar\C_\o^1(\M)_\varnothing$.
Applying Proposition \ref{empty stalk} with $i=0$, we see that the kernel of
$\partial^0_\varnothing$ is concentrated in degree $d - 1 > (d-2)/2$.
So to show that $\uNS^{<\frac{d-2}{2}}$ holds it suffices to show that each summand $\underline{\IH}(\M^F)[k]$ of $\bar{\C}^1_\o(\M)$ has no socle in degrees less than $(d-2)/2$.

Theorem \ref{thm:perverse complexes} and Proposition \ref{prop:Rouquier perversity} together imply that 
$\bar\C_\o^1(\M)_\varnothing$ is a direct sum of modules of the form 
	\[
	\IH_\o(\M^F)[k]_\varnothing \cong \underline{\IH}(\M^F)[k],
	\]
	where $F$ is nonempty
	and $k=(1-\crk F)/2 \leq 0$, which implies in particular that $F\ne E$.	
	The hypothesis $\hyperlink{uns}{\underline{\NS}(\M^F)}$ implies that the socle of $\underline{\IH}(\M^F)$ vanishes in degrees
	less than or equal to $(\rk F - 2)/2$, and therefore the socle of $\underline{\IH}(\M^F)[k]$ vanishes in degrees
	less than or equal to 
	\[
	\frac{\rk F - 2}{2} - \frac{1 - \crk F}{2} = \frac{d-3}{2} = \frac{d-2}{2} - \frac{1}{2}.
	\]
	We can therefore conclude 
	$\hyperlink{uns}{\underline{\NS}^{<\frac{d-2}{2}}(\M)}$.
\end{proof}

\subsection{Perversity of $\C^\bullet(\M)$ and $\C^\bullet_\o(\M)$}\label{sec:Rouquier perversity}

Next we turn to proving Proposition \ref{prop:Rouquier perversity}.  By Lemma \ref{terms of Rouquier}, $\C^\bullet(\M)$ is a complex of pure $\H(\M)$-modules and $\C^\bullet_\o(\M)$ is a complex of pure $\H_\o(\M)$-modules, so what remains is to prove the vanishing of the cohomology of the stalk and costalk complexes in the appropriate degrees.  Our first lemma uses the map $\psi_F$, which factors the map of multiplication by $y_F$ (Proposition \ref{Prop_ymult}), to reduce these questions to studying stalks and costalks at the empty flat.

\begin{lemma}\label{complex}
Let $F$ be a flat of a matroid $\M$.
\begin{enumerate}[(1)]\itemsep 5pt
\item
The map $\psi_F$ induces an isomorphism
$$\C^\bullet(\M_F)[-\rk F] \cong y_F\C^\bullet(\M)$$
of complexes of graded $\CH(\M)$-modules, where $\CH(\M)$ acts on the left-hand side
via the graded algebra homomorphism $\varphi^\Mo_F \colon \CH(\M)\to\CH(\M_F)$.
\item If $F$ is nonempty, $\psi_F$ also induces an isomorphism
$$\C^\bullet(\M_F)[-\rk F] \cong y_F\C_\o^\bullet(\M)$$
of complexes of graded $\CH(\M)$-modules.
\end{enumerate}
\end{lemma}

\begin{proof}
The pushforward map $\psi_F^\Mo \colon \CH(\M_F)[-\rk F] \to y_F\CH(\M)$ is an isomorphism
of graded $\CH(\M)$-modules by \cite[Proposition 2.31]{BHMPW}, where the action of $x_G \in \CH(\M)$ on the left side is multiplication by 
$\varphi_F(x_G) = x_{G\setminus F}$ if $G \ge F$ and is zero otherwise. 
It follows that we have an isomorphism
\[x_{F_1 \setminus F}\dots x_{F_i \setminus F}\CH(\M_F)[-\rk F] \cong y_Fx_{F_1}\dots x_{F_i}\CH(\M)\] for any chain of flats $\varnothing \le F_1 < \dots < F_i < E$, where both sides are zero unless $F \le F_1$.  Adding up over all such chains proves the statement (1).  
 
Since $x_\varnothing y_F = 0$ for any nonempty flat $F$, the projection from $\C^\bullet(\M)$ to $\C_\o^\bullet(\M)$ becomes an isomorphism after multiplying by $y_F$, and hence the second statement follows from the first one.
\end{proof}


Next we show that the stalk cohomology of $\C^\bullet(\M)$ and $\C^\bullet_\o(\M)$ at a proper flat actually vanishes in all degrees.  This is stronger than what we need for perversity, but we will need the full strength later when we prove Proposition \ref{prop:Leading term of Rouquier}.

\begin{lemma}\label{stalk acyclic}
\
\begin{enumerate}[(1)]\itemsep 5pt
\item If $F$ is a proper flat, then the stalk complex $\C^\bullet(\M)_F$ is acyclic.  The stalk complex $\C^\bullet(\M)_E$
is isomorphic to $\Q$ concentrated in degree zero.
\item If $F$ is a nonempty proper flat, the stalk complex $\C_\o^\bullet(\M)_F$ is acyclic.  If $E$ is nonempty, the stalk complex $\C_\o^\bullet(\M)_E$
is isomorphic to $\Q$ concentrated in degree zero.
\end{enumerate}
\end{lemma}


\begin{proof}
We begin by proving statement (1) when $F$ is the empty flat.
We observe that multiplication by $x_\varnothing$ defines a map of complexes
$$\C_\o^\bullet(\M)\to  x_\varnothing\C_\o^\bullet(\M)[1],$$
and (after shifting by $1$ in cohomological degree) the cone of this map is isomorphic to $\C^\bullet(\M)$.
To prove that $\C^\bullet(\M)_{\varnothing}$ is acyclic, it is therefore sufficient to prove that for all $i$, the map
from $\C_\o^i(\M)$ to $x_\varnothing\C_\o^i(\M)[1]$ induces an isomorphism on stalks at the empty flat.  By Lemma \ref{terms of Rouquier}, it is enough to show that multiplication by $x_\varnothing$ induces an isomorphism between the stalks at $\varnothing$ of $\CH(\M^F)$ and $x_\varnothing\CH(\M^F)$.
Because $x_\varnothing$ is annihilated by all $y_F$, $F\ne \varnothing$, we have 
$\left(x_\varnothing\CH(\M^F)\right)_\varnothing = x_\varnothing\CH(\M^F)$.  Then Lemma \ref{mutual annihilators} gives an exact sequence
\[0\to \sum_{i\in E} y_i \CH(\M^F) \to \CH(\M^F) \stackrel{x_\varnothing \cdot}{\longrightarrow} x_\varnothing\CH(\M^F)[1]\to 0.\]
Since $\CH(\M^F)_\varnothing$ is the cokernel of the first map, the result follows.

Next we prove statement (1) for arbitrary proper flats.  By Lemmas \ref{everything is empty} and \ref{complex} (1), 
$$\C^\bullet(\M)_F \cong (y_F\C^\bullet(\M)[\rk F])_\varnothing \cong \C^\bullet(\M_F)_\varnothing.$$
Since $F$ is proper, $\M_F$ has positive rank, and the statement follows from the previous paragraph.

It follows from the definition of $\C^\bullet(\M)$ that $\C^k(\M)_E = y_E\C^k(\M)[d]$ vanishes for $k\ne 0$, and for $k=0$, it is isomorphic to $\Q$ placed in grading degree zero.  This implies the second sentence of (1).

For any nonempty flat $F$, we have $y_F x_\varnothing = 0$. Therefore, the natural quotient $\C^\bullet(\M)\to \C_\o^\bullet(\M)$
induces an isomorphism on the stalk at $F$.  Thus, statement (2) follows from statement (1).
\end{proof}

\begin{proposition}\label{acyclic}
The complex $\C^\bullet(\M)$ has no cohomology except in degree zero, and $\HH^0(\C^\bullet(\M)) \cong \Q[-d]$.
\end{proposition}

\begin{proof}
Fix a total order $F_1,\ldots,F_r$ of $\cL(\M)$ refining the natural partial order, so that, for any $k$, the set 
$$\Sigma_k \coloneq \{F_k,\ldots,F_r\}$$
is an order ideal.  
By Proposition \ref{pure}, we have
\[
\C^\bullet(\M)_{\Sigma_k}/\C^\bullet(\M)_{\Sigma_{k+1}} \cong \C^\bullet(\M)_{F_k}[-\rk F_k],
\]
which is acyclic for all $1\le k<r$ and quasi-isomorphic to $\Q[-d]$ in degree zero when $k=r$ by Lemma~\ref{stalk acyclic} (1).
The result then follows from the spectral sequence relating the cohomology of a filtered complex to the cohomology
of its associated graded complex.
\end{proof}

Next we turn to the cohomology of the costalk complex of $\C^\bullet(\M)$, starting with the empty flat.

\begin{proposition}\label{rouquier empty costalk}
For any $i$, we have $\HH^i\bigl(\C^\bullet(\M)_{[\varnothing]}\bigr) \cong \OS^i(\M)^*[-d].$  In particular, it is nonzero only in degree $d-i$.
\end{proposition}

\begin{proof}
Let $\C^\bullet(\M)_!^\bullet$ be the double complex obtained by applying the functor $(\cdot)^\bullet_!$ of Section \ref{sec:costalk}
to each term in the complex $\C^\bullet(\M)$, so the $(i,j)$ term is 
\[\C^i(\M)^j_! = \bigoplus_{F \in \cL^j(\M)} \OS_F(\M)^* \otimes y_F\C^i(\M).\] 
Since each $\C^i(\M)$ is pure, Proposition \ref{costalk complex} 
implies that the map $\C^\bullet(\M)_{[\varnothing]} \to \C^\bullet(\M) = \C^\bullet(\M)^0_!$ induces a quasi-isomorphism between $\C^\bullet(\M)_{[\varnothing]}$ and the total complex of $\C^\bullet(\M)_!^\bullet$.

On the other hand, by Lemma \ref{complex} the $j$-th row $\C^\bullet(\M)^j_!$ of the double complex is equal to the direct sum over all rank $j$ flats $F$ of the complex
$$\OS_F(\M)^*\otimes y_F\C^\bullet(\M) \cong \OS_F(\M)^* \otimes \C^\bullet(\M_F)[-\rk F].$$
Proposition \ref{acyclic} and equation \eqref{eqn:Brieskorn} then imply that 
\[\HH^i(\C^\bullet(\M)^j_!) \cong \begin{cases}
\bigoplus_{F\in \cL^j(\M)} \OS_F(\M)^*[-d] = \OS^j(\M)^*[-d]  & \mbox{if} \; i = 0\\
0 & \mbox{if}\; i \ne 0. 
\end{cases}\]
Note that for $i=0$ this graded vector space is concentrated in (grading) degree $d - j$, which means that the differential $\HH^0(\C^\bullet(\M)^j_!) \to \HH^0(\C^\bullet(\M)^{j+1}_!)$ vanishes for degree reasons.  In particular, we see that 
the complex $\OS^\bullet(\M)^*[-d]$ with zero differential is quasi-isomorphic to the total complex of $\C^\bullet(\M)_!^\bullet$.

Putting together the two paragraphs above, we can conclude the proof.
\end{proof}

\begin{corollary}\label{rouquier costalk}
Let $F$ be a flat, and let $i$ be a nonnegative integer.
\begin{enumerate}[(1)] \itemsep 5pt
\item  We have 
$
\HH^i\!\left(\C^\bullet(\M)_{[F]}\right) \cong \OS^i(\M_F)^*[-\crk F].
$
\item If $F$ is nonempty, then 
$
\HH^i\!\left(\C_\o^\bullet(\M)_{[F]}\right) \cong \OS^i(\M_F)^*[-\crk F].
$
\end{enumerate}
\end{corollary}

\begin{proof}
By Lemma \ref{everything is empty} and Lemma \ref{complex} (1),
$$\C^\bullet(\M)_{[F]} \cong (y_F\C^\bullet(\M)[\rk F])_{[\varnothing]} \cong \C^\bullet(\M_F)_{[\varnothing]}.$$
Statement (1) then follows from Proposition \ref{rouquier empty costalk}. Similarly, we can deduce statement (2) using Lemma~\ref{complex},
which says that $y_F\C^\bullet(\M) \cong y_F\C_\o^\bullet(\M)$ when $F$ is nonempty.
\end{proof}

Corollary \ref{rouquier costalk} combines with Lemma \ref{stalk acyclic} to complete the proof of Proposition \ref{prop:Rouquier perversity}, because the only possible degree in which $\HH^i(\C^\bullet(\M)_{[F]})$ can be nonzero is $j = \crk F - i$.  That means that 
$i + 2j = \crk F + j \ge \crk F$, because our complex vanishes in negative grading degrees.

\subsection{Proof of Proposition \ref{empty stalk}}\label{sec:empty stalk}
Throughout this section, we assume that $E$ is nonempty. Our goal is to give a degree bound on the cohomology of the complex $\C_\o^\bullet(\M)_\varnothing$.  

Given any graded $\H(\M)$-module $\N$, the stalk $\N_\varnothing$ is a quotient of $\N$ and the costalk $\N_{[\varnothing]}$ is a subspace of $\N$.  As a result we have a natural transformation $\N_{[\varnothing]} \to \N_\varnothing$ from the costalk to the stalk.  Given a complex $\mathrm{Q}^\bullet$ of graded $\H(\M)$-modules, we denote by $\Delta(\mathrm{Q}^\bullet)$ the cone of the natural map $\mathrm{Q}^{\bullet-1}_{[\varnothing]}\to\mathrm{Q}^{\bullet-1}_{\varnothing}$. In particular, $\Delta(\mathrm{Q}^\bullet)^k=\mathrm{Q}^k_{[\varnothing]}\oplus \mathrm{Q}^{k-1}_{\varnothing}$,
and we have a distinguished triangle
\[
\mathrm{Q}^{\bullet-1}_{[\varnothing]}\to\mathrm{Q}^{\bullet-1}_{\varnothing} \to \Delta(\mathrm{Q}^\bullet)\to \mathrm{Q}^\bullet_{[\varnothing]}.
\]

\begin{lemma}\label{first}
The natural map $\Delta(\C^\bullet(\M))\to \C^\bullet(\M)_{[\varnothing]}$ is a quasi-isomorphism.
\end{lemma}

\begin{proof}
The first part of Lemma \ref{stalk acyclic} says that $\C^\bullet(\M)_{\varnothing}$ is acyclic, so the desired statement follows 
from the long exact sequence in cohomology associated with the distinguished triangle above.
\end{proof}

\begin{lemma}\label{second}
The map $\C^\bullet(\M)\to\C_\o^\bullet(\M)$ induces a quasi-isomorphism $\Delta(\C^\bullet(\M))\to\Delta(\C_\o^\bullet(\M))$.
\end{lemma}

\begin{proof}
Let $\C^\bullet_{-}(\M)$ be the kernel of $\C^\bullet(\M)\to\C_\o^\bullet(\M)$. In other words, the complex $\C^\bullet_{-}(\M)$ is defined by
\[
\C_{-}^i(\M) \coloneq \bigoplus_{\varnothing = F_1 < \cdots < F_i<E} x_{F_1}\cdots x_{F_i}\CH(\M)[i],
\]
and has differential defined by the same component-wise formula as in the definition of $\C^\bullet(\M)$. 
Since the exact sequences
\[0\to \C^i_-(\M) \to \C^i(\M)\to \C^i_\o(\M)\to 0\]
are split for every $i$, applying the stalk and costalk functors gives exact sequences, so we have an exact sequence
\[0 \to \Delta(\C^\bullet_-(\M)) \to \Delta(\C^\bullet(\M)) \to \Delta(\C^\bullet_\o(\M)) \to 0\]
of complexes.


Since $\C_{-}^{\bullet}(\M)$ is annihilated by $\Upsilon_{>\varnothing}$, we have 
$\C_{-}^{\bullet}(\M)_{[\varnothing]} = \C_{-}^{\bullet}(\M) = \C_{-}^{\bullet}(\M)_{\varnothing}$ and therefore $\Delta(\C_{-}^{\bullet}(\M))$ is acyclic. Now the long exact sequence in cohomology associated to the displayed exact sequence of complexes implies that the map $\Delta(\C^\bullet(\M))\to\Delta(\C_\o^\bullet(\M))$ is a quasi-isomorphism. 
\end{proof}

\begin{lemma}\label{third}
The complex $\Delta(\C_\o^\bullet(\M))$ is quasi-isomorphic to the cone of the map of complexes
$\C_\o^{\bullet-1}(\M)_{\varnothing}[-1]\to \C_\o^{\bullet-1}(\M)_{\varnothing}$
given by multiplication by $x_\varnothing$.
\end{lemma}

\begin{proof}
By Lemma \ref{mutual annihilators}, the annihilator of $\Upsilon_{>\varnothing}$ in $\CH(\M^F)$ is equal to $x_\varnothing\CH(\M^F)$ for all nonempty flats $F$. It follows that the natural maps $\CH(\M^F)_{[\varnothing]} \to \CH(\M^F) \to \CH(\M^F)_\varnothing$ are identified with
\[
x_\varnothing\CH(\M^F) \to \CH(\M^F) \to x_\varnothing\CH(\M^F)[1],
\]
where the first map is the inclusion and the second is multiplication by $x_\varnothing$.  This identification is natural with respect to maps between shifts of modules of the form $\CH(\M^F)$.
By Lemma \ref{terms of Rouquier}, each $\C_\o^i(\M)$ is isomorphic to a direct sum of shifts of such modules,
therefore 
\[
\C_\o^{\bullet-1}(\M)_{[\varnothing]} \cong \C_\o^{\bullet-1}(\M)_{\varnothing}[-1],
\]
and under this isomorphism the natural map $\C_\o^{\bullet-1}(\M)_{[\varnothing]} \to \C_\o^{\bullet-1}(\M)_{\varnothing}$ is multiplication by $x_\varnothing$.  The lemma follows.
\end{proof}

\begin{proof}[Proof of Proposition \ref{empty stalk}]
Combining Lemmas \ref{first}, \ref{second}, and \ref{third}, we find that $\C^\bullet(\M)_{[\varnothing]}$ is quasi-isomorphic
to the cone of the map $\C_\o^{\bullet-1}(\M)_{\varnothing}[-1]\to \C_\o^{\bullet-1}(\M)_{\varnothing}$
given by multiplication by $x_\varnothing$.  This induces a long exact sequence
\[
\cdots \to \HH^{i}\bigl(\C^\bullet(\M)_{[\varnothing]}\bigr) \to \HH^{i}\bigl(\C_\o^\bullet(\M)_\varnothing\bigr)[-1] \xrightarrow{\cdot x_\varnothing}\HH^{i}\bigl(\C_\o^\bullet(\M)_\varnothing\bigr) \to \HH^{i+1}\bigl(\C^\bullet(\M)_{[\varnothing]}\bigr) \to\cdots. \]
If $\HH^i(\C_\o^\bullet(\M)_\varnothing) \ne 0$, 
let $k$ be the smallest degree in which it does not vanish.  A nonzero element in degree $k$ is not in the image of multiplication by $x_\varnothing$, so the long exact sequence implies that $\HH^{i+1}(\C^\bullet(\M)_{[\varnothing]})$ is nonzero in degree $k$.  But that implies that $k = d - (i+1)$ by Proposition~\ref{rouquier empty costalk}.  Dually, if $k$ is the largest nonvanishing degree, then an element in degree $k$ is killed by multiplication by $x_\varnothing$, and our exact sequence implies that 
$\HH^{i}(\C^\bullet(\M)_{[\varnothing]})$ is nonzero in 
degree $k+1$, so we get $k+1 =d - i$ again by Proposition~\ref{rouquier empty costalk}.  Thus, the proposition follows.
\end{proof}

\subsection{Proof of Proposition \ref{prop:Leading term of Rouquier}}\label{sec:Leading term of Rouquier}

We will prove the first part of the proposition; the proof of the second part is identical.  By our construction in Section~\ref{sec:bigcomplexes}, $\bar{\C}^\bullet(\M)$ is a minimal subcomplex of $\C^\bullet(\M)$.  By Corollary \ref{cor:characterizing purity}, each module $\C^i(\M)$ is isomorphic to a direct sum of $\H(\M)$-modules of the form $\IH(\M^G)[k]$ for $G\in \cL(\M)$ and $k\in \Z$.  Furthermore, since condition A includes $\hyperlink{cd}{\CD(\M^G)}$ for all flats $G$, the $\H(\M)$-module $\C^0(\M) = \CH(\M)$ contains $\IH(\M)$ as a direct summand with multiplicity one.  By Lemma \ref{terms of Rouquier}, $\C^i(\M)$ does not contain $\IH(\M)$ as a direct summand if $i > 0$.  
So the first term $\bar{\C}^0(\M)$ of the minimal subcomplex must contain exactly one summand isomorphic to $\IH(\M)$.

Now take any proper flat $G$, and suppose that $\bar{\C}^0(\M)$ contains a direct summand isomorphic to $\IH(\M^G)[k]$.  By Proposition \ref{prop:Rouquier perversity} and Theorem \ref{thm:perverse complexes}, we must have $k = (-\crk G)/2$.  So the stalk $\bar{\C}^0(\M)_G$ is nonzero in degree $-k = (\crk G)/2$.  But any summand of $\bar{\C}^1(\M)$ is isomorphic to a module $\IH(\M^F)[\ell]$ with $\ell = (1-\crk F)/2$.  The stalk at $G$ of this module is zero unless $F > G$ (the case $F = G$ is impossible since $\crk F$ and $\crk G$ must have opposite parity), 
in which case Proposition \ref{zero map} says that this stalk vanishes in degrees greater than or equal to 
\[(\rk F - \rk G)/2 - \ell = (\crk G - 1)/2 < -k,\]
and so $\bar{\C}^1(\M)_G$ vanishes in degree $-k$.  This is a contradiction, since  Lemma \ref{stalk acyclic} says that the stalk complex $\bar{\C}^\bullet(\M)_G$ is acyclic.  So $\IH(\M)$ is the only direct summand of $\bar{\C}^0(\M)$.

\subsection{Multiplicities and inverse Kazhdan--Lusztig polynomials}\label{sec:multiplicities and inverse KL}
In this section we show that, under the assumption that Theorem~\ref{theorem_all} holds, the multiplicities of the modules $\IH(\M^F)$ in the complex $\bar{\C}^\bullet(\M)$ are given by coefficients of inverse Kazhdan--Lusztig polynomials of matroids.  This implies that these coefficients are nonnegative, providing a proof of Theorem \ref{thm:inverse}.

For a matroid $\M$, define a polynomial $\tilde{Q}_\M(t) = \sum_k \tilde{q}_kt^k \in \mathbb{N}[t]$, where $\tilde{q}_k$ is the multiplicity of the module 
$\IH(\M^\varnothing)[-k]$ in $\bar{\C}^{d - 2k}(\M)$.  Since the modules $\IH(\M^F)[-k]$ are indecomposable by Proposition \ref{prop:indecomposable}, and they are pairwise non-isomorphic, this number is well-defined.  We note that $\bar{\C}^i(\M) = 0$ for all $i>d$, so $\tilde{q}_k$ can only be nonzero when $k\ge 0$.

\begin{proposition}\label{prop:Rouquier multiplicities}
	The inverse Kazhdan--Lusztig polynomial $Q_\M(t)$ is equal to
 $\tilde Q_\M(t)$, so in particular it has nonnegative coefficients.
\end{proposition}

\begin{lemma} \label{lem:inverse mult at general flat}
Suppose that Theorem \ref{theorem_all} holds.
For any flat $F$ and any integer $k$, the multiplicity of $\IH(\M^F)[-k]$ in $\bar{\C}^{\crk F - 2k}(\M)$ is
equal to the coefficient of $t^{k}$ in $\tilde{Q}_{\M_F}(t)$.
\end{lemma}
\begin{proof}
Lemma \ref{mult by y} (1) gives an isomorphism
\[y_F\IH(\M^G)[\ell] \cong 
\begin{cases}
\IH(\M^G_F)[\ell-\rk F] & \text{if $F \le G$,} \\
0 & \text{otherwise.}
\end{cases}
\]
So, to find the multiplicity of $\IH(\M^F)$ with any shift in $\C^\bullet(\M)$, it is sufficient to find the multiplicity of $\IH(\M^F_F)$ in $y_F\C^\bullet(\M)$.  But $\M^F_F = (\M_F)^\varnothing$ has rank zero, so our result will follow if we can show that there is an isomorphism
\[y_F\bar{\C}^\bullet(\M) \cong \bar{\C}^\bullet(\M_F)[-\rk F].\]	

By Lemma \ref{complex} (1), we have an isomorphism
\[y_F{\C}^\bullet(\M) \cong {\C}^\bullet(\M_F)[-\rk F].\]  
We have a direct sum decomposition $\C^\bullet(\M) = \bar{\C}^\bullet(\M) \oplus \mathrm{P}^\bullet$ of complexes, where $\mathrm{P}^\bullet$ is chain-homotopy equivalent to zero.  It follows that the inclusion of $y_F\bar{\C}^\bullet(\M)$ into
$y_F\C^\bullet(\M) \cong {\C}^\bullet(\M_F)$ is a chain-homotopy equivalence.  On the other hand, by Proposition \ref{prop:Rouquier perversity} and Theorem \ref{thm:perverse complexes}, $y_F\bar{\C}^i(\M)$ can have $\IH(\M^G_F)[k]$ as a direct summand only if $k = (i-\crk G)/2-\rk F$.   Thus $y_F\bar{\C}^\bullet(\M)$ does not have any more two-step summands to which Lemma \ref{basic fact} would apply, so it is a minimal subcomplex of $y_F\C^\bullet(\M)$.  The result follows.
\end{proof}

\begin{proof}[Proof of Proposition \ref{prop:Rouquier multiplicities}, assuming Theorem \ref{theorem_all}]
If the rank $d$ of $\M$ is equal to zero, then $\C^\bullet(\M) = \bar{\C}^\bullet(\M)$ is a complex with only one nonzero term, which is $\IH(\M) = \IH(\M^\varnothing)$ placed in degree zero.  So
$Q_\M(t) = 1 = \tilde{Q}_\M(t)$ in this case.  

When $\rk \M > 0$, the inverse Kazhdan--Lusztig polynomial of $\M$ satisfies the following recursion \cite[Theorem 1.3]{GX}:
\begin{equation}\label{eqn:inverse KL}
\sum_{F \in \cL(\M)} (-1)^{\rk F}P_{\M^F}(t)Q_{\M_F}(t) = 0,
\end{equation}
or equivalently
$$Q_{\M}(t) = - \sum_{\varnothing\neq F\in\cL(\M)}(-1)^{\rk F}P_{\M^F}(t)Q_{\M_F}(t).$$
If we can show that $\tilde{Q}_\M(t)$ satisfies the same recursion, then the result will follow by induction on $\rk \M$.

Assume that $\rk \M > 0$.
By Lemma \ref{stalk acyclic} (1), the complex $\C^\bullet(\M)_\varnothing$ is acyclic, and since $\bar\C^\bullet(\M)_\varnothing$ is a direct summand of this complex, it is also acyclic.
By Proposition \ref{prop:Rouquier perversity} and Theorem \ref{thm:perverse complexes}, we have an isomorphism
\[\bar\C^i(\M)_\varnothing \cong \bigoplus_{k\ge 0}\,\bigoplus_{\crk F = i + 2k} \left(\IH(\M^F)_\varnothing[-k]\right)^{\oplus{q_k(\M_F)}},\]
where $q_k(\M_F)$ is the coefficient of $t^k$ in $\tilde{Q}_{\M_F}(t)$.
 Notice that for all terms of this sum, $i$ and $\crk F$ have the same parity. Since the Poincar\'e polynomial of $\IH(\M^F)_\varnothing$ is equal to $P_{\M^F}(t)$, the alternating sum of the Poincar\'e polynomials of $\bar\C^i(\M)_\varnothing$ for all $i$ is equal to
\[\sum_{F \in \cL(\M)} (-1)^{\crk F}P_{\M^F}(t)\tilde{Q}_{\M_F}(t)=(-1)^{\rk \M}\sum_{F \in \cL(\M)} (-1)^{\rk F}P_{\M^F}(t)\tilde{Q}_{\M_F}(t).\]
Since $\bar\C^\bullet(\M)_\varnothing$ is acyclic, the above sum is equal to zero. 
\end{proof}

Finally, we note the following explicit description of the second term $\bar{\C}^1(\M)$ of the Rouquier complex.  When the matroid $\M$ has odd rank $2\ell + 1$, the coefficient of $t^\ell$ in  \eqref{eqn:inverse KL} is nonzero only for $F = \varnothing$ and $F = E$, which implies that  
$P_\M(t)$ and $Q_\M(t)$ have the same coefficient of $t^\ell$.
Let $\tau(\M)$ denote this coefficient. (Note that it is only defined when the rank of $\M$ is odd.)

\begin{corollary}
	For any $\M$, there is an isomorphism 
	\[\bar{\C}^1(\M) \cong \bigoplus_{j\ge 0} \bigoplus_{\crk F = 2j+1} \IH(\M^F)[-j]^{\oplus \tau(\M_F)} .\]
\end{corollary}

\begin{proof}
	By Lemma \ref{lem:inverse mult at general flat}, $\bar{\C}^1(\M)$ is a direct sum of modules of the form $\IH(\M^F)[-j]$, where $2j+1$ is equal to the corank of $F$, and the multiplicity of this module
	is the coefficient of $t^j$ in $\tilde Q_{\M_F}(t)$.    By Proposition \ref{prop:Rouquier multiplicities}, this is the coefficient of $t^j$ in $Q_{\M_F}(t)$, which is by definition equal to $\tau(\M_F)$.
\end{proof}

\subsection{Equivariant inverse Kazhdan--Lusztig polynomials} As was mentioned in the introduction following Theorem \ref{thm:inverse}, the nonnegativity of coefficients of $Q_\M(t)$ extends to coefficients of equivariant inverse Kazhdan--Lusztig polynomials.  Let us explain the main ways the proof of Proposition \ref{prop:Rouquier multiplicities} must be modified to upgrade it from multiplicities to representations.

Suppose that a finite group $\Gamma$ acts on the matroid $\M$.  Then the big complex $\C^\bullet(\M)$ is $\Gamma$-equivariant, meaning that $\Gamma$ acts on it by chain maps, compatibly with the action on $\H(\M)$.  We wish to know that $\bar{\C}^\bullet(\M)$ is also $\Gamma$-equivariant.  Fixing a direct sum decomposition $\C^\bullet(\M) = \bar{\C}^\bullet(\M) \oplus \mathrm{P}^\bullet$, we can define the action of $g\in \Gamma$ on $\bar{\C}^\bullet(\M)$ by including into $\C^\bullet(\M)$, acting by $g$, and then projecting back.  Since the inclusion and projection are chain-homotopy equivalences, this defines an action up to chain homotopy.  But by Proposition \ref{prop:no chain homotopies}, all chain homotopies on $\bar{\C}^\bullet(\M)$ vanish, and so this defines an action by chain maps.  

\begin{remark}
	Although we have defined actions of $\Gamma$ on both $\bar{\C}^\bullet(\M)$ and $\C^\bullet(\M)$, the inclusion map is not necessarily $\Gamma$-equivariant, and the subcomplex $\bar{\C}^\bullet(\M)$ may not be invariant under the action of $\Gamma$ on $\C^\bullet(\M)$.  We do not need it for what follows, but we can require $\bar{\C}^\bullet(\M)$ to be $\Gamma$-invariant, by the following argument.  Let $\iota, \pi$ denote the inclusion and projection of $\bar{\C}^\bullet(\M)$ in $\C^\bullet(\M)$, and denote the action of $\Gamma$ on $\bar{\C}^\bullet(\M)$ defined above by $(g, x) \mapsto g\ast x = \pi(g \iota(x))$.  Then the map
	\[\tilde\iota\colon \bar{\C}^\bullet(\M) \to \C^\bullet(\M),\;\; \tilde\iota(x) = \frac{1}{|\Gamma|}\sum_{g\in \Gamma} g^{-1}(\iota(g\ast x)),\]
	is $\Gamma$-equivariant, and $\pi\tilde\iota$ is the identity on $\bar{\C}^\bullet(\M)$.  Thus the image of $\tilde\iota$ is a complementary summand to $\mathrm{P}^\bullet$ which is $\Gamma$-invariant.
\end{remark}

For a $\Gamma$-equivariant graded $\H(\M)$-module $\N$, the stalk $\N_\varnothing$ is a graded $\Gamma$-representation.  More generally, for any flat $F$, $\N_F$ is a graded representation of the stabilizer group $\Gamma_F$, and the sum $\bigoplus_{G\in \Gamma F} \N_G$ over the $\Gamma$-orbit of $F$ is a graded $\Gamma$-representation. Define a polynomial 
\[\tilde{Q}^\Gamma_\M(t) = \sum_{k\ge 0} \left[\bar{\C}^{d-2k}(\M)^k_\varnothing\right] t^k \in \VRep(\Gamma)[t],\]
where $\VRep(\Gamma)$ is the ring of virtual representations of $\Gamma$.
The first superscript $d-2k$ is the cohomological degree in the complex, and the second is the grading degree in the stalk.  Note that the only indecomposable summands of $\bar{\C}^{d-2k}(\M)$ which contribute to the stalk at $\varnothing$ in degree $k$ are the ones of the form $\IH(\M^\varnothing)[-k]$, so applying the dimension homomorphism $\VRep(\Gamma)[t] \to \Z[t]$ to $\tilde{Q}^\Gamma_\M(t)$ recovers the non-equivariant polynomial $\tilde{Q}_\M(t)$.
Then Lemma \ref{lem:inverse mult at general flat} generalizes as follows to the equivariant setting, with essentially the same proof.

\begin{lemma}
For any flat $F$, we have
\[\tilde{Q}^{\Gamma_F}_{\M_F}(t) = \sum_{k\ge 0} \left[\bar{\C}^{\crk F-2k}(\M)^k_F\right] t^k \in \VRep(\Gamma_F)[t].\]	
\end{lemma}

\begin{proposition}\label{prop:equivariant iso}
For any $i\ge 0$, there exists a $\Gamma$-equivariant isomorphism of $\H(\M)$-modules
\begin{equation}\label{eqn:equivariant Rouquier decomp}
\bar{\C}^i(\M) \cong \bigoplus_{F} \bar{\C}^i(\M)_F^{(\crk F - i)/2} \otimes \IH(\M^F)\left[\frac{i-\crk F}{2}\right],
\end{equation}
where the sum is over all flats such that $\crk F$ and $i$ have the same parity.  On the right side $\H(\M)$ acts on the second tensor factor only, and $g\in \Gamma$ acts by sending $\bar{\C}^i(\M)_F$ to $\bar{\C}^i(\M)_{gF}$ and sending $\IH(\M^F)$ to $\IH(\M^{gF})$. 
\end{proposition}
Assuming this for the moment, the fact that 
$\bar{\C}^\bullet(\M)_\varnothing$ is acyclic together with Lemma \ref{induced reps} gives
\[\sum_{F\in \cL(\M)} (-1)^{\rk F}\frac{|\Gamma_F|}{|\Gamma|}\Ind_{\Gamma_F}^\Gamma\left(P_{\M^F}^{\Gamma_F}(t)\tilde{Q}_{\M_F}^{\Gamma_F}(t)\right)=0.\]
This implies that $\tilde{Q}^\Gamma_\M(t)$ satisfies the recursion of Definition \ref{eqI}, and so $\tilde{Q}^\Gamma_\M(t) = {Q}^\Gamma_\M(t)$.  This immediately implies the following equivariant generalization of Theorem \ref{thm:inverse}.
\begin{theorem}
	The coefficients of $Q^\Gamma_\M(t)$ are honest $\Gamma$-representations.
\end{theorem}

To prove Proposition \ref{prop:equivariant iso}, first note that since $\bar{\C}^\bullet(\M)$ is a minimal perverse complex, Theorem \ref{thm:perverse complexes} gives an isomorphism
\[\bar{\C}^i(\M) \stackrel{\sim}\longrightarrow \bigoplus_F V_F \otimes \IH(\M^F)\left[\frac{i-\crk F}{2}\right]\]
for some finite-dimensional vector spaces $V_F$.  Taking the degree $(\crk F -i)/2$ part of the stalk at $F$ of both sides gives an isomorphism $\bar{\C}^i(\M)^{(\crk F-i)/2}_F \cong V_F$, since $\IH(\M^F)_F \cong \Q$ canonically, and the stalk at $F$ of all of the other summands vanishes in degree $(\crk F-i)/2$, by Proposition \ref{zero map}.
This gives an isomorphism of the form \eqref{eqn:equivariant Rouquier decomp}, but it may not be $\Gamma$-equivariant.  By construction, applying the stalk functor at a flat $F$ to this isomorphism and taking the 
 degree $(\crk F-i)/2$ part gives the identity isomorphism
\[\bar{\C}^i(\M)_F^{(\crk F - i)/2} \cong \bar{\C}^i (\M)_F^{(\crk F - i)/2} \otimes \Q,\]
which \emph{is} $\Gamma$-equivariant.  So averaging our isomorphism   
 over $g\in \Gamma$ gives a $\Gamma$-equivariant map which still induces the identity on the stalks at $F$ in degree $(\crk F - i)/2$.  The fact that the averaged map is an isomorphism now follows using Proposition \ref{prop:indecomposable} (2).


\section{Deletion induction for {$\IH(\M)$}}\label{sec:deletion induction}
Let $\M$ be a matroid of rank $d>0$ on the ground set $E$. 
The purpose of this section is to show that, 
if $\hyperlink{cd}{\CD^{<\frac{d}{2}}(\M)}$ holds, and all of the statements of Theorem \ref{theorem_all} hold for matroids whose ground
sets are proper subsets of $E$, then $\hyperlink{hli}{\HL_i(\M)}$ and $\hyperlink{hri}{\HR_i^{<\frac{d}{2}}(\M)}$ also hold.

Throughout this section, we assume the following hypotheses:
\begin{enumerate}[(1)] \itemsep 5pt
\item the element $i\in E$ is not a coloop and $\{i\}$ is a flat, that is, $i$ does not have any parallel element;
\item the statement $\hyperlink{cd}{\CD^{<\frac{d}{2}}(\M)}$ holds;
\item Theorem \ref{theorem_all} holds for any matroid whose ground set is a proper subset of $E$.
\end{enumerate}

In particular, $\hyperlink{pdo}{\PD_\o(\M)}$ and $\hyperlink{cdo}{\CD_\o(\M)}$ hold by Corollary \ref{CD for free}, and $\hyperlink{cd}{\CD^{>\frac{d}{2}}(\M)}$ holds by Remark \ref{remark_cd}. By Remark \ref{CDPD} and Proposition~\ref{CDNS}, the statement $\hyperlink{cd}{\CD^{<\frac{d}{2}}(\M)}$ implies $\hyperlink{pd}{\PD^{<\frac{d}{2}}(\M)}$ and $\hyperlink{ns}{\NS^{<\frac{d}{2}}(\M)}$.
Our goal is to show that these hypotheses imply $\hyperlink{hli}{\HL_i(\M)}$ and $\hyperlink{hri}{\HR_i^{<\frac{d}{2}}(\M)}$.

\subsection{The deletion map and the semi-small decomposition of $\CH(\M)$}\label{sec_deletion}
Fixing an element $i$ of $E$, there is an injective graded algebra homomorphism \cite[Section 3]{BHMPW}
\[
\theta_i = \theta_i^{\M} \colon \CH(\Mi)\to \CH(\M), \quad x_F\mapsto x_{F}+x_{F\cup i},
\]
where a variable in the target is set to zero if its label is not a flat of $\M$.  Just as we have done with the pushforward and pullback homomorphisms, we will omit the superscript when the ambient matroid is $\M$.  Then we have $\theta_i^{\Mo}(y_j)=y_j$ for any $j\in E\setminus i$. More generally, for any flat $G\in\cL(\M \setminus i)$, we have $\theta_i^{\Mo}(y_G)=y_{\bar{G}}$, where $\bar{G}$ is the closure of $G$ in $\M$. In particular, $\theta_i$ restricts to an injective homomorphism $\H(\M \setminus i) \to \H(\M)$.   

Let $\CHi$ be the image of the homomorphism $\theta_i^{\Mo}$, and let
\[
{\mathscr{S}}_i\coloneq \big\{F \ | \ \text{$F$ is a proper subset of $E \setminus i$ such that $F \in \mathscr{L}(\mathrm{M})$ and $F\cup i \in \mathscr{L}(\mathrm{M})$}\big\}.
\]
Note that, if $i$ is not a coloop, then $E\setminus i$ is not a flat of $\M$ and so $E \setminus i \notin \sS_i$ and $\rk (\M\setminus i) = \rk \M$. If $\{i\}$ is a flat, then $\varnothing \in \sS_i$. Moreover, for any $F\in {\mathscr{S}}_i$, $i$ is a coloop in the localization $\M^{F\cup i}$.
We will use the following result \cite[Theorem 1.5]{BHMPW}.  
\begin{theorem}\label{TheoremDecomposition}
If $i$ is not a coloop of $\mathrm{M}$,  there is  a direct sum decomposition of $\mathrm{CH}(\mathrm{M}) $ into indecomposable graded $\mathrm{CH}(\mathrm{M}\setminus i)$-modules
\[
\mathrm{CH}(\mathrm{M}) = \mathrm{CH}_{(i)} \oplus \bigoplus_{F \in \mathscr{S}_i} x_{F\cup i} \mathrm{CH}_{(i)}.
\]
All pairs of distinct summands are orthogonal for the Poincar\'e pairing of $\mathrm{CH}(\mathrm{M})$.
Moreover, for all $F \in\mathscr{S}_i$, we have 
\[x_{F\cup i}\CHi = \psi^{F\cup i}\!\left(\uCH(\M_{F\cup i})\otimes \theta^{\M^{F\cup i}}_i\CH(\Mip^F)\right),\]
where $\M^{F\cup i} \setminus i$ is identified with $\Mip^F$ because $i$ is a coloop in $\M^{F\cup i}$.  The homomorphism $\theta_i$ gives isomorphisms as $\H(\Mi)$-modules:
\[
\CHi\cong \CH(\Mi) \;\;\text{and}\;\; x_{F\cup i}\CHi\cong \uCH(\M_{F\cup i})\otimes \CH(\Mip^F)[-1],
\]
where the action on the first tensor factor is trivial.  If $F \ne \varnothing$, these are even isomorphisms of $\H_\o(\Mi)$-modules.
\end{theorem}

\begin{example}\label{ex:semismall U34}
Following Examples \ref{ex:rank2_2} and \ref{ex:uniform2}, we give explicit semi-small decompositions of $\CH(U_{2,3})$ and $\CH(U_{3,4})$. 

For the uniform matroid $U_{2,3}$, we have $\mathscr{S}_{3}=\{\varnothing\}$. Notice that $\{1\}, \{2\}\notin \mathscr{S}_{3}$, since $\{1, 3\}$ and $\{2, 3\}$ are not flats. 
For the deletion of the element $3$, the corresponding semi-small decomposition is
\[
\CH(U_{2,3})=\CH_{(3)}\oplus \, \Q x_3. 
\]
Here, $\Q x_3$ is a $\CH_{(3)}$-submodule, since $\CH^1(U_{2,2})$ is generated by $x_1$, $x_2$, and $y_1$, and evidently $\theta_3(x_1)\cdot x_3=\theta_3(x_2)\cdot x_3=\theta_3(y_1)\cdot x_3=0$. Therefore, the direct sum decomposition follows from the Poincar\'e duality of $\CH(U_{2,3})$ and 
\[
\deg_{\,U_{2,3}}(x_3^2)=\deg_{\,U_{2,3}}\!\big((y_1-x_\varnothing-x_2)x_3\big)=-\deg_{\,U_{2,3}}(x_\varnothing x_3)=-1.
\]
In this example, we see that the semi-small decomposition is in general different from the canonical decomposition $\CH(U_{2,3})=\H(U_{2,3})\oplus\Q x_\varnothing$ in  Example \ref{ex:rank2_2}.  

For the uniform matroid $U_{3,4}$, $\mathscr{S}_{4}=\{\varnothing, \{1\}, \{2\}, \{3\}\}$.  For the deletion of the element $4$, the corresponding semi-small decomposition of $\CH(U_{3,4})$ is 
\begin{align*}
\CH(U_{3,4})&=\CH_{(4)}\oplus \, x_{4}\CH_{(4)} \oplus \, x_{14}\CH_{(4)}\oplus\, x_{24}\CH_{(4)}\oplus\, x_{34}\CH_{(4)}\\
&=\CH_{(4)}\oplus \, (\Q x_{4}\oplus \Q x_4 x_{14})\oplus (\Q x_{14}\oplus \Q y_1 x_{14}) \oplus (\Q x_{24}\oplus \Q y_2 x_{24}) \oplus (\Q x_{34}\oplus \Q y_3 x_{34}).
\end{align*}
Since $\varphi^{4}(x_{14})=\varphi^{4}(x_{24})=\varphi^{4}(x_{34})$, we have $x_4 x_{14}=x_4 x_{24}=x_4 x_{34}$. Therefore, the above decomposition is invariant under permutations of $1$, $2$, and 3. 
Similar to the decomposition of $\CH(U_{2,3})$, it is also possible to directly check that each summand is a $\CH_{(4)}$-submodule and the restriction of the Poincar\'e pairing to every smaller summand is non-degenerate. The orthogonality between the smaller summands is obvious. 
For a more interesting example, the orthogonality between $\CH_{(4)}$ and $\Q x_{4}\oplus \Q x_4 x_{14}$ implies that $x_{14}(x_1+x_{14})^2 = 0$, which can also be verified by an explicit computation.
Indeed,
\[
x_{14}(x_1+x_{14})=x_{14}(x_\varnothing+x_1+x_{4}+x_{14})-x_{14}(x_\varnothing+x_4)=x_{14}y_{2}-x_{14}y_1=-x_{14}y_1.
\]
Therefore, 
\[
x_{14}(x_1+x_{14})^2=x_{14}y_1^2=0.
\]

Comparing this example with Example \ref{ex:uniform2}, we see that the semi-small decomposition and the canonical decomposition differ both in the number of terms and their dimensions. Moreover, although the summands $(\Q x_i\oplus y_i \Q x_{i4})$ for $i=1, 2, 3$ are $\H(\M)$-modules, the summand $(\Q x_4\oplus \Q x_4 x_{14})$ is not  an $\H(\M)$-module, since 
\[
x_4 x_{14}\notin \langle y_1, y_2, y_3, y_4\rangle \cdot x_4=\langle y_4 x_4\rangle=\langle x_\varnothing x_4\rangle.
\]
Thus, the semi-small decomposition in general is not a coarsening of the canonical decomposition, since the latter is a decomposition of $\H(\M)$-modules. 
\end{example}


\subsection{Pulling back to the deletion}

Let $\delta \colon \cL(\M)\to\cL(\M\setminus i)$ be the map given by $\delta(F) = F\setminus i$ for all flats $F$.  

\begin{lemma}\label{lemma:semi-small flats}
	The map $\delta$ is surjective and order-preserving.  For any flat $F \in \cL(\M\setminus i)$, we have
	\begin{itemize}[$\bullet$]\itemsep 5pt
		\item if $F \in \sS_i$, then
		$\delta^{-1}(F) = \{F, F\cup i\}$ and $\rk_{\M\setminus i} F = \rk_\M F = \rk_\M(F\cup i) - 1$, and
		\item if $F \notin \sS_i$, then $\delta^{-1}(F)$ is a single flat of $\M$ with the same rank as $F$.
	\end{itemize}
\end{lemma}
\begin{proof}
The deletion $\M\setminus i$ has the property that $\rk_{\M\setminus i}(S) = \rk_\M(S)$ for any subset $S$ of $E \setminus i$ (in fact, when matroids are defined using the rank function, this can be taken as the definition of $\M\setminus i$).  From this and the submodularity property of $\rk_\M$, it is easy to check that $F\setminus i$ is a flat of $\Mi$ for any flat $F$ of $\M$.  So the map $\delta$ is well-defined.

Given any flat $F \in \cL(\M\setminus i)$, we have $F\subseteq \overline{F}\setminus i$, where $\overline{F}$ is the closure of $F$ in $\M$, and 
\[
\rk_{\M\setminus i}\left(\overline{F}\setminus i\right)=\rk_{\M}\left(\overline{F}\setminus i\right)\leq \rk_{\M}\left(\overline{F}\right)=\rk_{\M}(F)=\rk_{\M\setminus i}(F).
\]
Since $F\subseteq \overline{F}\setminus i$ implies the opposite inequality, all of the ranks above must be equal.  If $\overline F$ contained an element $j \notin F \cup i$, then the fact that $F$ is a flat of $\Mi$ would imply
\[\rk_\M(\overline{F}) \ge \rk_\M(F \cup j) = \rk_{\Mi}(F\cup j) > \rk_{\Mi}(F),\]
which is a contradiction. It follows that $\overline{F}\setminus i=F$, and so  $\overline{F}$ equals either $F$ or $F\cup i$, and $\delta$ is surjective. Clearly, deleting a fixed element preserves the order. Finally, the two bullet points follow from the definition of $\sS_i$ and the above sequence of equalities of ranks. 
\end{proof}


Recall from Definition~\ref{ideal order filter} that, for any order ideal $\Sigma\subseteq\cL(\M)$, the ideal 
$\Upsilon_\Sigma\subseteq \H(\M)$ is equal to the $\Q$-vector space spanned by $\{y_G\mid G\in \Sigma\}$.
In this section, we will write $\Upsilon_\Sigma^\M$ for $\Sigma\subseteq\cL(\M)$ and $\Upsilon_\Sigma^{\Mi}$ for $\Sigma\subseteq\cL(\Mi)$
to make it clear which matroid we are working with at any given time.  The fact that for any $G \in \cL(\M\setminus i)$ we have $\theta_i^{\Mo}(y_G)=y_{\bar{G}}$, where $\bar{G}$ is the minimal element of $\delta^{-1}(G)$, immediately
implies the following lemma.

\begin{lemma}\label{deletion filters}
	For any order ideal $\Sigma$ in $\cL(\M\setminus i)$, we have
	\[
	\H(\M)\cdot \theta_i^{\Mo}(\Upsilon_{\Sigma}^{\M\setminus i})=\Upsilon_{\delta^{-1}(\Sigma)}^{\M}. 
	\]
\end{lemma}

For an $\H(\M)$-module $\N$, we let $\theta^*_i\N$ denote the $\H(\M\setminus i)$-module obtained by pulling back by the homomorphism $\theta^*_i$. 

\begin{proposition}\label{prop:semi-small stalk}
	Let $\N$ be a pure graded $\H(\M)$-module.  For any flat $F \in \cL(\M\setminus i)$, we have an isomorphism
	\[(\theta^*_i\N)_{F} \cong \bigoplus_{G\in \delta^{-1}(F)} \N_G[\rk_{\M\setminus i} F - \rk_\M G].\]
\end{proposition}

\begin{proof}
	By Lemma \ref{deletion filters} and Proposition \ref{pure}, we have
	\[(\theta^*_i\N)_F = \frac{\theta_i^{\Mo}\big(\Upsilon^{\Mi}_{\geq F}\big)\N[\rk F]}{\theta_i^{\Mo}\big(\Upsilon^{\Mi}_{> F}\big)\N[\rk F]}
	= \frac{\Upsilon^{\M}_{\delta^{-1}(\Sigma_{\geq F})}\N[\rk F]}{\Upsilon^{\M}_{\delta^{-1}(\Sigma_{>F})}\N[\rk F]}.\]
	If $\delta^{-1}(F)$ has a single element, then Proposition \ref{pure} immediately implies the result.  Otherwise, we can put $\delta^{-1}(F) = \{G_1, G_2\}$ with $G_1 < G_2$.  Let $\Sigma_j = \delta^{-1}(\Sigma_{>F}) \cup \{G_i \mid i>j\}$ for $j=0,1,2$.  Since $\delta$ is order-preserving, these are order ideals.  Then Proposition \ref{pure} gives isomorphisms
	\[\frac{\Upsilon^\M_{\Sigma_{j-1}}\N[\rk F]}{\Upsilon^\M_{\Sigma_j}\N[\rk F]} = \frac{\N_{\Sigma_{j-1}}[\rk F]}{\N_{\Sigma_j}[\rk F]} \cong \N_{G_j}[\rk F - \rk G_j]\]
	for $j = 1, 2$.   Since $\Sigma_0 = \delta^{-1}(\Sigma_{\ge F})$ and $\Sigma_2 = \delta^{-1}(\Sigma_{>F})$, the result follows.   
\end{proof}

\subsection{The hard Lefschetz theorem}\label{sec:delhl}
We would like to apply Proposition \ref{prop:semi-small stalk} to the pullback module $\theta^*_i\IH(\M)$, but since at this point in the induction we are not assuming that $\hyperlink{cd}{\CD(\M)}$ holds in middle degree, we do not yet know that $\IH(\M)$ is pure. Instead, we modify $\IH(\M)$ slightly to produce a module which we can show is a direct summand of $\CH(\M)$, and hence is pure.  This module will not be used outside of Section \ref{sec:deletion induction}.  Let 
\[
\tIH^k(\M) \coloneq 
\begin{cases}
\IH^k(\M) &\text{if $k\neq d/2$,}\\
\IH_\o^k(\M) &\text{if $k=d/2$.}
\end{cases}
\]
Equivalently, we can define
\[
\tuJ^k(\M) \coloneq 
\begin{cases}
\uJ^k(\M) &\text{if $k\neq d/2$,}\\
0 &\text{if $k= d/2$,}
\end{cases}
\]
and then define $\tIH(\M)$ to be the orthogonal complement to $\psi^\varnothing_\Mo\tuJ(\M)$ inside of $\IH_\o(\M)$. In particular, when $d$ is odd, $\tIH(\M)=\IH(\M)$ and $\tuJ(\M)=\uJ(\M)$.  Note that $\tIH(\M)$ satisfies hard Lefschetz for an element $y \in \H^1(\M)$ if and only if $\IH(\M)$ does, and the Hodge-Riemann relations with respect to $y$ are the same for $\tIH(\M)$ and $\IH(\M)$ except in degree $d/2$.

As in Example \ref{ex:rank2_2}, when $\M=U_{2,3}$, $\tIH(\M) =\CH(\M)$. Similarly, as in Example \ref{ex:uniform2}, when $\M=U_{3,4}$, the rank of $\M$ is odd, and hence $\tIH(\M) =\IH_\o(\M)=\IH(\M)$. 

\begin{lemma}\label{tilde PD}
The subspace $\tIH(\M)\subseteq\IH_\o(\M)$ is an $\H(\M)$-submodule.  
Moreover, $\tIH(\M)$ satisfies Poincar\'e duality and it is a direct summand of $\CH(\M)$.  In particular, it is pure.

\end{lemma}

\begin{proof}
The maximal ideal $\Upsilon_{>\varnothing}$ of $\H(\M)$ annihilates $x_\varnothing$, and hence annihilates the image of $\psi^\varnothing_\Mo$.  
Therefore, $\psi^\varnothing_\Mo\tuJ(\M)$ is an $\H(\M)$-submodule, and thus so is its orthogonal complement. The statement $\hyperlink{cd}{\CD^{<\frac{d}{2}}(\M)}$ implies that
$\psi^\varnothing_\Mo\tuJ(\M)$ satisfies Poincar\'e duality, and the statement $\hyperlink{cdo}{\CD_\o(\M)}$ implies that $\IH_\o(\M)$ satisfies Poincar\'e duality.
Therefore, $\tIH(\M)$ satisfies Poincar\'e duality and we have an orthogonal decomposition of $\H(\M)$-modules
\begin{equation}\label{ihcircdecomp}
\IH_\o(\M)=\tIH(\M)\oplus \psi^\varnothing_\Mo\tuJ(\M). 
\end{equation}
By $\hyperlink{cdo}{\CD_\o(\M)}$, $\IH_\o(\M)$ is a direct summand of $\CH(\M)$, and hence the lemma follows. 
\end{proof}


\begin{lemma}\label{tilde stalks}
The inclusion $\tIH(\M)\subseteq\IH_\o(\M)$
induces an isomorphism
\[
\tIH(\M)_F\cong \IH_\o(\M)_F
\]
for each nonempty flat $F$.
\end{lemma}

\begin{proof}
The isomorphism follows from multiplying Equation~\eqref{ihcircdecomp} by $y_F$, since the image of $\psi^\varnothing_\Mo$ is annihilated by $y_F$.  
\end{proof}


\begin{proposition}\label{prop:pullback perverse}
	The pullback module $\theta^*_i\tIH(\M)$ is a perverse $\H(\M\setminus i)$-module, when considered as a complex placed in degree $0$.
\end{proposition}

\begin{proof}
	Theorem \ref{TheoremDecomposition} implies that $\theta^*_i \CH(\M)$ is a pure $\H(\M\setminus i)$-module, and so the direct summand $\theta^*_i\tIH(\M)$ is pure.
	
	Take any flat $F\in \cL(\M\setminus i)$.  We will show that the stalk $(\theta^*_i\tIH(\M))_F$ vanishes in degrees strictly greater than $(\crk F)/2$.  By Proposition \ref{prop:semi-small stalk}, it is enough to prove that 
	\[\tIH(\M)_G[\rk_{\M\setminus i} F - \rk_\M G]\]
	vanishes in the same degrees for every $G\in \delta^{-1}(F)$.	
	
	The first case is $F = E \setminus i$.  Then $\delta^{-1}(F) = \{E\}$ and $\rk_\M E = \rk_{\Mi} F$, since $i$ is not a coloop in $\M$.  We have
	\[\tIH(\M)_E[\rk_{\M\setminus i} F - \rk_\M E] = \tIH(\M)_E \cong \Q,\]
	placed in degree $0 = (\crk F)/2$.  So the claimed vanishing holds in this case.
	
	Now suppose that $F$ is a proper flat of $\Mi$, and take any $G \in \delta^{-1}(F)$. Then 
	$\rk_\M G$ is either $\rk_{\Mi} F$ (if $F \notin \sS_i$ or $G = F$) or $\rk_{\Mi} F +1$ (if $F \in \sS_i$ and $G = F\cup i$).  Let us suppose first that $G \ne \varnothing$.  Then Lemma \ref{tilde stalks} and Proposition \ref{zero map} (2) show that
	\[\tIH(\M)_G[\rk_{\M\setminus i} F - \rk_\M G] \cong \IH_\o(\M)_G[\rk_{\M\setminus i} F - \rk_\M G]\]
	vanishes in degrees greater than or equal to 
	\[(\crk G)/2 + \rk_\M G - \rk_{\Mi} F = (\crk F)/2 + (\rk_\M G - \rk_{\Mi} F)/2 \le (\crk F)/2 + 1/2.\]
	In particular, it vanishes in degrees strictly greater than $(\crk F)/2$, as desired.
	
	Next, suppose that $G = F = \varnothing$.  Following the proof of Proposition \ref{zero map}, note that  
	Lemmas \ref{stalk-costalk} and \ref{tilde PD} give an isomorphism
	$\tIH(\M)_{\varnothing} \cong (\tIH(\M)_{[\varnothing]})^*[-d]$.
	On the other hand, the socles $\tIH(\M)_{[\varnothing]}$ and $\IH(\M)_{[\varnothing]}$ are clearly equal in all degrees except $d/2$, and so $\hyperlink{ns}{\NS^{<\frac{d}{2}}(\M)}$ implies that $\tIH(\M)_{[\varnothing]}$ vanishes in degrees below $d/2$.  Thus $\tIH(\M)_{\varnothing}$ vanishes in degrees above $d/2$.  
	
	Finally, to see that the costalk conditions hold, we note that Poincar\'e duality gives an isomorphism $\tIH(\M)^\ast \cong \tIH(\M)[d]$ and so Lemma \ref{stalk-costalk} implies that the costalk conditions follow from the stalk conditions.	
\end{proof}

Since all of our statements are true for $\M \setminus i$ by induction, Theorem \ref{thm:perverse complexes} allows us to deduce the following.

\begin{corollary}\label{centering}
The graded $\H(\Mi)$-module $\theta_i^*\tIH(\M)$ is isomorphic to a direct sum of modules of the form $\IH(\Mip^F)[-(\crk F)/2]$
for various flats $F\in\cL(\Mi)$ of even corank.
\end{corollary}

\begin{proposition}\label{HLi}
The statement $\hyperlink{hli}{\HL_i(\M)}$ holds.
\end{proposition}

\begin{proof}
	Let $y' = \sum_{j \in E \setminus i} c_jy_j$ where all $c_j > 0$.  Then for any flat $F \in \cL(\M\setminus i)$ of even corank, the statement $\HL((\M\setminus i)^F)$ holds because the ground set $F$ is a proper subset of $E$.  So for each $0 \le k \le (\rk F)/2$, multiplication by 
	$(y')^{\rk F - 2k}$ on $\IH(\Mip^F)[-(\crk F)/2]$ gives an isomorphism between degrees 
	\[k + \frac{\crk F}{2} \;\; \mbox{and}\;\; (\rk F - k) + \frac{\crk F}{2} = d - \left(k + \frac{\crk F}{2}\right).\]
	Since Corollary \ref{centering} says that $\theta_i^*\tIH(\M)$ is isomorphic to a direct sum of such modules, and since $y'$ is in the image of $\theta_i$, this shows that $y'$ acts as a degree $d$ Lefschetz operator on $\tIH(\M)$.  Since $\tIH(\M) = \IH(\M)$ except in the middle degree $d/2$, where the hard Lefschetz property is trivial, this proves the statement $\hyperlink{hli}{\HL_i(\M)}$.
\end{proof}


\subsection{The Hodge--Riemann relations away from middle degree}\label{sec:first HR}
Next we prove the statement $\hyperlink{hri}{\HR_i^{<\frac{d}{2}}(\M)}$, which says that the Hodge--Riemann inequalities hold for $\IH(\M)$ with respect to multiplication by an element $y' = \sum_{j\in E\setminus i} c_jy_j$ with all $c_j > 0$.  Since $y'$ can also be considered as an element of $\H(\Mi)$, we can show this by checking that this holds for each summand in the decomposition of $\theta_i^*\tIH(\M)$ provided by Corollary \ref{centering}. 

We will need a lemma comparing two natural pairings on these summands.
Let $F$ be a nonempty flat of $\Mi$ of even corank, and suppose we have any inclusion of $\H(\Mi)$-modules
\[f \colon \IH(\Mip^F)[-(\crk F)/2]\hookrightarrow \theta_i^*\tIH(\M),\]
such as the inclusions of direct summands in the decomposition of Corollary \ref{centering}. 
There are two pairings on $\IH(\Mip^F)$ that are \textit{a priori} different:  the one induced by the inclusion of $\IH(\Mip^F)$ into $\CH(\Mip^F)$,
and the one induced by the inclusion $f$ and the pairing on $\CH(\M)$.  Note that the shift by $(\crk F)/2$ ensures that these pairings have the same degree.

\begin{lemma}\label{pair and compare}
These two pairings are related by a constant factor $c\in\Q$ with $(-1)^{\frac{\crk F}{2}}c > 0$.
\end{lemma}

\begin{proof}
Both pairings are compatible with the $\H(\Mi)$-module structure in the sense that $\langle \eta\xi,\sigma\rangle = \langle \xi,\eta\sigma\rangle$ for any $\eta\in\H(\Mi)$ and $\xi,\sigma\in \IH(\Mip^F)$. Thus, both are given by isomorphisms 
\[
\IH(\Mip^F)^*\cong \IH(\Mip^F)[\rk F]
\]
of graded $\H(\Mi)$-modules.
Proposition \ref{prop:indecomposable} (1) implies that $\IH(\Mip^F)$ has only scalar endomorphisms, and hence any two such isomorphisms must be related by a nonzero scalar factor $c\in\Q$.  

To compute the sign of $c$, we pair $1\in\IH(\Mip^F)$ with $y_F = y_F\cdot 1 \in\IH(\Mip^F)$.
Inside $\CH(\Mip^F)$, they pair to $1$.  The second pairing equals the pairing 
of $f(1)$ and 
\[f(y_F) = \theta_i(y_F) f(1) = y_{\bar F} f(1)\]
inside $\tIH(\M) \subseteq \CH(\M)$, where $\bar{F}$ is the closure of $F$ in $\M$.
By Proposition \ref{lemma_degy} and Proposition \ref{Prop_ymult}, this pairing is equal to the Poincar\'e pairing of $\varphi^\Mo_{\bar F}(f(1))$ with itself inside $\varphi^\Mo_{\bar F}\tIH^{\frac{\crk F}{2}}(\M)\subseteq \CH(\M_{\bar F})$.
Since $F$ is nonempty, $(\crk F)/2$ is strictly less than $d/2$, so $\tIH^{\frac{\crk F}{2}}(\M) = \IH^{\frac{\crk F}{2}}(\M)$.
Applying Lemma~\ref{mult by y} (1), we see that $\varphi^\Mo_{\bar F}\IH^{\frac{\crk F}{2}}(\M)$ is equal to 
$\IH^{\frac{\crk F}{2}}(\M_{\bar F})$. Since $\varphi^\Mo_{\bar F}(f(1))$ is annihilated by $y_j$ for all $j\in E\setminus \bar F$, it is a primitive class in $\IH^{\frac{\crk F}{2}}(\M_{\bar F})$ with respect to multiplication by any positive sum $\sum_{j\in E \setminus \bar{F}} c_jy_j$. Therefore, the sign of its pairing with itself is equal to $(-1)^{\frac{\crk F}{2}}$ by $\hyperlink{hr}{\HR(\M_{\bar F})}$.
\end{proof}

\begin{example}\label{ex:Poincare U34}
For a rank 2 matroid $\M$, Lemma \ref{pair and compare} says that $\deg_\M(x_i^2)<0$. In fact, for any $j\neq i$, 
\[
0=y_jx_i=(x_\varnothing+x_i)x_i,
\]
and hence $\deg_\M(x_i^2)=-\deg_\M(x_\varnothing x_i)=-1$. 

As in the more interesting Example \ref{ex:semismall U34}, when $\M=U_{3,4}$ and $i=4$, Lemma \ref{pair and compare} asserts that the Poincar\'e pairing between $x_{4}$ and $x_{4}x_{14}$ is negative, and the Poincar\'e pairing between $x_{i4}$ and $y_{i}x_{i4}$ is also negative for $i=1, 2, 3$. For a more concrete illustration of the mechanism behind the abstract lemma, we verify these assertions directly from the definition. By the definition of augmented Chow ring, we have
\[
y_1=x_\varnothing+x_2+x_3+x_4+x_{23}+x_{24}+x_{34}.
\]
Multiplying the above equation by $x_4x_{14}$ and using the defining relations of the augmented Chow ring, we have
\[
0=x_{\varnothing}x_4 x_{14}+x_4^2 x_{14}. 
\]
By the definition of degree map (Definition \ref{DefinitionDegreemap}), we have $\deg_\M(x_{\varnothing}x_4 x_{14})=1$. Thus, $\deg_\M(x_4^2 x_{14})=-1$, which confirms the first assertion. Multiplying the above expression of $y_1$ by $y_2 x_{24}$, we have
\[
0=y_2x_2x_{24}+y_2x_{24}^2.
\]
Using the defining formula for $y_2$, we can deduce that $y_2 x_2=x_\varnothing x_2$. Thus, 
\[
\deg_\M(y_2x_2x_{24})=\deg_\M(x_\varnothing x_2x_{24})=1,
\]
and hence $\deg_\M(y_2 x_{24}^2)=-1$. By symmetry, we have also confirmed the second assertion. 

If we rescale the map $f$ by a nonzero constant $\lambda$, then the constant $c$ is multiplied by $\lambda^2$, and hence its sign remains unchanged.
\end{example}

\begin{corollary}\label{HRi}
The statement $\hyperlink{hri}{\HR_i^{<\frac{d}{2}}(\M)}$ holds.
\end{corollary}

\begin{proof}
Since the statement does not involve the middle degree, we can replace $\IH(\M)$ with $\tIH(\M)$.  By Corollary \ref{centering}, it suffices to prove that each summand $\IH(\Mip^F)[-(\crk F)/2]$ of $\theta_i^*\tIH(\M)$ satisfies the Hodge--Riemann relations.  Again, since the statement does not involve the middle degree, we can assume that $F$ is nonempty.
Then the statement follows from Lemma \ref{pair and compare} and $\hyperlink{hr}{\hyperlink{hr}{\HR(\Mip^F)}}$. 
\end{proof}

\section{Deletion induction for {$\uIHi$}}\label{sec:underlined deletion induction}
Let $\M$ be a matroid on the ground set $E$. The purpose of this section is to show that $\hyperlink{uhli}{\uHL_i(\M)}$ and $\hyperlink{uhri}{\uHR_i(\M)}$ hold under the following hypotheses, which we assume throughout this section:

\begin{enumerate}[(1)] \itemsep 5pt
\item the element $i\in E$ is not a coloop and $\{i\}$ is a flat, that is, $i$ does not have any parallel element;
\item $\hyperlink{cdo}{\CD_\o(\M)}$ and $\hyperlink{ucd}{\uCD(\M)}$ are true; and
\item Theorem \ref{theorem_all} holds for any matroid whose ground set is a proper subset of $E$. 
\end{enumerate}


The argument is similar to the one in the previous section.  
The homomorphism $\theta_i\colon \CH(\Mi) \to \CH(\M)$ induces a homomorphism $\underline{\theta}_i\colon \uCH(\Mi) \to \uCH(\M)$ which sends $\uH(\Mi)$ to $\uH_i(\M)$, and in particular sends $\ub_{\Mi} = \varphi^\varnothing_{\Mi}(-x_\varnothing)$ to \[\varphi^\varnothing_\M(\theta_i(-x_\varnothing)) = \varphi^\varnothing_\M(-x_\varnothing - x_{\{i\}}) = \ub_\M - x_{\{i\}}.\]
We show in Corollary \ref{uIH_i in terms of uIH_1} that pulling back $\uIH_i(\M)$ by $\underline{\theta}_i$ gives an $\uH(\Mi)$-module which is isomorphic to a direct sum of modules of the form $\uIH((\M \setminus i)^F)[-(\crk F)/2]$.  Since the matroids $(\M \setminus i)^F$ have smaller ground sets than $\M$, we know $\hyperlink{uhl}{\uHL((\M\setminus i)^F)}$ and $\hyperlink{uhr}{\uHR((\M\setminus i)^F)}$ by induction,
and we use this to deduce $\hyperlink{uhli}{\uHL_i(\M)}$ and $\hyperlink{uhri}{\uHR_i(\M)}$.

The geometric motivation of this direct sum decomposition is explained in Remark \ref{rmk:underlined semismall geometry}. The divisor $\ub_\M - x_{\{i\}}$ is a lef class on the blowup of $\underline{Y}$ at a point in the sense of \cite{dCM02}.

\begin{example}\label{ex:deletion1}
Let $\M$ be the rank $3$ uniform matroid on $\{1,2,3,4\}$. By Examples~\ref{ex:boolean1} and \ref{ex:uniform1},
we have the canonical decompositions
\[
\uCH(\M \setminus 4)=\uIH(\M \setminus 4) \oplus \bigoplus_{i=1}^3 \uK{i}{\M \setminus 4} \ \ \text{and} \ \ 
\uCH(\M)=\uIH(\M)\oplus \bigoplus_{i=1}^4 \uK{i}{\M},
\]
where $\underline{\mathrm{K}}_i$ is the $1$-dimensional summand spanned by $x_i$. 
The homomorphism $\underline{\theta}_4$ is given by 
\[
\underline{\theta}_4(x_1)=x_1+x_{14}, \ \ 
\underline{\theta}_4(x_2)=x_2+x_{24}, \ \ \underline{\theta}_4(x_3)=x_3+x_{34}, \ \ \text{and} \ \ \underline{\theta}_4(x_{ij})=x_{ij}.
\]
Note that $\underline{\theta}_4$ does not map $\uIH^1(\M \setminus 4)$ into $\uIH^1(\M)$ because $\underline{\theta}_4(\ub)=\ub-x_4$ pairs nontrivially with $x_4$.
Since $\ub-x_4$ annihilates $x_1,x_2,x_3$, we see that $\underline{\theta}_4$ does map $\uIH(\M \setminus 4)$ into  $\uIH_4(\M)$, and in fact, 
\[
\uIH_4(\M)=\uIH(\M) \oplus  \uK{4}{\M}\cong \uIH(\M \setminus 4) \oplus W
\]
as graded $\uH(\M \setminus 4)$-modules, where $W$ is the $3$-dimensional trivial module in degree $1$.
\end{example}

Note that $\uIH((\M\setminus i)^F)$ is not indecomposable as an  $\uH(\Mi)$-module, so we cannot produce our decomposition of the pullback $\underline{\theta}_i^*\uIH_i(\M)$ by directly following the arguments in the previous section.  Instead, we deduce it from a decomposition of a certain $\H_\o(\Mi)$-module, which we now define.
Let $\H_i(\M)$ denote the subalgebra of $\CH(\M)$ generated by $\H_\o(\M)$ and $x_{\{i\}}$.  Then $\theta_i$ sends $\H_\o(\Mi)$ into $\H_i(\M)$.

\begin{definition}
	We define the $\H_i(\M)$-submodule $\IH_i(\M)$ of $\CH(\M)$ by
	\[
	\IH_i(\M)\coloneq \left(\sum_{F\neq \varnothing,\{i\}} \ \K{F}{\M}\right)^\perp,
	\]
	where the sum is over all nonempty proper flats $F$ of $\M$ different from $\{i\}$. 
\end{definition}

Since we are assuming that $\hyperlink{cdo}{\CD_\o(\M)}$ holds, we have
$\IH_i(\M) = \IH_\o(\M) \oplus \K{\{i\}}{\M}$.

\begin{lemma}\label{lemma:IHi stalks}
	The stalks of $\IH_i(\M)$ are given by
	\begin{enumerate}[(a)]\itemsep 5pt
		\item $\IH_i(\M)_\varnothing = \uIH_i(\M)$, as a subset of $\CH(\M)_\varnothing = \uCH(\M)$, 
		\item $\IH_i(\M)_{\{i\}} \cong \uIH(\M_{\{i\}})$, and
		\item $\IH_i(\M)_F \cong \IH_\o(\M)_F$ for any flat $F \ne \varnothing, \{i\}$.
	\end{enumerate}
Furthermore, the isomorphism (b) comes from the equality
\[y_i\IH(\M_{\{i\}}) = \psi_{\{i\}}(\IH_\o(\M_{\{i\}})).\]
\end{lemma}

\begin{proof}
For the first statement, Corollary \ref{underline it}, Lemma \ref{push uIH} (1), and $\hyperlink{ucd}{\uCD(\M)}$ give
\begin{align*}
\IH_i(\M)_\varnothing 	& = \varphi^\varnothing\!\left(\IH_i(\M)\right) \\
& = \varphi^\varnothing\!\left(\IH_\circ(\M)\right) + \varphi^\varnothing\K{\{i\}}{\M}\\
& = \uIH(\M) + \uK{\{i\}}{\M}\\
& = \uIH_i(\M).
\end{align*}	

Next, since multiplication by $y_i$ is $\psi_{\{i\}}\varphi_{\{i\}}$, we have
\begin{align*}
y_i \IH_i(\M) & = \psi_{\{i\}}\varphi_{\{i\}}\! \left(\IH_\o(\M)\right) \oplus \psi_{\{i\}}\varphi_{\{i\}}(\K{\{i\}}{\M})\\
&= \psi_{\{i\}}\left(\IH(\M_{\{i\}})\right) \oplus \psi_{\{i\}} \K{\varnothing}{\M_{\{i\}}}  \\
&=\psi_{\{i\}}(\IH_\o(\M_{\{i\}}))\\
&\cong \IH_\o(\M_{\{i\}})[-1],
\end{align*}
where we have used Lemma \ref{mult by y} (2) for the first summand and Lemma \ref{lem_imageofIH} (2) for the second summand.  
Thus, by Lemma \ref{everything is empty}, we have
\[
\IH_i(\M)_{\{i\}}\cong (y_i\IH_i(\M)[1])_\varnothing \cong \IH_\o(\M_{\{i\}})_\varnothing \cong \uIH(\M_{\{i\}}).
\]

Finally, for any flat $F \ne \varnothing,\{i\}$, multiplication by $y_F$ annihilates the image of $\psi^{\{i\}}$, so we have the isomorphism (c).
\end{proof}

\begin{example}\label{ex:deletion2}
We continue Examples~\ref{ex:uniform1}, \ref{ex:uniform2}, \ref{ex:semismall U34}, \ref{ex:Poincare U34}, and \ref{ex:deletion1} for the rank $3$ uniform matroid $\M$ on $\{1,2,3,4\}$.
In this case, $\IH_4(\M)$ is the orthogonal complement in $\CH(\M)$, with respect to its Poincar\'e pairing, of the graded subspace
\[
\K{1}{\M} \oplus \K{2}{\M} \oplus \K{3}{\M}=\mathbb{Q}x_1 \oplus \mathbb{Q}x_2 \oplus \mathbb{Q}x_3 \oplus \mathbb{Q}x_1y_1 \oplus \mathbb{Q}x_2y_2 \oplus \mathbb{Q}x_3y_3.
\]
Since $\K{1}{\M} \oplus \K{2}{\M} \oplus \K{3}{\M}$ is preserved under the multiplication by $y_1,y_2,y_3,y_4,x_\varnothing$ and $x_4$, the same is true for its orthogonal complement $\IH_4(\M)$. The degree $2$ component of $\IH_4(\M)$ has a basis $y_1y_2,y_1y_3,y_1y_4,y_2y_3,y_2y_4,y_3y_4,x_\varnothing^2,x_4y_4$, and the degree $1$ component of $\IH_4(\M)$ admits the decomposition
\[
\IH^1_4(\M)=\IH^1(\M) \oplus \mathbb{Q} x_\varnothing \oplus \mathbb{Q} x_4.
\]
The restriction of the Poincar\'e pairing of $\CH(\M)$ to $\IH_4(\M)$ is non-degenerate, because the same is true for its restriction to the subspace $\K{1}{\M} \oplus \K{2}{\M} \oplus \K{3}{\M}$.

By Lemma~\ref{CH-to-uCH}, the kernel of the pullback map $\varphi^{\varnothing}\colon \CH(\M)\to \uCH(\M)$ is the ideal $\langle y_i\rangle_{1\leq i\leq 4}$. Thus, 
\[
\varphi^{\varnothing}(\K{1}{\M} \oplus \K{2}{\M} \oplus \K{3}{\M})=\mathbb{Q}x_1 \oplus \mathbb{Q}x_2 \oplus \mathbb{Q}x_3=\uK{1}{\M} \oplus \uK{2}{\M} \oplus \uK{3}{\M}\subseteq \uCH(\M)
\]
which is consistent with Example \ref{ex:deletion1}. From this, one can further deduce that $\varphi^\varnothing(\IH_4(\M))=\uIH_4(\M)$, which is Lemma \ref{lemma:IHi stalks} (a). In fact, by the surjectivity of $\varphi^\varnothing$, it suffices to show that $\varphi^\varnothing(\IH_4(\M))\subseteq\uIH_4(\M)$. For any $\alpha\in \IH_4(\M)$ and $\beta\in \K{1}{\M} \oplus \K{2}{\M} \oplus \K{3}{\M}$, by Lemma \ref{adjoint}, Proposition \ref{Prop_xmult} (2), and the fact that $\K{1}{\M} \oplus \K{2}{\M} \oplus \K{3}{\M}$ is closed under multiplication by $x_\varnothing$, we have
\[
\udeg_\M(\varphi^\varnothing(\alpha) \varphi^\varnothing(\beta))=\deg_\M(\alpha \psi^\varnothing(\varphi^\varnothing(\beta)))=\deg_\M(\alpha(x_\varnothing \beta))=0. 
\]
Hence, $\varphi^\varnothing(\IH_4(\M))\subseteq\uIH_4(\M)$.
\end{example}

Part (a) of Lemma~\ref{lemma:IHi stalks} shows that we can use the module $\IH_i(\M)$ to study $\uIH_i(\M)$.  However, when we pull back by $\theta_i$, we can only take the stalk of the $\H(\Mi)$-module $\theta_i^*\IH_i(\M)$ at flats of the matroid $\M\setminus i$.  The stalk at $\varnothing \in \cL(\M\setminus i)$ will be too large for what we want, because it includes a contribution from the stalk $\IH_i(\M)_{\{i\} \in \cL(\M)}$, by Proposition \ref{prop:semi-small stalk}.  To get around this problem,  we consider an $\H_\o(\Mi)$-submodule $\IH_i'(\M)\subseteq \theta_i^*\IH_i(\M)$ defined as follows.

Let $\mathscr{S}_i$ be the collection of subsets of $E\setminus i$ defined in Section \ref{sec_deletion}. Let \[
R \coloneq \CHi \oplus\bigoplus_{F\in\mathscr{S}_i\setminus \{\varnothing\}}x_{F\cup i}\CHi
\and
P \coloneq x_{\{i\}} \CH_{(i)}\subseteq \CH(\M).
\]
By Theorem \ref{TheoremDecomposition}, we have an orthogonal decomposition of $\CH(\Mi)$-modules
\[
\CH(\M) = R \oplus P.
\]
Then we define our $\H_\o(\Mi)$-submodule by
\[\IH_i'(\M) \coloneq \IH_i(\M)\cap R.\]
See Remark \ref{rmk:IHi geometry} below for the geometric motivation behind this definition.

\begin{example}
When $\M=U_{3, 4}$ is the uniform matroid and $i=4$, by Example \ref{ex:deletion2}, 
\[
P=x_4\CH_{(4)}=(\Q x_{4}\oplus \Q x_4 x_{14}),
\]
and $\IH_4(\M)$ is the orthogonal complement of 
\[
\K{1}{\M} \oplus \K{2}{\M} \oplus \K{3}{\M}=\mathbb{Q}x_1 \oplus \mathbb{Q}x_2 \oplus \mathbb{Q}x_3 \oplus \mathbb{Q}x_1y_1 \oplus \mathbb{Q}x_2y_2 \oplus \mathbb{Q}x_3y_3.
\]
Since $x_jx_4x_{14}=x_jy_jx_4=0$ for any $j=1,2,3$, $P$ is orthogonal to $\K{1}{\M} \oplus \K{2}{\M} \oplus \K{3}{\M}$, and hence is contained in $\IH_4(\M)$. Therefore, $\IH_4'(\M)$ is the orthogonal complement of $P$ in $\IH_4(\M)$. In general, $P$ may not be contained in $\IH_i(\M)$. The intersections $\IH_i(\M)\cap P$ and $\K{F}{\M}\cap P$ will be discussed in Lemma \ref{lemma:P}. 
\end{example}

The next two propositions give the properties of this module that we need to deduce $\hyperlink{uhli}{\uHL_i(\M)}$ and $\hyperlink{uhri}{\uHR_i(\M)}$.  The first is analogous to Proposition \ref{prop:pullback perverse} and has a similar proof.

\begin{proposition}\label{prop:IH'_i is pure}
	$\IH'_i(\M)$ is a pure $\H_\o(\Mi)$-module, and 
	it is $\o$-perverse when considered as a complex placed in degree zero. 
\end{proposition}

Our second proposition describes the stalk of the $\H(\Mi)$-module $\IH'_i(\M)$ at the empty flat $\varnothing \in \cL(\Mi)$.  This is contained in the stalk of $\theta_i^*\IH_i(\M)$, which 
by Lemma \ref{deletion filters} is 
\[\left(\theta_i^*\IH_i(\M)\right)_{\varnothing} = \frac{\theta_i^*\IH_i(\M)}{\Upsilon_{>\varnothing}^{\Mi}\cdot \theta_i^*\IH_i(\M)} = \frac{\IH_i(\M)}{\Upsilon_\Sigma^\M \cdot \IH_i(\M)},\]
where \[\Sigma = \cL(\M) \,\setminus\, \{\varnothing, \{i\}\} = \delta^{-1}(\cL(\Mi)\setminus\{\varnothing\}).\]
In particular, we have a natural quotient map from $\left(\theta_i^*\IH_i(\M)\right)_{\varnothing}$ to 
\[\IH_i(\M)_\varnothing = \frac{\IH_i(\M)}{\Upsilon^\M_{>\varnothing}\cdot\IH_i(\M)}.\]

\begin{proposition}\label{prop:IHiprime stalk}
	 The composition
	\begin{equation}\label{eq:strange stalk}
	\IH'_i(\M)_{\varnothing} \to \left(\theta_i^*\IH_i(\M)\right)_{\varnothing} \to \IH_i(\M)_{\varnothing} \cong \uIH_i(\M)
	\end{equation}
	is an isomorphism of $\uH(\Mi)$-modules, where the module structure on the target is via the homomorphism $\underline{\theta}_i\colon \uH(\Mi) \to \uH_i(\M)$.
\end{proposition}

The proofs of these two propositions will occupy Sections \ref{sec:IHiprime purity}, \ref{sec:IHiprime perversity}, and \ref{sec:IHiprime stalk} below.  But first we assume that they hold and show that they imply $\hyperlink{uhli}{\uHL_i(\M)}$.  

\begin{corollary}\label{uIH_i in terms of uIH_1}
    When considered as an $\uH(\Mi)$-module, $\uIH_i(\M)$ is isomorphic to a direct sum of copies of modules of the form 
	\[
	\uIH\big(\Mip^F\big)[-(\crk F)/2]
	\]
	for various nonempty flats $F\in\cL(\Mi)$ of even corank.  
\end{corollary}

\begin{proof}
    Proposition \ref{prop:IH'_i is pure} and Theorem \ref{thm:perverse complexes} imply that $\IH'_i(\M)$ is isomorphic as an $\H_\o(\Mi)$-module to a direct sum of modules of the form 
    \[\IH_\o(\Mip^F)[-(\crk F)/2],\]
    where $F \in \cL(\Mi)$ is a nonempty flat with even corank.  Taking stalks at $\varnothing \in \cL(\Mi)$ and using Proposition \ref{prop:IHiprime stalk} and Corollary \ref{underline it} gives the result. 
\end{proof}

\begin{corollary}\label{uHL_i}
	The statement  $\hyperlink{uhli}{\uHL_i(\M)}$ holds.
\end{corollary}

\begin{proof}
	This follows from Corollary \ref{uIH_i in terms of uIH_1} and $\hyperlink{uhl}{\uHL((\Mi)^F)}$ for all nonempty flats $F\in \cL(\Mi)$.
\end{proof}

In order to prove $\hyperlink{uhri}{\uHR_i(\M)}$ using these results, we need to make a careful comparison of intersection pairings in the decomposition provided by Corollary \ref{uIH_i in terms of uIH_1}.  We postpone this until Section \ref{sec:uHRi}.

\begin{remark}\label{rmk:IHi geometry}
	Let us explain the geometry behind the definition of $\IH'_i(\M)$ when $\M$ is realizable.  Following the notation of Section \ref{sec:realizable}, we have the arrangement Schubert variety $Y$ corresponding to $\M$ and its blow-up $Y_\o$ at the point stratum
	corresponding to the flat $\varnothing$ of $\M$.  Recall from Section~\ref{CD geometry} that the exceptional divisor $\underline{Y}\subseteq Y_\o$ has intersection cohomology $\uIH(\M)$. 
	Let $Y_i$ be the blow-up of $Y_\o$ along the proper transform of $\overline{U^{\{i\}}}$, the closure of the stratum indexed by $\{i\}$, and let $\underline{Y}_i \subseteq Y_i$ be the inverse image of $\underline{Y}$.  It is the blow-up of $\underline{Y}$ along $\underline{Y}\cap \overline{U^{\{i\}}}$, and its intersection cohomology is $\uIH_i(\M)$.
	
	As explained in Remark \ref{rmk:deletion geometry}, the arrangement Schubert variety corresponding to $\M \setminus i$ is 
	the image $Y'$ of $Y$ under the projection $(\mathbb{P}^1)^E\to (\mathbb{P}^1)^{E\setminus i}$.  Let $Y'_\o$ be the blow-up of $Y'$ at the point stratum.  The projection $Y \to Y'$ does not lift to a map $Y_\o \to Y'_\o$, but it does lift to a map $Y_i \to Y'_\o$.  The preimage of the exceptional divisor $\underline{Y}'$ of $Y'_\o$ under this map has two components: $\underline{Y}_i$ and the exceptional divisor $D$ of $Y_i \to Y_\o$.  Taking the stalk of the $\H(\M\setminus i)$-module $\theta_i^*\IH_i(\M)$ at $\varnothing\in\cL(\M\setminus i)$ gives the cohomology of the restriction of the IC sheaf of $Y_i$ to the union of both components.  Restricting to the component $\underline{Y}_i$ gives $\uIH_i(\M)$, but there is also a contribution from $D$.  
	Proposition~\ref{prop:IHiprime stalk} says that passing to the submodule $\IH'_i(\M)$ allows us to get only the part of this stalk that we want.
	
	To motivate the appearance of the decomposition $\CH(\M) = R \oplus P$ in our definition of $\IH'_i(\M)$, consider the map $\pi\colon X \to X'$ of augmented wonderful varieties obtained by blowing up the proper transforms of all remaining strata of $Y_i$ and $Y'_\o$.  This map is semi-small, and it is birational away from the union of the boundary divisors $D_{F\cup i}$, $F \in \sS_i$.  Each of these divisors is mapped by $\pi$ to a codimension two subvariety of $X'$ with $\mathbb{P}^1$ fibers.  Then $\pi$ induces a decomposition
	\[\pi_*\underline{\Q}_X \cong \underline{\Q}_{X'} \oplus \bigoplus_{F\in \sS_i} \IC^\bullet(\pi(D_{F\cup i}))[-2],\]
	which is canonical since $\pi$ is semi-small. Taking cohomology gives the decomposition of Theorem~\ref{TheoremDecomposition}.   The summand $P$ is the contribution of the boundary divisor $D_{\{i\}} \subseteq X$, which is birational to $D$, and the orthogonal summand $R$ collects the remaining summands.
	The fact that the embedding of $\IH(\M)$ into $\CH(\M)$ is compatible with this splitting is a special feature of our canonical decomposition, and it is the key to the proof of Proposition \ref{prop:IHiprime stalk}.
\end{remark}

\subsection{Proof of Proposition \ref{prop:IH'_i is pure} part I: purity}\label{sec:IHiprime purity}
The orthogonal complement of $R$ is \[P = x_{\{i\}}\CH_{(i)} = \psi^{\{i\}}\!\left(\uCH(\M_{\{i\}})\otimes \theta_i^{\M^{\{i\}}}\!\CH(\M^\varnothing)\right).\]
Since $\M^\varnothing = \M^{\{i\}} \setminus i$ is the matroid on the empty set, $\theta_i^{\M^{\{i\}}}\!\CH(\M^\varnothing)$ is just the degree zero part of $\CH(\M^{\{i\}})$.  
This means that $P$ is the image of the injective map
\[\sigma\colon \uCH(\M_{\{i\}}) \to \CH(\M), \;\; a \mapsto \psi^{\{i\}}(a \otimes 1).\]
Applying this map to the canonical decomposition $\hyperlink{ucd}{\uCD(\M_{\{i\}})}$, we see that $P$ is the direct sum of

\begin{enumerate}[(1)]\itemsep 5pt
	\item $\sigma(\uIH(\M_{\{i\}}))$, and
	\item $\sigma\big(\uK{F\setminus i}{\M_{\{i\}}}\big)$ for each flat $F > \{i\}$ in $\cL(\M)$.
\end{enumerate}
Our next result says that these terms are compatible with the decomposition of $\CH(\M)$ into $\IH_i(\M)$ and its orthogonal summands.  

%
%

\begin{lemma}\label{lemma:P}
	We have 
	\begin{enumerate}[(1)]\itemsep 5pt
		\item $\sigma(\uIH(\M_{\{i\}})) = \IH_i(\M) \cap P$, and
		\item $\sigma\big(\uK{F\setminus i}{\M_{\{i\}}}\big)
		= \K{F}{\M}\cap P$ for all flats $F > \{i\}$ in $\cL(\M)$.
	\end{enumerate}
		
\end{lemma}

\begin{proof}
	For the first statement, take any $a \in \uCH(\M_{\{i\}})$.  Then $\sigma(a) = \psi^{\{i\}}(a \otimes 1)$ is in $\IH_i(\M)$ if and only if it is orthogonal to $\K{F}{\M}$
	for all flats $F \in \cL(\M)$ other than $\varnothing$, $\{i\}$.  By Lemma \ref{adjoint}, this is true if and only if $a \otimes 1$ is orthogonal to $\varphi^{\{i\}}\K{F}{\M}$.  By Lemma \ref{lemma_incomparable}, it is enough to check this when $F > \{i\}$.
	In that case, we have \[\varphi^{\{i\}}\K{F}{\M}= \uK{F\setminus i}{\M_{\{i\}}} \otimes  \CH(\M^{\{i\}})\]
	by Lemma \ref{push uIH} (1).  If $\{i\} < F$ then $G = F \setminus \{i\}$ is a nonempty flat of 
	$\cL(\M_{\{i\}})$, and all such flats occur this way, so $\sigma(a)$ is in $\IH_i(\M)$ if and only if $a$ is orthogonal to $\uK{G}{\M_{\{i\}}}$ for all nonempty flats $G \in \cL(\M_{\{i\}})$.  This happens exactly when $a \in \uIH(\M_{\{i\}})$.
	
	To see the second statement, use Lemma \ref{push uIH} (1) again to get
\[\sigma\big(\uK{F\setminus i}{\M_{\{i\}}}\big) =  \psi^{\{i\}}\big(\uK{F\setminus i}{\M_{\{i\}}} \otimes \Q\big) \subseteq \K{F}{\M}.\]
	This gives containment in one direction.  The other direction follows from the fact that $P$ is the sum of all the terms of type (1) and (2).
\end{proof}

Since $R$ is the perpendicular space to $P$ and the terms of the form (2) are all orthogonal to $\IH_i(\M)$, we see that 
\[\IH_i(\M)\cap R \; \mbox{is the perpendicular space to $\IH_i(\M)\cap P$ inside $\IH_i(\M)$}.\]

\begin{lemma}\label{lem:nondegen pairings}
	The Poincar\'e pairing on $\CH(\M)$ restricts to non-degenerate pairings on $\IH_i(\M)\cap P$ and on 
	$\IH_i(\M)\cap R = \IH_i'(\M)$.
\end{lemma}

\begin{proof}
	For $a, b \in \uIH(\M_{\{i\}})$, Lemma \ref{we need this} gives
	\begin{align*}
	\big\langle \sigma(a), \sigma(b) \big\rangle_{\CH(\M)} & = 
	\deg_\M(\psi^{\{i\}}(a \otimes 1) \cdot \psi^{\{i\}}(b \otimes 1))\\
	& = - \udeg_{\M_{\{i\}}}\otimes \deg_{\M^{\{i\}}}\big( (ab \otimes 1)( 1\otimes  \a_{\M^{\{i\}}} + \ub_{\M_{\{i\}}}\otimes 1)\big)\\
	& = -\udeg_{\M_{\{i\}}}(ab)\deg_{\M^{\{i\}}}(\a_{\M^{\{i\}}}).
	\end{align*}
	Since $\deg_{\M^{\{i\}}}(\a_{\M^{\{i\}}}) \ne 0$ and the pairing on $\uIH(\M_{\{i\}})$ is non-degenerate, the non-degeneracy of the pairing on $\IH_i(\M)\cap P$ follows.  From this we deduce that $\IH_i(\M)$ is the orthogonal direct sum of $\IH_i(\M)\cap P$ and $\IH_i(\M)\cap R$, and since the pairing on $\IH_i(\M)$ is non-degenerate, it follows that the restriction of the Poincar\'e pairing to $\IH_i(\M)\cap R$ is non-degenerate. 
\end{proof}

Because the Poincar\'e pairing also restricts to a non-degenerate pairing on $R$, we can conclude that $\IH'_i(\M)$ is an $\H_\o(\Mi)$-direct summand of $R$.  Theorem \ref{TheoremDecomposition} then implies that $R$ is a pure $\H_\o(\Mi)$-module, and so $\IH'_i(\M)$ is a pure $\H_\o(\Mi)$-module as well.

\subsection{Proof of Proposition \ref{prop:IH'_i is pure} part II: $\o$-perversity}\label{sec:IHiprime perversity}

The proof that $\IH'_i(\M)$ is a $\o$-perverse $\H_\o(\Mi)$-module follows the same basic plan as the proof of Proposition \ref{prop:pullback perverse}, using Proposition \ref{prop:semi-small stalk} to compute the stalks of $\IH'_i(\M)$ at a nonempty flat $F$ of $\Mi$ in terms of the stalks of $\IH_i(\M)$ at flats $G \in \delta^{-1}(F) \subseteq \cL(\M)$.  

%

We cannot apply Proposition \ref{prop:semi-small stalk} to $\IH'_i(\M)$ directly, because it is not closed under multiplication by $y_i$, so it is not an $\H(\M)$-module.\footnote{For instance, one can easily check that $1 \in R$ but $y_i\notin R$.} But we have shown that $\theta_i^*\IH_i(\M)$ is the direct sum of $\IH'_i(\M) = \IH_i(\M)\cap R$ and $\IH_i(\M) \cap P$, and $P$ is annihilated by all $y_j$ for $j\in E \setminus i$.  Thus for any nonempty flat $F$ of $\Mi$, we have
\[\IH'_i(\M)_F \cong (\theta_i^*\IH_i(\M))_F \cong \bigoplus_{G\in \delta^{-1}(F)} \IH_i(\M)_G[\rk_{\Mi} F - \rk_\M G]\]
by Proposition \ref{prop:semi-small stalk}.

Now the proof that $\IH'_i(\M)_F$ vanishes in degrees strictly greater than $(\crk F)/2$ follows exactly the proof of Proposition \ref{prop:pullback perverse}, using the above equation in place of Proposition \ref{prop:semi-small stalk}, and omitting the case $G = \varnothing$, since $F$ is assumed to be nonempty.

Thus $\IH'_i(\M)$ satisfies the stalk conditions for perversity.  To see that the costalk conditions hold, we note that Lemma \ref{lem:nondegen pairings} gives an isomorphism $\IH'_i(\M)^\ast \cong \IH'_i(\M)[d]$ of graded $\H_\o(\Mi)$-modules, and so Lemma \ref{stalk-costalk} implies that the costalk conditions follow from the stalk conditions.	

\subsection{Proof of Proposition \ref{prop:IHiprime stalk}}\label{sec:IHiprime stalk}
The proposition states that the composition of two maps is an isomorphism.  The first map is an injection and the second is a surjection, thus the composition is an isomorphism if and only if
the natural map from the kernel of the second map to the cokernel of the first map is an isomorphism.

Consider the following two short exact sequences containing our maps:
\[0 \to \IH'_i(\M)_\varnothing \to (\theta^*_i\IH_i(\M))_\varnothing \to (\IH_i(\M)\cap P)_{\varnothing} \to 0\]
\[0 \to \frac{{\Upsilon^\M_{>\varnothing}} \IH_i(\M)}{\Upsilon^\M_\Sigma \IH_i(\M)} \to (\theta^*_i\IH_i(\M))_\varnothing \to \IH_i(\M)_\varnothing \to 0.\]
The first sequence is exact because $\theta^*_i\IH(\M) \cong \IH_i'(\M) \oplus (\IH_i(\M)\cap P)$ as $\H(\Mi)$-modules.
By Proposition \ref{pure} and Lemma \ref{lemma:IHi stalks}, the first term of the second sequence is 
isomorphic to 
\[\frac{\Upsilon^\M_{\ge\{i\}} \IH_i(\M)}{\Upsilon^\M_{>\{i\}} \IH_i(\M)} = \frac{y_i \IH_i(\M)}{\Upsilon^\M_{>\{i\}} \IH_i(\M)} \cong \IH_\o(\M_{\{i\}})_\varnothing \cong \uIH(\M_{\{i\}}).\]
So Proposition \ref{prop:IHiprime stalk} is equivalent to showing that the lower row of the following diagram is an isomorphism.
\[
\begin{tikzcd}
\IH(\M_{\{i\}}) \arrow[r,"\psi_{\{i\}}"] \arrow[d, "\varphi^\varnothing_{\M_{\{i\}}}"] & \theta_i^*\IH_i(\M) \arrow[r, "\pi"] \arrow[d, "\varphi^\varnothing_{\Mi}"]               & \IH_i(\M)\cap P \arrow[d, "\cong"] \arrow[r, equal] &[-28pt] \sigma(\uIH(\M_{\{i\}})) \\
\uIH(\M_{\{i\}}) \arrow[r, "\zeta"]          & (\theta_i^*\IH_i(\M))_\varnothing \arrow[r, "\pi_\varnothing"] & (\IH_i(\M)\cap P)_\varnothing                            &                         
\end{tikzcd}\]
Here $\pi$ is the orthogonal projection onto $\IH_i(\M)\cap P$, since $\IH_i(\M)\cap P$ and $\IH_i(\M)\cap R$ are orthogonal complements inside $\IH_i(\M)$.  The map $\pi_\varnothing$ is the induced map on stalks at $\varnothing\in \cL(\Mi)$, and the third vertical map is an isomorphism because $y_jP = 0$ for $j \in E \setminus i$.  The map $\zeta$ is the unique map making the left square commute. It exists because the kernel of $\varphi^\varnothing_{\M_{\{i\}}}$ is generated by elements $y_ja$, where $a \in \IH(\M_{\{i\}})$ and $j \in E\setminus i$, and we have
	\[\varphi^\varnothing_{\Mi}(\psi_{\{i\}}(y_ja)) = \varphi^\varnothing_{\Mi}(y_j\,\psi_{\{i\}}(a))=0.\]
Since $\pi$ is an orthogonal projection, it is defined by the property that, if $c \in \theta^*_i\IH(\M)$, then 
\[\deg_{\M}(c \cdot \sigma(\underline{a})) = \deg_{\M}(\pi(c)\sigma(\underline{a}))\]
for any $\underline{a} \in \uIH(\M_{\{i\}})$.

Now take any $\underline{b} \in \uIH(\M_{\{i\}})$ and suppose that $\underline{b} = \varphi^\varnothing_{\M_{\{i\}}}\!(b)$ for $b \in \IH(\M_{\{i\}})$.  Then $\pi_\varnothing\zeta(\underline{b}) = \varphi^\varnothing_{\Mi}(\pi\psi_{\{i\}}(b))$, and so we want to show that the map sending $\underline{b}$ to $\pi(\psi_{\{i\}}(b))$ is an isomorphism.  The element $\pi(\psi_{\{i\}}(b))$ is characterized by the following equation, for every $\underline{a} \in \uIH(\M_{\{i\}})$:

\begin{align*}
\deg_{\M}\left(\pi(\psi_{\{i\}}(b))\cdot \sigma(\underline{a})\right) & = \deg_{\M}\left(\psi_{\{i\}}(b)\cdot \sigma(\underline{a})\right)\\
 & = \deg_{\M}\left(\psi_{\{i\}}(b)\cdot \psi^{\{i\}}(\underline{a}\otimes 1)\right)   && \text{definition of $\sigma$} \\ 
 & = \deg_{\M_{\{i\}}}\left(b\cdot \varphi_{\{i\}}\psi^{\{i\}}(\underline{a}\otimes 1)\right) && \text{by Lemma \ref{subscript adjoint}}  \\
 & = \deg_{\M_{\{i\}}}\left(b\cdot \psi^\varnothing_{\M_{\{i\}}}(\underline{a})\right) && \text{by Lemma \ref{lemma_diagrams} (4)}\\
 & = \udeg_{\M_{\{i\}}}\left(\underline{b}\cdot \underline{a}\right)  && \text{by Lemma \ref{adjoint}.}
\end{align*}
The fact that our map is an isomorphism follows from $\hyperlink{upd}{\uPD(\M_{\{i\}})}$.

\subsection{The Hodge--Riemann relations}\label{sec:uHRi}
To prove $\hyperlink{uhri}{\uHR_i(\M)}$, we need to understand the Poincar\'e pairing on the direct summands of $\uIH_i(\M)$ provided by Corollary \ref{uIH_i in terms of uIH_1}.  In order to do this, we first consider the Poincar\'e pairing on the summands of $\IH'_i(\M)$.  Our first result says that these summands are rigid, in the sense that their only endomorphisms as graded $\H_\o(\Mi)$-modules are multiplication by scalars.

\begin{lemma}\label{lem:IHo endomorphisms}
	Let $\M$ be a matroid on a nonempty ground set $E$.  Suppose that $\hyperlink{cd}{\CD(\M_F)}$, $\hyperlink{pd}{\PD(\M_F)}$, and $\hyperlink{ns}{\NS(\M_F)}$ hold for all proper flats $F$, and that $\hyperlink{uhl}{\uHL(\M)}$ holds.  Then an endomorphism of 
	$\IH_\o(\M)$ as a graded $\H_\o(\M)$-module that induces the zero map on the stalk $\IH_\o(\M)_E \cong \Q$ must be zero.  In particular, 	
	the only endomorphisms of $\IH_\o(\M)$ as a graded $\H_\o(\M)$-module are multiplication by scalars.
\end{lemma}

\begin{proof}
Take an endomorphism $f$ of $\IH_\o(\M)$ which induces the zero endomorphism on 
$\IH_\o(\M)_E$.  Then following the argument of Proposition \ref{prop:indecomposable}, but using Lemma \ref{mult by y}~(2) in place of Lemma \ref{mult by y}~(1), we see that $f$ vanishes on $y_F\IH_\o(\M)$ for any nonempty flat $F$, and so it induces a homomorphism
\[\uf\colon \IH_\o(\M)_{\varnothing} \to \IH_\o(\M)_{[\varnothing]}.\]
Corollary \ref{underline it} gives an isomorphism 
\[\IH_\o(\M)_\varnothing \cong \varphi^\varnothing(\IH_\o(\M)) = \uIH(\M).\]
By Lemma \ref{annihilators} and Corollary \ref{very convenient}, we have an isomorphism
\[\IH_\o(\M)_{[\varnothing]} = \langle x_\varnothing\rangle\cap \IH_\o(\M) =\psi^\varnothing(\uIH(\M))\cong \uIH(\M)[-1].\]
Using these isomorphisms, we can write $\uf$ as a map $\uIH(\M) \to \uIH(\M)[-1]$ satisfying 
$f = \psi^\varnothing \uf \varphi^\varnothing$.

Both $\varphi^\varnothing$ and $\psi^\varnothing$ are homomorphisms of $\H_\o(\M)$-modules, where $x_\varnothing$ acts on $\uIH(\M)$ as multiplication by $\varphi^\varnothing(x_\varnothing) = -\ub$.  
Since $\psi^\varnothing$ is injective and $\varphi^\varnothing(\IH_\o(\M)) = \uIH(\M)$ by Corollary \ref{underline it}, we see that $\uf$ commutes with multiplication by $\ub$, or in other words it is a homomorphism of $\uH(\M)$-modules.

Take an element $a \in \uIH^k(\M)$, and 
suppose that $a$ is primitive, so $\ub^{d-2k}a = 0$.  This gives
\[0 = \uf(\ub^{d-2k}a) = \ub^{d-2k} \cdot \uf(a).\]
But $\uf(a) \in \IH^{k-1}(\M)$, and so $\hyperlink{uhl}{\uHL(\M)}$ implies that
\[(\ub^{d-2k+1}\cdot\;) \colon \uIH^{k-1}(\M) \to \uIH^{d-k}(\M)\]
is an isomorphism, which gives $\uf(a) = 0$.  
Furthermore, we have $\uf(\beta^\ell a) = \beta^\ell \uf(a)=0$ for any $\ell \ge 0$.

By $\hyperlink{uhl}{\uHL(\M)}$, $\uIH^k(\M)$ is spanned by $\ub^{k-i}$ times primitive classes of degree $i$ for all $0\leq i\leq k$.  Since we have shown that $\uf$ vanishes on all such classes,
we can conclude that $\uf = 0$, and therefore that $f=0$, as well.
\end{proof}

Now we can proceed with an analysis analogous to the one at the beginning of Section \ref{sec:first HR}.
Let $F$ be a nonempty flat of $\Mi$ of even corank, and suppose we have an inclusion 
\begin{equation*}\label{eq_f}
f \colon \IH_\o\bigl(\Mip^F\bigr)[-(\crk F)/2]\hookrightarrow \IH'_i(\M)
\end{equation*}
of $\H_\o(\M)$-modules.
We have two pairings on $\IH_\o(\Mip^F)$ that are \textit{a priori} different:  the one induced by the inclusion of $\IH_\o(\Mip^F)$ into $\CH(\Mip^F)$,
and the one induced by the inclusion of $\IH_\o(\Mip^F)[-(\crk F)/2]$ into $\IH'_i(\M)$.  

\begin{lemma}\label{pair and compare again}
These two pairings on $\IH_\o(\Mip^F)$ are related by a constant factor $c\in\Q$ with $(-1)^{\frac{\crk F}{2}}c > 0$.  
\end{lemma}

\begin{proof} This proof is essentially the same as the proof of Lemma \ref{pair and compare}.
Both pairings are compatible with the $\H_\o(\Mi)$-module structure in the sense that $\langle \eta\xi,\sigma\rangle = \langle \xi,\eta\sigma\rangle$ for any $\eta\in\H_\o(\Mi)$
and $\xi,\sigma\in \IH_\o(\Mip^F)$. Thus both pairings are given by isomorphisms $\IH_\o(\Mip^F)^*\cong \IH_\o(\Mip^F)[\rk F]$ of graded $\H_\o(\Mi)$-modules.
By Lemma \ref{lem:IHo endomorphisms}, the $\H_\o(\Mi)$-module $\IH_\o(\Mip^F)$ has only scalar endomorphisms, so any two such isomorphisms
must be related by a scalar factor $c\in\Q$.  

To compute the sign of $c$, we pair the class $1\in\IH_\o(\Mip^F)$ with the class $y_F\in\IH_\o(\Mip^F)$.
Inside of $\CH(\Mip^F)$, they pair to $1$.  Since $\theta_i^{\Mo}(y_F)=y_{\bar{F}}$, by Proposition \ref{lemma_degy},  Proposition~\ref{Prop_ymult}, Lemma \ref{lemma:IHi stalks} (c), and Lemma \ref{mult by y} (2), the pairing of their images in $\IH'_i(\M)$, or equivalently in $\CH(\M)$, is equal to the Poincar\'e pairing of 
$\varphi_{\bar F}(f(1))$ with itself inside of $\IH^{\frac{\crk F}{2}}(\M_{\bar F})$.
The class $\varphi_{\bar F}(f(1))$ is annihilated by $y_j$ for all $j\in E\setminus \bar F$, so it is primitive, and therefore
the sign of its Poincar\'e pairing with itself is equal to $(-1)^{\frac{\crk F}{2}}$ by $\hyperlink{hr}{\HR(\M_{\bar F})}$.
\end{proof}

Taking stalks at the empty flat $\varnothing\in \cL(\Mi)$ and using Proposition \ref{prop:IHiprime stalk}, the inclusion $f$ induces an inclusion 
\[
\underline{f} \colon \uIH(\Mip^F)[-(\crk F)/2]\hookrightarrow \IH'_i(\M)_\varnothing \cong \uIH_i(\M)
\]
of $\uH(\Mi)$-modules.  All of the summands of $\uIH_i(\M)$ provided by Corollary \ref{uIH_i in terms of uIH_1} are images of maps of this form.

There are two pairings on $\uIH(\Mip^F)$ that are {\em a priori} different:  the one induced by the inclusion of 
$\uIH(\Mip^F)$ into $\uCH(\Mip^F)$,
and the one induced by the above inclusion $\underline{f}$.

\begin{lemma}\label{keep pairing and keep comparing}
These two pairings on $\uIH(\Mip^F)$ are related by the same constant factor $c\in\Q$ as in Lemma~\ref{pair and compare again} with $(-1)^{\frac{\crk F}{2}}c > 0$.
\end{lemma}

\begin{proof}
We need to compare the Poincar\'e pairings in the Chow rings and the augmented Chow rings. 
Given two classes $\eta, \xi\in \IH_\o(\Mip^F)$, we denote their images in $\uIH(\Mip^F)$ by 
$\underline{\eta}, \underline{\xi}$.
By Propositions~\ref{lemma_xdegree} and~\ref{Prop_xmult}, we have
\[
\langle \underline{\eta}, \underline{\xi}\rangle_{\uCH((\Mi)^F)}
= \langle \varphi^\varnothing\eta, \varphi^\varnothing{\xi}\rangle_{\uCH((\Mi)^F)}
= \langle \eta, \psi^\varnothing\varphi^\varnothing\xi\rangle_{\CH((\Mi)^F)}
= \langle \eta, x_\varnothing\xi\rangle_{\CH((\Mi)^F)}
\]
and
\[
\langle \underline{f}(\underline{\eta}), \underline{f}(\underline{\xi})\rangle_{\uCH(\M)}
= \langle \varphi^\varnothing f(\eta), \varphi^\varnothing f(\xi))\rangle_{\uCH(\M)}
= \langle f(\eta), \psi^\varnothing\varphi^\varnothing f(\xi)\rangle_{\CH(\M)}
= \langle f(\eta), x_\varnothing f(\xi)\rangle_{\CH(\M)}.
\]
We further have
\[
\langle f(\eta), x_\varnothing f(\xi)\rangle_{\CH(\M)}
= \langle f(\eta), (\theta_i^{\Mo}(x_\varnothing)-x_{\{i\}})f(\xi)\rangle_{\CH(\M)}
=\langle f(\eta),  f(x_\varnothing\xi)\rangle_{\CH(\M)},
\]
where the last equality follows from the next lemma and $f$ being an $\H_\o(\Mi)$-module homomorphism.
Thus, the two pairings are related by the same constant factor $c$ as in Lemma \ref{pair and compare again}.
\end{proof}

\begin{lemma}
For any $\mu, \nu\in R$, we have
$
\langle \mu, x_{\{i\}} \nu \rangle_{\CH(\M)}=0.
$
\end{lemma}
\begin{proof}
By \cite[Lemma 3.9]{BHMPW}, for any $F\in \mathscr{S}_i\setminus \{\varnothing\}$, we have
\[
x_{\{i\}} x_{F\cup i} \CHi\subseteq x_{\{i\}}\CHi.
\]
Recall that we defined $R$ to be the direct sum of $\CHi$ and $x_{F\cup i}\CHi$ for all $F\in \mathscr{S}_i\setminus \{\varnothing\}$.  It therefore follows from the above equation
that $x_{\{i\}} R=x_{\{i\}}\CHi=P$, which is orthogonal to $R$ with respect to the Poincar\'e pairing of $\CH(\M)$. Thus, the lemma follows. 
\end{proof}

\begin{corollary}\label{uHR_i}
The statement $\hyperlink{uhri}{\uHR_i(\M)}$ holds.
\end{corollary}

\begin{proof}
This follows from Corollary \ref{uIH_i in terms of uIH_1}, Lemma \ref{keep pairing and keep comparing}, and $\hyperlink{uhr}{\uHR(\Mip^F)}$ for  $\varnothing\neq F\in \cL(\Mi)$.
\end{proof}

\section{Deformation arguments}\label{sec:deform}
This section is devoted to arguments that establish hard Lefschetz or Hodge--Riemann properties by considering families of Lefschetz arguments.
We continue to fix an element $i\in E$ satisfying the three assumptions stated at the beginning of Section \ref{sec:underlined deletion induction}.

\subsection{Establishing $\HR^{<\frac d 2}(\M)$}
\begin{proposition}\label{delnodel}
We have
\[
\hyperlink{hl}{\HL(\M)},\; \hyperlink{hli}{\HL_i(\M)}, \;\;\text{and}\;\; \hyperlink{hri}{\HR_i^{< \frac d 2}(\M)}\enspace \Longrightarrow\enspace \hyperlink{hr}{\HR^{<\frac d 2}(\M)}.
\]
\end{proposition}

\begin{proof}
Given $y = \sum_{j\in E} c_{j} y_{j}$ with every $c_j>0$, to show that $\IH(\M)$ satisfies the Hodge--Riemann relations with respect to multiplication by $y$ in degrees less than $d/2$, we consider 
\[
y_t\coloneq t \cdot c_i y_i+\sum_{j\in E, j\neq i}c_jy_j.
\]
By $\hyperlink{hl}{\HL(\M)}$ and $\hyperlink{hli}{\HL_i(\M)}$,  $\IH(\M)$ satisfies the hard Lefschetz theorem with respect to multiplication by $y_t$ for any $t\geq 0$. Therefore, for any $k< d/2$, the Hodge--Riemann form on $\IH^{k}(\M)$ associated with any $y_t$ with $t\geq 0$ has the same signature.  Given the hard Lefschetz theorem, the Hodge--Riemann relations are conditions on the signature of the Hodge--Riemann forms \cite[Proposition 7.6]{AHK}, thus the fact that 
$\IH(\M)$ satisfies the Hodge--Riemann relations with respect to multiplication by $y_0$ implies that it satisfies the Hodge--Riemann relations 
for any $y_t$ with $t\geq 0$. 
\end{proof}

\subsection{Establishing $\uHR(\M)$}
The purpose of this section is to prove Proposition \ref{alt}, which gives us a way to pass from $\hyperlink{uhri}{\uHR_i(\M)}$ to $\hyperlink{uhr}{\uHR(\M)}$.
If $\{i\}$ is not a flat,
then $i$ has a parallel element and the statements $\hyperlink{uhri}{\uHR_i^{\leq k}(\M)}$ and $\hyperlink{uhr}{\uHR^{\leq k}(\M)}$ are the same. So, without loss of generality, 
we may assume that 
$\{i\}$ is a flat.  To simplify the notation we will denote this flat without braces in this section, so we write $x_i$ instead of $x_{\{i\}}$, $\upsi^i$ instead of $\upsi^{\{i\}}$, etc.

For any  $t\geq 0$, consider the degree one linear operator $\uL_t$ on $\uIH_i(\M)$
given by multiplication by $\ub - t x_{i}$.   We will assume $\hyperlink{ucd}{\uCD(\M)}$ throughout this section, so that we have $\uIH_i(\M) = \uIH(\M) \oplus \uK{i}{\M}$.  Thus, if $k < (d-1)/2$ we have decompositions
\[
\uIH^k_i(\M) = \uIH^k(\M) \oplus \upsi^i_\Mo\big( \uJ^{k-1}(\M_i)\big)
\;\;\text{and}\;\; \uIH_i^{d-k-1}(\M) = \uIH^{d-k-1}(\M) \oplus \upsi^i_\Mo\big( \uJ^{d-k-2}(\M_i)\big).
\]
Here we have used the fact that $i$ has rank one (so $\uCH(\M^i) = \Q$) to suppress the second tensor factor in the source of $\upsi^i$.
\begin{lemma}\label{orthogonal}
The map 
\[
\uL_t^{d-2k-1} \colon \uIH_i^k(\M)\to\uIH_i^{d-k-1}(\M)
\]
is block diagonal with respect to the above direct sum decompositions.
\end{lemma}

\begin{proof}
Since $\ub=\sum_{i\notin G\neq\varnothing}x_G$, we have
$\ub x_i=0$.  Since the image of $\upsi^i_\Mo$ is equal to the ideal of $\uCH(\M)$ generated by $x_i$, multiplication by $\ub$ annihilates the image of $\upsi^i_\Mo$. 
Thus, we have
\begin{align*}
\uL_t^{d-2k-1} \upsi^i_\Mo\big( \uJ^{k-1}(\M_i)\big) &\subseteq 
x_i^{d-2k-1}  \upsi^i_\Mo\big( \uJ^{k-1}(\M_i)\big) && \\
&= \upsi^i_\Mo\Big(\uvarphi_\Mo^i(x_i^{d-2k-1}) \uJ^{k-1}(\M_i)\Big) && \text{by Proposition \ref{lemma_uab} (5)}\\
&= \upsi^i_\Mo\Big(\b_{\M_i}^{d-2k-1}  \uJ^{k-1}(\M_i)\Big) && \text{since $\ua_{\M^i} = 0$}\\
&= \upsi^i_\Mo\Big(\b_{\M_i}^{d-2k-1} \uIH^{k-1}(\M_i)\Big) && \text{by the definition of $\uJ$ }\\
&= \upsi^i_\Mo\big(\uJ^{d-k-2}(\M_i)\big)  && \text{by the definition of $\uJ$}.
\end{align*}
Thus $\uL_t^{d-2k-1}$ maps $\upsi^i_\Mo\!\left( \uJ^{k-1}(\M_i)\right)$ to $\upsi^i_\Mo\!\left( \uJ^{d-k-2}(\M_i)\right)$.

On the other hand, by  the same argument above, we have
\begin{align*}
x_i^{d-2k-1}\uIH^k(\M) \cdot \upsi^i_\Mo\left( \uJ^{k-1}(\M_i)\right) &= \uIH^k(\M)\cdot\upsi^i_\Mo\Big(\uvarphi_\Mo^i\big(x_i^{d-2k-1}\big)\uJ^{k-1}(\M_i)\Big) \\
&=   \uIH^k(\M) \cdot \upsi^i_\Mo\left(\uJ^{d-k-2}(\M_i)\right). 
\end{align*}
By Proposition \ref{summands}, the subspaces $\uIH^k(\M)$ and $\upsi^i_\Mo(\uJ^{d-k-2}(\M_i))$ are orthogonal in $\uCH(\M)$, and hence the above product is zero. Therefore, the subspaces $x_i^{d-2k-1}\uIH^k(\M)$ and $\upsi^i_\Mo\left( \uJ^{k-1}(\M_i)\right)$ are orthogonal in $\uIH_i(\M)$. 
Since $\uIH(\M)$ is the orthogonal complement of $\upsi^i_\Mo\!\left( \uJ(\M_i)\right)$ in $\uIH_i(\M)$, it follows that $$x_i^{d-2k-1}\uIH^k(\M) \subseteq \uIH^{d-k-1}(\M).$$ 
But, since $\ub x_i = 0$, we have $(\ub-tx_i)^{d-2k-1} = \ub^{d-2k-1} + (-tx_i)^{d-2k-1}$, and $\uIH(\M)$ is preserved by multiplication by $\ub$, so this shows that $\uL_t^{d-2k-1}$ maps $\uIH^k(\M)$ to $\uIH^{d-k-1}(\M)$.
\end{proof}

\begin{lemma}\label{HLt}
Let $k\leq (d-1)/2$ be given, and suppose that the statements $\hyperlink{uhri}{\uHR_i(\M)}$ and $\hyperlink{uhl}{\uHL^{\leq k}(\M)}$ hold. For any $0 < t \leq 1$, the map
\[
\uL_t^{d-2k-1} \colon \uIH^k(\M)\to\uIH^{d-k-1}(\M)
\]
is an isomorphism.
\end{lemma}

\begin{proof}
First note that the statement for $t=1$ holds by 
Lemma \ref{orthogonal} and $\hyperlink{uhri}{\uHR_i(\M)}$.  
For $0 < t < 1$, assume for the sake of contradiction that $0 \ne \eta \in \uIH^k(\M)$ and
\begin{equation}\label{negateHL}
\Big(\b^{d-2k-1} + (-tx_i)^{d-2k-1}\Big)\eta = 0.
\end{equation}
Multiplying this equation by $\b$ and by $x_i$ gives
\begin{equation*}\label{prim}
\b^{d-2k}\eta = 0 \;\;\text{and}\;\; x_i^{d-2k}\eta = 0.
\end{equation*}
Thus $\eta$ is a primitive class in $\uIH^k(\M)$ with respect to $\b - x_i$.  By $\hyperlink{uhri}{\uHR_i(\M)}$,
\[
(-1)^k\udeg_{\M}\Bigl(\Big(\b^{d-2k-1} + (-x_i)^{d-2k-1}\Big)\eta^2\Big) > 0.
\]
But by an application of \eqref{negateHL}, this inequality is equivalent to
\begin{align*}
0 &< (-1)^k \udeg_{\M} \Bigl(\Big(\b^{d-2k-1} + (-tx_i)^{d-2k-1} - (-tx_i)^{d-2k-1} + (-x_i)^{d-2k-1}\Big)\eta^2\Big) \\
&=(-1)^k \udeg_{\M} \Bigl(\Big(- (-tx_i)^{d-2k-1} + (-x_i)^{d-2k-1}\Big)\eta^2\Big) \\
&=(-1)^{d-k-1}\udeg_{\M} \Bigl(x_i^{d-2k-1}(-t^{d-2k-1}+1)\eta^2\Big).
\end{align*}
Since $0 < t < 1$, this inequality reduces to
\begin{equation*}\label{ineq}
(-1)^{d-k-1}\udeg_{\M}(x_i^{d-2k-1}\eta^2) > 0.
\end{equation*}
On the other hand, by Lemma \ref{push uIH} (3), we know that $\uvarphi_\Mo^i\!\left(\uIH(\M)\right)\subseteq \uIH(\M_i)$. Since $(\b_{\M})^{d-2k}\eta = 0$ and $\uvarphi_\Mo^i(\b_{\M})=\b_{\M_i}$, it follows that $(\b_{\M_i})^{d-2k}\uvarphi_\Mo^i(\eta)=0$. In other words, $\uvarphi_\Mo^i(\eta)\in  \uIH^k(\M_i)$ is a primitive class with respect to $\b_{\M_i}$.  
Thus, by Proposition \ref{lemma_uab} and Proposition \ref{Prop_uxmult}, we have
\begin{align*}
0 &\leq  (-1)^{k} \udeg_{\M_i} \Bigl(\b^{d-2k-2}\uvarphi_\Mo^i(\eta)^2\Big)\\
&= (-1)^{d-k} \udeg_{\M_i} \Bigl(\uvarphi_\Mo^i\big(x_i^{d-2k-2}\eta^2\big)\Big)\\
& = (-1)^{d-k} \udeg_{\M}\Bigl(\upsi^i_\Mo \uvarphi_\Mo^i\big(x_i^{d-2k-2}\eta^2\big)\Big)\\
&= (-1)^{d-k}\udeg_{\M}(x_i^{d-2k-1}\eta^2).
\end{align*}
Now, we have a contradiction between the above two sets of inequalities.
\end{proof}

The following proposition allows us to reduce the proof of the Hodge--Riemann relations to a signature computation, assuming that Poincar\'e duality and the hard Lefschetz theorem are already known.
We follow the notation from the start of Section \ref{sec:statements}: let $\N=\bigoplus_{k\geq 0}\N^k$ be a finite-dimensional graded $\mathbb{Q}$-vector space endowed with a bilinear form 
\[
\langle -,- \rangle \colon \N\times \N\to \Q
\]
and  a linear operator $\oL \colon \N\to \N$  of degree $1$ that satisfies 
$\langle \oL(\eta), \xi \rangle=\langle \eta, \oL(\xi) \rangle$ for all $\eta, \xi\in \N$.

\begin{proposition}\label{prop_sig}
Suppose that $(\N, \oL)$ satisfies Poincar\'e duality and the hard Lefschetz theorem of degree $d$. Then, it satisfies the Hodge--Riemann relations in degrees between $0$ and $k \le d/2$ if and only if 
\[
\sig_{\oL} \N^j- \sig_{\oL} \N^{j-1}= (-1)^j\left(\dim \N^j - \dim \N^{j-1}\right), \quad\text{for any }j\leq k,
\]
where $\sig_{\oL}$ denotes the signature of the Hodge--Riemann form associated with $\oL$. 
\end{proposition}
\begin{proof}
  This was proved in \cite[Proposition 7.6]{AHK}.  There it was assumed that $\N$ was a graded ring and $\oL$ is multiplication by an element of $\N^1$, but under our hypotheses the same argument works.
\end{proof}

\begin{proposition}\label{alt}
For any $k\leq (d-1)/2$, we have
\[
\hyperlink{uhr}{\uHR(\M_i)},\; \hyperlink{upd}{\uPD(\M)},\; \hyperlink{uhri}{\uHR_i(\M)}, \;\;\text{and}\;\; \hyperlink{uhl}{\uHL^{\leq k}(\M)}\enspace \Longrightarrow\enspace \hyperlink{uhr}{\uHR^{\leq k}(\M)}.
\]
\end{proposition}

\begin{proof}
By induction on $k$, we may assume $\hyperlink{uhr}{\uHR^{< k}(\M)}$.  
To prove $\hyperlink{uhr}{\uHR^{k}(\M)}$,
by Proposition \ref{prop_sig}, it suffices to show that 
\[
\sig_{\uL_0} \uIH^k(\M) - \sig_{\uL_0} \uIH^{k-1}(\M) = (-1)^k\left(\dim \uIH^k(\M) - \dim \uIH^{k-1}(\M)\right),
\]
where $\sig_{\uL_0}$ denotes the signature of the Hodge--Riemann form associated with $\uL_0$.

By Lemma \ref{HLt} and $\hyperlink{upd}{\uPD(\M)}$, the Hodge--Riemann form associated with $\uL_t$ is non-degenerate for all $0< t\leq 1$, and by $\hyperlink{uhl}{\uHL^{\leq k}(\M)}$, the Hodge--Riemann form is also non-degenerate when $t=0$.  Thus, both $\sig_{\uL_t} \uIH^k(\M)$ and $\sig_{\uL_t} \uIH^{k-1}(\M)$ are constant as $t$ varies in the closed interval $[0, 1]$. 
Therefore, it suffices to show that 
\begin{equation}\label{eq_sig0}
\sig_{\uL_1} \uIH^k(\M) - \sig_{\uL_1} \uIH^{k-1}(\M) = (-1)^k\left(\dim \uIH^k(\M) - \dim \uIH^{k-1}(\M)\right).
\end{equation}
By Lemma \ref{orthogonal}, 
we have
\begin{equation}\label{eq_sig}
\sig_{\uL_1} \uIH_i^k(\M) = \sig_{\uL_1} \uIH^k(\M) + \sig_{\uL_1} \upsi^i_\Mo\big( \uJ^{k-1}(\M_i)\big).
\end{equation}
For any $\eta,\xi\in \uJ^{k-1}(\M_i)$, since  $\b$ annihilates the image of $\upsi^i_\Mo$, we have
\[
\uL_1^{d-2k-1}\big(\upsi^i_\Mo(\eta) \cdot \upsi^i_\Mo(\xi)\big)=(-x_i)^{d-2k-1}\big(\upsi^i_\Mo(\eta) \cdot \upsi^i_\Mo(\xi)\big),
\]
and hence
\[
\udeg_{\M}\left(\uL_1^{d-2k-1}\big(\upsi^i_\Mo(\eta) \cdot \upsi^i_\Mo(\xi)\big)\right)=\udeg_{\M}\left((-x_i)^{d-2k-1}\big(\upsi^i_\Mo(\eta) \cdot \upsi^i_\Mo(\xi)\big)\right)=\udeg_{\M}\left(\upsi^i_\Mo(\b_{\M_i}^{d-2k-1}\eta) \cdot \upsi^i_\Mo(\xi)\right).
\]
Since $\M^{\{i\}}$ has rank $1$, its Chow ring is isomorphic to $\Q$, and we therefore have $\ua_{\M^{\{i\}}}=0$.
By Lemma \ref{we need this} (2) with $F=\{i\}$,
this implies that
\[
\udeg_{\M}\left(\upsi^i_\Mo(\b_{\M_i}^{d-2k-1}\eta) \cdot \upsi^i_\Mo(\xi)\right)=
 -\udeg_{\M_i}\left(\b_{\M_i}^{d-2k}\eta\xi\right).
\]
Combining the previous two sets of equations, we find that 
$$\udeg_{\M}\left(\uL_1^{d-2k-1}\big(\upsi^i_\Mo(\eta) \cdot \upsi^i_\Mo(\xi)\big)\right)=
 -\udeg_{\M_i}\left(\b_{\M_i}^{d-2k}\eta\xi\right),$$ and therefore
\[
\sig_{\uL_1} \upsi^i_\Mo\big( \uJ^{k-1}(\M_i)\big)=-\sig_{\b_{\M_i}}\uJ^{k-1}(\M_i)=-\sig_{\b_{\M_i}}\uIH^{k-1}(\M_i).
\]
This implies that 
\begin{eqnarray*}
\sig_{\uL_1} \upsi^i_\Mo\big( \uJ^{k-1}(\M_i)\big) - \sig_{\uL_1} \upsi^i_\Mo \big(\uJ^{k-2}(\M_i)\big) 
&=& -\sig_{\b_{\M_i}}\uIH^{k-1}(\M_i) + \sig_{\b_{\M_i}}\uIH^{k-2}(\M_i)\\
&=& (-1)^{k-1}\left(\dim \uIH^{k-1}(\M_i) - \dim \uIH^{k-2}(\M_i)\right),
\end{eqnarray*}
where the last equality follows from $\hyperlink{uhr}{\uHR(\M_i)}$.
Similarly, by $\hyperlink{uhri}{\uHR_i(\M)}$, we have
\[
\sig_{\uL_1} \uIH_i^k(\M) - \sig_{\uL_1} \uIH_i^{k-1}(\M) = (-1)^k\left(\dim \uIH_i^k(\M) - \dim \uIH_i^{k-1}(\M)\right).
\]
The above two equations together with Equation \eqref{eq_sig} imply the desired Equation \eqref{eq_sig0}. 
\end{proof}

\subsection{Establishing {$\HL_\o(\M)$} and {$\HR_\o^{<\frac d 2}(\M)$}}
We now use similar arguments to those in the previous subsection in order to obtain the statements $\hyperlink{hlo}{\HL_\o(\M)}$ and $\hyperlink{hro}{\HR_\o^{<\frac d 2}(\M)}$.
Fix a positive sum $$y = \sum_{j\in E} c_j y_j.$$
For any  $t\geq 0$, consider the degree one linear operator $\oL_t$ on $\IH_\o(\M)$
given by multiplication by $y - t x_\varnothing$.  
As in the previous section, we will assume $\hyperlink{ucd}{\uCD(\M)}$.
We will also assume $\hyperlink{cd}{\CD^{<\frac{d}{2}}(\M)}$, so that for any
$k< d/2$, we have a direct sum decomposition
\[
\IH_\o^k(\M) = \IH^k(\M) \oplus \psi^\varnothing \big(\uJ^{k-1}(\M)\big)\;\;\text{and}\;\; \IH_\o^{d-k}(\M) = \IH^{d-k}(\M) \oplus \psi^\varnothing_\Mo\big( \uJ^{d-k-1}(\M)\big).
\]

\begin{lemma}\label{orthogonal again}
For any $t\geq 0$, the linear map 
\[
\oL_t^{d-2k} \colon \IH_\o^k(\M)\to \IH_\o^{d-k}(\M)
\]
is block diagonal with respect to the above decompositions. 
\end{lemma}

\begin{proof}
Since $y x_\varnothing=0$ and $y$ annihilates the image of $\psi^\varnothing_\Mo$, we have
\begin{eqnarray*}
\oL_t^{d-2k}\psi^\varnothing_\Mo\big( \uJ^{k-1}(\M)\big) &=& y^{d-2k}\psi^\varnothing_\Mo \big(\uJ^{k-1}(\M)\big) + (-t)^{d-2k}(x_\varnothing)^{d-2k}\psi^\varnothing_\Mo\big( \uJ^{k-1}(\M)\big)\\
&=& (-t)^{d-2k}(x_\varnothing)^{d-2k}\psi^\varnothing_\Mo\big( \uJ^{k-1}(\M)\big)\\
&=& t^{d-2k}\psi^\varnothing_\Mo\left(\b^{d-2k}\uJ^{k-1}(\M)\right)\\
&=& t^{d-2k}\psi^\varnothing_\Mo\big(\uJ^{d-k-1}(\M)\big),
\end{eqnarray*}
which is equal to $\psi^\varnothing_\Mo\big(\uJ^{d-k-1}(\M)\big)$ if $t> 0$ and $0$ if $t=0$. In either case, we have
\[
\oL_t^{d-2k}\psi^\varnothing_\Mo\big( \uJ^{k-1}(\M)\big)\subseteq \psi^\varnothing_\Mo\big( \uJ^{d-k-1}(\M)\big).
\]
By the above inclusion, for any $\eta\in \IH^k(\M)$ and $\xi \in \psi^\varnothing_\Mo\big( \uJ^{k-1}(\M)\big)$, we have
\[
\deg_{\M}\big(\oL_t^{d-2k}(\eta)\cdot \xi\big)=\deg_{\M}\big(\eta\cdot \oL_t^{d-2k}(\xi)\big)=0.
\]
Notice that the graded subspace $\IH(\M)\subseteq \IH_\o(\M)$ is the orthogonal complement of $\psi^\varnothing_\Mo\big( \uJ(\M)\big)$. Thus, we also have
\[
\oL_t^{d-2k}\IH^{k}(\M)\subseteq \IH^{d-k}(\M). \qedhere
\]
\end{proof}

\begin{proposition}\label{block_HL}
We have
\[
\hyperlink{cd}{\CD^{<\frac{d}{2}}(\M)},\;\;\hyperlink{uhl}{\uHL^{<\frac{d-2}{2}}(\M)},\;\;\text{and}\;\;\hyperlink{hl}{\HL(\M)}\;\; \Longrightarrow\;\;\hyperlink{hlo}{\HL_\o(\M)}.
\]
\end{proposition}
\begin{proof}
By Lemma \ref{orthogonal again}, we need to show that $\oL_t^{d-2k}$ induces isomorphisms $\IH^k(\M)\cong \IH^{d-k}(\M)$ and $\psi^\varnothing_\Mo\big( \uJ^{k-1}(\M)\big)\cong \psi^\varnothing_\Mo \big(\uJ^{d-k-1}(\M)\big)$ for some $t>0$. 
As shown in the proof of Lemma \ref{orthogonal again}, when $t>0$, $\hyperlink{uhl}{\uHL^{<\frac{d-2}{2}}(\M)}$ implies that 
\[
\oL_t^{d-2k}\psi^\varnothing_\Mo\big( \uJ^{k-1}(\M)\big)=t^{d-2k}\psi^\varnothing_\Mo\big(\uJ^{d-k-1}(\M)\big)=\psi^\varnothing_\Mo\big(\uJ^{d-k-1}(\M)\big).
\]
Thus, the second block of the block diagonal map 
\[
\oL_t^{d-2k} \colon \IH^k(\M) \oplus \psi^\varnothing \big(\uJ^{k-1}(\M)\big)\to \IH^{d-k}(\M) \oplus \psi^\varnothing_\Mo\big( \uJ^{d-k-1}(\M)\big)
\]
is an isomorphism. 

The statement $\hyperlink{hl}{\HL(\M)}$ implies that  $\oL_0^{d-2k}\colon \IH^k(\M)\to \IH^{d-k}(\M)$ is an isomorphism.  Therefore, for sufficiently small $t>0$, the map $\oL_t^{d-2k} \colon \IH^k(\M)\to \IH^{d-k}(\M)$ is also an isomorphism. 
\end{proof}

\begin{proposition}\label{block1}
We have 
\[
\hyperlink{cd}{\CD^{<\frac{d}{2}}(\M)}, \;\hyperlink{hl}{\HL(\M)}, \;\hyperlink{hr}{\HR^{<\frac{d}{2}}(\M)}, \; \hyperlink{uhl}{\uHL^{<\frac{d-2}{2}}(\M)},\;\;\text{and}\;\;\hyperlink{uhr}{\uHR^{<\frac{d-2}{2}}(\M)}\enspace \Longrightarrow \enspace\hyperlink{hro}{\HR_\o^{<\frac{d}{2}}(\M)}.
\]
\end{proposition}

\begin{proof}
For $k<d/2$, we prove $\hyperlink{hro}{\HR_\o^{k}(\M)}$ by induction on $k$.
It is clear that $\IH_\o(\M)$ satisfies the Hodge--Riemann relations in degree zero with respect to $\oL_t$ for $t$ sufficiently small. 
Now fix $0<k<d/2$ and suppose that $\hyperlink{hro}{\HR_\o^{<k}(\M)}$ holds.
We need to show that, for $t$ sufficiently small, 
\[
\sig_{\oL_t}\IH_\o^k(\M) - \sig_{\oL_t}\IH_\o^{k-1}(\M) = (-1)^k\left(\dim\IH_\o^k(\M) - \dim\IH_\o^{k-1}(\M)\right).
\]
By Lemma \ref{orthogonal again}, we have
\[
\sig_{\oL_t}\IH_\o^k(\M) = \sig_{\oL_t}\IH^k(\M) + \sig_{\oL_t}\psi^\varnothing_\Mo\big( \uJ^{k-1}(\M)\big).
\]

For $\eta,\xi\in \uJ^{k-1}(\M)= \uIH^{k-1}(\M)$, since each $y_i$ annihilates the image of $\psi^\varnothing_\Mo$, we have
\[
\oL_t^{d-2k}\left(\psi^\varnothing_\Mo(\eta) \cdot  \psi^\varnothing_\Mo(\xi)\right)=(-t x_\varnothing)^{d-2k}\left(\psi^\varnothing_\Mo(\eta) \cdot  \psi^\varnothing_\Mo(\xi)\right),
\]
and hence
\begin{align*}
\deg_{\M}\left(\oL_t^{d-2k}\left(\psi^\varnothing_\Mo(\eta) \cdot  \psi^\varnothing_\Mo(\xi)\right)\right)
&=\deg_{\M}\left(\left(-t x_\varnothing\right)^{d-2k} \psi^\varnothing_\Mo(\eta) \cdot  \psi^\varnothing_\Mo(\xi)\right)\\
&=t^{d-2k}\deg_{\M}\left(\psi^\varnothing_\Mo\big(\b^{d-2k}\eta\big)\cdot \psi^\varnothing_\Mo(\xi)\right). 
\end{align*}
Note that $\CH(\M^\varnothing) = \Q$, hence $\a_{\M^\varnothing}=0$.
By Lemma \ref{we need this} (1) with $F=\varnothing$, we therefore have
\[
\deg_{\M}\left(\psi^\varnothing_\Mo\big(\b^{d-2k}\eta\big)\cdot \psi^\varnothing_\Mo(\xi)\right)
=-\udeg_{\M}\left(\b^{d-2k+1}\eta\xi\right).
\]
When $t$ is positive, by the above two sets of equations, we have
\[
\sig_{\oL_t}\psi^\varnothing_\Mo\big( \uJ^{k-1}(\M)\big) = -\sig_{\b}\uJ^{k-1}(\M) = -\sig_{\b}\uIH^{k-1}(\M),
\]
and therefore
\begin{equation}\label{eq_r1}
\sig_{\oL_t}\IH_\o^k(\M) = \sig_{\oL_t}\IH^k(\M) -\sig_{\b}\uIH^{k-1}(\M).
\end{equation}
By $\hyperlink{hl}{\HL(\M)}$ and $\hyperlink{hr}{\HR^{<\frac{d}{2}}(\M)}$, the Hodge--Riemann forms on $\IH^k(\M)$ and $\IH^{k-1}(\M)$ associated with $\oL_0$ are non-degenerate. Thus, for $t$ sufficiently small, we have
\begin{equation}\label{eq_r2}
\begin{split}
\sig_{\oL_t}\IH^k(\M) - \sig_{\oL_t}\IH^{k-1}(\M) &= \sig_{\oL_0}\IH^k(\M) - \sig_{\oL_0}\IH^{k-1}(\M)\\
&= (-1)^k\left(\dim\IH^k(\M) - \dim\IH^{k-1}(\M)\right).
\end{split}
\end{equation}
We also have 
\begin{equation}\label{eq_r3}
\sig_{\b}\uIH^{k-1}(\M) - \sig_{\b}\uIH^{k-2}(\M) = (-1)^{k-1}\left(\dim\uIH^{k-1}(\M) - \dim\uIH^{k-2}(\M)\right)
\end{equation}
by $\hyperlink{uhl}{\uHL^{<\frac{d-2}{2}}(\M)}$ and $\hyperlink{uhr}{\uHR^{<\frac{d-2}{2}}(\M)}$.  
Therefore, we have
\begin{align*}
&\sig_{\oL_t}\IH_\o^k(\M) - \sig_{\oL_t}\IH_\o^{k-1}(\M) \\
&= \left(\sig_{\oL_t}\IH^k(\M)-\sig_{\oL_t}\IH^{k-1}(\M)\right) -\left( \sig_{\b}\uIH^{k-1}(\M)- \sig_{\b}\uIH^{k-2}(\M) \right)  \\
&= (-1)^k\left(\dim\IH^k(\M) - \dim\IH^{k-1}(\M)\right) - (-1)^{k-1}\left(\dim\uIH^{k-1}(\M) - \dim\uIH^{k-2}(\M)\right) \\
&= (-1)^k\left(\dim\IH_\o^k(\M) - \dim\IH^{k-1}_\o(\M)\right), 
\end{align*}
where the first equality follows from \eqref{eq_r1} and the second equality follows from \eqref{eq_r2} and \eqref{eq_r3}. 
\end{proof}

\section{Proof of the main theorem}\label{SectionProof}
Sections \ref{sec:main proof} and \ref{sec:boo} are devoted to combining the results that we have obtained in the previous sections
in order to complete the proof of Theorem \ref{theorem_all}.  In Section \ref{sec:the rest} we prove Propositions~\ref{prop:no socle} and~\ref{prop:gr}, thus concluding the proof of Theorem \ref{thm:KL}.

\subsection{Proof of Theorem \ref{theorem_all} for non-Boolean matroids}\label{sec:main proof}
We now complete the inductive proof of Theorem \ref{theorem_all} when $\M$ is not the Boolean matroid; the Boolean case will be addressed
in Section \ref{sec:boo}.  Let $\M$ be a matroid that is not Boolean, and assume that Theorem \ref{theorem_all} holds 
for any matroid whose ground set is a proper subset of $E$.  Since $\M$ is not Boolean,
we may fix an element $i\in E$ which is not a coloop.
If $\{i\}$ is not a flat, then $i$ has a parallel element and all of our statements about $\M$ are equivalent to the corresponding statements about 
$\Mi$. Thus, we may assume that $\{i\}$ is a flat.  We will continue our convention of writing $x_i$ in place of $x_{\{i\}}$.

We recall the main results in the previous five sections. 
By Corollary \ref{CD for free}, we have $\hyperlink{pdo}{\PD_\o(\M)}$, $\hyperlink{upd}{\uPD(\M)}$, $\hyperlink{cdo}{\CD_\o(\M)}$, and $\hyperlink{ucd}{\uCD(\M)}$. 
By Proposition \ref{uNS1}, we also have $\hyperlink{uns}{\uNS^{<\frac{d-2}{2}}(\M)}$.
By Corollaries \ref{uHL_i} and \ref{uHR_i}, we have both $\hyperlink{uhli}{\uHL_i(\M)}$ and $\hyperlink{uhri}{\uHR_i(\M)}$.

\begin{proposition}\label{NS under half}
The statement $\hyperlink{uhl}{\uHL^{<\frac{d-2}{2}}(\M)}$ holds.
\end{proposition}

\begin{proof}
Given $1\leq k<d/2$, let  
$\eta\in \uIH^{k-1}(\M)$ be a nonzero class such that 
\[
\ub^{d-2k+1} \eta=0.
\]
Recall from the proof of Lemma \ref{orthogonal} that $\ub x_i=0$, and therefore
\[
(\ub-x_i)^{d-2k}\cdot (\ub \eta)=0.
\]
In other words, the class $\ub \eta$ is primitive in $\uIH^k_i(\M)$  with respect to multiplication by $\ub-x_i$. By $ \hyperlink{uns}{\uNS^{<\frac{d-2}{2}}(\M)}$, we have $\ub\eta\neq 0$. Now, $\hyperlink{uhri}{\uHR_i(\M)}$ implies that 
\[
0< (-1)^k\udeg_{\M}\Big((\ub-x_i)^{d-2k-1}\cdot (\ub \eta)^2\Big)=(-1)^k\udeg_{\M}\big(\b^{d-2k+1}\cdot \eta^2\big).
\]
This contradicts the assumption that $\ub^{d-2k+1} \eta=0$.
\end{proof}


By Proposition \ref{CDNS}, Proposition \ref{HLi}, and Corollary \ref{HRi}, we have
\[
\hyperlink{cd}{\CD^{<\frac{d}{2}}(\M)}\;\; \Longrightarrow \; \;  \hyperlink{ns}{\NS^{< \frac{d}{2}}(\M)}, \;\;\hyperlink{hli}{\HL_i^{<\frac{d}{2}}(\M)}, \;\;\text{and}\;\; \hyperlink{hri}{\HR_i^{<\frac{d}{2}}(\M)}.
\]

\begin{proposition}\label{dCM}
We have $$
\hyperlink{ns}{\NS^{<\frac{d}{2}}(\M)}\enspace \Longrightarrow \enspace\hyperlink{hl}{\HL(\M)}.$$
\end{proposition}

\begin{proof}
Given positive numbers $c_j$ for $j\in E$, we let $y = \sum_{j\in E} c_j y_j$.  Take $k < d/2$, and suppose that $\eta\in \IH^k(\M)$ satisfies $y^{d-2k} \eta = 0$.  
For any rank one flat $G$, we have $\varphi^\Mo_G(y) = \sum_{j\notin G} c_j y_j\in \CH^1(\M_G)$.
Since $y^{d-2k} \eta = 0$, we have 
\[
\varphi^\Mo_G(y)^{d-2k}\cdot \varphi^\Mo_G(\eta) = 0.
\]
By Lemma \ref{lem_imageofIH} (1), we know that $\varphi^\Mo_G(\eta)\in \IH^k(\M_G)$.  Thus, the class $\varphi^\Mo_G(\eta)\in \IH^k(\M_G)$ is primitive with respect to $\varphi^\Mo_G(y)$.
By $\hyperlink{hr}{\HR(\M_G)}$, Proposition \ref{lemma_degy}, and Proposition \ref{Prop_ymult}, for every rank one flat $G$ we have
\[
0 \le (-1)^k\deg_{\M_G}\Big(\varphi^\Mo_G(y)^{d-2k-1}\cdot \varphi^\Mo_G(\eta)^2\Big)=
(-1)^k\deg_{\M}\Big(y_G\cdot  y^{d-2k-1} \eta^2\Big),
\]
and the equality holds if and only if $\varphi^\Mo_G(\eta)=0$.

On the other hand, since $y^{d-2k} \eta = 0$, we have
\begin{align*}
0 &= (-1)^k \deg_{\M} \left(y^{d-2k} \eta^2\right)\\
&= (-1)^k \deg_{\M}\bigg( \bigg(\sum_{j\in E} c_j y_j\bigg)\cdot y^{d-2k-1} \eta^2\bigg)\\
&= (-1)^k\sum_{j\in E} c_j \deg_{\M}\Big(y_j\cdot  y^{d-2k-1} \eta^2\Big).
\end{align*}
Since each $c_j > 0$, the above two sets of equations imply that $\varphi^\Mo_G(\eta) = 0$ for every rank one flat $G$. Thus,
\[
y_G \eta = \psi_G \big(\varphi^\Mo_G(\eta) \big) = \psi_G(0) = 0
\]
for every rank one flat $G$.  By $\hyperlink{ns}{\NS^{<\frac{d}{2}}(\M)}$, it follows that $\eta=0$.

We have proved that multiplication by $y^{d-2k}$ is an injective map from $\IH^k(\M)$ to $\IH^{d-k}(\M)$.  To conclude it is an isomorphism, it is enough to know that these spaces have the same dimension.  We know that $\hyperlink{pdo}{\PD_\o(\M)}$ holds, and since $\IH(\M)$ is the perpendicular space to $\psi_\Mo^\varnothing(\uJ(\M))$ in $\IH_\o(\M)$, it is enough to know that $\dim \uJ^{k-1}(\M) = \dim \uJ^{d-k-1}(\M)$. This follows from $\hyperlink{uhl}{\uHL^{<\frac{d-2}{2}}(\M)}$.
\end{proof}

\begin{proposition}\label{HR1NS1}
We have 
\[
\hyperlink{hro}{\HR_\o(\M)}\;\; \Longrightarrow\;\;\hyperlink{nso}{\NS_\o(\M)}.
\]
\end{proposition}

\begin{proof}
Let $y = \sum_{j\in E} y_j$.  By $\hyperlink{hro}{\HR_\o(\M)}$, we can choose $\epsilon > 0$ such that $\IH_\o(\M)$ satisfies the Hodge--Riemann relations with respect to multiplication by $y - \epsilon x_\varnothing$.
Suppose that $\eta$ is a nonzero element of the socle of $\IH^{k}_\o(\M)$ for some $k\leq d/2$.  By $\hyperlink{hro}{\HR_\o(\M)}$, we have 
\begin{equation}\label{eq:hreq}
(-1)^{k}\deg_{\M}\left((y - \epsilon x_\varnothing)^{d-2k}\eta^2\right)>0.
\end{equation}
Since $\eta$ is annihilated by every $y_j$, Lemma \ref{mutual annihilators} implies that $\eta$ is a multiple of $x_\varnothing$.  On the other hand, since $\eta$ is annihilated by $x_\varnothing$, Lemma \ref{mutual annihilators}
implies that $\eta$ is in the ideal spanned by the $y_j$.  Thus another application of Lemma \ref{mutual annihilators}
implies that $\eta^2=0$, which contradicts Equation~\eqref{eq:hreq}.
\end{proof}

\begin{proposition}\label{NS1uNS1}
We have 
\[
\hyperlink{nso}{\NS_\o(\M)}\;\; \Longrightarrow\;\;\hyperlink{uns}{\uNS(\M)}.
\]
\end{proposition}

\begin{proof}
Suppose that $k\leq d/2$ and $\eta\in\uIH^{k-1}(\M)$ is an element of the socle, that is, $\b \eta=0$.  
By Corollary \ref{very convenient}, it follows that $\psi_\Mo^\varnothing(\eta)$ is a multiple of $x_\varnothing$, and hence annihilated by each $y_j$ by Lemma~\ref{mutual annihilators}.
Furthermore, by Proposition \ref{lemma_xdegree}, 
we have \[
x_\varnothing\psi_\Mo^\varnothing(\eta) = \psi^\varnothing_\Mo\!\left(\varphi_\Mo^\varnothing(x_\varnothing)\eta\right) = \psi_\Mo^\varnothing(-\b\eta) = 0.
\]
Thus, $\psi_\Mo^\varnothing(\eta)\in\IH_\o^k(\M)$ is annihilated by each $y_j$ and $x_\varnothing$.
Then $\hyperlink{nso}{\NS_\o(\M)}$ implies that $\psi_\Mo^\varnothing(\eta)=0$,
and the injectivity of $\psi^\varnothing_\Mo$ implies that $\eta=0$.
\end{proof}

\begin{proposition}\label{NS1uHL}
We have
\[
 \hyperlink{uhl}{\uHL^{< \frac{d-2}{2}}(\M)}\;\;\text{and}\;\; 
\hyperlink{uns}{\uNS(\M)}\enspace \Longrightarrow \enspace \hyperlink{uhl}{\uHL(\M)}.
\]
\end{proposition}
\begin{proof}
When $d$ is odd, the statement $\hyperlink{uhl}{\uHL(\M)}$ is identical to $\hyperlink{uhl}{\uHL^{< \frac{d-2}{2}}(\M)}$.
When $d$ is even, the only missing case is $\hyperlink{uhl}{\uHL^{\frac{d-2}{2}}(\M)}$, which is exactly the same as $\hyperlink{uns}{\uNS^{\frac{d-2}{2}}(\M)}$. 
\end{proof}



\begin{proposition}\label{non-boolean}
Suppose that $\M$ is a matroid on $E$ that is not Boolean, and that Theorem \ref{theorem_all} holds for all matroids whose ground sets are proper subsets of $E$.
Then Theorem \ref{theorem_all} holds for $\M$.
\end{proposition}

\begin{proof}
We have already established
$\hyperlink{pdo}{\PD_\o(\M)}$, $\hyperlink{upd}{\uPD(\M)}$, $\hyperlink{cdo}{\CD_\o(\M)}$, and $\hyperlink{ucd}{\uCD(\M)}$.
The statements $\hyperlink{cd}{\CD(\M)}$, $\hyperlink{ns}{\NS(\M)}$, $\hyperlink{nso}{\NS_\o(\M)}$, $\hyperlink{hl}{\HL(\M)}$, $\hyperlink{hlo}{\HL_\o(\M)}$, $\hyperlink{uhl}{\uHL(\M)}$, $\hyperlink{hr}{\HR(\M)}$, $\hyperlink{hro}{\HR_\o(\M)}$, and $\hyperlink{uhr}{\uHR(\M)}$
are obtained from the implications shown in Figure~\ref{Proof diagram}.
The statement $\hyperlink{pd}{\PD(\M)}$ follows from $\hyperlink{hl}{\HL(\M)}$ and $\hyperlink{hr}{\HR(\M)}$.
The statement $\hyperlink{uns}{\uNS(\M)}$ is proved in Proposition \ref{NS1uNS1}.
\end{proof}

\subsection{Proof of Theorem \ref{theorem_all}: Boolean case}\label{sec:boo}

Suppose $\M$ is the Boolean matroid on $E=\{1, 2, \ldots, d\}$ with $d>0$.
\begin{proposition} The canonical decomposition $\hyperlink{ucd}{\uCD(\M)}$ of $\uCH(\M)$ holds.
We have $\uIH(\M) = \uH(\M)$, and
the space $\uJ(\M)$ is spanned by $1, \ub, \dots, \ub^{d-2}$. 
\end{proposition}

\begin{proof}  
Let $\uJ'(\M)$ be the subspace of 
$\uH(\M)$ spanned by $1, \ub, \dots, \ub^{d-2}$. 
We have $\uH(\M) \subseteq \uIH(\M)$, since $\uIH(\M)$ is an $\uH(\M)$-module that contains $1$.  Since $\ub^{d-2}$ is not zero, by Definition \ref{def:thesubspaces}, we have 
 $\uJ'(\M) \subseteq \uJ(\M)$.

Thus if we can show there is a direct sum decomposition
\begin{eqnarray}\label{eq:boolean sum}
\uCH(\M) = \uH(\M) \oplus \bigoplus_{\varnothing < F < E} \upsi^F_\M\left( \uJ'(\M_F)\otimes \uCH(\M^F)\right),
\end{eqnarray}
the proposition will follow.  

For a Boolean matroid $\M$, $\uCH(\M)$ admits an automorphism 
\[
\tau \colon \uCH(\M)\to \uCH(\M), \quad x_F\mapsto x_{E\setminus F}.
\]
The automorphism $\tau$ exchanges $\underline{\alpha}$ and $\b$.  It is then easy to see that the decomposition \eqref{eq:boolean sum} is the result of applying $\tau$ to the decomposition $(\underline{\mathrm D}_3)$ of \cite{BHMPW}.

Alternatively, one can use the basis of $\uCH(\M)$ given by Feichtner and Yuzvinsky \cite[Corollary 1]{FY}.  Their basis is given by all products
\[x_{G_1}^{m_1}x_{G_2}^{m_2}\cdots x_{G_k}^{m_k}\ua^{m_{k+1}},\]
where $G_1 < G_2 < \dots < G_k$ is a (possibly empty) flag of nonempty proper flats and we have $m_1 < \rk G_1$, $m_i < \rk G_i - \rk G_{i-1}$ for $1 < i \le k$, and $m_{k+1} < \crk G_k$ (when $k=0$, $m_{k+1}<\rk \M$).
Applying $\tau$ gives
\[\ub^{m_{k+1}}(x_{F_k})^{m_k}\cdots (x_{F_1})^{m_1},\]
where $F_i = E \setminus G_i$. 
If $k \ne 0$ this is in $\upsi_\Mo^{F_k}\big((\ub_{\M_{F_k}})^{m_k} \otimes \uCH(\M^{F_k})\big)$, while if $k=0$ it is in $\uH(\M)$.  The direct sum decomposition \eqref{eq:boolean sum} follows.
\end{proof}

Since $\uIH(\M)$ is isomorphic to $\uH(\M)$, which is spanned by $1, \b, \b^2, \ldots, \b^{d-1}$, we immediately deduce $\hyperlink{uns}{\uNS(\M)}$ and $\hyperlink{uhl}{\uHL(\M)}$. Notice that the involution $\tau$ induces the identity map on $\uCH^{d-1}(\M)$. Therefore, $\udeg_{\M}(\b^{d-1})=\udeg_{\M}(\underline{\alpha}^{d-1})=1$, and we have $\hyperlink{upd}{\uPD(\M)}$ and $\hyperlink{uhr}{\uHR(\M)}$.
By Proposition \ref{uHL1CD}, we also get $\hyperlink{cd}{\CD(\M)}$. By Lemma \ref{mutual annihilators} and Corollary \ref{underline it}, we have an isomorphism of graded vector spaces
\[
\IH_\o(\M)_\varnothing\cong \varphi_\Mo^\varnothing(\IH_\o(\M))=\uIH(\M).
\]
Since $\psi_\Mo^\varnothing(\b^i)=\psi^\varnothing\varphi^\varnothing((-x_\varnothing)^{i}) = (-1)^i(x_\varnothing)^{i+1}$, it follows that $\psi^\varnothing_\Mo \uJ(\M)$ is spanned by $x_\varnothing, \ldots, x_\varnothing^{d-1}$. Since $x_\varnothing y_j=0$ for any $j\in E$, we have an isomorphism of graded vector spaces 
\[
\big(\psi^\varnothing_\Mo \uJ(\M)\big)_\varnothing \cong\psi^\varnothing_\Mo \uJ(\M).
\]
By $\hyperlink{cd}{\CD(\M)}$, we have $\IH_\o(\M)_\varnothing \cong \IH(\M)_\varnothing \oplus \big(\psi^\varnothing_\Mo \uJ(\M)\big)_\varnothing$.  Since $\uIH(\M)$ has total dimension $d$ and $\uJ(\M)$ has total dimension $d-1$, the stalk $\IH(\M)_\varnothing$ is one-dimensional, and hence $\IH(\M)_\varnothing\cong \IH^0(\M)\cong \Q$. Therefore, $\IH(\M)$ is generated in degree zero as a module over $\H(\M)$. Equivalently, $\IH(\M)$ is isomorphic to a quotient of $\H(\M)$. 

On the other hand, since $\M$ is Boolean, $\H(\M)=\Q[y_1, \ldots, y_d]/(y_1^2, \ldots, y_d^2)$ is a Poincar\'e duality algebra. Since $\IH^d(\M)$ is one-dimensional, the quotient map $\H(\M)\to \IH(\M)$ is an isomorphism in degree $d$. Therefore, the quotient map must be an isomorphism, that is, 
\[
\IH(\M)\cong \H(\M)=\Q[y_1, \ldots, y_d]/(y_1^2, \ldots, y_d^2).
\]
One can explicitly verify that $\H(\M)$ 
satisfies Poincar\'e duality, the hard Lefschetz theorem, and the Hodge--Riemann relations,
or one can infer it from the fact that $\H(\M)$ is isomorphic to the Chow ring of the projective variety $(\mathbb{P}_{\mathbb{C}}^1)^d$.
The statement $\hyperlink{pdo}{\PD_\o(\M)}$ follows from $\hyperlink{pd}{\PD(\M)}$, $\hyperlink{upd}{\uPD(\M)}$, and $\hyperlink{uhl}{\uHL(\M)}$. By Lemma \ref{orthogonal}, the statement $\hyperlink{hlo}{\HL_\o(\M)}$ follows from $\hyperlink{hl}{\HL(\M)}$ and $\hyperlink{uhl}{\uHL(\M)}$, and the statement $\hyperlink{hro}{\HR_\o(\M)}$ follows from $\hyperlink{hr}{\HR(\M)}$ and $\hyperlink{uhr}{\uHR(\M)}$. 


\subsection{Proofs of Propositions~\ref{prop:no socle} and~\ref{prop:gr}}\label{sec:the rest}

Recall from Section~\ref{sec:strategy} that the proof of Theorem~\ref{thm:KL} relies on Theorem~\ref{prop:kahler}, which we have already proved as part of Theorem~\ref{theorem_all}, as well as on Propositions~\ref{prop:no socle} and~\ref{prop:gr}. In this subsection, we will prove these remaining two propositions.







\begin{proof}[Proof of Proposition~\ref{prop:no socle}]
As parts of Theorem~\ref{theorem_all}, we have already obtained $\hyperlink{pd}{\PD(\M)}$ and $\hyperlink{ns}{\NS(\M)}$. By $\hyperlink{pd}{\PD(\M)}$, the socle of $\IH(\M)$ is equal to the orthogonal complement $(\mm \IH(\M))^\perp$ in $\IH(\M)$. 
By $\hyperlink{ns}{\NS(\M)}$, we know that $(\mm \IH(\M))^\perp =0$ in degrees less than or equal to $d/2$. Thus, $\mm \IH(\M)= \IH(\M)$ in degrees greater than or equal to $d/2$, or equivalently, $\IH(\M)_{\varnothing}=0$ in degrees greater than or equal to $d/2$. 
\end{proof}

\begin{proof}[Proof of Proposition~\ref{prop:gr}]
Choose any ordering $F_1, \ldots, F_r$ of $\cL(\M)$ such that $\rk{F_i}\leq \rk{F_j}$ whenever $i\leq j$.  Then there exist $\mu$, $\nu$ so that 
$\Sigma_{\mu}=\{F_{\mu}, \ldots, F_r\}=\cL^{\geq k}(\M)$ and $\Sigma_{\nu}=\{F_{\nu}, \ldots, F_r\}=\cL^{\geq k+1}(\M)$. 
By definition,
\begin{equation}\label{eq_191}
\mm^k\IH(\M)/\mm^{k+1}\IH(\M)\cong 
\frac{\IH(\M)_{\Sigma_{\mu}}}{\IH(\M)_{\Sigma_{\nu}}}.
\end{equation}
Consider the natural maps
\begin{equation}\label{eq_192} 
\bigoplus_{F\in\cL^k(\M)}\frac{\IH(\M)_{\geq F}}{\IH(\M)_{>F}} \to
\bigoplus_{F\in \cL^k(\M)}\frac{\IH(\M)_{\Sigma_\nu\cup\{F\}}}{\IH(\M)_{\Sigma_\nu}} \to \frac{\IH(\M)_{\Sigma_{\mu}}}{\IH(\M)_{\Sigma_{\nu}}}.
\end{equation}
The first map is an isomorphism by Proposition \ref{pure} (1), and the second map is a surjection since $\Upsilon_{\Sigma_\mu}$ is the ideal generated by $y_F$, $F \in \cL^k(\M)$.  To show that it is an injection, suppose that there are elements $a_F \in \IH(\M)_{\Sigma_\nu\cup\{F\}}$ representing a nonzero element of the kernel.  Let $S$ be the set of flats for which $a_F \notin \IH(\M)_{\Sigma_\nu}$; we can assume that the $a_F$ are chosen so $S$ is as small as possible.  Each component of the map \eqref{eq_192} is injective, so $|S|>1$.  Let $S = S_1 \cup S_2$ be a nontrivial partition of $S$.  Then 
\[b_i := \sum_{F \in S_i} a_F \in \IH(\M)_{\Sigma_\nu\cup S_i}\]
is not in $\IH(\M)_{\Sigma_\nu}$, by the minimality of $S$.  But $b_1+b_2$
is in $\IH(\M)_{\Sigma_\nu}$, which implies that $b_1$ and $b_2$
are in \[\IH(\M)_{\Sigma_\nu\cup S_1} \cap \IH(\M)_{\Sigma_\nu\cup S_2} = 
\IH(\M)_{\Sigma_\nu},\]
by Proposition \ref{prop:purity characterization}.  This is a contradiction, so we conclude that \eqref{eq_192} is an isomorphism.


By Lemma \ref{everything is empty} and Lemma \ref{mult by y} (1), for any flat $F$, we have natural isomorphisms
\begin{equation}\label{eq_193}
\frac{\IH(\M)_{\geq F}}{\IH(\M)_{>F}} = \big(\IH(\M)[-\rk F]\big)_F \cong  \big(y_F\IH(\M)\big)_\varnothing\cong \IH(\M_F)_\varnothing[-\rk F].
\end{equation}


Now, the proposition follows from the isomorphisms in Equations \eqref{eq_191}, \eqref{eq_192}, and \eqref{eq_193}.
\end{proof}


\subsection{Proof of Theorem \ref{thm:monotonicity}}\label{subsec:monotonicity}
By Theorem \ref{theorem_all}, all of our statements hold for $\M$ and $\M_F$, and so in particular Lemma 
\ref{mult by y} says that the pullback $\varphi_F$ restricts to a surjection $\IH(\M) \to \IH(\M_F)$.  Because we are assuming that $y_F$ is fixed by the action of $\Gamma$, this surjection is $\Gamma$-equivariant.  Since $\varphi_F$ is a ring homomorphism that sends the maximal ideal $\mm$ of $\H(\M)$ to the maximal ideal $\mm_F$ of $\H(\M_F)$, it follows that we have a $\Gamma$-equivariant surjection 
\[\IH(\M)_\varnothing \to \IH(\M_F)_\varnothing.\]
The result now follows by taking $\Gamma$-equivariant Poincar\'e polynomials.

\appendix
\section{Equivariant polynomials}\label{Appendix}
The purpose of this appendix is to give precise definitions of equivariant Kazhdan--Lusztig polynomials, equivariant $Z$-polynomials,
and equivariant inverse Kazhdan--Lusztig polynomials.
We also prove an equivariant analogue of the characterization of Kazhdan--Lusztig polynomials
and $Z$-polynomials that appears  in \cite[Theorem 2.2]{BV}.

Let $\Gamma$ be a finite group, and let $\VRep(\Gamma)$ be the ring of virtual representations of $\Gamma$ over $\Q$ with coefficients in $\mathbb{Q}$.
For any finite-dimensional representation $V$ of $\Gamma$, let $[V]$ be its class in $\VRep(\Gamma)$.
If $\Gamma$ acts on a set $S$ and $x\in S$, we write $\Gamma_x\subseteq \Gamma$ for the stabilizer of $x$.
We use the following standard lemma  \cite[Lemma 2.7]{eq-incidence}.

\begin{lemma}\label{induced reps}
Let $V = \bigoplus_{x\in S} V_x$ be a vector space that decomposes as a direct sum of pieces indexed by a finite set $S$,
and suppose that $\Gamma$ acts linearly on $V$ and acts by permutations on $S$.
If $\gamma\cdot V_x = V_{\gamma\cdot x}$
 for all $x\in S$ and $\gamma\in\Gamma$,
then
$$[V] = \sum_{[x]\in S/\Gamma}\Ind_{\Gamma_x}^{\Gamma} [V_x] = \sum_{x\in S}\frac{|\Gamma_x|}{|\Gamma|}\Ind_{\Gamma_x}^{\Gamma} [V_x] \in \VRep(\Gamma).$$
\end{lemma}

Let $\M$ be a matroid  on the ground set $E$, and let $\Gamma$ be a finite group acting on $\M$.  In other words,
the set $E$ is equipped with an action of $\Gamma$ by permutations that take flats of $\M$ to flats of $\M$.
We define the {\bf equivariant characteristic polynomial} 
$$\chi^\Gamma_\M(t) \coloneq \sum_{k=0}^{\rk \M} (-1)^k [\OS^k(\M)]\, t^{\rk \M -k} \in  \VRep(\Gamma)[t],$$
where $\OS^k(\M)$ is the degree $k$ part of the Orlik--Solomon algebra of $\M$.
The dimension homomorphism from $\VRep(\Gamma)[t]$ to $\Z[t]$ takes the equivariant characteristic polynomial
$\chi^\Gamma_\M(t)$ to the ordinary characteristic polynomial $\chi_\M(t)$; see \cite[Chapter 3]{Orlik-Terao}.
The following statement appears in \cite[Theorem 2.8]{GPY}.


\begin{theorem}\label{eqP}
To each matroid $\M$ and symmetry group $\Gamma$,
there is a unique way to assign 
 a polynomial $P_\M^\Gamma(t)$ with coefficients in $\VRep(\Gamma)$ 
with the following properties:
\begin{enumerate}[(a)]\itemsep 5pt
\item If the ground set of $\M$ is empty, then $P_{\M}^\Gamma(t) = 1$ (the trivial representation).
\item For every matroid $\M$ on a nonempty ground set, the degree of $P_\M^\Gamma(t)$ is strictly less than  $\rk \M / 2$.	
\item For every matroid $\M$, we have $t^{\rk \M} P_\M(t^{-1})\; = \displaystyle\sum_{F\in \cL(\M)} \frac{|\Gamma_F|}{|\Gamma|}\Ind_{\Gamma_F}^\Gamma\left(\chi_{\M^F}^{\Gamma_F}(t) P_{\M_F}^{\Gamma_F}(t)\right)$.
\end{enumerate}
The polynomial $P_\M^\Gamma(t)$ is called the {\bf equivariant Kazhdan--Lusztig polynomial} of $\M$ with respect to the action of $\Gamma$.
\end{theorem}

The following definition appears in \cite[Section 6]{PXY}.

\begin{definition}\label{eqZ}
The {\bf equivariant \boldmath{$Z$}-polynomial} of
$\M$ with respect to the action of $\Gamma$ is $$Z_\M^\Gamma(t)\; \coloneq 
\sum_{F\in \cL(\M)} \frac{|\Gamma_F|}{|\Gamma|}\Ind_{\Gamma_F}^\Gamma\left(P_{\M_F}^{\Gamma_F}(t)\right)\, t^{\rk F} \in \VRep(\Gamma)[t].$$
\end{definition}

A polynomial $f(t)\in\VRep(\Gamma)[t]$ is called {\bf palindromic} if $t^{\deg f(t)}f(t^{-1}) = f(t)$.  
The fact that the equivariant $Z$-polynomial is palindromic is asserted without proof in \cite[Section 6]{PXY};
a full proof appears in \cite[Corollary 4.5]{eq-incidence}.

\begin{lemma}\label{duh}
For any polynomial $f(t)$ of degree $d$, 
there is a unique polynomial $g(t)$ of degree strictly less than $d/2$ such that $f(t) + g(t)$
is palindromic.
\end{lemma}

\begin{proof}
We must take $g(t)$ to be the truncation of $t^d f(t^{-1}) - f(t)$ to degree $\lfloor (d-1)/2\rfloor$.
\end{proof}

The following proposition is an equivariant analogue of \cite[Theorem 2.2]{BV}.

\begin{corollary}\label{Zcor}
Let $\M$ be a nonempty matroid, let $\tilde P_\M^\Gamma(t)$ be a polynomial of degree strictly less than $\rk \M/2$ in  $\VRep(\Gamma)[t]$,
and let
\[
\tilde Z_\M^\Gamma(t) \coloneq \tilde P^\Gamma_\M(t)\; + \sum_{\varnothing \neq F\in \cL(\M)} \frac{|\Gamma_F|}{|\Gamma|}\Ind_{\Gamma_F}^\Gamma\left(P_{\M_F}^{\Gamma_F}(t)\right)\, t^{\rk F}.
\]
If  $\tilde Z_\M^\Gamma(t)$ is a palindromic polynomial,
then $\tilde P_\M^\Gamma(t) = P_\M^\Gamma(t)$ and $\tilde Z_\M^\Gamma(t) = Z_\M^\Gamma(t)$.
\end{corollary}

\begin{proof}
By definition of $Z_\M^\Gamma(t)$, we have
$$Z_\M^\Gamma(t) = P^\Gamma_\M(t)\; + \sum_{\varnothing \neq F\in \cL(\M)} \frac{|\Gamma_F|}{|\Gamma|} \Ind_{\Gamma_F}^\Gamma\left(P_{\M_F}^{\Gamma_F}(t)\right)\, t^{\rk F}.$$  The corollary then follows from Lemma \ref{duh} and the palindromicity of $\tilde Z_\M^\Gamma(t)$.
\end{proof}

When the rank of $\M$ is positive, by \cite[Theorem 1.3]{GX}, the inverse Kazhdan--Lusztig polynomial $Q_\M(t)$ of $\M$  
satisfies the equation
\[
\sum_{F \in \cL(\M)} (-1)^{\rk F}P_{\M^F}(t)Q_{\M_F}(t) = 0.
\]
We use the recurrence relation to define an equivariant analogue of $Q_\M(t)$.

\begin{definition}\label{eqI}
The {\bf equivariant inverse Kazhdan--Lusztig polynomial} of
$\M$ with respect to the action of $\Gamma$ is defined by the condition that $Q^\Gamma_{\M}(t)$ is equal to the trivial representation if the ground set of $\M$ is empty,
and otherwise
$$\sum_{F\in \cL(\M)} (-1)^{\rk F}\frac{|\Gamma_F|}{|\Gamma|}\Ind_{\Gamma_F}^\Gamma\left(P_{\M^F}^{\Gamma_F}(t)Q_{\M_F}^{\Gamma_F}(t)\right)=0.$$
Equivalently, we recursively put
$$Q^\Gamma_{\M}(t) = - \sum_{\varnothing \neq F\in\cL(\M)} (-1)^{\rk F}\frac{|\Gamma_F|}{|\Gamma|}\Ind_{\Gamma_F}^\Gamma\left(P_{\M^F}^{\Gamma_F}(t)Q_{\M_F}^{\Gamma_F}(t)\right)\in \VRep(\Gamma)[t].$$
\end{definition}

For  equivalent definitions of $P^\Gamma_\M(t)$,  $Z^\Gamma_\M(t)$, and  $Q^\Gamma_\M(t)$ 
in the framework of equivariant incidence algebras and equivariant Kazhdan--Lusztig--Stanley theory, 
we refer to  \cite[Section 4]{eq-incidence}.

\bibliography{refs}

@incollection {Zaslavsky,
    AUTHOR = {Zaslavsky, Thomas},
     TITLE = {The {M}\"obius function and the characteristic polynomial},
 BOOKTITLE = {Combinatorial geometries},
    SERIES = {Encyclopedia Math. Appl.},
    VOLUME = {29},
     PAGES = {114--138},
 PUBLISHER = {Cambridge Univ. Press, Cambridge},
      YEAR = {1987},
      ISBN = {0-521-33339-3},
   MRCLASS = {03B35 (05A15 05C15)},
  MRNUMBER = {921071},
}

@article {McMullen,
    AUTHOR = {McMullen, Peter},
     TITLE = {On simple polytopes},
   JOURNAL = {Invent. Math.},
  FJOURNAL = {Inventiones Mathematicae},
    VOLUME = {113},
      YEAR = {1993},
    NUMBER = {2},
     PAGES = {419--444},
      ISSN = {0020-9910,1432-1297},
   MRCLASS = {52B05 (52B35)},
  MRNUMBER = {1228132},
MRREVIEWER = {Bernd\ Kind},
       DOI = {10.1007/BF01244313},
       URL = {https://doi-org.uoregon.idm.oclc.org/10.1007/BF01244313},
}

@article {BN07,
    AUTHOR = {Bar-Natan, Dror},
     TITLE = {Fast {K}hovanov homology computations},
   JOURNAL = {J. Knot Theory Ramifications},
  FJOURNAL = {Journal of Knot Theory and its Ramifications},
    VOLUME = {16},
      YEAR = {2007},
    NUMBER = {3},
     PAGES = {243--255},
      ISSN = {0218-2165,1793-6527},
   MRCLASS = {57M25},
  MRNUMBER = {2320156},
       DOI = {10.1142/S0218216507005294},
       URL = {https://doi.org/10.1142/S0218216507005294},
}

@article {Atiyah,
    AUTHOR = {Atiyah, M.},
     TITLE = {On the {K}rull-{S}chmidt theorem with application to sheaves},
   JOURNAL = {Bull. Soc. Math. France},
  FJOURNAL = {Bulletin de la Soci\'{e}t\'{e} Math\'{e}matique de France},
    VOLUME = {84},
      YEAR = {1956},
     PAGES = {307--317},
      ISSN = {0037-9484},
   MRCLASS = {53.3X},
  MRNUMBER = {86358},
MRREVIEWER = {S. Eilenberg},
       URL = {http://www.numdam.org/item?id=BSMF_1956__84__307_0},
}

@article {CF,
    AUTHOR = {Camillo, V. P. and Fuller, K. R.},
     TITLE = {On graded rings with finiteness conditions},
   JOURNAL = {Proc. Amer. Math. Soc.},
  FJOURNAL = {Proceedings of the American Mathematical Society},
    VOLUME = {86},
      YEAR = {1982},
    NUMBER = {1},
     PAGES = {1--5},
      ISSN = {0002-9939},
   MRCLASS = {16A03 (16A10 16A35 16A51)},
  MRNUMBER = {663852},
MRREVIEWER = {C. N\u{a}st\u{a}sescu},
       DOI = {10.2307/2044382},
       URL = {https://doi.org/10.2307/2044382},
}

@article {GG,
    AUTHOR = {Gordon, Robert and Green, Edward L.},
     TITLE = {Graded {A}rtin algebras},
   JOURNAL = {J. Algebra},
  FJOURNAL = {Journal of Algebra},
    VOLUME = {76},
      YEAR = {1982},
    NUMBER = {1},
     PAGES = {111--137},
      ISSN = {0021-8693},
   MRCLASS = {16A64 (16A46)},
  MRNUMBER = {659212},
MRREVIEWER = {Idun Reiten},
       DOI = {10.1016/0021-8693(82)90240-X},
       URL = {https://doi.org/10.1016/0021-8693(82)90240-X},
}

@unpublished{Melvin-Slofstra,
  author = {George Melvin and William Slofstra},
  title  = {Soergel bimodules and the shape of {B}ruhat intervals},
  note   = {Preprint},
  year = {2020},
}

@book {EMTW,
    AUTHOR = {Elias, Ben and Makisumi, Shotaro and Thiel, Ulrich and
              Williamson, Geordie},
     TITLE = {Introduction to {S}oergel bimodules},
    SERIES = {RSME Springer Series},
    VOLUME = {5},
 PUBLISHER = {Springer, Cham},
      YEAR = {2020},
     PAGES = {xxv+588},
      ISBN = {978-3-030-48825-3; 978-3-030-48826-0},
   MRCLASS = {20C08 (17B10 18M30 20F55)},
  MRNUMBER = {4220642},
       DOI = {10.1007/978-3-030-48826-0},
       URL = {https://doi.org/10.1007/978-3-030-48826-0},
}

@article {coron2024operadic,
    AUTHOR = {Coron, Basile},
     TITLE = {Operadic {K}azhdan-{L}usztig-{S}tanley theory},
   JOURNAL = {Int. Math. Res. Not. IMRN},
  FJOURNAL = {International Mathematics Research Notices. IMRN},
      YEAR = {2025},
    NUMBER = {3},
     PAGES = {Paper No. rnaf009, 28},
      ISSN = {1073-7928,1687-0247},
   MRCLASS = {18M80 (06A11 13P10)},
  MRNUMBER = {4857948},
MRREVIEWER = {Beno\^it\ Fresse},
       DOI = {10.1093/imrn/rnaf009},
       URL = {https://doi.org/10.1093/imrn/rnaf009},
}

@article {AR16,
    AUTHOR = {Achar, Pramod N. and Riche, Simon},
     TITLE = {Modular perverse sheaves on flag varieties, {II}: {K}oszul
              duality and formality},
   JOURNAL = {Duke Math. J.},
  FJOURNAL = {Duke Mathematical Journal},
    VOLUME = {165},
      YEAR = {2016},
    NUMBER = {1},
     PAGES = {161--215},
      ISSN = {0012-7094},
   MRCLASS = {14M15 (14F05 20G44)},
  MRNUMBER = {3450745},
MRREVIEWER = {Nicolas Perrin},
       DOI = {10.1215/00127094-3165541},
       URL = {https://doi.org/10.1215/00127094-3165541},
}

@article {Mak17,
    AUTHOR = {Makisumi, Shotaro},
     TITLE = {Modular {K}oszul duality for {S}oergel bimodules},
   JOURNAL = {Adv. Math.},
  FJOURNAL = {Advances in Mathematics},
    VOLUME = {428},
      YEAR = {2023},
     PAGES = {Paper No. 109158, 38},
      ISSN = {0001-8708,1090-2082},
   MRCLASS = {14F08 (17B10 20F55 20G15)},
  MRNUMBER = {4601783},
       DOI = {10.1016/j.aim.2023.109158},
       URL = {https://doi.org/10.1016/j.aim.2023.109158},
}

@incollection {Sp82,
    AUTHOR = {Springer, T. A.},
     TITLE = {Quelques applications de la cohomologie d'intersection},
 BOOKTITLE = {Bourbaki {S}eminar, {V}ol. 1981/1982},
    SERIES = {Ast\'{e}risque},
    VOLUME = {92},
     PAGES = {249--273},
 PUBLISHER = {Soc. Math. France, Paris},
      YEAR = {1982},
   MRCLASS = {32C38 (14G20 14L30 20G05)},
  MRNUMBER = {689533},
MRREVIEWER = {Robert D. MacPherson},
}

@article {DeCP,
    AUTHOR = {De Concini, C. and Procesi, C.},
     TITLE = {Wonderful models of subspace arrangements},
   JOURNAL = {Selecta Math. (N.S.)},
  FJOURNAL = {Selecta Mathematica. New Series},
    VOLUME = {1},
      YEAR = {1995},
    NUMBER = {3},
     PAGES = {459--494},
      ISSN = {1022-1824},
   MRCLASS = {14D99 (32G13 52B30)},
  MRNUMBER = {1366622},
MRREVIEWER = {V. Leksin},
       DOI = {10.1007/BF01589496},
       URL = {https://doi.org/10.1007/BF01589496},
}

@article {DL,
    AUTHOR = {Denef, J. and Loeser, F.},
     TITLE = {Weights of exponential sums, intersection cohomology, and
              {N}ewton polyhedra},
   JOURNAL = {Invent. Math.},
  FJOURNAL = {Inventiones Mathematicae},
    VOLUME = {106},
      YEAR = {1991},
    NUMBER = {2},
     PAGES = {275--294},
      ISSN = {0020-9910},
   MRCLASS = {14F32 (11L03 11T23 14F20 14M25)},
  MRNUMBER = {1128216},
MRREVIEWER = {F. Baldassarri},
       DOI = {10.1007/BF01243914},
       URL = {https://doi.org/10.1007/BF01243914},
}

@article {BrLu,
    AUTHOR = {Bressler, Paul and Lunts, Valery A.},
     TITLE = {Intersection cohomology on nonrational polytopes},
   JOURNAL = {Compositio Math.},
  FJOURNAL = {Compositio Mathematica},
    VOLUME = {135},
      YEAR = {2003},
    NUMBER = {3},
     PAGES = {245--278},
      ISSN = {0010-437X},
   MRCLASS = {52B05 (05A99 14F43 14M25)},
  MRNUMBER = {1956814},
MRREVIEWER = {Christian Haase},
       DOI = {10.1023/A:1022232232018},
       URL = {https://doi.org/10.1023/A:1022232232018},
}

@article {BrLu05,
    AUTHOR = {Bressler, Paul and Lunts, Valery A.},
     TITLE = {Hard {L}efschetz theorem and {H}odge-{R}iemann relations for
              intersection cohomology of nonrational polytopes},
   JOURNAL = {Indiana Univ. Math. J.},
  FJOURNAL = {Indiana University Mathematics Journal},
    VOLUME = {54},
      YEAR = {2005},
    NUMBER = {1},
     PAGES = {263--307},
      ISSN = {0022-2518},
   MRCLASS = {14F43 (14M25)},
  MRNUMBER = {2126725},
MRREVIEWER = {Tom C. Braden},
       DOI = {10.1512/iumj.2005.54.2528},
       URL = {https://doi.org/10.1512/iumj.2005.54.2528},
}

@article {BBFK,
    AUTHOR = {Barthel, Gottfried and Brasselet, Jean-Paul and Fieseler,
              Karl-Heinz and Kaup, Ludger},
     TITLE = {Combinatorial intersection cohomology for fans},
   JOURNAL = {Tohoku Math. J. (2)},
  FJOURNAL = {The Tohoku Mathematical Journal. Second Series},
    VOLUME = {54},
      YEAR = {2002},
    NUMBER = {1},
     PAGES = {1--41},
      ISSN = {0040-8735},
   MRCLASS = {14F43 (14M25 32S60 52B20)},
  MRNUMBER = {1878925},
MRREVIEWER = {Tom C. Braden},
       URL = {http://projecteuclid.org/euclid.tmj/1113247177},
}

@article {BrMacP,
    AUTHOR = {Braden, Tom and MacPherson, Robert},
     TITLE = {From moment graphs to intersection cohomology},
   JOURNAL = {Math. Ann.},
  FJOURNAL = {Mathematische Annalen},
    VOLUME = {321},
      YEAR = {2001},
    NUMBER = {3},
     PAGES = {533--551},
      ISSN = {0025-5831},
   MRCLASS = {14F43 (14M15 32S60)},
  MRNUMBER = {1871967},
       DOI = {10.1007/s002080100232},
       URL = {https://doi.org/10.1007/s002080100232},
}

@article {Irv,
    AUTHOR = {Irving, Ronald S.},
     TITLE = {The socle filtration of a {V}erma module},
   JOURNAL = {Ann. Sci. \'{E}cole Norm. Sup. (4)},
  FJOURNAL = {Annales Scientifiques de l'\'{E}cole Normale Sup\'{e}rieure. Quatri\`eme
              S\'{e}rie},
    VOLUME = {21},
      YEAR = {1988},
    NUMBER = {1},
     PAGES = {47--65},
      ISSN = {0012-9593},
   MRCLASS = {17B35},
  MRNUMBER = {944101},
MRREVIEWER = {Anthony Joseph},
       URL = {http://www.numdam.org/item?id=ASENS_1988_4_21_1_47_0},
}

@article {Braden-CICF,
    AUTHOR = {Braden, Tom},
     TITLE = {Remarks on the combinatorial intersection cohomology of fans},
   JOURNAL = {Pure Appl. Math. Q.},
  FJOURNAL = {Pure and Applied Mathematics Quarterly},
    VOLUME = {2},
      YEAR = {2006},
    NUMBER = {4, Special Issue: In honor of Robert D. MacPherson. Part
              2},
     PAGES = {1149--1186},
      ISSN = {1558-8599},
   MRCLASS = {14F43 (14M25)},
  MRNUMBER = {2282417},
MRREVIEWER = {Julianna Tymoczko},
       DOI = {10.4310/PAMQ.2006.v2.n4.a10},
       URL = {https://doi.org/10.4310/PAMQ.2006.v2.n4.a10},
}

@article {Fieseler,
    AUTHOR = {Fieseler, Karl-Heinz},
     TITLE = {Rational intersection cohomology of projective toric
              varieties},
   JOURNAL = {J. Reine Angew. Math.},
  FJOURNAL = {Journal f\"{u}r die Reine und Angewandte Mathematik. [Crelle's
              Journal]},
    VOLUME = {413},
      YEAR = {1991},
     PAGES = {88--98},
      ISSN = {0075-4102},
   MRCLASS = {14F32 (14M25 55N33 57N80)},
  MRNUMBER = {1089798},
MRREVIEWER = {V. P. Snaith},
       DOI = {10.1515/crll.1991.413.88},
       URL = {https://doi.org/10.1515/crll.1991.413.88},
}

@article {KL79,
    AUTHOR = {Kazhdan, David and Lusztig, George},
     TITLE = {Representations of {C}oxeter groups and {H}ecke algebras},
   JOURNAL = {Invent. Math.},
  FJOURNAL = {Inventiones Mathematicae},
    VOLUME = {53},
      YEAR = {1979},
    NUMBER = {2},
     PAGES = {165--184},
      ISSN = {0020-9910},
   MRCLASS = {20H15 (17B35 20G05 22E47)},
  MRNUMBER = {560412},
MRREVIEWER = {Vinay V. Deodhar},
       DOI = {10.1007/BF01390031},
       URL = {https://doi.org/10.1007/BF01390031},
}

@inproceedings {KL80,
    AUTHOR = {Kazhdan, David and Lusztig, George},
     TITLE = {Schubert varieties and {P}oincar\'{e} duality},
 BOOKTITLE = {Geometry of the {L}aplace operator ({P}roc. {S}ympos. {P}ure
              {M}ath., {U}niv. {H}awaii, {H}onolulu, {H}awaii, 1979)},
    SERIES = {Proc. Sympos. Pure Math., XXXVI},
     PAGES = {185--203},
 PUBLISHER = {Amer. Math. Soc., Providence, R.I.},
      YEAR = {1980},
   MRCLASS = {14M15 (55N35 57R19)},
  MRNUMBER = {573434},
}

@article {BGS,
    AUTHOR = {Beilinson, Alexander and Ginzburg, Victor and Soergel,
              Wolfgang},
     TITLE = {Koszul duality patterns in representation theory},
   JOURNAL = {J. Amer. Math. Soc.},
  FJOURNAL = {Journal of the American Mathematical Society},
    VOLUME = {9},
      YEAR = {1996},
    NUMBER = {2},
     PAGES = {473--527},
      ISSN = {0894-0347},
   MRCLASS = {17B10 (14F10 16W50)},
  MRNUMBER = {1322847},
MRREVIEWER = {Rolf K\"{a}llstr\"{o}m},
       DOI = {10.1090/S0894-0347-96-00192-0},
       URL = {https://doi.org/10.1090/S0894-0347-96-00192-0},
}

@article {Ginsburg-perverse,
    AUTHOR = {Ginsburg, Victor},
     TITLE = {Perverse sheaves and {${\bf C}^*$}-actions},
   JOURNAL = {J. Amer. Math. Soc.},
  FJOURNAL = {Journal of the American Mathematical Society},
    VOLUME = {4},
      YEAR = {1991},
    NUMBER = {3},
     PAGES = {483--490},
      ISSN = {0894-0347},
   MRCLASS = {14F32 (14L30 22E47 55N33 57R19)},
  MRNUMBER = {1091465},
MRREVIEWER = {Ivan Mirkovi\'{c}},
       DOI = {10.2307/2939265},
       URL = {https://doi.org/10.2307/2939265},
}

@article {BE,
    AUTHOR = {Bj\"{o}rner, Anders and Ekedahl, Torsten},
     TITLE = {On the shape of {B}ruhat intervals},
   JOURNAL = {Ann. of Math. (2)},
  FJOURNAL = {Annals of Mathematics. Second Series},
    VOLUME = {170},
      YEAR = {2009},
    NUMBER = {2},
     PAGES = {799--817},
      ISSN = {0003-486X},
   MRCLASS = {05E99 (05A05 06A11 14F20 14F43 14M15 20F55)},
  MRNUMBER = {2552108},
MRREVIEWER = {Frank Sottile},
       DOI = {10.4007/annals.2009.170.799},
       URL = {https://doi.org/10.4007/annals.2009.170.799},
}

@incollection {BBD,
    AUTHOR = {Be\u{\i}linson, A. A. and Bernstein, J. and Deligne, P.},
     TITLE = {Faisceaux pervers},
 BOOKTITLE = {Analysis and topology on singular spaces, {I} ({L}uminy,
              1981)},
    SERIES = {Ast\'{e}risque},
    VOLUME = {100},
     PAGES = {5--171},
 PUBLISHER = {Soc. Math. France, Paris},
      YEAR = {1982},
   MRCLASS = {32C38},
  MRNUMBER = {751966},
MRREVIEWER = {Zoghman Mebkhout},
}

@article {BV,
    AUTHOR = {Braden, Tom and Vysogorets, Artem},
     TITLE = {Kazhdan-{L}usztig polynomials of matroids under deletion},
   JOURNAL = {Electron. J. Combin.},
  FJOURNAL = {Electronic Journal of Combinatorics},
    VOLUME = {27},
      YEAR = {2020},
    NUMBER = {1},
     PAGES = {Paper No. 1.17, 17},
   MRCLASS = {05B35},
  MRNUMBER = {4061059},
}

@article {BHMPW,
    AUTHOR = {Braden, Tom and Huh, June and Matherne, Jacob P. and
              Proudfoot, Nicholas and Wang, Botong},
     TITLE = {A semi-small decomposition of the {C}how ring of a matroid},
   JOURNAL = {Adv. Math.},
  FJOURNAL = {Advances in Mathematics},
    VOLUME = {409},
      YEAR = {2022},
     PAGES = {Paper No. 108646},
      ISSN = {0001-8708},
   MRCLASS = {Prelim},
  MRNUMBER = {4477425},
       DOI = {10.1016/j.aim.2022.108646},
       URL = {https://doi.org/10.1016/j.aim.2022.108646},
}

@misc {BHMPW-IH,
   AUTHOR = {Braden, Tom and Huh, June and Matherne, Jacob P. and
              Proudfoot, Nicholas and Wang, Botong},
	 TITLE = {The parity module of a matroid},
	  NOTE = {In preparation},
}

@article {SparsePaving,
    AUTHOR = {Lee, Kyungyong and Nasr, George D. and Radcliffe, Jamie},
     TITLE = {A combinatorial formula for {K}azhdan-{L}usztig polynomials of
              sparse paving matroids},
   JOURNAL = {Electron. J. Combin.},
  FJOURNAL = {Electronic Journal of Combinatorics},
    VOLUME = {28},
      YEAR = {2021},
    NUMBER = {4},
     PAGES = {Paper No. 4.44, 29},
   MRCLASS = {05B35},
  MRNUMBER = {4395926},
       DOI = {10.37236/10602},
       URL = {https://doi.org/10.37236/10602},
}

@article {eq-incidence,
    AUTHOR = {Proudfoot, Nicholas},
     TITLE = {Equivariant incidence algebras and equivariant
              {K}azhdan-{L}usztig-{S}tanley theory},
   JOURNAL = {Algebr. Comb.},
  FJOURNAL = {Algebraic Combinatorics},
    VOLUME = {4},
      YEAR = {2021},
    NUMBER = {4},
     PAGES = {675--681},
   MRCLASS = {05B35 (06A11)},
  MRNUMBER = {4311083},
       DOI = {10.5802/alco.174},
       URL = {https://doi.org/10.5802/alco.174},
}

@article {GX,
    AUTHOR = {Gao, Alice L. L. and Xie, Matthew H. Y.},
     TITLE = {The inverse {K}azhdan-{L}usztig polynomial of a matroid},
   JOURNAL = {J. Combin. Theory Ser. B},
  FJOURNAL = {Journal of Combinatorial Theory. Series B},
    VOLUME = {151},
      YEAR = {2021},
     PAGES = {375--392},
      ISSN = {0095-8956},
   MRCLASS = {05B35},
  MRNUMBER = {4294228},
MRREVIEWER = {Eva Ferrara Dentice},
       DOI = {10.1016/j.jctb.2021.07.004},
       URL = {https://doi.org/10.1016/j.jctb.2021.07.004},
}

@article {Nelson,
    AUTHOR = {Nelson, Peter},
     TITLE = {Almost all matroids are nonrepresentable},
   JOURNAL = {Bull. Lond. Math. Soc.},
  FJOURNAL = {Bulletin of the London Mathematical Society},
    VOLUME = {50},
      YEAR = {2018},
    NUMBER = {2},
     PAGES = {245--248},
      ISSN = {0024-6093},
   MRCLASS = {05B35},
  MRNUMBER = {3830117},
MRREVIEWER = {Eva Ferrara Dentice},
       DOI = {10.1112/blms.12141},
       URL = {https://doi.org/10.1112/blms.12141},
}

@article {AHK,
    AUTHOR = {Adiprasito, Karim and Huh, June and Katz, Eric},
     TITLE = {Hodge theory for combinatorial geometries},
   JOURNAL = {Ann. of Math. (2)},
  FJOURNAL = {Annals of Mathematics. Second Series},
    VOLUME = {188},
      YEAR = {2018},
    NUMBER = {2},
     PAGES = {381--452},
      ISSN = {0003-486X},
   MRCLASS = {05B35 (05E99 14C25 14T05)},
  MRNUMBER = {3862944},
       DOI = {10.4007/annals.2018.188.2.1},
       URL = {https://doi.org/10.4007/annals.2018.188.2.1},
}

@article {BK,
    AUTHOR = {Basterfield, J. G. and Kelly, L. M.},
     TITLE = {A characterization of sets of {$n$} points which determine
              {$n$} hyperplanes},
   JOURNAL = {Proc. Cambridge Philos. Soc.},
  FJOURNAL = {Proceedings of the Cambridge Philosophical Society},
    VOLUME = {64},
      YEAR = {1968},
     PAGES = {585--588},
      ISSN = {0008-1981},
   MRCLASS = {05.27 (50.00)},
  MRNUMBER = {233719},
MRREVIEWER = {H. Hanani},
       DOI = {10.1017/s0305004100043243},
       URL = {https://doi.org/10.1017/s0305004100043243},
}

@article {dBE,
    AUTHOR = {de Bruijn, N. G. and Erd\"{o}s, P.},
     TITLE = {On a combinatorial problem},
   JOURNAL = {Indagationes Math.},
  FJOURNAL = {Proceedings. Akadamie van Wetenschappen Amsterdam.
              North-Holland, Amsterdam},
    VOLUME = {10},
      YEAR = {1948},
     PAGES = {421--423},
      ISSN = {0370-0348},
   MRCLASS = {09.0X},
  MRNUMBER = {28289},
MRREVIEWER = {M. Hall, Jr.},
}

@article {DW1,
    AUTHOR = {Dowling, Thomas A. and Wilson, Richard M.},
     TITLE = {The slimmest geometric lattices},
   JOURNAL = {Trans. Amer. Math. Soc.},
  FJOURNAL = {Transactions of the American Mathematical Society},
    VOLUME = {196},
      YEAR = {1974},
     PAGES = {203--215},
      ISSN = {0002-9947},
   MRCLASS = {05B35},
  MRNUMBER = {345849},
MRREVIEWER = {Ann Miller},
       DOI = {10.2307/1997023},
       URL = {https://doi.org/10.2307/1997023},
}

@article {DW2,
    AUTHOR = {Dowling, Thomas A. and Wilson, Richard M.},
     TITLE = {Whitney number inequalities for geometric lattices},
   JOURNAL = {Proc. Amer. Math. Soc.},
  FJOURNAL = {Proceedings of the American Mathematical Society},
    VOLUME = {47},
      YEAR = {1975},
     PAGES = {504--512},
      ISSN = {0002-9939},
   MRCLASS = {05B35},
  MRNUMBER = {354422},
MRREVIEWER = {Ann Miller},
       DOI = {10.2307/2039773},
       URL = {https://doi.org/10.2307/2039773},
}

@article {Karu,
    AUTHOR = {Karu, Kalle},
     TITLE = {Hard {L}efschetz theorem for nonrational polytopes},
   JOURNAL = {Invent. Math.},
  FJOURNAL = {Inventiones Mathematicae},
    VOLUME = {157},
      YEAR = {2004},
    NUMBER = {2},
     PAGES = {419--447},
      ISSN = {0020-9910},
   MRCLASS = {52B20 (14F43 52B05)},
  MRNUMBER = {2076929},
MRREVIEWER = {Boris J. Kazarnovski\u{\i}},
       DOI = {10.1007/s00222-004-0358-3},
       URL = {https://doi.org/10.1007/s00222-004-0358-3},
}

@article {EW,
    AUTHOR = {Elias, Ben and Williamson, Geordie},
     TITLE = {The {H}odge theory of {S}oergel bimodules},
   JOURNAL = {Ann. of Math. (2)},
  FJOURNAL = {Annals of Mathematics. Second Series},
    VOLUME = {180},
      YEAR = {2014},
    NUMBER = {3},
     PAGES = {1089--1136},
      ISSN = {0003-486X},
   MRCLASS = {20C08 (20F55)},
  MRNUMBER = {3245013},
MRREVIEWER = {Chi Kin Mak},
       DOI = {10.4007/annals.2014.180.3.6},
       URL = {https://doi.org/10.4007/annals.2014.180.3.6},
}

@article {dCM02,
    AUTHOR = {de Cataldo, Mark Andrea A. and Migliorini, Luca},
     TITLE = {The hard {L}efschetz theorem and the topology of semismall
              maps},
   JOURNAL = {Ann. Sci. \'{E}cole Norm. Sup. (4)},
  FJOURNAL = {Annales Scientifiques de l'\'{E}cole Normale Sup\'{e}rieure. Quatri\`eme
              S\'{e}rie},
    VOLUME = {35},
      YEAR = {2002},
    NUMBER = {5},
     PAGES = {759--772},
}

@article {dCM05,
    AUTHOR = {de Cataldo, Mark Andrea A. and Migliorini, Luca},
     TITLE = {The {H}odge theory of algebraic maps},
   JOURNAL = {Ann. Sci. \'Ecole Norm. Sup. (4)},
  FJOURNAL = {Annales Scientifiques de l'\'Ecole Normale Sup\'erieure.
              Quatri\`eme S\'erie},
    VOLUME = {38},
      YEAR = {2005},
    NUMBER = {5},
     PAGES = {693--750},
      ISSN = {0012-9593},
   MRCLASS = {14D07 (32S60)},
  MRNUMBER = {2195257},
MRREVIEWER = {Christophe\ Mourougane},
       DOI = {10.1016/j.ansens.2005.07.001},
       URL = {https://doi.org/10.1016/j.ansens.2005.07.001},
}

@article {JMW,
    AUTHOR = {Juteau, Daniel and Mautner, Carl and Williamson, Geordie},
     TITLE = {Parity sheaves},
   JOURNAL = {J. Amer. Math. Soc.},
  FJOURNAL = {Journal of the American Mathematical Society},
    VOLUME = {27},
      YEAR = {2014},
    NUMBER = {4},
     PAGES = {1169--1212},
      ISSN = {0894-0347},
   MRCLASS = {14F05 (14M25 20G15)},
  MRNUMBER = {3230821},
MRREVIEWER = {Markus Banagl},
       DOI = {10.1090/S0894-0347-2014-00804-3},
       URL = {https://doi.org/10.1090/S0894-0347-2014-00804-3},
}

@article {EPW,
    AUTHOR = {Elias, Ben and Proudfoot, Nicholas and Wakefield, Max},
     TITLE = {The {K}azhdan-{L}usztig polynomial of a matroid},
   JOURNAL = {Adv. Math.},
  FJOURNAL = {Advances in Mathematics},
    VOLUME = {299},
      YEAR = {2016},
     PAGES = {36--70},
      ISSN = {0001-8708},
   MRCLASS = {05B35},
  MRNUMBER = {3519463},
MRREVIEWER = {Eva Ferrara Dentice},
       DOI = {10.1016/j.aim.2016.05.005},
       URL = {https://doi.org/10.1016/j.aim.2016.05.005},
}

@article {FY,
    AUTHOR = {Feichtner, Eva Maria and Yuzvinsky, Sergey},
     TITLE = {Chow rings of toric varieties defined by atomic lattices},
   JOURNAL = {Invent. Math.},
  FJOURNAL = {Inventiones Mathematicae},
    VOLUME = {155},
      YEAR = {2004},
    NUMBER = {3},
     PAGES = {515--536},
      ISSN = {0020-9910},
   MRCLASS = {14C15 (14M25)},
  MRNUMBER = {2038195},
MRREVIEWER = {G. K. Sankaran},
       DOI = {10.1007/s00222-003-0327-2},
       URL = {https://doi.org/10.1007/s00222-003-0327-2},
}

@article {GPY,
    AUTHOR = {Gedeon, Katie and Proudfoot, Nicholas and Young, Benjamin},
     TITLE = {The equivariant {K}azhdan-{L}usztig polynomial of a matroid},
   JOURNAL = {J. Combin. Theory Ser. A},
  FJOURNAL = {Journal of Combinatorial Theory. Series A},
    VOLUME = {150},
      YEAR = {2017},
     PAGES = {267--294},
      ISSN = {0097-3165},
   MRCLASS = {05B35 (05E05 05E18 20C30)},
  MRNUMBER = {3645577},
MRREVIEWER = {Eva Ferrara Dentice},
       DOI = {10.1016/j.jcta.2017.03.007},
       URL = {https://doi.org/10.1016/j.jcta.2017.03.007},
}

@article {Greene,
    AUTHOR = {Greene, Curtis},
     TITLE = {A rank inequality for finite geometric lattices},
   JOURNAL = {J. Combinatorial Theory},
  FJOURNAL = {Journal of Combinatorial Theory},
    VOLUME = {9},
      YEAR = {1970},
     PAGES = {357--364},
      ISSN = {0021-9800},
   MRCLASS = {06.40},
  MRNUMBER = {266824},
MRREVIEWER = {R. J. Bumcrot},
}

@article {Heron,
    AUTHOR = {Heron, A. P.},
     TITLE = {A property of the hyperplanes of a matroid and an extension of
              {D}ilworth's theorem},
   JOURNAL = {J. Math. Anal. Appl.},
  FJOURNAL = {Journal of Mathematical Analysis and Applications},
    VOLUME = {42},
      YEAR = {1973},
     PAGES = {119--131},
      ISSN = {0022-247X},
   MRCLASS = {05B35},
  MRNUMBER = {376403},
MRREVIEWER = {Stephen Tanny},
       DOI = {10.1016/0022-247X(73)90124-8},
       URL = {https://doi.org/10.1016/0022-247X(73)90124-8},
}

@article {HW,
    AUTHOR = {Huh, June and Wang, Botong},
     TITLE = {Enumeration of points, lines, planes, etc},
   JOURNAL = {Acta Math.},
  FJOURNAL = {Acta Mathematica},
    VOLUME = {218},
      YEAR = {2017},
    NUMBER = {2},
     PAGES = {297--317},
      ISSN = {0001-5962},
   MRCLASS = {05A15 (52C10)},
  MRNUMBER = {3733101},
       DOI = {10.4310/ACTA.2017.v218.n2.a2},
       URL = {https://doi.org/10.4310/ACTA.2017.v218.n2.a2},
}

@article {KungRadonI,
    AUTHOR = {Kung, Joseph P. S.},
     TITLE = {The {R}adon transforms of a combinatorial geometry. {I}},
   JOURNAL = {J. Combin. Theory Ser. A},
  FJOURNAL = {Journal of Combinatorial Theory. Series A},
    VOLUME = {26},
      YEAR = {1979},
    NUMBER = {2},
     PAGES = {97--102},
      ISSN = {0097-3165},
   MRCLASS = {05B35},
  MRNUMBER = {530281},
MRREVIEWER = {Ethan D. Bolker},
       DOI = {10.1016/0097-3165(79)90059-1},
       URL = {https://doi.org/10.1016/0097-3165(79)90059-1},
}

@article {KungRadonII,
    AUTHOR = {Kung, Joseph P. S.},
     TITLE = {The {R}adon transforms of a combinatorial geometry. {II}.
              {P}artition lattices},
   JOURNAL = {Adv. Math.},
  FJOURNAL = {Advances in Mathematics},
    VOLUME = {101},
      YEAR = {1993},
    NUMBER = {1},
     PAGES = {114--132},
      ISSN = {0001-8708},
   MRCLASS = {05B35 (05A18 06A07 44A12 51D20)},
  MRNUMBER = {1239455},
MRREVIEWER = {J. R. Griggs},
       DOI = {10.1006/aima.1993.1044},
       URL = {https://doi.org/10.1006/aima.1993.1044},
}

@incollection {KungLinesPlanes,
    AUTHOR = {Kung, Joseph P. S.},
     TITLE = {On the lines-planes inequality for matroids},
      NOTE = {In memory of Gian-Carlo Rota},
   JOURNAL = {J. Combin. Theory Ser. A},
  FJOURNAL = {Journal of Combinatorial Theory. Series A},
    VOLUME = {91},
      YEAR = {2000},
    NUMBER = {1-2},
     PAGES = {363--368},
      ISSN = {0097-3165},
   MRCLASS = {05B35 (05B25 06C10 51D20 51D25)},
  MRNUMBER = {1780029},
MRREVIEWER = {Walter Wenzel},
       DOI = {10.1006/jcta.2000.3101},
       URL = {https://doi.org/10.1006/jcta.2000.3101},
}

@incollection {KungRadon,
    AUTHOR = {Kung, Joseph P. S.},
     TITLE = {Radon transforms in combinatorics and lattice theory},
 BOOKTITLE = {Combinatorics and ordered sets ({A}rcata, {C}alif., 1985)},
    SERIES = {Contemp. Math.},
    VOLUME = {57},
     PAGES = {33--74},
 PUBLISHER = {Amer. Math. Soc., Providence, RI},
      YEAR = {1986},
   MRCLASS = {05B05 (05C60 06B99 44A15)},
  MRNUMBER = {856232},
MRREVIEWER = {L. M. Kelly},
       DOI = {10.1090/conm/057/856232},
       URL = {https://doi.org/10.1090/conm/057/856232},
}

@inproceedings {Mason,
    AUTHOR = {Mason, J. H.},
     TITLE = {Matroids: unimodal conjectures and {M}otzkin's theorem},
 BOOKTITLE = {Combinatorics ({P}roc. {C}onf. {C}ombinatorial {M}ath.,
              {M}ath. {I}nst., {O}xford, 1972)},
     PAGES = {207--220},
      YEAR = {1972},
   MRCLASS = {05B35},
  MRNUMBER = {0349445},
MRREVIEWER = {W. Dorfler},
}

@article {Motzkin,
    AUTHOR = {Motzkin, Theodore},
     TITLE = {The lines and planes connecting the points of a finite set},
   JOURNAL = {Trans. Amer. Math. Soc.},
  FJOURNAL = {Transactions of the American Mathematical Society},
    VOLUME = {70},
      YEAR = {1951},
     PAGES = {451--464},
      ISSN = {0002-9947},
   MRCLASS = {48.0X},
  MRNUMBER = {41447},
MRREVIEWER = {P. Erd\H{o}s},
       DOI = {10.2307/1990609},
       URL = {https://doi.org/10.2307/1990609},
}

@book {Oxley,
    AUTHOR = {Oxley, James},
     TITLE = {Matroid theory},
    SERIES = {Oxford Graduate Texts in Mathematics},
    VOLUME = {21},
   EDITION = {Second},
 PUBLISHER = {Oxford University Press, Oxford},
      YEAR = {2011},
     PAGES = {xiv+684},
      ISBN = {978-0-19-960339-8},
   MRCLASS = {05-01 (05B35 90C27)},
  MRNUMBER = {2849819},
MRREVIEWER = {Maruti M. Shikare},
       DOI = {10.1093/acprof:oso/9780198566946.001.0001},
       URL = {https://doi.org/10.1093/acprof:oso/9780198566946.001.0001},
}

@article {subdivisions,
    AUTHOR = {Stanley, Richard P.},
     TITLE = {Subdivisions and local {$h$}-vectors},
   JOURNAL = {J. Amer. Math. Soc.},
  FJOURNAL = {Journal of the American Mathematical Society},
    VOLUME = {5},
      YEAR = {1992},
    NUMBER = {4},
     PAGES = {805--851},
      ISSN = {0894-0347},
   MRCLASS = {52B20 (05E99 06A07 13D40 55U10)},
  MRNUMBER = {1157293},
MRREVIEWER = {Margaret M. Bayer},
       DOI = {10.2307/2152711},
       URL = {https://doi.org/10.2307/2152711},
}

@article {KLS,
    AUTHOR = {Proudfoot, Nicholas},
     TITLE = {The algebraic geometry of {K}azhdan-{L}usztig-{S}tanley
              polynomials},
   JOURNAL = {EMS Surv. Math. Sci.},
  FJOURNAL = {EMS Surveys in Mathematical Sciences},
    VOLUME = {5},
      YEAR = {2018},
    NUMBER = {1-2},
     PAGES = {99--127},
      ISSN = {2308-2151},
   MRCLASS = {14N10 (06A11 14C17 14N15)},
  MRNUMBER = {3880222},
MRREVIEWER = {Gergely B\'{e}rczi},
       DOI = {10.4171/EMSS/28},
       URL = {https://doi.org/10.4171/EMSS/28},
}

@article {PXY,
    AUTHOR = {Proudfoot, Nicholas and Xu, Yuan and Young, Ben},
     TITLE = {The {$Z$}-polynomial of a matroid},
   JOURNAL = {Electron. J. Combin.},
  FJOURNAL = {Electronic Journal of Combinatorics},
    VOLUME = {25},
      YEAR = {2018},
    NUMBER = {1},
     PAGES = {Paper 1.26, 21},
   MRCLASS = {05B35 (11B83 55N33)},
  MRNUMBER = {3785005},
MRREVIEWER = {Daniel Matei},
}

@book {StanleyAC,
    AUTHOR = {Stanley, Richard P.},
     TITLE = {Algebraic combinatorics},
    SERIES = {Undergraduate Texts in Mathematics},
      NOTE = {Walks, trees, tableaux, and more,
              Second edition of [ MR3097651]},
 PUBLISHER = {Springer, Cham},
      YEAR = {2018},
     PAGES = {xvi+263},
      ISBN = {978-3-319-77172-4; 978-3-319-77173-1},
   MRCLASS = {05-01 (05Axx 05Exx)},
  MRNUMBER = {3823204},
       DOI = {10.1007/978-3-319-77173-1},
       URL = {https://doi.org/10.1007/978-3-319-77173-1},
}

@book {Welsh,
    AUTHOR = {Welsh, D. J. A.},
     TITLE = {Matroid theory},
      NOTE = {L. M. S. Monographs, No. 8},
 PUBLISHER = {Academic Press [Harcourt Brace Jovanovich, Publishers],
              London-New York},
      YEAR = {1976},
     PAGES = {xi+433},
   MRCLASS = {05B35},
  MRNUMBER = {0427112},
MRREVIEWER = {W. T. Tutte},
}

@article {Orlik-Solomon,
    AUTHOR = {Orlik, Peter and Solomon, Louis},
     TITLE = {Combinatorics and topology of complements of hyperplanes},
   JOURNAL = {Invent. Math.},
  FJOURNAL = {Inventiones Mathematicae},
    VOLUME = {56},
      YEAR = {1980},
    NUMBER = {2},
     PAGES = {167--189},
      ISSN = {0020-9910},
   MRCLASS = {32C40 (05A99 06B99)},
  MRNUMBER = {558866},
MRREVIEWER = {M. Sebastiani},
       DOI = {10.1007/BF01392549},
       URL = {https://doi-org.uoregon.idm.oclc.org/10.1007/BF01392549},
}

@book {Orlik-Terao,
    AUTHOR = {Orlik, Peter and Terao, Hiroaki},
     TITLE = {Arrangements of hyperplanes},
    SERIES = {Grundlehren der Mathematischen Wissenschaften [Fundamental
              Principles of Mathematical Sciences]},
    VOLUME = {300},
 PUBLISHER = {Springer-Verlag, Berlin},
      YEAR = {1992},
     PAGES = {xviii+325},
      ISBN = {3-540-55259-6},
   MRCLASS = {52B30 (14F35 20F36 20F55 32S25 57N65)},
  MRNUMBER = {1217488},
MRREVIEWER = {Michel Yves Jambu},
       DOI = {10.1007/978-3-662-02772-1},
       URL = {https://doi-org.ezproxy.library.wisc.edu/10.1007/978-3-662-02772-1},
}

@article {dCM,
    AUTHOR = {de Cataldo, Mark Andrea A. and Migliorini, Luca},
     TITLE = {The decomposition theorem, perverse sheaves and the topology
              of algebraic maps},
   JOURNAL = {Bull. Amer. Math. Soc. (N.S.)},
  FJOURNAL = {American Mathematical Society. Bulletin. New Series},
    VOLUME = {46},
      YEAR = {2009},
    NUMBER = {4},
     PAGES = {535--633},
      ISSN = {0273-0979},
   MRCLASS = {14C30 (14F05 14F43 18E30)},
  MRNUMBER = {2525735},
MRREVIEWER = {Adrian Langer},
       DOI = {10.1090/S0273-0979-09-01260-9},
       URL = {https://doi.org/10.1090/S0273-0979-09-01260-9},
}
\bibliographystyle{amsalpha}

\end{document}